\documentclass[12pt,letterpaper]{amsart}
\usepackage{amsmath}
\usepackage{amstext}
\usepackage{amssymb}
\usepackage{amsfonts}
\usepackage{enumerate}
\usepackage{amsthm}
\usepackage{mathrsfs}
\usepackage[all]{xy}
\usepackage{hyperref}
\usepackage{xcolor}
\usepackage{wasysym}

\AtBeginDocument{%
   \def\MR#1{}
}

\numberwithin{equation}{section}

\newtheorem{theorem}{Theorem}[section]
\newtheorem{lemma}[theorem]{Lemma}
\newtheorem{proposition}[theorem]{Proposition}
\newtheorem{corollary}[theorem]{Corollary}

\newtheorem{example}[theorem]{Example}
\theoremstyle{definition}
\newtheorem{remark}[theorem]{Remark}
\newtheorem{definition}[theorem]{Definition}

\DeclareMathOperator{\cb}{cb}
\DeclareMathOperator{\mb}{mb}
\DeclareMathOperator{\jcb}{jcb}
\DeclareMathOperator{\CB}{CB}
\DeclareMathOperator{\MB}{MB}
\DeclareMathOperator{\JCB}{JCB}
\DeclareMathOperator{\tr}{tr}

\DeclareMathOperator{\Max}{Max}
\DeclareMathOperator{\proj}{proj}
\DeclareMathOperator{\im}{im}
\DeclareMathOperator{\dual}{dual}
\DeclareMathOperator{\sur}{sur}
\DeclareMathOperator{\inj}{inj}

\DeclareMathOperator{\nuc}{nuc}
\DeclareMathOperator{\eh}{eh}
\DeclareMathOperator{\nh}{\sigma h}

\DeclareMathOperator{\ONORM}{ONORM}
\DeclareMathOperator{\OFIN}{OFIN}
\DeclareMathOperator{\OCOFIN}{OCOFIN}

\DeclareMathOperator{\OLOC}{OLOC}
\DeclareMathOperator{\COFIN}{COFIN}
\DeclareMathOperator{\OBAN}{OBAN}

\newcommand{\n}[1]{ \left\|#1\right\| }

\newcommand{\N}{{\mathbb{N}}}

\newcommand{\C}{{\mathbb{C}}}

\newcommand{\bl}{B^\lambda}

\newcommand{\pair}[2]{{\langle #1, #2 \rangle}}
\newcommand{\mpair}[2]{{\langle\langle #1, #2 \rangle\rangle}}

\newcommand{\eps}{{\varepsilon}}
\newcommand{\on}{{\quad\text{on}\quad}}

\newcommand\restr[2]{{% we make the whole thing an ordinary symbol
  \left.\kern-\nulldelimiterspace % automatically resize the bar with \right
  #1 % the function
  \right|_{#2} % this is the delimiter
  }}

\begin{document}

\title{\textbf{Operator space tensor norms}}
\author[J.A. Ch\'avez-Dom\'inguez]{Javier Alejandro Ch\'avez-Dom\'inguez}
\address{Department of Mathematics, University of Oklahoma, Norman, OK 73019-3103,
USA} \email{jachavezd@ou.edu}

\author[V. Dimant]{Ver\'onica Dimant}
\address{Departamento de Matem\'{a}tica y Ciencias, Universidad de San
Andr\'{e}s, Vito Dumas 284, (B1644BID) Victoria, Buenos Aires,
Argentina and CONICET} \email{vero@udesa.edu.ar}

\author[D. Galicer] {Daniel Galicer}
\address{Departamento de Matem\'{a}tica, Facultad de Ciencias Exactas y Naturales, Universidad de Buenos Aires, (1428) Buenos Aires,
Argentina and CONICET} \email{dgalicer@dm.uba.ar}
\thanks{The first-named author was partially supported by NSF grant DMS-1900985. The second author was partially supported by CONICET PIP 11220200101609CO and ANPCyT PICT 2018-04104. The third author was partially supported by CONICET-PIP 11220200102366CO
 and ANPCyT PICT 2018-4250.}

\makeatletter
\@namedef{subjclassname@2020}{\textup{2020} Mathematics Subject Classification}
\makeatother

\subjclass[2020]{Primary: 47L25, 46M05, 46B28. Secondary: 46L06.  }

\keywords{Operator spaces, tensor norms, mapping ideals}

\begin{abstract}
The use of a tensor product perspective has enriched functional analysis and other important areas of mathematics and physics. The context of operator spaces is clearly no exception.	
The aim of this manuscript 
is to kick off the development of a \emph{systematic} theory of tensor products and tensor norms for operator spaces and its interplay with their associated  mapping ideals. 
Based on the theory of tensor products in Banach spaces, we provide the corresponding natural definitions in the operator space framework. 
The theory is not a mere translation of what is known in the classical setting and new insights, techniques, ideas or hypotheses are required in many cases.  As a consequence, notable differences in the theory appear when compared to the classical one.
\end{abstract}

\maketitle

%------------------------------------------------------
%------------------------------------------------------
%------------------------------------------------------

\tableofcontents
\markboth{}{}

\section{Introduction}
In the thirties von Neumann in collaboration with Murray  began a ``mathematical quantization program'', aiming to formulate an operator version of integration theory.
They were successful, with the key idea being to replace function spaces by $*$-algebras of bounded operators on Hilbert spaces.
This was the seed that started a long road of ``noncommutative analysis'' which lead to what is now known as the theory of operator spaces.
Along this road and prior to the formalization of this theory, many concepts which are now typical in this framework (e.g., completely bounded mappings) were introduced in the context of the $C^*$-algebras. 

An operator space is a Banach space $\textbf X$ with an extra ``matricial norm structure'':
in addition to the usual norm on $\textbf X$, we have norms on all the spaces $M_n(\textbf X)$ of $n \times n$ matrices with entries from $\textbf X$ (often called the matrix levels), where these norms must satisfy certain consistency requirements.
The connection with operator algebras comes from the fact that when $\textbf X$ is a subspace of an algebra $\mathcal{B}(H)$ of bounded linear operators on a Hilbert space $H$, there is a canonical norm on $M_n(\textbf X)$ coming from the identification between $M_n(\mathcal{B}(H))$ and $\mathcal{B}(H^n)$.
The natural morphisms in this category are no longer just bounded maps, but the \emph{completely bounded ones}: they are required to be uniformly bounded on all the matrix levels.
Operator spaces are thus a quantized or noncommutative version of Banach spaces, giving rise to a theory that not only is mathematically attractive but it is also naturally well-positioned to have applications to quantum physics.
The systematic study of operator spaces begun with Ruan's work, and
was deepened mainly by Effros and Ruan, Blecher and Paulsen, Pisier and Junge (see the monographs \cite{Junge-Habilitationschrift,Pisier-Asterisque-98,Blecher-LeMerdy-book,Effros-Ruan-book,Pisier-Operator-Space-Theory} and the references therein).
Operator spaces have often provided an appropriate framework to develop many areas of non-commutative analysis.

It is quite natural to investigate to what extent the classical theory of Banach spaces can be translated to the  noncommutative context, and this has been extensively studied throughout the years. Though some properties do carry over, many do not and these differences are one of the reasons making the new theory so interesting and rich. And even for those properties that do admit a generalization, the proofs are often different and require the development of new tools.
Along these lines, the main goal of this monograph is to take the first step towards the systematic development of a complete theory of tensor products and tensor norms in the operator space setting and its connections with the parallel theory of their associated mapping ideals.

It should be noted that both ideals of operators and tensor norms have inspired important developments in the operator space setting.
Noncommutative versions of nuclear, integral, summing, and other ideals of operators have played significant roles in, e.g.,
\cite{Effros-Ruan-book, Pisier-Asterisque-98,Junge-Habilitationschrift,Effros-Junge-Ruan,Junge-Parcet-Maurey-factorization},
whereas various specific tensor norms for operator spaces (most notably the Haagerup one) have also played a significant role (see \cite[Sec. 1.5]{Blecher-LeMerdy-book} for a good summary).

Blecher and Paulsen \cite{Blecher-Paulsen-Tensors} have showed that part of elementary and basic theory of tensor norms of Banach spaces carries over to operator spaces, initiating a ``tensor norm program'' for operator spaces further developed in \cite{blecher1991tensor}.
However, many of the abstract tools of operators ideals and tensor norms in the sense of \cite{Defant-Floret} do not appear to have been treated yet.
Moreover, while it is clear to specialists that there are obstacles making it impossible to fully translate the classical tensor product theory into the operator space setting, 
it has not been highlighted exactly which portions do translate and which ones do not.
Hence, we will focus on seeing to what extent the classical theory can be transferred. In many situations, this is not the case and additional assumptions about the structure of the operator spaces or the tensor norms involved are really necessary.
A good example is local reflexivity: since all Banach spaces have this property but not all operator spaces do, it is well known that operator space versions of classical results will often need local reflexivity as an additional assumption.
However, when it comes to tensor norms it turns out that in some cases we can still get operator space versions of some classical results without requiring local reflexivity, as long as the tensor norm satisfies an extra condition.

Although the theory of operator spaces has its own peculiarities, the classical theory will be our reference and roadmap to obtain and develop non-commutative versions of the theory of tensor norms.
Therefore, let us now recall briefly summarize the history of the classical theory's development .
Tensor products burst into the world of functional analysis in the late thirties with the works of Murray, von Neumann and Schatten,  but it was Grothendieck who first truly revealed the richness of properties for tensor products.
His famous article  \emph{``R\'{e}sum\'{e} de la th\'{e}orie m\'{e}trique  des produits tensoriels
topologiques''} \cite{grothendieck1956resume} established the basis of what we know today as `local theory' --- that is, the study of Banach spaces in terms of their finite-dimensional subspaces --- and showed the relevance of the use of tensor products in the theory of normed spaces.  %(and, of course, later on its interplay with the theory of operator ideals). 
All this helped him to establish a very fruitful theory of duality. 

While nowadays the \emph{``R\'{e}sum\'{e}''} is considered foundational work in the area, this masterpiece remained mostly unnoticed for many years. 
It was not until the end of the sixties when it started to be better valued by the scientific community.  Lindenstrauss and Pe{\l}czy\'nski \cite{lindenstrauss1968absolutely}
presented some important connections and applications to the theory of absolutely $p$-summing operators, translating results written in terms of certain tensor norms by Grothendieck into properties of operator ideals. 
On the other hand, almost simultaneously, a general theory of operator ideals on the class of Banach spaces was developed by  Pietsch (and his school in Jena) but without the use --- and the language --- of tensor norms. Pietsch's book \emph{``Operator Ideals''} \cite{pietsch1978operator}, which appeared at the end of the seventies, led this theory to  become one of the central themes of study in Banach space theory.

At the time functional analysts generally chose the language of operator ideals over that of tensor products, so the former received more attention.
The theory of tensor products in Banach spaces  became stronger and more ubiquitous during the eighties with the seminal work of Pisier \cite{pisier1983counterexamples,pisier1986factorization}, which showed the usefulness of these tools.

With time, the theory of tensor norms became an interesting field in its own right.
Defant and Floret, in their famous monograph \emph{``Tensor Norms and Operator Ideals''} \cite{Defant-Floret}, made it clear that the theory of tensor products and the  theory of operator ideals are two sides of the same coin. Their work had a tremendous legacy: nowadays authors use indistinctly both languages. This same approach has also been developed in various other books on tensor products of Banach spaces \cite{Ryan,diestel2008metric}.

As stated in \cite{Defant-Floret}, ``both theories, the theory of tensor norms and of norm operator ideals, are more easily understood and also richer if one works with both simultaneously''.
With this in mind,  certain problems are more easily handled using either the \emph{categorical perspective} due to Pietsch, or Grothendieck's \emph{cycle of ideas on tensor products}.
With the present work, it is our intent to bring the benefits of the aforementioned parallel perspectives to the operator space setting.
As mentioned above both the approaches of ideals and tensor norms have been repeatedly used in the operator space literature, but there has not been a methodical exploration of the connections between the two.
It is worth mentioning that this work does not pretend to be an extensive monograph like \cite{Defant-Floret}. Our purpose is much more modest: initiate a program that, over time, 
can be enriched with the contributions of the operator space community. Thus, numerous topics covered in \cite{Defant-Floret} have been left out.
Some expected results that do not appear in the literature can be found in this monograph, and also some new insights that do not have a counterpart or differ from the theory of Banach spaces.

\bigskip
The manuscript is organized as follows. Section \ref{Basic definitions} is devoted to background material. We set some notation, give basic definitions that will be needed throughout the manuscript (in particular what we mean by an operator space tensor norm, or o.s. tensor norm), and present some known examples of operator space tensor norms.
Also, we introduce a new family of Haagerup-style tensor norms which are examples of
 $\lambda$-tensor products (developed in \cite{Defant-Wiesner, Wiesner}).
 Additionally, we emphasize the main relevant differences between the operator space and the Banach space frameworks which have an impact in the study of tensor norms in both settings.
 
In Section \ref{hulls} we deal with the notions of finitely and cofinitely generated o.s. tensor norms. The former essentially means that the tensor norm can be approximated using finite-dimensional subspaces, while the latter means we can approximate it using quotients over subspaces of cofinite dimension.
We also define the finite and cofinite hull of a given o.s. tensor norm.

In Section \ref{fivebasic} we present the “Five Basic Lemmas” (see Section 13 in \cite{Defant-Floret}) for the operator space setting. These are the \emph{Approximation Lemma},
the \emph{Extension Lemma}, the \emph{Embedding Lemma}, the \emph{Density Lemma} and the $\mathcal L_p$-\emph{Technique Lemma}. 
To prove some of these results, unlike the classical case, we additionally need certain strong hypotheses such as local reflexivity.
However, in the case of $\lambda$-o.s. tensor norms this extra assumption is often not necessary.

In Section \ref{dual section} we deal with dual tensor norms and develop an operator space version of the \emph{Duality Theorem}. We also introduce right and left-accessible tensor norms and prove that any $\lambda$-o.s. tensor norm is accessible. Further, in order to prove 
some duality relations we define the notion of local accessibility, a weaker version of accessibility which involves locally reflexive operator spaces. 

In Section \ref{complete bap} we discuss completely metric and completely bounded approximation properties and relate these properties with the weak density of a multiple of $ B_{E' \otimes_{\min} F}$ in $B_{\CB(E,F)}$. Also, we consider a weak completely bounded approximation property which provides an accurate characterization in terms of  the duality between the $\proj$ and  the $\min$ o.s. tensor norms.

In Section \ref{mapping ideals}  we discuss the notion of mapping ideal (defined in \cite[Sec. 12.2]{Effros-Ruan-book}) and  present some examples already studied in the literature.  To maintain the core of the Banach space concept and allow the natural relationship  with tensor norms, we have to slightly strengthen the definition of mapping ideal from \cite[Sec. 12.2]{Effros-Ruan-book} and add an extra condition. Also, we introduce some typical procedures:  injective/surjective hulls of a given mapping ideal, the dual mapping ideal, and present some properties. 

Sections \ref{represent sec}  and \ref{minimal sec} deal with maximal and minimal mapping ideals and their \emph{Representation theorems}. 
As the names indicate, maximal and minimal ideals associated to a given tensor norms are the largest and smallest with respect to the inclusion which coincides over finite dimensional spaces. 
The representation theorem for maximal mapping ideals shows how the ideal can be seen as a dual of a tensor product. Again, the hypothesis of local reflexivity is needed here in order to obtain a full characterization. As a consequence, we prove an \textsl{Embedding Theorem} (which gives a natural completely isometric inclusion from a tensor product into the mapping ideal) and show how dual and adjoint ideals behave in terms of their associated tensor norms. 
For minimal ideals, the representation theorem provides a natural completely quotient mapping from a tensor product onto the mapping ideal. 
Therefore, in a sense, both  theorems (for maximal and minimal ideals) relate a mapping ideal with a tensor product. 

In Section \ref{capsulas inyectivas}, we give the definitions of right and left completely injective and completely projective hulls of an o.s. tensor
norm. We examine some of their interesting properties and see when they are finitely generated. Also, we define (both sides) completely injective and completely projective hulls showing that in both cases left and right procedures commute.  

We establish in Section \ref{mapping accessibility sec} that, unlike what happens in the Banach space setting,  accessibility is a required property of a tensor norm in order to obtain the usual relations between duality and injective/projective hulls. Accessibility is also needed to associate a left (or right) injective hull of a norm to the surjective (respectively, injective) hull of a mapping ideal.
Based on the classical theory, we also provide a definition of accessibility for mapping ideals. As expected, right-accessible finitely-generated tensor norms are associated with right-accessible mapping ideals but, surprisingly, the left version of this result does not hold. We see that a left-accessible o.s. tensor norm might have an associated mapping ideal which is not left-accessible (although we show in Proposition \ref{prop-equivalences-accessibility} that, in this case, the mapping ideal does have a weaker version of left-accessibility). 

Section \ref{normasnaturales sec} deals with the structure of \emph{natural} norms, in the sense of Grothendieck. 
We prove that some results from the Banach space theory about equivalences or inequalities between injective and projective hulls remain valid in this new context but the whole picture is undoubtedly substantially different. In the classical theory Grothendieck's inequality (which for instance implies that there is not possible for a tensor norm to be both injective and projective) has a strong impact in the description of natural tensor norms. The existence of an injective and projective Haagerup tensor norm shows us that the situation in the operator space setting could not be at all a translation from the classical case.
We use the term \textsl{$\proj$ family} for the set of norms obtained from the norm  $\proj$ after taking left or right injective/projective hulls finitely many times. Analogously, we define the \textsl{$\min$ family}. Unlike the classical setting, these families are not intertwined: each member of the $\proj$ family dominates any member of the $\min$ family. We present incomplete pictures of both of these families, proving some dominations and non-equivalences and leaving a large list of open questions about them.
On the other hand, we completely describe the list of all natural norms that come from applying to $\min$ or $\proj$ two-sided hull operations (injective or projective). Precisely, it consists of six o.s. tensor norms. Again, this differs from the Banach space case where there are actually four (see Section \ref{twosided}).

\section{Basic definitions} \label{Basic definitions}

We only assume familiarity with the basic theory of operator spaces, although we recall some elementary definitions. The books \cite{Pisier-Operator-Space-Theory} and \cite{Effros-Ruan-book} are excellent references on this topic.
Our notation follows closely that from \cite{Pisier-Asterisque-98, Pisier-Operator-Space-Theory}.
For technical reasons, unlike the general literature in the subject, we will not assume that our operator spaces are complete. Thus, to avoid confusion, we  choose to call them  
 \emph{normed operator spaces} (when completeness is not assumed) and   \emph{Banach operator spaces} (when completeness is required).
The letters $E$, $F$ and $G$ will always denote normed operator spaces, that is, normed vector spaces with an additional structure at the matricial level that we now explain.

As usual, $M_n(E)$ will stand for the set of $n\times n$-matrices of elements in $E$ and for $a\in M_{n, m}$ ($n\times m$-matrices of complex scalars), $\|a\|$ denotes its norm as an operator from $\ell_2^m$ to $\ell_2^n$.

Given $x=(x_{i,j})\in M_n(E)$ and $y=(y_{k,l})\in M_m(E)$,  we consider $x\oplus y\in M_{n+m}(E)$ to be the matrix
$
x\oplus y=\left(
            \begin{array}{cc}
              (x_{i,j}) & 0 \\
              0 & (y_{k,l}) \\
            \end{array}
          \right).
$ 

\begin{definition}
$E$ is a  \textit{normed operator space} if, for each $n$, there is  a norm $\|\cdot\|_n$ on $M_n(E)$ satisfying:
\begin{enumerate}
\item[\textbf{M1}] $\|x\oplus y\|_{M_{n+m}(E)}=\max\big\{\|x\|_{M_n(E)},\|y\|_{M_m(E)}\big\}$, for all $x\in M_n(E)$ and $y\in M_m(E)$. 
\item[\textbf{M2}] $\|a x b\|_{M_n(E)}\leq \|a\|\cdot \|x\|_{M_m(E)} \cdot \|b\|$, for all $x\in M_m(E)$, $a\in M_{n, m}$ and $b\in M_{m, n}$.
\end{enumerate}
\end{definition}

If $H$ is a Hilbert space and $E\subset \mathcal B(H)$ is a subspace, there is a natural norm in $M_n(E)$ given through the identification
$M_n(\mathcal B(H))=\mathcal B(H^n)$, which endows
$E$ with a normed operator space structure.
On  the other hand, by Ruan's representation theorem for any  operator space $E$ there exist a Hilbert space $H$ and an inclusion  $E\subset\mathcal B(H)$ which is an isometry for all the matrix levels.

Every  linear mapping $T:E\to F$ induces, for each $n\in\mathbb{N}$,  a linear mapping  $T_n: M_n(E)\to M_n(F)$ (called the \emph{amplification of $T$}) given by
$$
T_n(x)=\left(T(x_{ij})\right), \textrm{ for all }x=(x_{ij})\in M_n(E).
$$

The \textit{completely bounded norm} of $T$ is defined by
$$
\|T\|_{\cb}=\sup_{n\in\mathbb{N}}\|T_n\|=\sup_{n\in\mathbb{N}} \big\{\|T_n(x)\|_{M_n(F)}:\, \|x\|_{M_n(E)}\le 1\big\}.
$$

We say that $T$ is \textit{completely bounded} if $\|T\|_{\cb}$ is finite and we denote by $\CB(E,F)$ the space of completely bounded mappings from $E$ to $F$. This is an operator space with the structure given by the identification $M_n(\CB(E,F))=\CB(E,M_n(F))$.

We say that $T$ is a \textit{complete isometry} if each $T_n:M_n(E)\to M_n(F)$ is an isometry
and we write $E = F$ to indicate that $E$ and $F$ are completely isometrically isomorphic.

The operator space structure of a dual space $E'$ comes from the identification $M_n(E')=\CB(E,M_n)$. The norm of a matrix in $M_n(E')$ can also be computed through the \textit{matrix pairing} $\langle\langle \cdot, \cdot\rangle\rangle: M_n(E)\times M_n(E')\to M_{n^2}$,  where $\langle\langle (x_{ij}),(\varphi_{kl})\rangle\rangle =(\varphi_{kl}(x_{ij}))$. Precisely,
$$
\|(\varphi_{kl})\|_{M_n(E')}=\sup\big\{\|\langle\langle (x_{ij}),(\varphi_{kl})\rangle\rangle\|_{M_{n^2}}: \|(x_{ij})\|_{M_n(E)}\le 1\big\}.
$$

Given a bilinear mapping $\phi:E\times F\to G$, we denote by $\phi_n:M_n(E)\times M_n(F)\to M_{n^2}(G)$ the mapping defined, for each $n\in \mathbb{N}$, as follows:
$$
\phi_n(v,w)=\left(\phi(v_{ij},w_{kl})\right), \textrm{ for all }v=(v_{ij})\in M_n(E), w=(w_{kl})\in M_n(F).
$$

We say that $\phi$ is \textit{jointly completely bounded} when
$$
\|\phi\|_{\jcb}\equiv\sup_{n\in\mathbb{N}}\|\phi_n\|<\infty.
$$
 The space $\JCB(E\times F,G)$ of all  jointly completely bounded bilinear mappings from $E\times F$ to $G$ has an operator space structure given by the identification
$$M_n\left(\JCB(E\times F,G)\right)= \JCB\left(E\times F,M_n(G)\right).$$

Another natural associate to $\phi$ is the bilinear mapping  $\phi_{(n)}:M_n(E)\times M_n(F)\to M_n(G)$ related with the matrix product and given by
$$
\phi_{(n)}(v,w)=\left(\sum_{k=1}^n\phi(v_{ik},w_{kl})\right), \textrm{ for all }v=(v_{ij})\in M_n(E), w=(w_{kl})\in M_n(F).
$$

The bilinear mapping $\phi$ is \textit{multiplicatively bounded} if
$$
\|\phi\|_{\mb}=\sup_{n\in\mathbb{N}}\|\phi_{(n)}\|<\infty.
$$

Again, for the space $\MB(E\times F,G)$ of  all  multiplicatively bounded bilinear mappings from $E\times F$ to $G$ an operator space structure is provided via the identification
$$
M_n\left(\MB(E\times F,G)\right) = \MB\left(E\times F,M_n(G)\right).
$$

\subsection{Definition of o.s. tensor norm}
We begin by recalling the definitions of the three most common tensor norms in operator space theory: the minimal, projective, and Haagerup tensor norms. 

The \emph{operator space injective tensor norm} (also known as \emph{minimal tensor norm}) of $u \in M_n(E \otimes F)$ is defined as 
\begin{equation}\label{eqn-injective-tensor-norm}
    \min(u) = \sup \{ \|(f \otimes g)_n(u)\| : f\in M_p(E'), g \in M_q(F'), \Vert f \Vert \leq 1, \Vert g \Vert \leq 1 \}
\end{equation}

The \emph{operator space projective tensor norm} of $u \in M_n(E \otimes F)$ is defined as 
\begin{equation}\label{eqn-projective-tensor-norm}
    \proj(u) = \inf \{ \Vert a \Vert \cdot \Vert x\Vert \cdot \Vert y\Vert \cdot \Vert b\Vert \},
\end{equation}
where the infimum is taken over all the representations of $u= a(x \otimes y)b$ with $x \in M_p(E)$, $y \in M_q(F)$, $a \in M_{n,p \cdot q}$, $b \in M_{p\cdot q,n}$ for some $p,q\in \mathbb{N}.$

The \emph{Haagerup tensor norm} of $u \in M_n(E \otimes F)$ is defined as 
\begin{equation}\label{eqn-Haagerup-tensor}
    h(u) = \inf \{ \Vert x \Vert \cdot \Vert y\Vert : u = x \odot y, x \in M_{n,r}(E), y \in M_{r,n}(F), r\in \mathbb{N} \},
\end{equation}
where $\odot$ denotes the standard  matrix  product. Precisely,
$$
x \odot y = \left(\sum_{k=1}^r x_{ik} \otimes y_{kj}\right)_{i,j}.
$$

The three aforementioned well-known tensor norms can be framed into the following concept.

An \emph{operator space cross-norm} $\alpha$, on the class $\ONORM$ of all  normed operator spaces,
is an assignment of a normed operator space $E \otimes_\alpha F$ to each pair $(E,F)$ of normed operator spaces, in such a way that $E \otimes_\alpha F$ is the algebraic tensor product $E \otimes F$ together with a matricial norm structure on $E \otimes F$, that we write as $\alpha_n$ or $\n{\cdot}_{\alpha_n}$, and such that

\begin{equation}  \label{norma producto matrices}
    \alpha_{nm}(x \otimes y) = \n{x}_{M_n(E)} \cdot \n{y}_{M_m(F)} \text{ for every } x \in M_n(E), y \in M_m(F).
\end{equation}

This implies that the identity map $E \otimes_{\proj} F \to E \otimes_{\alpha} F$ is completely contractive \cite[Thm. 5.5]{Blecher-Paulsen-Tensors}.
If in addition the identity map $E \otimes_\alpha F \to E \otimes_{\min} F$ is also completely contractive, we say that $\alpha$ is  \emph{reasonable}. Note that if both mappings $E \otimes_{\proj} F \to E \otimes_{\alpha} F$ and $E \otimes_\alpha F \to E \otimes_{\min} F$ are completely contractive then we obviously obtain \eqref{norma producto matrices}.

Moreover, an operator space cross-norm $\alpha$ is called \emph{uniform} if additionally
it satisfies the \emph{complete metric mapping property}:
if $S \in \CB(E_1,E_2)$ and $T \in \CB(F_1,F_2)$, then
$$
\left\| S \otimes T \colon  E_1\otimes_\alpha F_1 \to E_2\otimes_\alpha F_2 \right\|_{\cb} \le \n{S}_{\cb}\n{T}_{\cb}.
$$

An operator space cross-norm on $\ONORM$ that is both reasonable and uniform will be called an \emph{operator space tensor norm}, or \emph{o.s. tensor norm} for simplicity.

\begin{definition}
An operator tensor norm $\alpha$ is an assignment  to each  $E,F \in \ONORM$  of a matricial norm in $E \otimes F$  such that the following two conditions are satisfied:
\begin{enumerate}
    \item $\alpha$ is reasonable: The mappings $E\otimes_{\proj} F\to E \otimes_\alpha F \to E \otimes_{\min} F$ are completely contractive.
    \item $\alpha$ is uniform: If $S \in \CB(E_1,E_2)$ and $T \in \CB(F_1,F_2)$, then
$$
\left\| S \otimes T \colon  E_1\otimes_\alpha F_1 \to E_2\otimes_\alpha F_2 \right\|_{\cb} \le \n{S}_{\cb}\n{T}_{\cb}.
$$
\end{enumerate}
\end{definition}

A degree of caution is required when consulting different works dealing with operator space tensor products, since the term ``tensor norm'' is not always taken to have the exact same meaning.
We are using the same definition as  \cite{Effros-Ruan-book,dimant2015biduals},
which is slightly different than the one from \cite{Blecher-Paulsen-Tensors,blecher1991tensor};
we refer to \cite[Sec. 6.1]{Wittstock} for  an explicit comparison.
It is clear that analogous definitions for operator space tensor norms can also be made on a subclass of operator spaces,
for example the class $\OFIN$ of all finite-dimensional operator spaces,
or the class of dual operator spaces.

\subsection{Basic properties of tensor norms}
An o.s. tensor norm $\alpha$ is \emph{symmetric} if $E \otimes_{\alpha}F$ and $F \otimes_{\alpha}E$ are canonically completely isometric (via the transposition map). A norm $\alpha$ is \emph{associative} if $(E \otimes_{\alpha}F)\otimes_{\alpha}G$ and $E \otimes_{\alpha}(F\otimes_{\alpha}G)$ are canonically completely isometric (via the identity map).
It is known that $\min$ and $\proj$ are both symmetric and associative, while the Haagerup o.s. tensor norm $h$ is associative but not symmetric \cite[Prop. 9.3.4]{Effros-Ruan-book}. 

For operator space tensor norms $\alpha$ and $\beta$ and a constant $c$, we write ``$\alpha \le c \beta$ on $E\otimes F$'' to indicate that the identity map $E \otimes_\beta F \to E \otimes_\alpha F$ has $\cb$-norm at most $c$.
If no reference to spaces is made, we mean that the inequality holds for any pair of normed operator spaces.

\subsection{The Schatten spaces and their vector-valued versions}\label{Schatten-spaces}
Recall that for a Hilbert space $H$ the Schatten class $S_p(H)$ is defined for $1 \le p < \infty$ as the space of all compact operators $T$ on $H$ such that $\tr\big( |T|^p\big) < \infty$ equipped with the norm $\n{T}_{S_p} = \big( \tr\big( |T|^p \big) \big)^{1/p}$; 
in the case $p=\infty$, we denote by $S_\infty(H)$ the space of all compact operators on $H$ (also denoted by $\mathcal K(H)$) endowed with the operator norm.
To endow the spaces $S_p(H)$ with an operator space structure, we follow Pisier's approach \cite{Pisier-Asterisque-98}:
since $S_\infty(H)$ is a $C^*$-algebra it has a canonical operator space structure \cite[p. 21]{Effros-Ruan-book}, by the duality $S_1(H)'= B(H)$ we also get a natural operator space structure on $S_1(H)$, 
and by complex interpolation for operator spaces \cite[Sec. 2.7]{Pisier-Operator-Space-Theory}
we get an operator space structure for each of the intermediate $S_p(H)$ spaces.
In the special case $H=\ell_2$, we simply write $S_p$ instead of $S_p(\ell_2)$.
For an operator space $E$ and $1\le p \le \infty$ we can similarly obtain an operator space which is an $E$-valued version of $S_p$ and will be denoted by $S_p[E]$: we define
$S_\infty[E] = S_\infty \otimes_{\min} E$ and
$S_1[E] = S_1 \otimes_{\proj} E$,
and once again in the case $1 < p < \infty$ we define $S_p[E]$ via complex interpolation between $S_\infty[E]$ and $S_1[E]$.

Replacing $S_1$ by the space $S_1^n$ of $n\times n$ matrices with the trace norm, and $S_\infty$ by the space $M_n$, the same procedure analogously yields operator spaces $S_p^n$ and $S_p^n[E]$.
We refer the reader to \cite{Pisier-Asterisque-98} for a detailed study of these spaces.

\subsection{Direct sums}\label{direct-sums} \cite[Sec. 2.6]{Pisier-Operator-Space-Theory}
Given a family $\{E_\gamma\}_{\gamma \in \Gamma}$ of operator spaces, its $\ell_\infty$ direct sum 
 which we will denote by $\ell_\infty( \{E_\gamma \;:\; \gamma \in \Gamma\} )$ is the direct sum $E = \bigoplus_{\gamma \in \Gamma} E_\gamma$ with the operator space structure given as follows: 
for $u = (u_\gamma)_{\gamma \in \Gamma} \in  M_n(E)$,
$$
\|u\|_{\ell_\infty( \{E_\gamma \;:\; \gamma \in \Gamma\}) } = \sup \big\{ \|u_\gamma\|_{M_n(E_\gamma)} \;:\; \gamma \in \Gamma \}.
$$
When we have only two summands, we will often use the notation $E_1 \oplus_\infty E_2$ instead.

The $\ell_1$-sum of $\{E_\gamma\}_{\gamma \in \Gamma}$ will be similarly denoted by $\ell_1( \{E_\gamma \;:\; \gamma \in \Gamma\} )$; for simplicity, below we only give the full details in the case 
$E_1 \oplus_1 E_2$ where we have two summands. For $u  \in M_n(E_1 \oplus E_2)$,
$$
\|u\|_{M_n(E_1 \oplus_1 E_2)} = \sup\big\{ \n{(T_1 \oplus T_2)_n(u)}_{M_n(\mathcal{B}(H))} \;:\; \n{T_j : E_j \to \mathcal{B}(H)}_{\cb} \le 1 \text{ for } j=1,2\big\}
$$
where $T_1 \oplus T_2 : E_1 \oplus E_2 \to \mathcal{B}(H)$ is the mapping $(T_1\oplus T_2)(x,y) = x+y$.

\subsection{Injections and projections}\label{injections-and-projections}

A linear map $q : E \to F$ between normed operator spaces is called a \emph{complete quotient} or \emph{complete metric surjection} or \emph{complete projection}  if it is onto and the associated map from $E/\ker(q)$ to $F$ is a completely isometric isomorphism.
In \cite[Sec. 2.4]{Pisier-Operator-Space-Theory},  it is proved that
a linear map $q:E \to F$ is a complete quotient if and only if its adjoint $q' : F' \to E'$ is a completely isometric embedding.
Note that if a linear map $q : E \to F$ between normed operator spaces is a complete quotient, then for every $x \in M_n(F)$ we have
$$
\n{x}_{M_n(F)} = \inf \left\{ \n{y}_{M_n(E)} \; : \; y \in M_n(E), \;  q_n(y) = x   \right\}.
$$

A completely isometric embedding $T : E \to F$ will be called a \emph{complete injection} for short.

An o.s. tensor norm $\alpha$ is called \emph{completely left projective} (resp.  \emph{completely left injective}) if for any normed operator space $G$ and any complete projection (resp. complete injection) $T : E \to F$, the map $T \otimes id_G : E \otimes_\alpha G \to F \otimes_\alpha G$ is a complete projection (resp. complete injection) as well.
Completely right projective and completely right injective operator space tensor norms are defined analogously, and an o.s. tensor norm will be called completely projective (resp. completely injective) when it is both completely left and right projective (resp. injective).

A normed operator space $E$ is said to be \emph{completely injective} if whenever $F \subseteq G$ are operator spaces and $T : F \to E$ is a completely bounded linear map, there exists an extension $\widetilde{T} : G \to E$ with $\n{T}_{\cb} = \|\widetilde{T}\|_{\cb}$.
By the Arveson extension theorem, $\mathcal{B}(H)$ is completely injective for any Hilbert space $H$.
Let us now observe that any completely injective normed operator space $E$ has to be complete. Indeed, consider the diagram
$$
\xymatrix{
\widetilde{E} \ar[dr]^T &\\
E \ar[u] \ar[r]_{id_E} & E}
$$
where the vertical arrow is simply the inclusion of $E$ in its completion $\widetilde{E}$, and the operator $T$ is given by the injectivity of $E$.
If we take a Cauchy sequence $(v_n)_n$ in $E$, and we take its limit $v \in \widetilde{E}$, then $T(v) \in E$ is the limit of $(v_n)_n$ since
$$
T(v) = T\Big( \lim_{n\to\infty} v_n \Big) = \lim_{n\to\infty} T(v_n) = \lim_{n\to\infty}  v_n.
$$

A normed operator space $E$ is said to be \emph{completely projective} if for any completely bounded linear map $T : E \to G/F$ into a quotient space and any $\varepsilon>0$ there exists a lifting $\widetilde{T} : E \to G$ of $T$ with $\|\widetilde{T}\|_{\cb} \le (1+\varepsilon) \n{T}_{\cb}$.
The basic examples of projective operator spaces are the spaces $S_1^n$ for any $n\in\N$, and we will also use the fact that an $\ell_1$-sum of projective operator spaces is again projective.

Every Banach operator space $E$ can be seen as a quotient of a completely projective space. Indeed, there is a set $I$ and a family $(n_i)_{i \in I} \subset \mathbb{N}$ such that $E$ is the quotient of $\ell_1(\{S_1^{n_i} : i \in I \})$ (see e.g., \cite[Prop. 2.12.2]{Pisier-Operator-Space-Theory}). We denote the latter space by $Z_E$ and $q_E : Z_E \twoheadrightarrow E$ the corresponding complete quotient mapping. The fact that $Z_E$ is completely projective can be tracked in \cite[Chap. 24]{Pisier-Operator-Space-Theory}.

See \cite[Chap. 24]{Pisier-Operator-Space-Theory} for more on completely injective and completely projective operator spaces.

\subsection{Examples of o.s. tensor norms} \label{Subsection: Examples}

In addition to the three fundamental o.s. tensor norms defined above (injective, projective, Haagerup), a number of other examples have appeared in the literature and we list some of them below.
The first three are discussed in \cite{Effros-Ruan-Hopf}, where it is shown that they satisfy the complete metric mapping property.
To conclude that they are o.s. tensor norms, we then just need to check that they are between $\min$ and $\proj$.

\subsubsection{The nuclear tensor product}
$E \otimes_{\nuc} F$ is defined as the quotient of $E \otimes_{\proj} F$ by the kernel of the canonical identity map $E \otimes_{\proj} F \to E \otimes_{\min} F$.
From the definition it is clear that $\min \le \nuc \le \proj$, so $\nuc$ is reasonable.

\subsubsection{The extended Haagerup tensor product \cite{Effros-Kishimoto}}
\label{subsubsection-extended-Haagerup}
$E \otimes_{\eh} F$ is the space of maps $u : E' \times F' \to \C$ which are normal (i.e. weak$^*$ continuous in each variable) and multiplicatively bounded.
If we denote the space of such maps by $\MB^\sigma(E'\times F', \C)$,
the matrix norms on $E \otimes_{\eh} F$ are given by the identification ${M_n\big( \MB^\sigma(E'\times F', \C) \big)} = \MB^\sigma(E'\times F', M_n)$.
It then follows from \eqref{eqn-injective-tensor-norm} that $\min \le \eh$.

The operator space structure on $E \otimes_{\eh} F$ can be described in the same way as \eqref{eqn-Haagerup-tensor} but replacing $r\in\N$ with an arbitrary set \cite[Eqn. 5.10]{Effros-Ruan-Hopf}, hence the \emph{extended} in the name.
This in particular shows $\eh \le h$, and thus $\eh \le \proj$.
The extended Haagerup tensor product is associative \cite[p. 149]{Effros-Ruan-Hopf}, and completely injective but not completely projective \cite[Lemma 5.4 and Prop. 5.5]{Effros-Ruan-Hopf}.

\subsubsection{The normal Haagerup tensor product \cite{Effros-Kishimoto}}
This is only defined for dual operator spaces, by $E' \otimes_{\nh} F' = (E \otimes_{\eh} F)'$.
Since the identity $E \otimes_{\proj} F \to E\otimes_{\eh} F$ is a complete contraction, dualizing yields a contraction
$$
E' \otimes_{\nh} F' = (E \otimes_{\eh} F)' \to (E \otimes_{\proj} F)' = \CB(E,F'),
$$
which implies that the identity $E' \otimes_{\nh} F' \to E' \otimes_{\min} F'$ is a complete contraction.
The fact that $\nh \le \proj$ is a complete contraction, i.e. the identity $E' \otimes_{\proj} F' \to (E \otimes_{\eh} F)'$ is a complete contraction, follows easily from \eqref{eqn-projective-tensor-norm} and the description of $E \otimes_{\eh} F$ as normal multiplicatively bounded mappings (see   \ref{subsubsection-extended-Haagerup}).

Recall from above that the
extended Haagerup tensor product is not completely projective, that is, there exist complete projections $\pi_j : E_j \to F_j$, $j=1,2$ such that 
$\pi_1 \otimes \pi_2 : E_1 \otimes_{\eh} E_2 \to F_1 \otimes_{\eh} F_2$ is not a complete projection.
By dualizing, we have complete injections $\pi_j' : F_j' \to E_j'$, $j=1,2$
such that $\pi_1' \otimes \pi_2' : F_1' \otimes_{\nh} F_2' \to E_1' \otimes_{\nh} E_2'$
is not a complete injection, so 
the normal Haagerup tensor
product is not completely injective.

\subsubsection{The symmetrized Haagerup tensor norms}
Given normed operator spaces $E$ and $F$ and $u \in M_n(E \otimes F)$, we define
$$
(h \cap h^t)_n(u;E,F) = \max\big\{ h_n(u; E,F), h^t_n(u; E,F) \big\}.
$$
It is easy to see that $h \cap h^t$ is a symmetric o.s. tensor norm, and from the definition it is obvious that there is a completely isometric embedding
$$
E \otimes_{ h \cap h^t} F \hookrightarrow (E \otimes_h F) \oplus_\infty (E \otimes_{h_t} F)
$$
given by $u \mapsto (u,u)$.
Since $h$ and $h^t$ are completely injective, it is clear that $h \cap h^t$ is completely injective as well.
Notice that if $\alpha$ is a symmetric o.s. tensor norm such that $h \le \alpha$, by transposing we have $ h^t \le \alpha$ and therefore $h \cap h^t \le \alpha$.
Thus, we call $h \cap h^t$ the minimal symmetrized Haagerup tensor norm.
This tensor norm has appeared in \cite{pisier2002grothendieck,haagerup2008effros,dimant2015bilinear}.

On the other hand, for normed operator spaces $E$ and $F$ and $u \in M_n(E \otimes F)$, we define
$$
(h+h^t)_n(u;E,F) = \inf\big\{ \|(v,w)\|_{M_n\big( (E \otimes_h F)\oplus_1(E \otimes_{h^t} F) \big)} \;:\; u = v+w \big\}.
$$
That is, by definition the mapping
\[
q : (E \otimes_h F) \oplus_1 (E \otimes_{h^t} E) \to E \otimes_{h+h^t} F
\]
given by $q(u,v) = u + v$ is a complete quotient.
Once again, it is easy to check that this defines a symmetric o.s. tensor norm which is completely projective because so is $h$.
The o.s. tensor norm $h+h^t$ was introduced in \cite{Oikhberg-Pisier}, were it was denoted by $\mu$ and it was shown that it is neither associative nor completely injective.
We will call $h+h^t$ the maximal symmetrized Haagerup tensor norm because of the following property. Suppose that $\beta$ is a symmetric tensor norm with $\beta \le h$. Then $\beta \le h^t$, so that the formal identity maps
$$
E \otimes_h F \to E \otimes_\beta F, \quad E \otimes_{h^t} F \to E \otimes_\beta F
$$
are complete contractions. Therefore, the map $(v,w) \mapsto v+w$ is a complete contraction
$$
(E \otimes_h F) \oplus_1 (E \otimes_{h^t} E) \to E \otimes_\beta F
$$
and by the standard properties of quotients \cite[Prop. 2.4.1]{Pisier-Operator-Space-Theory}, the identity map $E \otimes_{h+h^t} F \to E \otimes_\beta F$ is a complete contraction as well, that is, $\beta \le {h+h^t}$.

\subsubsection{The Chevet-Saphar tensor products}

Inspired by the definition of the Banach space case,  for $1 \leq p \leq \infty$, the right and left $p$-Chevet-Saphar o.s. tensor norms, $d_p $ and $g_p $, are defined in  \cite{CD-Chevet-Saphar-OS}   in following way: given normed operator spaces $E$ and $F$, for $u \in E \otimes F$,
\begin{align}\label{dp-os}
d_p (u) = \inf \left\{\Vert (x_{ij}) \Vert_{S^n_{p'}\otimes_{\min} E} \; \Vert(y_{ij})\Vert_{S^n_p[F]} : u = \sum_{i,j=1}^n x_{ij} \otimes y_{ij} \right\}, \\
g_p (u) = \inf \left\{ \Vert(x_{ij})\Vert_{S^n_p[E]} \; \Vert (y_{ij}) \Vert_{S^n_{p'}\otimes_{\min} F}   : u = \sum_{i,j=1}^n x_{ij} \otimes y_{ij} \right\}
\end{align}
where the infimum runs over  all the possible ways in which the tensor $u$ can be written as a finite sum as above.

Note that the above expressions only define norms on $E \otimes F$, the full operator space structure of the tensor products $E \widehat \otimes_{d_p } F$ and  $E \widehat \otimes_{g_p } F$ are given by the following quotient mappings  (see \cite[Sec. 3]{CD-Chevet-Saphar-OS})
\begin{align}
 q_{d_p} : (S_{p'} \widehat{\otimes}_{\min} E) \widehat{\otimes}_{\proj} S_p[F] \to E \widehat{\otimes}_{d _p} F, \\
 q_{g_p} : S_p[E] \widehat{\otimes}_{\proj} (S_{p'} \widehat{\otimes}_{\min} F) \to E \widehat{\otimes}_{g _p} F,
 \end{align}
where $\widehat \otimes $ means the completion of the corresponding tensor product.
In \cite[Prop. 3.5.]{CD-Chevet-Saphar-OS} it is shown that both $d_p $ and $g_p $ are o.s. tensor norms.

\subsubsection{$\lambda$-tensor products \cite{Wiesner,Defant-Wiesner}} 
For each $k \in \N$, let $$\bl_k: M_k\times M_k \to M_{\tau(k)}$$ (where $\tau(k) \in \N$ is a natural number only depending on $k$) be a bilinear map;
we will denote the sequence $(\bl_k)_k$ by $\lambda$.
Two basic examples to keep in mind are the tensor product $\otimes$ (i.e. when $\bl_k(x,y) = x \otimes y$ for all $k$) and the matrix product $\odot$ (i.e. when $\bl_k(x,y) = x \odot y$ for all $k$).
We point out that the notation above is slightly different from that of \cite{Wiesner,Defant-Wiesner}, in order to keep consistency in the notation of the present work.
For $u \in M_k(E \otimes F)$, define
\begin{equation}\label{lnorm}
\lambda_k(u)=\inf\{\|a\|\|v_1\|\|v_2\|\|b\|\}
\end{equation}
where  the infimum is taken over arbitrary decompositions $u = a \otimes_{\bl_j} (v_1,v_2) b$ with $a\in M_{k, \tau(j)}$, $b\in M_{\tau(j), k}$, $v_1\in M_j(E)$, $v_2\in M_j(F)$, where $\otimes_{\bl_j} : M_j(E) \times M_j(F) \to M_{\tau(j)}(E \otimes F)$ is the bilinear map given by $(a_1 \otimes x, a_2 \otimes y) \mapsto \bl_j(a_1,a_2) \otimes x \otimes y$.
Observe that the case $\lambda=\otimes$ corresponds to the projective tensor norm, and $\lambda=\odot$ yields the Haagerup tensor norm.

In order to guarantee that the norms defined above give an o.s. tensor norm, we will need for $\lambda$ to satisfy certain technical conditions. First, some notation.
 We let $\varepsilon_{i,j}:=\varepsilon^{[k,l]}_{i,j}\in M_{k,l}$ denote the matrix which is 1 in the $(i,j)$-th entry and zero elsewhere, $\varepsilon^{[k]}_{i,j}:=\varepsilon^{[k,k]}_{i,j},\; \varepsilon_i:=\varepsilon_{i,i}, \;\varepsilon^{[k]}_i:=\varepsilon^{[k,k]}_{i,i}$ and $\varepsilon^{[k,l]}_{i,j}=0$ if $(i,j) \notin \{1,\ldots,k\} \times \{1,\ldots,l\}$. As usual we will denote by $I_n$ the identity $n\times n$-matrix.

\begin{proposition}[{\cite[Prop. 4.1]{Defant-Wiesner} and \cite[Prop. 6.1]{Wiesner}}]
For any operator spaces $E$ and $F$ the sequence $\lambda_k(\cdot)$ defined above gives an operator space structure on $E \otimes F$, whenever $\lambda$ satisfies the following conditions:
\begin{itemize}
\item[(E1)] For all $k \in \mathbb{N}$ there exist $p \in \N$ and matrices $S \in M_{k,\tau (p)}$, $T \in M_{\tau (p),k}$, $a_1 , \cdots , a_k \in M_p$ such that for all $ j_1 , j_2
 \in \{1, \cdots, k\}$:
$$
S \bl_p(a_{j_1},a_{j_2}) T= \left\{
        \begin{array}{ll}
         \varepsilon_j^{[k]} & \qquad \text{if}\;\; j_1=j_2=j, \\
          0 & \qquad \text{otherwise.}
          \end{array}
    \right.
$$ 
\item[(E2)] For all $ r, s \in  \mathbb{N}$ there exist matrices  $P \in M_{\tau (r)+\tau (s), \tau(r+s)}$, with $\|P\|\leq 1$ such that for all $ (i_k , j_k ) \in 
 \{1, \cdots , r\}^2 \cup \{r + 1, \cdots , r + s\}^2$ with $k=1,2$:
 
\begin{align}
\begin{split}
& P \bl_{r+s}( \varepsilon_{i_1, j_1}^{[r+s]} , \varepsilon_{i_2 ,j_2 }^{[r+s]})P^* \\
&= \mathrm{diag}\Big( \bl_r (\varepsilon_{i_1 ,j_1}^{[r]} ,  \varepsilon_{i_2,j_2}^{[r]}), \bl_{s}(\varepsilon_{i_1-r ,j_1-r}^{[s]} , \varepsilon_{i_2-r,j_2-r}^{[s]})\Big).\notag
\end{split}
\end{align}
\item[(E3)] $\bl_1(1,1)=1$ and $\sup_{k\in \N}\|\bl_k\| <\infty$.
\end{itemize}
In this case, we denote by $E \otimes_\lambda F$ the corresponding operator space.
Moreover, the complete metric mapping property is satisfied.
\end{proposition}

Note that, in an abuse of notation, we are using the symbol $\lambda$ to denote both the sequence of bilinear maps $M_k \times M_k \to M_{\tau(k)}$ and the operator space structure induced on the tensor product.

Just as for the projective and Haagerup tensor products, the dual of a $\lambda$-tensor product can be identified with a certain space of bilinear forms.
For any bilinear map $\phi : E \times F \to W$ where $E$, $F$, $W$ are operator spaces, define the bilinear maps
$\phi_{\bl_k} : M_k(E) \times M_k(F) \to M_{\tau(k)}(W)$ given on elementary tensors by
$$
\phi_{\bl_k}\big( a_1 \otimes v_1, a_2 \otimes v_2 \big) = \bl_k(a_1,a_2) \otimes \phi(v_1,v_2), \qquad a_1, a_2 \in M_k, v_1 \in E, v_2 \in F.
$$
We also define
$$
\n{\phi}_{\cb,\lambda} = \sup_{k\in\N} \left\{\n{\phi_{\bl_k}(x,y)}_{M_{\tau(k)}(W)} \;:\; \n{x}_{M_k(E)} \le 1, \n{y}_{M_k(F)} \le 1 \right\}
$$
and
$$
\CB_\lambda(E \times F; W) = \left\{ \phi : E \times F \to W \text{ bilinear} \;:\; \n{\phi}_{\cb,\lambda}<\infty \right\}.
$$
This space is an operator space with the identification
$$
M_k\big(\CB_\lambda(E \times F; W)\big) = \CB_\lambda\big(E \times F; M_k(W)\big).
$$
Note that for the cases  $\lambda=\otimes$ (projective tensor product) and $\lambda=\odot $ (Haagerup tensor product) we recover the usual dual spaces $\CB_\otimes=\JCB$ and $\CB_\odot=\MB$.

\begin{theorem}[{\cite[Thm. 6.2]{Defant-Wiesner}}]\label{duality-lambda-tensor}
If $\lambda$ satisfies (E1), (E2), and (E3), then
the natural identification yields a canonical complete isometry
$$
\CB( E\otimes_{\lambda} F, W ) = \CB_\lambda(E \times F;W)
$$
so in particular
$$
(E\otimes_{\lambda} F)' = \CB_\lambda(E \times F;\C) =: \CB_\lambda(E \times F).
$$
\end{theorem}

Let us denote by $\odot^t : M_k \times M_k \to M_k$ the transposition of the usual matrix product, that is, $A \odot^t B = B \odot A$ for $A, B \in M_k$.
It is clear that taking $\lambda=\odot^t$, the associated operator space structure on $E \otimes F$ is precisely $E \otimes_{h^t} F$, the transposition of the Haagerup tensor product.
We now introduce a new family of Haagerup-style o.s. tensor products which are $\lambda$-tensor products, and which can be understood as a sort of  ``interpolation'' between $h$ and $h^t$.
For any $\theta \in [0,1]$ and $k\in\N$, we let
$\odot^\theta_k : M_k \times M_k \to M_k$ be $(1-\theta) \odot + \theta \odot^t$, that is, $\odot^\theta_k(A,B) = (1-\theta)AB+\theta BA$ for $A,B\in M_k$. It is clear that each $\odot^\theta_k$ is bilinear, and let us now show that they all induce operator space structures on the tensor product.

\begin{proposition}\label{htheta-satisfies-Es}
For each $\theta\in[0,1]$, $\odot^\theta$ satisfies (E1), (E2) and (E3).
\end{proposition}

\begin{proof}
Clearly $1 \odot 1 = 1 \odot^t 1 = 1$, and $\n{\odot_k} = \n{\odot^t_k} = 1$ (since $\n{A\odot B} \le \n{A} \n{B}$ for $A,B\in M_k$). Taking a convex combination yields (E3) for $\odot^\theta$.

As in the proof of \cite[Prop. 7.6]{Wiesner}, (E1) is satisfied for $\odot$ by taking $p=k$, $S=T=I_k$, and $a_{j_i} = \varepsilon^{[k]}_{j_i}$. Since all the matrices involved are diagonal all the products commute, and it follows that the same choices yield (E1) for $\odot^t$. Taking convex combinations, the same choices once again yield (E1) for $\odot^\theta$.
The exact same type of argument works for (E2): this is proved for $\odot$ in \cite[Prop. 7.6]{Wiesner} using $P=I_{r+s}$.
\end{proof}

\begin{proposition}\label{htheta-is-cross}
For each $\theta\in[0,1]$, $\odot^\theta$ is a cross-norm.
\end{proposition}

\begin{proof}
Let $E$, $F$ be operator spaces, $x \in M_m(E)$, $y \in M_n(F)$.
It is well-known that $x\otimes y = (x\otimes I_n)\odot (I_m \otimes y)$, see e.g. \cite[Eqn. (9.1.10)]{Effros-Ruan-book}. The exact same calculation shows $x\otimes y = (x\otimes I_n)\odot^t (I_m \otimes y)$, and therefore $x\otimes y = \odot^\theta_{mn}(x\otimes I_n, I_m \otimes y)$. Thus, by the definition of $\odot^\theta_{mn}$,
$$
\odot^\theta_{mn}( x\otimes y ) \le \n{x\otimes I_n}_{M_{mn}(E)} \n{ I_m \otimes y }_{M_{mn}(F)} =  \n{x}_{M_m(E)} \cdot \n{y}_{M_n(F)}.
$$
\end{proof}

In order to conclude that $\odot^\theta$ is an o.s. tensor norm, the only missing ingredient is to check that $\min \le \odot^\theta$.
We prove a more general result.

\begin{proposition}\label{min-less-than-lambda}
Suppose that $\lambda$ satisfies (E1), (E2), (E3), and $\n{\bl_k}_{\jcb} \le 1$ for each $k\in\N$.
Then the identity mapping $E\otimes_\lambda F \to E \otimes_{\min} F$ is a complete contraction.
\end{proposition}

\begin{proof}
By \cite[Thm. 5.1]{Blecher-Paulsen-Tensors}, it suffices to show that for every $U \in M_k(E \otimes F)$, $\phi \in M_m(E')$, $\psi \in M_n(F')$ we have
$$
\n{\mpair{\phi \otimes \psi}{U}}_{M_{mnk}} \le \n{\phi}_{M_m(E')} \n{\psi}_{M_n(F')}\lambda_k(U).
$$
Now, by the identification $(E\otimes_{\lambda} F)' = \CB_\lambda(E \times F)$ we already know that
$$
\n{\mpair{\phi \otimes \psi}{U}}_{M_{mnk}} \le \n{\phi \otimes \psi}_{M_{mn}(\CB_\lambda(E \times F))}\lambda_k(U).
$$
Recalling that $M_{mn}\big(\CB_\lambda(E \times F)\big) = \CB_\lambda\big(E \times F; M_{mn}\big)$, all we need is to show that
$$
\n{\varphi}_{\lambda,\cb} \le \n{\phi}_{M_m(E')} \n{\psi}_{M_n(F')}
$$
where $\varphi =\phi \otimes \psi : E \times F \to M_{mn}$; note that here we are interpreting $\phi \in M_m(E') = \CB(E,M_m)$ and $\psi \in M_n(F') = \CB(F,M_n)$.

To calculate the norm of $\varphi$ in $\CB_\lambda(E \times F ; M_{mn})$, take $u \in M_k(E)$ and $v \in M_k(F)$ with $\n{u}_{M_k(E)}, \n{v}_{M_k(F)}\le 1$.
Represent $u = \sum_i a_i \otimes u_i$, $v = \sum_j b_j \otimes v_j$ where $a_i,b_j \in M_k$, $u_i \in E$, $v_j \in F$.
Then, by the definition of $\varphi_{\bl_k}$,
\begin{multline*}
   \varphi_{\bl_k}(u,v) = \sum_{i,j} \bl_k(a_i,b_j) \otimes \varphi(u_i,v_j) =  \sum_{i,j} \bl_k(a_i,b_j) \otimes \phi(u_i) \otimes \psi(v_j) \\
   = \otimes_{\bl_k} \Big(\sum_i a_i \otimes \phi(u_i),   \sum_j b_j \otimes \psi(v_j) \Big) 
   = \otimes_{\bl_k} \big( (I_k \otimes \phi)u, (I_k\otimes \phi)v \big).
\end{multline*}
Note that the assumption $\n{\bl_k}_{\jcb} \le 1$ precisely means that for any $A \in M_k(M_m)$ and $B \in M_k(M_n)$ we have
$$
\n{ \otimes_{\bl_k}(A,B) }_{M_{\tau(k)mn}} \le \n{A}_{M_k(M_m)} \n{B}_{M_k(M_n)},
$$
so
\begin{multline*}
\n{\varphi_{\bl_k}(u,v)} \le \n{(I_k \otimes \phi)u}_{M_k(M_m)} \n{(I_k\otimes \phi)v}_{M_k(M_n)} \\
\le \n{I_k \otimes \phi} \n{u} \n{I_k\otimes \phi} \n{v} \le \n{\phi} \n{\phi}    
\end{multline*}
which yields the desired conclusion.
\end{proof}

\begin{theorem}
For each $\theta \in [0,1]$, $\odot^\theta$ is an o.s. tensor norm. 
\end{theorem}

\begin{proof}
This is an immediate consequence of Propositions \ref{htheta-satisfies-Es}, \ref{htheta-is-cross} and \ref{min-less-than-lambda}, where for the latter we use the fact that $\n{\odot}_{\jcb} \le 1$ and therefore $\n{\odot^\theta_k}_{\jcb} \le 1$.
\end{proof}

Throughout the rest of this paper, whenever we consider any $\lambda$-tensor norm
we will always be assuming that it satisfies (E1), (E2), (E3) and additionally it is an o.s. tensor norm. To emphasize this point, we will call them \emph{$\lambda$-o.s. tensor norms}.

\begin{remark}
We have previously pointed out that our notion of uniformity in the definition of o.s. tensor norm is weaker than that of \cite{Blecher-Paulsen-Tensors}.
This has the advantage of allowing the $\odot^\theta$ to be covered by the theory: the only assumptions we are aware of which imply that a $\lambda$-tensor norm satisfies the uniformity condition in \cite{Blecher-Paulsen-Tensors} are conditions (W1) and (W2) in \cite[Prop. 12.2]{Wiesner}, but $\odot^\theta$ only satisfies (W2) in the extreme cases $\theta=0,1$.
\end{remark}

\subsection{Some usual notation} \label{usual notation}
For a normed operator space $E$, we denote by $\OFIN(E)$ (resp. $\OCOFIN(E)$) the set of all finite-dimensional (resp. finite-codimensional) subspaces of $E$.
Given $L \in \OCOFIN(E)$, let $q^E_L \colon E \to E/L$ be the canonical projection
and given a subspace $F$ of $E$ let $i^E_F \colon  F \to E$ be the canonical injection. To avoid an overload of notation, when the inclusion is the typical complete isometry for operator spaces $E\subset \mathcal{B}(H)$ we just denote the canonical injection by $i_E:E\to \mathcal B(H)$.
Also, following the usual Banach space notation, for a normed operator space $E$, $\kappa_E \colon  E \to E''$ denotes the canonical injection into the bidual.

Given vector spaces $E$ and $F$, we identify a linear map $T : E \to F'$ with a bilinear map $\beta_T : E \times F \to \C$ via $\beta_T(x,y) = (Tx)(y)$; note that $\beta_T$ can also be identified with an element of $(E \otimes F)'$.

A normed operator space $E$ is an $\mathscr{OS}_{p,C}$ space \cite[Sec. 2]{Junge-Nielsen-Ruan-Xu}
 if there is a family $(F_i)_{i\in I}$ of finite-dimensional subspaces of $E$ whose union is dense in $E$ and such that for every index $i$ there is a natural number $n_i$ such that $d_{\cb}(S_p^{n_i}, F_i) \leq C$ where $d_{\cb}$ stands for the completely bounded Banach-Mazur distance, i.e., 
 $$
 d_{\cb}(E, F):= \inf\big\{\Vert T \Vert_{cb} \Vert T^{-1} \Vert_{cb} : T \in \CB(E,F) \mbox{ is a complete isomorphism} \big\}.
 $$
If $E$ is an $\mathscr{OS}_{p,C'}$ space for every $C'>C$, we say that $E$ is an $\mathscr{OS}_{p,C+}$ space.

If $W$ is a Banach space, its maximal operator space structure $\Max(W)$ is defined by, for $A \in M_n(W)$ 
$$
\n{A}_{M_n(\Max(W))} = \big\{ \n{  u_n(A) }_{M_n(B(H))} \;:\; \n{ u : W \to B(H) } \le 1 \big\}.
$$
See \cite[Chap. 3]{Pisier-Operator-Space-Theory} for more details.

\subsection{Operator approximation property}
A normed operator space $E$ is said to have the \emph{operator approximation property} (OAP)
\cite[Sec. 11.2]{Effros-Ruan-book} if there exists a net of finite rank mappings $T_\eta \in\CB(E,E)$ such that the net $id_\mathcal{K} \otimes T_\eta$ converges pointwise to the identity in $K_\infty(E)=\mathcal{K} \otimes_{\min} E$, where $\mathcal{K}$ denotes the compact operators on $\ell_2$.
Similarly, $E$ is said to have the 
\emph{$C$-completely bounded approximation property} ($C$-CBAP) \cite[p. 205]{Effros-Ruan-book} if there exists a net of finite rank mappings $T_\eta \in\CB(E,E)$ such that $\n{T_\eta}_{\cb} \le C$ and $\n{T_\eta x - x} \to 0$ for every $x \in E$.
Note that by \cite[Thm. 11.3.3]{Effros-Ruan-book}, we also have $\n{(T_\eta)_n x - x} \to 0$ for every $x \in M_n(E)$.
In the particular case $C=1$, we say that $E$ has the \emph{completely metric approximation property} (CMAP).

\subsection{Relevant differences with the Banach space setting} \label{Relevant differences}
It is well-known that the theory of Banach spaces cannot be painlessly transferred to the operator space framework, and the theory of tensor norms is no exception. Let us point out three important differences in this context. 

\subsubsection{Local reflexivity} \label{locally reflexive}

A normed operator space $E$ is said to be  \emph{locally reflexive} \cite{Effros-Junge-Ruan}  if for each finite-dimensional operator space $F$, any complete contraction $\varphi: F \to E''$ may be approximated in the point-weak$^*$ topology by a net of complete contractions $\varphi_\eta : F \to E$.
The class of all  locally reflexive normed operator spaces is denoted by OLOC.

Since local reflexivity is relevant to prove many properties related to Banach space tensor norms, it is often an additional hypothesis in the operator space versions of those results.
Moreover, in some cases this is known to be necessary: \cite[Thm. 14.3.1]{Effros-Ruan-book} is an example of an identity involving tensor norms whose validity turns out to characterize local reflexivity.
Nevertheless, there are some particular situations where we can omit this assumption (for instance, for $\lambda$-o.s. tensor norms, see Lemma \ref{extension-lemma} and Lemma \ref{embedding-lemma}). 

We point out, for future reference, that given a complete operator space $E$, the completely projective space $Z_E$ introduced in Section~\ref{injections-and-projections} is in fact locally reflexive. Indeed,   its dual  $\ell_\infty(\{M_{n_i} : i \in I \})$ is clearly a $C^*$-algebra and thus a  von Neumann algebra, and preduals of von Neumann algebras are locally reflexive  (see e.g., \cite[Theorem 18.7]{Pisier-Operator-Space-Theory}).

\subsubsection{Exactness}
Every finite-dimensional Banach space embeds $(1+\varepsilon)$-isomorphically into a space of the form $\ell_\infty^n$, but the analogous property in the operator space setting does not hold in general:
an operator space $E$ is said to be \emph{$C$-exact} if for any $\varepsilon>0$ every finite-dimensional subspace of $E$ is $(C+\varepsilon)$-completely isomorphic to a subspace of an $M_n$ space.
When an operator space is $1$-exact we simply say that it is exact.

Remark~\ref{rmk: bad news}  gives an unsurprising example where this causes a significant difference between the classical tensor norm theory and the one for operator spaces.

Another useful  property (for Banach tensor products arguments) which is not transferable to operator spaces due to exactness issues is the one proved in \cite[Lemma 2.2.2]{diestel2008metric} about finite dimensional pieces of a quotient:

\begin{theorem}
The following statement is false:
\begin{quote}
    For every complete quotient between Banach
    operator spaces $q : E \twoheadrightarrow F$, every $F_0 \in \OFIN(F)$, and every $\varepsilon>0$, there exists $E_0 \in \OFIN(E)$ such that $q(E_0)=F_0$ and for every $y \in M_n(F_0)$ there exists $x \in M_n(E_0)$ such that $q_n(x)=y$ and $\n{x}_{M_n(E_0)} \le (1+\varepsilon) \n{y}_{M_n(F_0)}$.
\end{quote}
\end{theorem}

\begin{proof}
First let us remark that if $E_0 \in \OFIN(E)$ satisfies the conditions in the statement, then it is clear that $F_0$ is $(1+\varepsilon)$-completely isomorphic to $E_0/\ker(\restr{q}{E_0})$ via the canonical map $E_0/\ker(\restr{q}{E_0}) \to F_0$ induced by $\restr{q}{E_0}$.

Let $F_0$ be a finite-dimensional operator space.
By \cite[Cor. 2.12.3]{Pisier-Operator-Space-Theory}, there exists a complete quotient $q : S_1 \twoheadrightarrow F_0$.
Given $\varepsilon>0$, suppose that there exists $E_0 \in \OFIN(S_1)$
satisfying the conditions in the statement.

Since $E_0 \in \OFIN(S_1)$, we can find $G_0 \in \OFIN(S_1)$ such that $E_0 \subseteq G_0$ and $G_0$ is $(1+\varepsilon)$-completely isomorphic to $S_1^N$ for some $N\in\N$ using perturbation arguments as in \cite[Sec. 2.13]{Pisier-Operator-Space-Theory}.
First, take a basis of $E_0$ and approximate it by finitely supported matrices.
This yields a subspace $E_1 \subset S_1$ which is completely isomorphic to $E_0$ and naturally sits inside a subspace $G_1 \subset S_1$ that is completely isometric to $S_1^N$ for some $N$. Perturbing $G_1$ to ``put $E_1$ back onto $E_0$'' yields the desired $G_0$.

Because $G_0 \supseteq E_0$, it is clear that $G_0$ also satisfies the conditions in the statement.
This gives that $F$ is $(1+\varepsilon)^2$-completely isomorphic to a quotient of $S_1^N$, and therefore $F'$ is $(1+\varepsilon)^2$-completely isomorphic to a subspace of $M_N$.

But this is, in general, false: from \cite[Thm. 7]{pisier1995exact}, for any $n \ge 2$ any complete isomorphism from $\Max(\ell_1^n)$ onto a subspace of an $M_N$ space has constant at least $\frac{n}{2\sqrt{n-1}}$.
\end{proof}

As one would expect, this idea of approximating quotients using finite-dimensional pieces is useful when investigating projectivity of Banach space tensor norms.
For example, it is used in \cite[Prop 7.5]{Ryan} to show that a tensor norm is projective when its dual tensor norm is injective.
Not only does this argument not work in the operator space setting, but the corresponding result is in fact not true: see Remark \ref{consecuencias min} below, where we also point out that a related duality result for completely injective and completely projective hulls claimed without proof in \cite{blecher1991tensor} does not hold.

\subsubsection{Approximation properties of completely  injective spaces}
Each Banach space $E$  isometrically embeds into $\ell_\infty(B_{E'})$, which is an injective space in the Banach space setting. The fact that $\ell_\infty(B_{E'})$ has the approximation property plays a relevant role in many tensor product results in the Banach space framework, especially those related with the Approximation Lemma and completely injective tensor norms.
In contrast, for an operator space $E$ the standard complete isometry into a completely injective space is the one into $\mathcal B(H)$ for some Hilbert space $H$. The lack of the approximation property in $\mathcal B(H)$
is another important dissimilarity between both frameworks.

%------------------------------------------------------
%------------------------------------------------------
%------------------------------------------------------
\section{Finite and cofinite hulls} \label{hulls}

Given an o.s. tensor norm $\alpha$ on $\OFIN$, we can use the following two procedures to extend it to the class of all operator spaces.

\begin{definition}\label{def-finite-and-cofinite-hulls}
Given operator spaces $E$ and $F$, and $u \in M_n(E \otimes F)$, let
\begin{align*}
\overrightarrow{\alpha}_n(u ; E, F) &= \inf\left\{\alpha_n(u; E_0,F_0) \colon E_0 \in \OFIN(E), F_0 \in\OFIN(F), u \in M_n(E_0 \otimes F_0)\right\}
\end{align*}
and
\begin{align*}
\overleftarrow{\alpha}_n(u ; E, F) = \sup\left\{\alpha_n\big(  (q_K^E \otimes q_L^F)_n(u); E/K,F/L\big) \colon K \in \OCOFIN(E), L \in\OCOFIN(F) \right\}
\end{align*}
An o.s. tensor norm $\alpha$ on $\ONORM$ is called
\emph{finitely-generated} if $\alpha = \overrightarrow{\alpha}$,
and \emph{cofinitely-generated} if $\alpha = \overleftarrow{\alpha}$.
\end{definition}
Clearly, if $\alpha \le c \beta$ it follows that $\overleftarrow{\alpha} \le c \overleftarrow{\beta}$
and
$\overrightarrow{\alpha} \le c \overrightarrow{\beta}$.
Since $\min$ and $h$ are completely injective, $\min = \overrightarrow{\min}$ and $h = \overrightarrow{h}$.
Observe also that the description of the projective norm as an infimum gives $\proj = \overrightarrow{\proj}$.

For the same reason, any $\lambda$-o.s. tensor norm is finitely generated.

\begin{proposition}\label{prop-order-of-hulls}
Let  $\alpha$ be an o.s. tensor norm on $\OFIN$. Then the \emph{finite hull} $\overrightarrow{\alpha}$ of $\alpha$
 and the \emph{cofinite hull} $\overleftarrow{\alpha}$ of $\alpha$ are o.s. tensor norms on $\ONORM$ with
 $$
 \min \le \overleftarrow{\alpha} \le \overrightarrow{\alpha} \le \proj, \qquad  \overleftarrow{\alpha}\vert_{\OFIN} = \overrightarrow{\alpha}\vert_{\OFIN} = \alpha.
 $$
 If $\alpha$ is defined on $\ONORM$ (and not just on $\OFIN$) then
 $$
 \overleftarrow{\alpha} \le \alpha \le \overrightarrow{\alpha}.
 $$
\end{proposition}

\begin{proof}
The complete metric mapping property of $\alpha$ yields $\overleftarrow{\alpha} \le \alpha \le \overrightarrow{\alpha}$;
the remarks before the statement of the Proposition then show $\min \le \overleftarrow{\alpha} \le \overrightarrow{\alpha} \le \proj$.
Since it is obvious that $\alpha$ agrees with $\overleftarrow{\alpha}$ and $\overrightarrow{\alpha}$ on finite-dimensional spaces, the only thing left to prove is that both hulls are in fact o.s. tensor norms.
But we have already proved that they are between $\min$ and $\proj$, so we just have to check the complete metric mapping property.

To that effect, let $S \in \CB(E_1,E_2)$, $T \in \CB(F_1,F_2)$ and $u \in M_n(E_1 \otimes F_1)$.
Then, using the complete metric mapping property of $\alpha$,
\begin{multline*}
\overrightarrow{\alpha}_n\big( (S\otimes T)_n u ; E_2, F_2\big) \\
= 
 \begin{aligned}[t] 
\inf\Big\{\alpha_n\big((S\otimes T)_n u; E^2_0,F^2_0\big) \colon  E^2_0 \in \OFIN(E_2), F^2_0 \in\OFIN(F_2), (S\otimes T)_nu \in M_n(E^2_0 \otimes F^2_0)\Big\} 
\end{aligned}
\\
\le 
 \begin{aligned}[t] 
 \inf\Big\{\alpha_n\big((S\otimes T)_nu&; SE^1_0,TF^1_0\big) \colon \\ &E^1_0 \in \OFIN(E_1), F^1_0 \in\OFIN(F_1), u \in M_n(E^1_0 \otimes F^1_0)\Big\}
\end{aligned}
\\
\le \n{S}_{\cb}\n{T}_{\cb} 
 \begin{aligned}[t] 
\inf\Big\{\alpha_n\big(u; &E^1_0,F^1_0\big) \colon \\&E^1_0 \in \OFIN(E_1), F^1_0 \in\OFIN(F_1), u \in M_n(E^1_0 \otimes F^1_0)\Big\} 
\end{aligned}
\\
= \n{S}_{\cb}\n{T}_{\cb} \overrightarrow{\alpha}_n\big(  u ; E_1, F_1\big),
\end{multline*}
which shows the complete metric mapping property for $\overrightarrow{\alpha}$.

Now, consider $K^2_0 \in \OCOFIN(E_2)$ and $L^2_0 \in\OCOFIN(F_2)$.
Observe that $S^{-1}K^2_0 \in \OCOFIN(E_1)$ and $T^{-1}L^2_0 \in\OCOFIN(F_1)$.
Moreover, by the basic properties of quotients there exist maps
$S_{K^2_0} : E_1/S^{-1}K^2_0 \to E_2/K^2_0$ and $T_{L^2_0} : F_1/T^{-1}L^2_0 \to F_2/L^2_0$
making the following diagrams commutative
$$
	\xymatrix{
	E_1 \ar[d]_{ q^{E_1}_{S^{-1}K^2_0} } \ar[r]^{S}  & E_2 \ar[d]^{ q^{E_2}_{K^2_0} }\\
	E_1/S^{-1}K^2_0 \ar[r]_{S_{K^2_0}}  & E_2/K^2_0
	}
	\qquad \qquad
	\xymatrix{
	F_1 \ar[d]_{ q^{F_1}_{T^{-1}L^2_0} } \ar[r]^{T}  & F_2 \ar[d]^{ q^{F_2}_{L^2_0} }\\
	F_1/T^{-1}L^2_0 \ar[r]_{T_{L^2_0}}  & F_2/L^2_0
	}
$$
and moreover we have 
that $\n{S_{K^2_0}}_{\cb} = \n{ q^{E_2}_{K^2_0} S }_{\cb} \le \n{S}_{\cb}$ and $\n{T_{L^2_0}}_{\cb} = \n{ q^{F_2}_{L^2_0} T }_{\cb} \le \n{T}_{\cb}$.
Therefore, once again using the complete metric mapping property of $\alpha$,
\begin{multline*}
  \overleftarrow{\alpha}_n\big( (S\otimes T)_n u ; E_2, F_2\big) \\
  =
  \begin{aligned}[t]
  \sup\Big\{\alpha_n\big(  \big(q^{E_2}_{K^2_0} \otimes q^{F_2}_{L^2_0}\big)_n (S\otimes T)_n (u) &; E_2/K^2_0,F_2/L^2_0\big) \colon \\ 
  &K^2_0 \in \OCOFIN(E_2), L^2_0 \in\OCOFIN(F_2) \Big\}
  \end{aligned}
  \\
  = 
  \begin{aligned}[t] 
  \sup\Big\{\alpha_n\big( (S_{K^2_0}\otimes T_{L^2_0})_n \big(q^{E_1}_{S^{-1}K^2_0} \otimes &q^{F_1}_{T^{-1}L^2_0} \big)_n  (u); E_2/K^2_0,F_2/L^2_0\big) \\
  &\colon K^2_0 \in \OCOFIN(E_2), L^2_0 \in\OCOFIN(F_2) \Big\} 
  \end{aligned}
  \\
  \le \n{S}_{\cb}\n{T}_{\cb} 
   \begin{aligned}[t] 
  \sup\Big\{\alpha_n\big( \big(q^{E_1}_{K^1_0} \otimes q^{F_1}_{L^1_0} \big)_n  &(u); E_1/K^1_0,F_1/L^1_0\big) \colon\\ 
  &K^1_0 \in \OCOFIN(E_1), L^1_0 \in\OCOFIN(F_1) \Big\}
  \end{aligned}
  \\
  = \n{S}_{\cb}\n{T}_{\cb}  \overleftarrow{\alpha}_n\big( u ; E_1, F_1\big),
\end{multline*}
which shows the complete metric mapping property for $\overleftarrow{\alpha}$.
\end{proof}

As a consequence of the previous proposition, since we already observed that $\min = \overrightarrow{\min}$ it follows that $\min = \overrightarrow{\min}= \overleftarrow{\min}$.

We will see also in Remark \ref{h cofin} that the Haagerup o.s. tensor norm satisfies $h = \overrightarrow{h}= \overleftarrow{h}$.

Next we show that, just as in the classical theory, the projective norm is not cofinitely-generated.

\begin{proposition} \label{proj no es cofinitamente generada}
\begin{enumerate}[(a)]
    \item Let $W$ be a Banach space and $F$ a normed operator space. Then ${\Max(W) \otimes_{\overleftarrow{\proj}} F}$ and $W \otimes_{\overleftarrow{\pi}} F$ are isometrically isomorphic as Banach spaces.
    \item $\overleftarrow{\proj} \neq \proj$.
\end{enumerate}
\end{proposition}

\begin{proof}
$(a)$ 
Let $u \in W \otimes F$. From the definitions, 
\begin{multline*}
\overleftarrow{\proj}(u ; \Max(W), F) = \sup\big\{\proj\big(  (q_K^W \otimes q_L^F)(u); \Max(W)/K,F/L\big) \colon \\ K \in \OCOFIN(\Max(W)), L \in\OCOFIN(F) \big\}
\end{multline*}
and
\begin{align*}
\overleftarrow{\pi}(u ; W, F) =
\sup\left\{\pi\big(  (q_K^W \otimes q_L^F)(u); W/K,F/L\big) \colon K \in \COFIN(W), L \in \COFIN(F) \right\}.
\end{align*}
From \cite[Prop. 3.3]{Pisier-Operator-Space-Theory} we have that $\Max(W/K) = \Max(W)/K$.
Therefore, it follows from \cite[Prop. 1.5.12.(1)]{Blecher-LeMerdy-book} that
$$
\proj\big(  (q_K^W \otimes q_L^F)(u); \Max(W)/K,F/L\big) = \pi\big(  (q_K^W \otimes q_L^F)(u); W/K,F/L\big)
$$
and thus $\overleftarrow{\proj}(u ; \Max(W), F) = \overleftarrow{\pi}(u ; W, F)$.

$(b)$
Since $\overleftarrow{\pi} \neq \pi$, there exist Banach spaces $W$ and $V$ such that $W \otimes_\pi V$ and $W \otimes_{\overleftarrow{\pi}} V$ are different; it then follows from part $(a)$ and \cite[Prop. 1.5.12.(1)]{Blecher-LeMerdy-book} that
${\Max(W) \otimes_{\overleftarrow{\proj}} \Max(V)}$ and 
${\Max(W) \otimes_{\proj} \Max(V)}$ are different (even just as Banach spaces).
\end{proof}

\begin{remark}
The right-finite hull is presented in Definition~\ref{def:right-finite hull}. The left-finite hull can be defined analogously.
\end{remark}

\begin{proposition}\label{prop-dp-and-gp-finitely-generated}
For $1\le p \le \infty$, the o.s. tensor norms $d_p $ and $g_p $ are finitely generated.
\end{proposition}

\begin{proof}
By symmetry, it suffices to prove it for $d_p $.
Observe that the result will follow immediately if we can prove that for any $u \in M_n(E \otimes F)$ we have
\begin{equation}\label{eqn-dp-and-gp-finitely-generated}
(d_p )_n(u; E ,F) = \inf\{ \n{w}_{M_n( (S_{p'}^k \otimes_{\min}E) \otimes_{\proj} S_p^k[F]  )} \;:\; k\in \N, (q^{p,k})_nw = u \}
\end{equation}
where $q^{p,k} : (S_{p'}^k \otimes_{\min}E) \otimes_{\proj} S_p^k[F] \to E \otimes F$ is the tensor contraction.
Indeed, let $\varepsilon>0$ be given and choose $k\in\N$, $w \in M_n( (S_{p'}^k \otimes_{\min}E) \otimes_{\proj} S_p^k[F]  )$ such that $(q^{p,k})_nw = u$ and
$\n{w}_{M_n( (S_{p'}^k \otimes_{\min}E) \otimes_{\proj} S_p^k[F]  )}\le (1+\varepsilon)(d_p )_n(u; E ,F)$.
Since $\proj$ is finitely generated, we can $1+\varepsilon$ approximate the norm of $w$ by passing to finite-dimensional subspaces of $S_{p'}^k \otimes_{\min}E$ and $S_p^k[F]$. Without loss of generality, by enlarging the subspaces if necessary, we can assume that the aforementioned finite-dimensional subspaces are of the form $S_{p'}^k \otimes_{\min}E_0$ and $S_p^k[F_0]$ for some $E_0 \in\OFIN(E), F_0\in\OFIN(F)$, and also that $u \in M_n(E_0\otimes F_0)$.
It now follows that
$$
(d_p )_n(u; E_0 ,F_0 ) \le (1+\varepsilon)^2 (d_p )_n(u; E ,F ),
$$
showing that $d_p $ is finitely generated.

Let us now prove \eqref{eqn-dp-and-gp-finitely-generated}.
Recall that the tensor contraction 
${q}^p : \big( S_{p'}\widehat{\otimes}_{\min} E \big) \widehat{\otimes}_{\proj} S_p[F] \to E \widehat{\otimes}_{d_p } F$ is a complete 1-quotient.
By dualizing, there is a completely isometric embedding
$(E \otimes_{d_p } F)' \hookrightarrow \CB(S_{p'}\widehat{\otimes}_{\min} E , S_{p'}[F'])$.
Since the union of the spaces $S_{p'}^k \otimes_{\min} E$ is dense in $S_{p'}\widehat{\otimes}_{\min} E$, norms in $\CB(S_{p'}\widehat{\otimes}_{\min} E , S_{p'}[F'])$ (and more generally, norms of matrices over $\CB(S_{p'}\widehat{\otimes}_{\min} E , S_{p'}[F'])$) can be calculated by taking the supremum over the restrictions to $\CB(S_{p'}^k\otimes_{\min} E , S_{p'}^k[F'])$, as $k$ ranges over $\N$.
Thus, the norm of $u$ in $M_n(E \otimes_{d_p } F)$ can be calculated as the supremum of the norms of all pairings of $u$ against a $\varphi$ which is a matrix of norm at most one in $\CB(S_{p'}\widehat{\otimes}_{\min} E , S_{p'}[F'])$, and by the above we can consider instead  $\varphi$ to be a matrix of norm at most one in $\CB(S_{p'}^k \otimes_{\min} E , S_{p'}^k[F'])$, but this is exactly the same as the right-hand side of  \eqref{eqn-dp-and-gp-finitely-generated}.
\end{proof}

\begin{proposition}\label{prop-symmetrized-Haagerup-finitely-generated}
The symmetrized Haagerup o.s. tensor norms $h \cap h^t$ and $h + h^t$ are finitely generated.
\end{proposition}

\begin{proof}
Since $h$ and $h^t$ are completely injective, it is clear that so is $h\cap h^t$ and therefore the latter is finitely generated.

Consider operator spaces $E$ and $F$, $u \in M_n(E \otimes F)$, and $\varepsilon>0$.
Let $v,w \in M_n(E \otimes F)$
such that $u = v+w$ and 
$$ \|(v,w)\|_{M_n\big( (E \otimes_h F)\oplus_1(E \otimes_{h^t} F) \big)} < (h+h^t)_n(u; E,F) + \varepsilon.$$
Let $E_0 \in \OFIN(E)$ and $F_0 \in\OFIN(F)$ such that $v, w \in M_n(E_0 \otimes F_0)$.
Since $h$ and $h^t$ are both completely injective,
and so are $\ell_1$-sums (this is clear from the description of $\ell_1$-sums in Section \ref{direct-sums} and the injectivity of $\mathcal{B}(H)$), 
we have that
$$
\|(v,w)\|_{M_n\big( (E_0 \otimes_h F_0)\oplus_1(E_0 \otimes_{h^t} F_0) \big)} = \|(v,w)\|_{M_n\big( (E \otimes_h F)\oplus_1(E \otimes_{h^t} F) \big)}.
$$
Since 
$(h+h^t)_n(u; E_0, F_0) \le \|(v,w)\|_{M_n\big( (E_0 \otimes_h F_0)\oplus_1(E_0 \otimes_{h^t} F_0) \big)}$, combining the inequalities gives
$(h+h^t)_n(u; E_0, F_0) < (h+h^t)_n(u; E,F) + \varepsilon$ which implies that $h+h^t$ is finitely generated.

\end{proof}

%------------------------------------------------------
%------------------------------------------------------
%------------------------------------------------------
\section{The five basic lemmas} \label{fivebasic}

In the theory of tensor products  of normed spaces, “The Five Basic Lemmas” (see Section 13 in Defant and Floret’s book \cite{Defant-Floret}) are rather simple results which turn out to be “basic for the
understanding and use of tensor norms”. Namely, they are the \textsl{Approximation Lemma}, the
\textsl{Extension Lemma}, the \textsl{Embedding Lemma}, the \textsl{Density Lemma} and the \textsl{Local Technique
Lemma}.  We present here the
analogous results for the operator space setting and also exhibit some applications as example of
their potential. Our presentation follows the
lines of \cite{Defant-Floret}. Although the proofs are similar to the Banach space case, the operator space nature of
our tensor products introduces some difficulties, as we can see, for example, in our
version of the Extension Lemma \ref{extension-lemma}, whose proof involves additional hypothesis of local reflexivity.
%------------------------------------------------------
\subsection{The approximation lemma}

\begin{lemma}\label{approximation-lemma}
	Let $\alpha$ and $\beta$ be o.s. tensor norms (on $\ONORM$), $E$ and $F$ normed operator spaces, $c \ge 1$ and
	$$
	\alpha \le c \beta \on E \otimes N
	$$
	for cofinally many $N \in \OFIN(F)$.
	If $F$ has the completely bounded approximation property with constant $C \ge 1$, then
	$$
	\alpha \le  C c\beta \text{ on } E \otimes F.
	$$
\end{lemma}

\begin{proof} 	
	Fix $n\in\N$ and $z\in M_n(E \otimes F)$, and take $\varepsilon >0$.
	Since $F$ has the $C$-completely bounded approximation property,
	 there is a net $(T_{\eta})_\eta$ of finite rank operators in $\CB(F,F)$ with $\cb$-norm bounded by $C$ such that
	 $\n{(T_\eta)_n x - x} \to 0$ for every $x \in M_n(E)$.
	Therefore $(id_E \otimes T_\eta)_n(z)$ converges to $z$ in the projective norm, and thus also in the  $\alpha$ norm.
	Hence we have $\alpha_n(z-(id_E \otimes T_\eta)_n(z) ; E,F )<\varepsilon$ for some $\eta$ large enough. 	
	If we take a subspace $N$ containing $T_\eta(E)$ satisfying the hypothesis of the lemma, by the complete metric mapping property of the o.s. tensor norm $\beta$ we have
	\begin{align*}
	\alpha_n(z;E,F) & \leq  \alpha_n(z-(id_E \otimes T_\eta)_n(z) ; E,F ) + \alpha_n((id_E \otimes T_\eta)_n(z) ; E,F ) \\
	& \leq \varepsilon + \alpha_n((id_E \otimes T_\eta)_n(z) ; E,N ) \\
	& \leq \varepsilon + c \beta_n((id_E \otimes T_\eta)_n(z) ; E,N ) \\
	& \leq \varepsilon + c \n{T_\eta: E \to N }_{\cb} \beta_n(z; E,F) \\
	& \leq \varepsilon + C c  \beta_n(z; E,F).
	\end{align*}
	Since this holds for every $\varepsilon > 0$,  we have $\alpha_n(z;E,F) \leq  C c  \beta_n(z; E,F).$
\end{proof}

By Proposition \ref{prop-order-of-hulls} and the previous lemma we obtain the following.

\begin{corollary}\label{cor-cbap-implies-equivalence-of-hulls} 
	Let $\alpha$ be an o.s. tensor norm (on $\OFIN$), $E$ and $F$ normed operator spaces with the completely bounded approximation property with constants $C_E$ and $C_F$, respectively. Then
	$$
	 \overleftarrow{\alpha} \le  \overrightarrow{\alpha} \le C_E C_F \overleftarrow{\alpha} \on E \otimes F.
	$$
	In particular, if $E$ and $F$ both have the completely metric approximation property then $\overleftarrow{\alpha} =  \overrightarrow{\alpha} $ on $E\otimes F$.
\end{corollary}

%------------------------------------------------------
\subsection{The extension lemma}

Every $\varphi \in (E\otimes_{\proj} F)' = \CB(E, F')$ has a canonical extension $\varphi^\wedge \in (E \otimes_{\proj} F'')' = \CB(E, F'')$ which satisfies for every $x\in E$ and $y''\in F''$ the relation
$$
\pair{\varphi^\wedge}{x \otimes y''} = \pair{L_\varphi x}{y''}_{F',F''},
$$
where $L_\varphi \in \CB(E,F')$ is the linear map associated with $\varphi$.
The extension lemma deals with this situation for more general o.s. tensor norms.
Unlike in the Banach space case, local reflexivity does not come for free and we are forced to require it. The item $(a)$ of the following statement previously appeared  in \cite{dimant2015biduals}. We include the proof here for the reader's benefit. Also,  the proof of item $(b)$ is inspired by \cite[Proposition 3.1]{dimant2015biduals}) which presents the same argument just for the norm $\proj$.

\begin{lemma}\label{extension-lemma} (Right extension lemma)
	Let $\alpha$ be a  o.s. tensor norm (on $\ONORM$), $E$ and $F$  be normed operator spaces and let $\varphi \in (E\otimes_{\proj} F)'$. Then
	$$
	\varphi \in (E\otimes_{\alpha} F)' \text{ if and only if } \varphi^\wedge \in (E \otimes_\alpha F'')'
	$$
	and the map $\varphi \mapsto \varphi^\wedge$ is a complete isometry from $(E\otimes_{\alpha} F)'$ to $(E \otimes_{\alpha} F'')'$ in the following cases:
	\begin{enumerate}[(a)]
		\item $\alpha$ is finitely-generated and $F$ is locally reflexive.
        \item $\alpha$ is a $\lambda$-o.s. tensor norm. 
	\end{enumerate}
\end{lemma}

\begin{proof}
	The complete metric mapping property implies that $id_E\otimes \kappa_F:E \otimes_\alpha F \hookrightarrow E \otimes_\alpha F''$ is a complete contraction,
	and hence so is the restriction map $(id_E\otimes \kappa_F)':(E \otimes_{\alpha} F'')' \to (E\otimes_{\alpha} F)'$.
	This shows that for any $(\varphi_{ij}) \in M_m\big( (E\otimes_{\proj} F)' \big)$,
	$$
	\n{ (\varphi_{ij}) }_{ M_m( (E\otimes_{\alpha} F)' ) } \le \n{ (\varphi^\wedge_{ij}) }_{ M_m( (E\otimes_{\alpha} F'')' ) }.
	$$
	
$(a)$	Now consider $n\in\N$, $z_0 \in M_n(E \otimes F'')$, $M \in \OFIN(E)$ and $N \in \OFIN(F'')$ such that $z_0 \in M_n(M \otimes N)$.
	Since $F$ is locally reflexive, by \cite[Lemma 14.3.3]{Effros-Ruan-book}
	given $\varepsilon > 0$ there exists $R \in \CB(N,F)$ with $\n{R}_{\cb} \le 1 + \varepsilon$ such that for every $y'' \in N$, $x \in M$ and $1 \le i, j \le m$
	$$
	\pair{y''}{L_{\varphi_{ij}} x}_{F'',F'} = \pair{Ry''}{L_{\varphi_{ij}} x}_{F,F'}.
	$$
	This means that
	$$
	\pair{\varphi_{ij}^\wedge}{x \otimes y''} = \pair{\varphi_{ij}}{(id_E \otimes R)(x \otimes y'')},
	$$
	thus
	$$
	\mpair{(\varphi_{ij}^\wedge)}{z_0} = \mpair{(\varphi_{ij})}{(id_E \otimes R)z_0}
	$$
	and hence
	\begin{align*}
	\n{ \mpair{(\varphi_{ij}^\wedge)}{z_0} }_{M_{mn}} &\le \n{ (\varphi_{ij}) }_{ M_m( (E\otimes_{\alpha} F)' ) } \alpha_n((id_E \otimes R)z_0; E,  F) \\
	&\le
	 \n{ (\varphi_{ij}) }_{ M_m( (E\otimes_{\alpha} F)' ) }
	 \n{id_E \otimes R}_{\cb} \alpha_n(z_0 ; E, N) \\
	&\le
	 \n{ (\varphi_{ij}) }_{ M_m( (E\otimes_{\alpha} F)' ) }
	 \n{R}_{\cb} \alpha_n(z_0 ; E, N) \\
	&\le
	\n{ (\varphi_{ij}) }_{ M_m( (E\otimes_{\alpha} F)' ) }
	(1+\varepsilon) \alpha_n(z_0 ; E, N)
	\end{align*}
	which implies the result since $\alpha$ is finitely-generated.
	
$(b)$ Recall that by Theorem \ref{duality-lambda-tensor} we can identify $(E\otimes_{\lambda} F)' = \CB_\lambda(E \times F)$, so
all we need to show is that the map $\varphi \mapsto \varphi^\wedge$ gives a completely isometric embedding $\CB_\lambda(E \times F) \to \CB_\lambda(E \times F'')$. Notice that at the beginning of the proof it has already been shown that this is a complete contraction. That we get a complete isometry is then clear from the definition of $\n{\cdot}_{\cb,\lambda}$ and the fact that the unit ball of $M_k(F)$ is weak$^*$-dense in that of $M_k(F'')=M_k(F)''$ by Goldstine's theorem.
\end{proof}

In a similar way a left extension for each $\varphi \in (E\otimes_{\proj} F)'$ is defined: let  $^\wedge\varphi \in (E'' \otimes_\alpha F)'$  given by, for $x'' \in E''$ and $y\in F$,
$$
\pair{^\wedge\varphi}{x'' \otimes y} = \pair{x''} { R_\varphi y}_{E'',E^{'}},
$$ where $R_\varphi\in\CB(F,E') $ is the mapping $R_\varphi (y)(x)=\langle\varphi,x\otimes y\rangle$.

\begin{lemma}\label{left-extension-lemma} (Left extension lemma)
Let $\alpha$ be a  o.s. tensor norm (on $\ONORM$), $E$ and $F$  be normed operator spaces and let $\varphi \in (E\otimes_{\proj} F)'$. Then
	$$
	\varphi \in (E\otimes_{\alpha} F)' \text{ if and only if } \ ^\wedge\varphi \in (E'' \otimes_\alpha F)'.
	$$ 
	and the map $\varphi \mapsto \ ^\wedge\varphi$ is a complete isometry from $(E\otimes_{\alpha} F)'$ to $(E'' \otimes_{\alpha} F)'$ in the following cases:
	\begin{enumerate}[(a)]
		\item $\alpha$ is finitely-generated and $E$ is locally reflexive.
        \item $\alpha$ is a $\lambda$-o.s. tensor norm. 
	\end{enumerate}
 \end{lemma}
 
 Now, both procedures can be applied to a given $\varphi \in (E\otimes_{\alpha} F)'$ to obtain two (possibly different) extensions to $(E''\otimes_{\alpha} F'')'$.
 
\begin{corollary}\label{cor-extension}
Let $\alpha$ be a  o.s. tensor norm (on $\ONORM$), $E$ and $F$  be normed operator spaces and let $\varphi \in (E\otimes_{\proj} F)'$. Then the following maps from $ (E\otimes_{\alpha} F)'$ to $ (E''\otimes_{\alpha} F'')'$ are complete isometries:
$$
\varphi\mapsto \ ^\wedge(\varphi^\wedge)\quad \textrm{ and } \quad \varphi\mapsto (^\wedge\varphi)^\wedge
$$ in the following cases:
\begin{enumerate}[(a)]
		\item $\alpha$ is finitely-generated and $E$ and $F$ are locally reflexive.
        \item $\alpha$ is a $\lambda$-o.s. tensor norm. 
	\end{enumerate}
\end{corollary}

As we have already mentioned, the fact that we do not need local reflexivity in the extension lemmas if the o.s. tensor norm involved is $\proj$ was proved in \cite[Proposition 3.1]{dimant2015biduals}.
In the same article this fact is also shown  for  the Haagerup o.s. tensor norm $h$ using that it satisfies the following property: for any operator spaces $E$ and $F$ and $\varphi \in (E \otimes_h F)'$, there is a unique extension $\widetilde\varphi \in (E'' \otimes_h F'')'$ which is separately $w^*$-continuous, see \cite[1.6.7]{Blecher-LeMerdy-book}. 

A word can be said about the equality/non-equality of the two extensions. Repeating the arguments of \cite[Cor. 1.9]{Defant-Floret} we obtain the analogous result in our setting. That is, given $\varphi \in (E\otimes_{\proj} F)'$, we have $ ^\wedge(\varphi^\wedge)=(^\wedge\varphi)^\wedge$ if and only if $L_\varphi:E\to F'$ is weakly compact. This is also equivalent to the existence of an extension of $\varphi$ to $ (E''\otimes_{\proj} F'')'$ that is separately $w^*$-continuous. Note that, by the comment above, if $\varphi \in (E \otimes_h F)'$ then $L_\varphi$ is weakly compact.
%\end{remark}
%------------------------------------------------------

 Note that we have the separating duality pair $\big\langle E \otimes F'', (E\otimes_{\alpha}F)' \big\rangle$ given by
 $$ \langle z, \varphi \rangle = \varphi^{\wedge}(z).$$

Using the bipolar theorem it is easy to see that, in the Banach space realm, the unit ball
$B_{E\otimes_{\alpha}F}$ is $\sigma(E\otimes F'',(E\otimes_{\alpha}F)')$-dense in the unit ball $B_{E \otimes_\alpha F''}$, whenever $\alpha$ is a finitely generated tensor norm.
In other words, given $z \in B_{E \otimes\alpha F''}$ there is a net $(u_{\eta})_{\eta} \subset B_{E\otimes_{\alpha}F}$ such that for every $\varphi \in (E\otimes_{\alpha}F)'$, 
$$\varphi(u_{\eta}) \to \varphi^{\wedge}(z).$$

In our setting, we can consider for each $n \in \mathbb{N}$, the following duality pair:

$$ \big\langle M_n(E \otimes F''),S_1^n[(E\otimes_{\alpha}F)'] \big\rangle$$ defined as
$$ \langle (z_{ij}),(\varphi_{ij}) \rangle = \sum_{i,j} \varphi_{ij}^{\wedge}(z_{ij}).$$

It is easy to see that this dual pairing system is separating.
Moreover, we have the following density result.

\begin{corollary}
	Let $\alpha$ be a  o.s. tensor norm (on $\ONORM$), $E$ and $F$  be normed operator spaces.
	
	For every $n\in \mathbb{N}$, the unit ball $B_{M_n(E\otimes_{\alpha} F)}$ is $\sigma(M_n(E \otimes F''),S_1^n[(E\otimes_{\alpha}F)'])$-dense in the unit ball $B_{M_n(E\otimes_{\alpha} F'')}$
	in the following cases:
	\begin{enumerate}[(a)]
		\item $\alpha$ is finitely-generated and $F$ is locally reflexive.
        \item $\alpha$ is a $\lambda$-o.s. tensor norm.
        \end{enumerate}
\end{corollary}

\begin{proof}
Since $\left(M_n(E\otimes_{\alpha} F)\right)'$ is $S_1^n[(E\otimes_{\alpha} F)']$ isometrically, it is clear that the polar  $\left(B_{M_n(E\otimes_{\alpha} F)}\right)^{\circ}$ (in this pairing system) is $B_{S_1^n[(E\otimes_{\alpha} F)']}.$
By the bipolar theorem it is enough to prove that $B_{M_n(E\otimes_{\alpha} F'')} \subset \left(B_{S_1^n[(E\otimes_{\alpha} F)']}\right)^{\circ}.$
Indeed, given $(z_{ij}) \in B_{M_n(E\otimes_{\alpha} F'')}$ and $(\varphi_{ij})\in B_{S_1^n[(E\otimes_{\alpha} F)']}$; by the Extension Lemma \ref{extension-lemma} along with \cite[Thm. 4.1.8]{Effros-Ruan-book} we know that $(\varphi_{ij}^{\wedge})\in B_{S_1^n[(E\otimes_{\alpha} F'')']}$. Using again 
$\left(M_n(E\otimes_{\alpha} F'')\right)'=S_1^n[(E\otimes_{\alpha} F'')']$ isometrically, we derive that that 
$$ \left|\langle (z_{ij}),(\varphi_{ij}) \rangle\right| = \left|\sum_{i,j} \varphi_{ij}^{\wedge}(z_{ij})\right| \leq 1.$$ This concludes the proof.
\end{proof}

\subsection{The embedding lemma}

Operator space tensor norms generally do not respect subspaces, but the embedding into the bidual is respected under certain conditions.

\begin{lemma}\label{embedding-lemma}
	Let $\alpha$ be an o.s. tensor norm (on $\ONORM$).
	The mapping
		$$
	id_E \otimes \kappa_F : E \otimes_\alpha F \to E \otimes_\alpha F''
	$$
	is a complete isometry in the following cases:
	\begin{enumerate}[(a)]
		\item $\alpha$ is finitely-generated and $F$ is locally reflexive.
		\item $\alpha$ is cofinitely-generated.
        \item $\alpha$ is a $\lambda$-o.s. tensor norm. % $\alpha= \proj$.
	\end{enumerate}
\end{lemma}

\begin{proof}
	 Let $n\in\N$ and $z \in M_n(E \otimes_\alpha F)$.
	By the complete metric mapping property
	$$
	id_E \otimes \kappa_F : E \otimes_\alpha F \to E \otimes_\alpha F''
	$$
	is always a complete contraction, so
	$$
	\alpha_n( z; E, F'') \le \alpha_n( z; E, F).
	$$
	$(a)$ By the extension lemma,
		\begin{align*}
			\alpha_n(z ; E, F) &= \sup \big\{ \n{\mpair{z}{ \Phi }}_{M_{nm}} \; : \; \Phi \in M_m((E \otimes_\alpha F )'),  \n{\Phi}_{M_m( (E \otimes_\alpha F )')} \le 1 \big\} \\
			&= \sup \big\{ \n{\mpair{z}{ \Phi^\wedge }}_{M_{nm}} \; : \; \Phi \in M_m((E \otimes_\alpha F )'),  \n{\Phi}_{M_m( (E \otimes_\alpha F )')} \le 1 \big\} \\
			&\le \sup \big\{ \n{\mpair{z}{ \Psi }}_{M_{nm}} \; : \; \Psi \in M_m((E \otimes_\alpha F'' )'),  \n{\Psi}_{M_m( (E \otimes_\alpha F'' )')} \le 1 \big\} \\
			&= \alpha_n(z ; E ,F''),
		\end{align*}
		which gives the reverse inequality.

$(b)$ Let $L \in \OCOFIN(F)$, then $L^{00}$ (formed in $F''$) is in $\OCOFIN(F'')$ and the map
		$$
		\kappa_{F/L} : F/L \to (F/L)'' = F''/L^{00}
		$$
		is completely isometric and surjective; moreover, $q^{F''}_{L^{00}} \circ \kappa_F = \kappa_{F/L} \circ q^F_L$.
		Therefore,
		\begin{align*}
			\alpha_n\big((id_E \otimes q_L^F)_n(z); E,F/L\big)
			&= \alpha_n\big((id_E \otimes (\kappa_{F/L} \circ q^F_L) )_n(z); E,(F/L)''\big) \\
			&= \alpha_n\big((id_E \otimes q^{F''}_{L^{00}}) \circ (id_E \otimes \kappa_F) )_n(z); E, F''/L^{00} \big) \\
			&\le \overleftarrow{\alpha}_n( (id_E \otimes \kappa_F) )_n(z); E, F'' ) \\
			&=  \overleftarrow{\alpha}_n( z; E, F'' ) = \alpha_n(z ; E, F'')
		\end{align*}
		The same argument with $E/K$ instead of $E$ allows us to conclude $\alpha_n(z ; E, F) \le \alpha_n(z ; E, F'')$.

$(c)$ The same proof as in the case $(a)$ works here because a $\lambda$-o.s. tensor norm is finitely-generated and the hypothesis of local reflexivity of $F$ is not needed for the extension lemma with this norm. 

\end{proof}

We point out that part (c) above is a  generalization of \cite[Lemma 11]{Janson-Kumar-spectra}, where the spaces are assumed to be $C^*$-algebras and the embedding is only shown to be an isometry instead of a complete isometry.

Note that the completion of an operator space $F$ has the same bidual as $F$. Therefore, as a consequence of the previous result we have:

\begin{corollary} \label{Cor:extender al completado}
Let $\alpha$ be an o.s. tensor norm (on $\ONORM$) and let $\widetilde F$ the completion of a normed operator space $F$.
	The mapping
	%If $\alpha$ is a finitely or cofinitely-generated cross-norm , then
	$$
	id_E \otimes \iota_F : E \otimes_\alpha F \to E \otimes_\alpha \widetilde F,
	$$
	where $\iota_F : F \to \widetilde F$ is the canonical inclusion, 
	is a complete isometry with dense range in the following cases:
	\begin{enumerate}[(a)]
		\item $\alpha$ is finitely-generated and $F$ is locally reflexive.
		\item $\alpha$ is cofinitely-generated.
        \item $\alpha$ is a $\lambda$-o.s. tensor norm. % $\alpha= \proj$.
	\end{enumerate}
\end{corollary}

\subsection{The density lemma}

We now state the so-called density lemma, which asserts that a given completely bounded mapping defined on the tensor product of certain normed operator spaces can be extended to a completely bounded mapping defined on the tensor product of the completions of the spaces. As usual, we need a local reflexivity hypothesis for the finitely-generated case. The proof easily follows from Corollary \ref{Cor:extender al completado}.

\begin{lemma}
Let $\alpha$ be an o.s. tensor norm  (on $\ONORM$), $E, F$ and $G$ normed operator spaces, $E_0$ and $F_0$ dense subspaces of $E$ and $F$ respectively. Suppose $T \in CB(E \otimes_{\proj}F; G)$ such that $T|_{E_0 \otimes_{\alpha}F_0} \in CB(E_0 \otimes_{\alpha}F_0; G)$.  Then $T \in CB(E \otimes_{\alpha}F; G)$ and
$$
\Vert T \Vert_{ CB(E_0 \otimes_{\alpha}F_0; G)}=\Vert T \Vert_{ CB(E \otimes_{\alpha}F; G)},
$$
in the following cases:
\begin{enumerate}[(a)]
		\item $\alpha$ is finitely-generated and $F$ is locally reflexive.
		\item $\alpha$ is cofinitely-generated.
		\item $\alpha$ is a $\lambda$-o.s. tensor norm.
        %\item $\alpha= \proj$.
	\end{enumerate}
\end{lemma}

%------------------------------------------------------
\subsection{The $\mathscr{OS}_{p}$-local technique lemma}
In the Banach space setting the $\mathscr L_p$-local technique lemma allow us to transfer results from tensor products of $\ell_p^n$ spaces to general $\mathscr L_p$-spaces (which, loosely speaking, are space whose finite dimensional subspaces ``uniformly look''  like $\ell_p^n$'s).
In the operator space setting, one of the possible extension of the $\mathscr L_p$-spaces are the $\mathscr{OS}_{p}$-spaces  (see Section~\ref{usual notation} for the definition).
For this extension, we have the following local technique lemma.

\begin{lemma}\label{local-technique-lemma}
	Let $\alpha$ and $\beta$ be o.s. tensor norms, $c \ge 0$ and $E$ a normed operator space such that for all $k \in \N$,
	$$
	\alpha \le c \beta \on E \otimes S_p^k.
	$$
	Then $\alpha \le C c \overrightarrow{\beta} \on E \otimes F$
	for every $\mathscr{OS}_{p,C}$-space $F$.
\end{lemma}

\begin{proof}
	For $N \in \OFIN(F)$, by \cite[Lemma 2.1]{Junge-Nielsen-Ruan-Xu} we can take a factorization
	$$
	\xymatrix{
	N \ar[rr]^{id_N} \ar[dr]_{R} & &F\\
	& S_p^k \ar[ru]_{S} &
	}
	$$
	with $\n{R}_{\cb} \n{S}_{\cb} \le C(1+\varepsilon)$. Let $M := S(S_p^k)$.
	Then for every $z \in M_n(E \otimes N)$,
	\begin{multline*}
		\alpha_n(z ; E, M) = \alpha_n( (id_E \otimes (S \circ R))_n z; E , M )
		\le \n{S}_{\cb} \alpha_n( (id_E \otimes R)_nz ; E, S_p^k ) \\
		\le \n{S}_{\cb} c \beta_n( (id_E \otimes R)_n z ; E , S_p^k )
		\le \n{S}_{\cb} \n{R}_{\cb} c \beta_n(z ; E , N).
	\end{multline*}
	The statement follows.
\end{proof}

\begin{remark}
Note that more generally the proof above applies whenever $F$ has the $\gamma_p$-AP \cite[Sec. 1]{Junge-Nielsen-Ruan-Xu}, thanks to \cite[Lemma 1.4]{Junge-Nielsen-Ruan-Xu}.
	Also note that the same argument yields versions of the local technique lemma for other operator space versions of the $\mathscr{L}_p$ spaces, e.g. the $\mathscr{OL}_p$ spaces from \cite{Junge-Nielsen-Ruan-Xu}.
\end{remark}

%------------------------------------------------------
%------------------------------------------------------
%------------------------------------------------------
\section{Dual o.s. tensor norms}\label{dual section}

In a nutshell, the dual o.s. tensor norm  $\alpha'$ for the o.s. tensor norm  $\alpha$ is the one that makes
$$
E \otimes_{\alpha'} F = (E' \otimes_\alpha F')'
$$
into a complete isometry for any $E,F \in \OFIN$.
More explicitly, when $E,F \in\OFIN$ for $z \in M_n( E \otimes F)$ we define
$$
\alpha'_n(z ; E ,F) := \sup \big\{  \n{ \mpair{u}{z} }_{M_{mn}}  \mid \alpha_m(u ; E' , F') \le 1 \big\}.
$$

Let us check that, for example, the minimal and projective o.s. tensor norms  are in duality with each other.
Consider finite-dimensional operator spaces $E$ and $F$. The operator space structure on $E \otimes_{\min} F$ is, by definition, the one induced by the embedding $E \otimes_{\min} F \to\CB(E',F)$. Now, $(E' \otimes_{\proj} F')'$ is completely isometric to $\CB(E', F'') = \CB(E',F)$, from where we conclude that the canonical identification
$$
E \otimes_{\min} F = (E' \otimes_{\proj} F')'
$$
is a complete isometry. Taking duals and exchanging the roles of $E,F$ and $E',F'$ we also have a complete isometry
$$
E \otimes_{\proj} F = (E' \otimes_{\min} F')'.
$$

\begin{proposition}
If $\alpha$ is an o.s. tensor norm on $\OFIN$, then so is $\alpha'$.
\end{proposition}

\begin{proof}
Let $E, F$ be finite-dimensional operator spaces.
It is clear that $\alpha'$ is an operator space structure on $E \otimes F$.
From $\min \le \alpha \le \proj$, taking duals we conclude $\min = \proj' \le \alpha' \le \min'=\proj$ (for finite-dimensional spaces).
The same sort of duality argument shows that $\alpha'$ is uniform.
\end{proof}

The finite hull $\overrightarrow{\alpha'}$ of $\alpha'$ will be called the \emph{dual o.s. tensor norm} $\alpha'$ (on $\ONORM$) of the o.s. tensor norm  $\alpha$ (on $\OFIN$ or $\ONORM$).

The following properties are easy to check:
\begin{proposition}\label{prop-properties-dual-cross-norm}
\begin{enumerate}[(a)]
\item If $\alpha \le c \beta$, then $\beta' \le c \alpha'$.
\item $\alpha = \alpha''$ on $\OFIN$, and $\overrightarrow{\alpha}= \alpha''$.
\item $\alpha = \alpha''$ on $\ONORM$ if and only if $\alpha$ is finitely-generated.
\end{enumerate}
\end{proposition}

Since $\min$ and $\proj$ are finitely-generated, $\min'=\proj$ and $\proj'=\min$. Also, $h'=h$ (see, for instance, \cite[Chapter 9]{Effros-Ruan-book}).
Additionally, the symmetrized Haagerup o.s. tensor norms are in duality with each other.

\begin{proposition}\label{prop-duality-symmetrized-Haagerup}
$(h \cap h^t)' = h + h^t$ and $(h+h^t)' = h \cap h^t$.
\end{proposition}

\begin{proof}
By Proposition \ref{prop-symmetrized-Haagerup-finitely-generated} both $h \cap h^t$ and $h+h^t$ are finitely-generated, so it suffices to check that they are in duality for finite-dimensional spaces. Let $E, F \in \OFIN$. By taking the adjoint of the canonical complete isometry $E \otimes_{h \cap h^t} F \hookrightarrow (E \otimes_h F) \oplus_\infty (E \otimes_{h^t} F)$ we get a canonical complete quotient 
$(E \otimes_h F)' \oplus_1 (E \otimes_{h^t} F)' \twoheadrightarrow (E \otimes_{h \cap h^t} F)'$.
But the self-duality of $h$ and $h^t$ implies $(E \otimes_h F)' = E' \otimes_h F'$ and $(E \otimes_{h^t} F)' = E' \otimes_{h^t} F'$, so we get a complete quotient 
$(E' \otimes_h F') \oplus_1 (E' \otimes_{h^t} F') \twoheadrightarrow (E \otimes_{h \cap h^t} F)'$. This is exactly the same quotient giving the o.s. structure to $E' \otimes_{h+h^t} F'$, and thus $E' \otimes_{h+h^t} F' = (E \otimes_{h \cap h^t} F)'$ as claimed.
\end{proof}

\begin{definition}
Given an o.s. tensor norm  $\alpha$, its \emph{adjoint or contragradient} o.s. tensor norm is defined as $\alpha^* = (\alpha^t)' = (\alpha')^t$.
\end{definition}

As in the classical case, the following theorem tells us in what sense the completely isometric embedding
$$
E \otimes_{\alpha'} F \to (E' \otimes_\alpha F')',
$$
valid for finite-dimensional spaces, extends to infinite-dimensional ones.

\begin{theorem}[The duality theorem]\label{duality-theorem}
Let $\alpha$ be an o.s. tensor norm  (on $\OFIN$) and let $E$, $F$ be normed operator spaces. The following natural mappings are complete isometries:
\begin{align}
\label{eqn-duality-3}E \otimes_{\overleftarrow{\alpha}} F  &\hookrightarrow (E' \otimes_{\alpha'} F')'
\\
\label{eqn-duality-2}E' \otimes_{\overleftarrow{\alpha}} F  &\hookrightarrow (E \otimes_{\alpha'} F')' \qquad \textrm{ whenever } E \textrm{ is locally reflexive}\\
\label{eqn-duality-1}E' \otimes_{\overleftarrow{\alpha}} F'  &\hookrightarrow (E \otimes_{\alpha'} F)' \qquad \textrm{ whenever } E \textrm{ and } F \textrm{ are locally reflexive}
\end{align}
\end{theorem}

\begin{proof}
To prove \eqref{eqn-duality-3}, observe first that
$$
\OFIN(E') = \{ K^0 \mid K \in \OCOFIN(E) \}
$$
and that for $(K,L) \in \OCOFIN(E) \times \OCOFIN(F)$, $z \in M_n(E \otimes F)$ and $u \in M_m(K^0 \otimes L^0 ) \subset M_m(E' \otimes F' )$,
we have
$$
\mpair{z}{u} = \mpair{ (q_K^E \otimes q_L^F)(z) }{u}
$$
Now, by the valid duality relation for finite-dimensional spaces,
\begin{align*}
\overleftarrow{\alpha}_n(z ; E, F) &= \sup_{K,L}\left\{\alpha_n\big(  (q_K^E \otimes q_L^F)(z); E/K,F/L\big)  \right\}\\
&= \sup_{K,L} \sup \left\{ \n{  \mpair{(q_K^E \otimes q_L^F)(z)}{ u} }_{M_{mn}} \mid \alpha'_m(u ; K^0, L^0) <1 \right\} \\
&= \sup \left\{ \n{  \mpair{z}{ u} }_{M_{mn}} \mid \overrightarrow{\alpha'}_m(u ; E', F') <1 \right\} \\
\end{align*}
and this is precisely \eqref{eqn-duality-3}.

Consider now the commutative diagram
$$
\xymatrix{
	E' \otimes_{\overleftarrow{\alpha}} F\ar[r]^{\phi_1} \ar[dr]  & (E'' \otimes_{\alpha'} F')'\\
	& (E \otimes_{\alpha'} F')' \ar[u]_{\phi_2}  \\
}
$$
Note that $\phi_1$ is a complete isometry by \eqref{eqn-duality-3}, whereas $\phi_2$ is a complete isometry by
the Extension Lemma \ref{extension-lemma}; \eqref{eqn-duality-2} follows.
In the same way, in the diagram
$$
\xymatrix{
	E' \otimes_{\overleftarrow{\alpha}} F' \ar[r]^{\psi_1} \ar[dr]  & (E \otimes_{\alpha'} F'')'\\
	& (E \otimes_{\alpha'} F)' \ar[u]_{\psi_2}  \\
}
$$
$\psi_1$ is a complete isometry by \eqref{eqn-duality-2} and $\psi_2$ is a complete isometry the Extension Lemma \ref{extension-lemma}; \eqref{eqn-duality-1} follows.
\end{proof}

\begin{remark}\label{extension-sin-loc-refl}
The Extension Lemma \ref{extension-lemma} is valid without the hypothesis of local reflexivity for $\lambda$-o.s. tensor norms.
 Hence, in equations \eqref{eqn-duality-2} and \eqref{eqn-duality-1} of the Duality Theorem, if $\alpha'$ is a $\lambda$-tensor norm the local reflexivity of the involved spaces is not needed. 
 
 In particular, for $\alpha'=\proj$, recalling that $\overleftarrow{\min}=\min$  we obtain for any normed operator spaces $E$ and $F$  the following known completely isometric embeddings
$$
E' \otimes_{\min} F  \hookrightarrow (E \otimes_{\proj} F')'\quad \textrm{ and }\quad E' \otimes_{\min} F'  \hookrightarrow (E \otimes_{\proj} F)'.
$$
\end{remark}

One nice consequence of the duality theorem is the following:

\begin{corollary}%DF 15.7
If $E$ and $F$ are normed operator spaces with the completely metric approximation property,
%\textcolor{red}{CMAP not defined! CBAP defined in the next section}, UPDATE: DEFINED IN INTRODUCTION
then the natural mappings
$$
E \otimes_\alpha F \to E\otimes_{\overleftarrow{\alpha}} F \qquad \text{and} \qquad E \otimes_{\alpha} F \to (E' \otimes_{\alpha'} F')'
$$
are complete isometries.
\end{corollary}

\begin{proof}
Proposition \ref{prop-order-of-hulls} and Corollary \ref{cor-cbap-implies-equivalence-of-hulls} imply the complete isometry between  $E \otimes_\alpha F$ and $E\otimes_{\overleftarrow{\alpha}} F$, now  \eqref{eqn-duality-3} implies the rest.
\end{proof}

Observe that if we replace the CMAP property by the CBAP property the above identifications become complete isomorphisms (with norms controlled by the CBAP constant).

We now show that the Haagerup o.s. tensor norm  is both finitely and cofinitely-generated.

\begin{remark} \label{h cofin}
Let $h$ be the Haagerup o.s. tensor norm. Then,
$$h=\overleftarrow{h}=\overrightarrow{h}.$$
\end{remark}

\begin{proof}
The fact that $h=\overrightarrow{h}$ easily follows since the Haagerup o.s. tensor norm respects complete isometries \cite[Proposition 9.25]{Effros-Ruan-book}.

To see that $h=\overleftarrow{h}$ notice that by the Duality Theorem \ref{duality-theorem} the canonical embedding below 
is a complete isometry
\begin{equation}\label{inclusion h1}
E \otimes_{\overleftarrow{h}} F  \hookrightarrow (E' \otimes_{h'} F')'=(E' \otimes_{h} F')',
\end{equation}
where we have used in the last equality the self-dual property of the o.s. tensor norm  $h$  \cite[Corollary 9.4.8]{Effros-Ruan-book}. 

On the other hand,  by \cite[Theorem 9.4.7.]{Effros-Ruan-book} and the Embedding Lemma \ref{embedding-lemma} we have that
\begin{equation}\label{inclusion h2}
E \otimes_{h} F  \hookrightarrow  E'' \otimes_{h} F'' \hookrightarrow (E' \otimes_{h} F')',
\end{equation}
is a complete isometry.

Since both inclusions in equations \eqref{inclusion h1}, \eqref{inclusion h2} are complete isometries, $h$ and $\overleftarrow{h}$ must coincide as o.s. tensor norms.
\end{proof}

In general, o.s. tensor norms are not both finitely and cofinitely-generated, as we have just seen for the Haagerup o.s. tensor norm and the minimal o.s. tensor norm.
For many purposes it is often enough that the equality $\overrightarrow{\alpha}=\overleftarrow{\alpha}$
holds just for finite dimensional spaces or locally reflexive spaces. This motivates the following definitions that we will use in the sequel.

\begin{definition}
An o.s. tensor norm  $\alpha$ is called \emph{right-accessible} if
$$
\overrightarrow{\alpha}_n(\cdot ; M, F) = \overleftarrow{\alpha}_n(\cdot ; M, F)  \qquad \text{for all }  M \in \OFIN, F \in \ONORM, n\in \N
$$
and \emph{left-accessible} if
$$
\overrightarrow{\alpha}_n(\cdot ; E, M) = \overleftarrow{\alpha}_n(\cdot ; E, M)  \qquad \text{for all }  M \in \OFIN, E \in \ONORM, n \in \N.
$$
An o.s. tensor norm that is both left and right-accessible is called \emph{accessible}, and \emph{totally accessible} means that $\overrightarrow{\alpha} = \overleftarrow{\alpha}$.
\end{definition}

In many cases, local reflexivity  naturally appears when dealing with accessibility. Based on this we introduce the following:

\begin{definition}
An o.s. tensor norm $\alpha$ is called \emph{locally right-accessible if}
$$
\overrightarrow{\alpha}_n(\cdot ; M, F) = \overleftarrow{\alpha}_n(\cdot ; M, F)  \qquad \text{for all }  M \in \OFIN, F \in \OLOC, n\in \N
$$
and \emph{locally left-accessible} if
$$
\overrightarrow{\alpha}_n(\cdot ; E, M) = \overleftarrow{\alpha}_n(\cdot ; E, M)  \qquad \text{for all }  M \in \OFIN,E \in \OLOC, n \in \N.
$$
An o.s. tensor norm that is both locally left and locally right-accessible is called \emph{locally accessible}, and \emph{locally totally accessible} means that $\overrightarrow{\alpha} = \overleftarrow{\alpha}$ in $\OLOC \otimes \OLOC$.
\end{definition}

We have already seen that the minimal and the Haagerup o.s. tensor norms are both finitely and cofinitely-generated, or in other words they are totally accessible. 
Recall that, as we saw in Proposition \ref{proj no es cofinitamente generada}, the projective o.s. tensor norm  is not cofinitely-generated, hence it is not totally accessible. Nevertheless, as in the Banach space framework, $\proj$ is indeed accessible.
To see this, let  $E\in \ONORM$ and $M \in \OFIN$. We always have that the natural mapping 
\begin{equation} \label{inclusion natural proj}
E \otimes_{\proj}M \to (E' \otimes_{\min}M')'
\end{equation}
is a complete isometry (see \cite[Theorem 2.2]{effros1990approximation} or Remark \ref{isometria-proj-min}). On the other hand, as a consequence of the Duality Theorem \ref{duality-theorem} we also have that the mapping

$$
E \otimes_{\overleftarrow{\proj}}M \to (E' \otimes_{\min}M')'
$$
 is a complete isometry. Thus, both o.s. tensor norms, $\overrightarrow{\proj}=\proj$ and $\overleftarrow{\proj}$, must coincide on $E \otimes M$. This shows the left-accessibility of $\proj$. Analogously, we have that $\proj$ is also right-accessible and therefore accessible.

More generally, a similar argument will show that any $\lambda$-o.s. tensor norm is accessible. 
Let us start with a formula for calculating the dual of a $\lambda$-o.s. tensor norm, which is essentially \cite[Prop. 5.3]{Wiesner}.

\begin{proposition}
Let $E$ and $F$ be operator spaces, and $u \in M_n(E \otimes F)$. For any $\lambda$-o.s. tensor norm,
$$
\lambda'_n(u; E,F) = \sup\Big\{ \n{ \big( \otimes_{\bl_k}(\phi,\psi)\big)_n(u) }_{M_{n\tau(k)}} \;:\; k\in\N,  \n{\phi}_{M_k(E')}, \n{\psi}_{M_k(F')} \le 1 \Big\}
$$
\end{proposition}

\begin{proof}
Assume first that $E$ and $F$ are finite-dimensional.
By the definition of $\lambda'$ and Theorem \ref{duality-lambda-tensor},
$$
E \otimes_{\lambda'}F = (E' \otimes_\lambda F')' = \CB_\lambda(E'\times F').
$$
Therefore, $\lambda'_n(u; E,F)$ is the norm $\n{\Phi_u}_{\cb,\lambda}$ of the bilinear map $\Phi_u : E' \times F' \to M_n$ associated to $u$.
By definition, said norm is
$$
\n{\Phi_u}_{\cb,\lambda} = \sup_{k\in\N} \big\{\n{(\Phi_u)_{\bl_k}(\phi,\psi)}_{M_{n\tau(k)}} \;:\; \n{\phi}_{M_k(E')}, \n{\psi}_{M_k(F')} \le 1 \},
$$
so it suffices to check
$$
\big( \otimes_{\bl_k}(\phi,\psi)\big)_n(u) = (\Phi_u)_{\bl_k}(\phi,\psi)
$$
which follows easily from the definitions.

Let us now consider the case of arbitrary $E$ and $F$.
Let $E_0 \in \OFIN(E)$ and $F_0 \in \OFIN(F)$ such that
$u \in M_n(E_0 \otimes F_0)$.
By the previous argument,
$$
\lambda'_n(u; E_0,F_0) = \sup\Big\{ \n{ \big( \otimes_{\bl_k}(\phi_0,\psi_0)\big)_n(u) }_{M_{n\tau(k)}} \;:\; k\in\N,  \n{\phi_0}_{M_k(E_0')}, \n{\psi_0}_{M_k(F_0')} \le 1 \Big\}.
$$
By the Arveson extension theorem, $\phi_0 \in M_k(E_0') = \CB(E_0,M_k)$ and $\psi_0 \in M_k(F_0') = \CB(F_0,M_k)$
admit respective norm-preserving extensions $\phi \in \CB(E,M_k)$ and $\psi \in \CB(F,M_k)$.
Since clearly
$$
\big( \otimes_{\bl_k}(\phi,\psi)\big)_n(u) = \big( \otimes_{\bl_k}(\phi_0,\psi_0)\big)_n(u),
$$
it follows that
$$
\lambda'_n(u; E_0,F_0) = \sup\Big\{ \n{ \big( \otimes_{\bl_k}(\phi,\psi)\big)_n(u) }_{M_{n\tau(k)}} \;:\; k\in\N,  \n{\phi}_{M_k(E')}, \n{\psi}_{M_k(F')} \le 1 \Big\},
$$
and therefore $\lambda'_n(u; E,F)$ is also equal to the same quantity.
\end{proof}

As an immediate consequence, we have the following result (which in \cite[Def. 5.1]{Wiesner} is taken as the definition for the dual of $\lambda$).

\begin{corollary}
Let $E$ and $F$ be operator spaces. For any $\lambda$-o.s. tensor norm, the natural map
$$
E \otimes_{\lambda'} F \hookrightarrow \CB_\lambda(E' \times F')
$$
is a complete isometry.
\end{corollary}

\begin{corollary}
Let $E$ and $F$ be operator spaces. For any $\lambda$-o.s. tensor norm, the natural map
$$
E' \otimes_{\lambda'} F' \hookrightarrow \CB_\lambda(E \times F)
$$
is a complete isometry.
\end{corollary}

\begin{proof}
From the previous result, the natural map
\begin{equation}\label{eqn-embedding-lambda-prime}
  E' \otimes_{\lambda'} F' \hookrightarrow \CB_\lambda(E'' \times F'')  
\end{equation}
is a complete isometry.
As we have already observed (see the proof of the Extension Lemma \ref{extension-lemma}), we also have a canonical completely isometric embedding
\begin{equation}\label{eqn-embedding-lambda-prime-2}
\CB_\lambda(E \times F) \hookrightarrow \CB_\lambda(E'' \times F'').
\end{equation}
It is clear that the image of the map in \eqref{eqn-embedding-lambda-prime} is contained in that of the map in \eqref{eqn-embedding-lambda-prime-2}, which gives the desired result.
\end{proof}

\begin{proposition}
Let $E$ and $F$ be operator spaces, with $E$ finite-dimensional. For any $\lambda$-o.s. tensor norm, the natural map
$$
E \otimes_{\lambda} F \hookrightarrow (E' \otimes_{\lambda'} F')'
$$
is a complete isometry.
\end{proposition}

\begin{proof}
From the previous result, the natural map
$$
  E' \otimes_{\lambda'} F' \hookrightarrow \CB_\lambda(E'' \times F'')  
$$
is a complete isometry, and since $E$ is finite-dimensional it follows that it is surjective and therefore 
$E' \otimes_{\lambda'} F' = \CB_\lambda(E'' \times F'')$.
Taking duals and using Theorem \ref{duality-lambda-tensor}  we have
$$
(E \otimes_\lambda F'')''=\CB_\lambda(E'' \times F'')'=(E' \otimes_{\lambda'} F')',
$$
which, together with  the Embedding Lemma \ref{embedding-lemma}, the desired result follows.
\end{proof}

\begin{theorem}
 Any $\lambda$-o.s. tensor norm is accessible.
\end{theorem}

\begin{proof}
Let $E$ and $F$ be operator spaces, with $E$ finite-dimensional. From the previous result the natural map
$$
E \otimes_{\lambda} F \hookrightarrow (E' \otimes_{\lambda'} F')'
$$
is a complete isometry, and from the Duality Theorem \ref{duality-theorem} so is
$$
E \otimes_{\overleftarrow{\lambda}} F  \hookrightarrow (E' \otimes_{\lambda'} F')',
$$
which means that  $E \otimes_{\lambda} F = E \otimes_{\overleftarrow{\lambda}} F$ and therefore $\lambda$ is right-accessible.
The proof for left-accessibility is analogous.
\end{proof}

Left-accessibility of an o.s. tensor norm  $\alpha$ provides a complete isometry analogous to the one that appears (for the  o.s. tensor norm $\proj$) in Equation \eqref{inclusion natural proj}. This fact combined with a CMAP hypothesis imply the following result.
\begin{corollary}
\begin{enumerate}[(a)]
\item Let $E  \in \ONORM$, $F$ a normed operator space with the CMAP and $\alpha$ a left-accessible o.s. tensor norm. Then,
    $$
E \otimes_{\overrightarrow{\alpha}} F \to E\otimes_{\overleftarrow{\alpha}} F \qquad \text{and} \qquad E \otimes_{\overrightarrow{\alpha}} F \to (E' \otimes_{\alpha'} F')'
$$ are complete isometries.

\item Let $E  \in \OLOC$, $F$ a normed operator space with the CMAP and $\alpha$ a locally left-accessible o.s. tensor norm.  Then,
$$
E \otimes_{\overrightarrow{\alpha}} F \to E\otimes_{\overleftarrow{\alpha}} F  \qquad \text{and}  \qquad E \otimes_{\overrightarrow{\alpha}} F \to (E' \otimes_{\alpha'} F')'
$$ are complete isometries.
\item The analogous results hold with right-accessibility/local right-accesibility.
\end{enumerate}
\end{corollary}

\begin{proof}
We will only prove $(a)$ since the proofs of the  other statements are similar. Note that, by the left-accesibility of the o.s. tensor norm we have
\begin{equation}
    \overrightarrow{\alpha} = \overleftarrow{\alpha} \qquad  \text{on} \qquad  E \otimes M
\end{equation} 
for every $M \in \OFIN(F).$ Now by the Approximation Lemma \ref{approximation-lemma}, since $F$ has the CMAP then $\overrightarrow{\alpha}$ and $\overleftarrow{\alpha}$ coincide on $E \otimes F$, and therefore 
$$
E \otimes_{\overrightarrow{\alpha}} F \to E\otimes_{\overleftarrow{\alpha}} F
$$
is a complete isometry.
The fact that $E \otimes_{\overrightarrow{\alpha}} F \to (E' \otimes_{\alpha'} F')'$ is also  a complete isometry follows from the previous identification and the Duality Theorem \ref{duality-theorem}.
\end{proof}

Note that by Proposition \ref{prop-order-of-hulls} in the statement of the previous corollary we can change $\overrightarrow{\alpha}$ by $\alpha$ and, of course, everything remains the same. 
Let us now consider the issue of the accessibility of dual o.s. tensor norms.

\begin{proposition}\label{prop-accessibility-for-dual-cross-norms} %DF 15.6
Let $\alpha$ be an o.s. tensor norm on $\ONORM$.
\begin{enumerate}[(a)]
\item $\alpha$ is right-accessible (resp.  left-accessible,  accessible) then $\alpha'$ is locally right-accessible (resp. locally left-accessible, locally accessible).
\item If $\alpha$ is accessible then the transposed o.s. tensor norm $\alpha^t$  is accessible and the adjoint o.s. tensor norm $\alpha^*$ is locally accessible.
\end{enumerate}
\end{proposition}

\begin{proof}
Assume that $\alpha$ is  right-accessible. Then for all $M \in \OFIN$ and $F \in \OLOC$ we have complete isometries
$$
M' \otimes_{\alpha''} F' = M' \otimes_{\overrightarrow{\alpha}} F' = M' \otimes_{\overleftarrow{\alpha}} F' = (M \otimes_{\alpha'} F)'
$$
where the first equality follows from Proposition \ref{prop-properties-dual-cross-norm},
the second one from right-accessibility, and the third one from the Duality Theorem \ref{duality-theorem}.
Therefore we also have complete isometries
$$
M \otimes_{\alpha'} F \hookrightarrow (M \otimes_{\alpha'} F)'' = (M' \otimes_{\alpha''} F')',
$$
which means that $\overleftarrow{\alpha'} = \alpha' = \overrightarrow{\alpha'}$ (again by the Duality Theorem) in $\OFIN \otimes \OLOC$.
This shows that $\alpha'$ is locally right-accessible.
Part (b) follows trivially.
\end{proof}

In the Banach space setting the dual of a right(left)-accessible tensor norm is again right(left)-accessible. For operator spaces we obtain in the previous proposition a weaker statement since we have proved that the dual norm is  \textsl{locally} right(left)-accessible. However, we do not know whether this notion is truly weaker since we do not have any example of an o.s. tensor norm which is locally right(left)-accessible but not right(left)-accessible.

%------------------------------------------------------
%------------------------------------------------------
%------------------------------------------------------
\section{The completely bounded approximation property} \label{complete bap}

Recall that given $C \ge 1$, a normed operator space is said to have the \emph{$C$-completely bounded approximation property} ($C$-CBAP for short)  if there exists a net of finite-rank mappings $\phi_i : E \to E$ such that $\n{\phi_i}_{\cb} \le C$ for all $i$ and for every $x \in E$, $\n{\phi_i(x) -  x} \to 0$.
One of the goals in this section is to study the CBAP via various conditions involving tensor products, in the spirit of the characterizations for the operator space approximation property appearing in \cite[Sec. 11.2]{Effros-Ruan-book}. We also deal with a weaker version, called W*CBAP, which carries over many of the interesting equivalences that appear in the classical theory of tensor products \cite[Sect. 16]{Defant-Floret}.

For a normed operator space $E$, we use the notation $K_\infty(E) = \mathcal{K} \otimes_{\min} E$ to emphasize that we think of this space as consisting of infinite $E$-valued matrices; see \cite[Sec. 10.1]{Effros-Ruan-book} for the specific details of how to interpret $K_\infty(E)$ as a completion of the union of the spaces $M_n(E)$ of $E$-valued matrices. 
Given normed operator spaces $E$ and $F$,
recall that a net $(\varphi^\alpha)$ in $\CB(E,F)$ converges to $\varphi$ in the \emph{stable point-norm topology} if  for every $x \in K_\infty(E)$ we have that
$
\varphi^\alpha_\infty(x) \to \varphi_\infty(x)
$
in the norm topology of $K_\infty(F)$
 \cite[Sec. 11.2]{Effros-Ruan-book}.

Our first characterization of CBAP (Proposition \ref{prop-characterization-CBAP-with-tau} below) corresponds to \cite[Lemma 11.2.1]{Effros-Ruan-book}, but instead of the stable point-norm topology we simply use the point-norm topology (which we denote by $\tau$).
The reason is clarified in the following remark.

\begin{remark}\label{Remark-stable-point-norm-topology}
Observe that for a bounded net in $\CB(E,F)$, convergence in the stable point-norm topology coincides with convergence in the point-norm topology. Indeed, suppose $\varphi^\alpha, \varphi \in \CB(E,F)$ and $x \in K_\infty(E)$. If $P^m : K_\infty(E) \to K_\infty(E)$ denotes the truncation operator,
\begin{multline*}
\n{ \varphi^\alpha_\infty(x) - \varphi_\infty(x) }_{K_\infty(F)}
\le 
\n{\varphi^\alpha}_{\cb}\n{ x - P^m(x) }_{K_\infty(E)} \\+
\n{ \varphi^\alpha_m( P^m(x) ) - \varphi_m( P^m(x) ) }_{M_m(F)}
+ \n{\varphi}_{\cb} \n{ P^m(x) - x }_{K_\infty(E)},
\end{multline*}
which shows that when $\varphi^\alpha$ is bounded, convergence in the point-norm topology implies convergence in the stable point-norm topology (the other implication is trivial and holds in general).
\end{remark}

%The above ``triangle inequality argument'' will be used repeatedly below.

\begin{proposition}\label{prop-characterization-CBAP-with-tau}
Suppose that $E$ is a normed operator space and $C \ge 1$. The following are equivalent:
\begin{enumerate}[(a)]
\item $E$ has the $C$-CBAP.
\item $C\cdot B_{E' \otimes_{\min} E}$ is  $\tau$-dense in $B_{\CB(E,E)}$.
\item For any normed operator space $F$, $C\cdot B_{E' \otimes_{\min} F}$ is  $\tau$-dense in $B_{\CB(E,F)}$.
\item For any normed operator space $F$, $C\cdot B_{F' \otimes_{\min} E}$ is  $\tau$-dense in $B_{\CB(F,E)}$.
\end{enumerate}
\end{proposition}

\begin{proof}
$(a) \Rightarrow (b)$:
Suppose that $E$ has $C$-CBAP, so that there exists a net of finite-rank mappings $\phi_i : E \to E$ such that $\n{\phi_i}_{\cb} \le C$ for all $i$ and for every $x \in E$, $\n{x - \phi_i x} \to 0$. This precisely means that $\phi_i \xrightarrow{\tau} id_E$, and
note that each $\phi_i$ can be identified with an element of $C\cdot B_{E' \otimes_{\min} E}$.
For any other $\varphi \in B_{\CB(E,E)}$, the net $(\phi_i \circ \varphi)_i$ does the job.
Reversing the argument gives $(b) \Rightarrow (a)$.

Now, if $(\phi_i)$ is a net in $C\cdot B_{E' \otimes_{\min} E}$ converging to $id_E$ in the $\tau$ topology, then for any $\psi \in B_{\CB(E,F)}$ the net $\psi \circ \phi_i \in C\cdot B_{E' \otimes_{\min} F}$ converges to $\psi$ in the $\tau$ topology.
Similarly, for $\psi \in B_{\CB(F,E)}$ the net $\phi_i \circ \psi \in C\cdot B_{F' \otimes_{\min} E}$ converges to $\psi$ in the $\tau$ topology. This proves $(b) \Rightarrow (c)$ and $(b) \Rightarrow (d)$.
The implications $(c) \Rightarrow (b)$ and $(d) \Rightarrow (b)$ are clear by specializing to $F=E$.
\end{proof}

It should be noted that although completeness plays an important role in  \cite[Sec. 11.2]{Effros-Ruan-book}, we do not assume it.
This was to be expected, as the same is true of the Banach space case:
note that in \cite{Defant-Floret} the section on BAP deals with normed spaces, whereas the one on AP deals with Banach spaces.
In this regard we point out the following easy lemma, which shows that when it comes to the CBAP it makes no difference to work with a space or with its completion. 
% The reason is clarified in the next couple of lemmas, 
% which show that in the conditions appearing in Proposition \ref{prop-characterization-CBAP-with-tau} we can freely pass between a space and its completion by using triangle inequality arguments.
As usual, $\widetilde{E}$ denotes the completion of $E$.
Moreover, given $\varphi \in \CB(E,F)$ we denote by $\widetilde{\varphi}$ its unique extension to a map in $\CB(\widetilde{E},\widetilde{F})$.

\begin{lemma}\label{CBAP-and-completion}
$E$ has $C$-CBAP if and only if $\widetilde{E}$ has $C$-CBAP.
\end{lemma}

\begin{proof}
$(\Rightarrow)$ Let $(\varphi^\alpha)$ be a net of finite-rank maps in $\CB(E,E)$ of $\cb$-norm at most $C$ converging to the identity of $E$ in the point-norm topology.
Let $\widetilde{x} \in \widetilde{E}$. By a triangle inequality argument, for $x\in E$
$$
\n{\widetilde{\varphi}^\alpha (\widetilde{x}) - \widetilde{x}} \le
\n{\widetilde{\varphi}^\alpha} \n{\widetilde{x}-x} + \n{ \varphi^\alpha(x) - x } + \n{ x - \widetilde{x} },
$$
which shows $(\widetilde{\varphi}^\alpha)$ is a net of finite-rank maps of $\cb$-norm at most $C$ converging to the identity of $\widetilde{E}$ in the point-norm topology.

$(\Leftarrow)$ Let $(\psi^\alpha)$ be a net of finite-rank maps in $\CB(\widetilde{E},\widetilde{E})$ of $\cb$-norm at most $C$ converging to the identity of $\widetilde{E}$ in the point-norm topology.
For each $\alpha$ let $E_\alpha$ be the range of $\psi^\alpha$, and for each $\delta>0$
use perturbation to find a linear map $R_\delta : E_\alpha \to E$ with
 $\n{R_\delta-Id_{E_\alpha}}_{\cb}< \delta$.
 Let $\varphi^{\alpha,\delta}$ be the restriction to $E$ of $\frac{1}{1+\delta}R_\delta \circ \psi^\alpha$.
 Then for any $x \in E$,
$$
\n{\varphi^{\alpha,\delta}(x)-x} \le \frac{1}{1+\delta}\n{R_\delta \psi^\alpha(x) - \psi^\alpha(x)} + \frac{\delta}{1+\delta}\n{\psi^\alpha(x)} +  \n{\psi^\alpha(x)-x},
$$
so  $(\varphi^{\alpha,\delta})$ is a net of finite-rank maps in $\CB(E,E)$ of $\cb$-norm at most $C$ converging to the identity of $E$ in the point-norm topology.
\end{proof}

For normed operator spaces $E$ and $F$,
recall that a net $(\varphi^\alpha)$ in $\CB(E,F)$ converges to $\varphi$ in the \emph{stable point-weak topology} (denoted $\tau_w$) if  for every $x \in K_\infty(E)$ we have that
$
\varphi^\alpha_\infty(x) \to \varphi_\infty(x)
$
in the weak topology of $K_\infty(F)$
\cite[Sec. 11.2]{Effros-Ruan-book}.
Let us remark that although \cite[Sec. 11.2]{Effros-Ruan-book} works with Banach operator spaces, the proof of \cite[Prop. 11.2.2]{Effros-Ruan-book} applies more generally to normed operator spaces.
Therefore, the proof of \cite[Cor. 11.2.3]{Effros-Ruan-book} together with Proposition \ref{prop-characterization-CBAP-with-tau} yield the following result.

\begin{corollary}\label{cor-characterization-CBAP-with-tau-weak}
Suppose that $E$ is a normed operator space and $C \ge 1$. The following are equivalent:
\begin{enumerate}[(a)]
\item $E$ has the $C$-CBAP.
\item $ C \cdot B_{E' \otimes_{\min} E}$ is  $\tau_w$-dense in $B_{\CB(E,E)}$.
\item For any operator space $F$, $C \cdot  B_{E' \otimes_{\min} F}$ is  $\tau_w$-dense in $B_{\CB(E,F)}$.
\item For any operator space $F$,  $C \cdot B_{F' \otimes_{\min} E}$ is  $\tau_w$-dense in $B_{\CB(F,E)}$.
\end{enumerate}
\end{corollary}

Consider now the natural embedding
$$
\CB(E,F) \subseteq \CB(E, F'') = (E \otimes_{\proj} F')'.
$$
In  \cite[Lemma 11.2.4]{Effros-Ruan-book} it is proved that  when $E$ and $F$ are complete,
the stable point-weak topology on $\CB(E,F)$ is just the relative weak$^*$ topology determined by $E \otimes_{\proj} F'$, and thus each stable point-weakly continuous functional on $\CB(E,F)$ is determined by an element of $E \widehat{\otimes}_{\proj} F'$.
Let us once again remark that the proof of \cite[Lemma 11.2.4]{Effros-Ruan-book} applies more generally to normed operator spaces.

The bounded approximation property, in the context of normed spaces, has many important equivalences in the tensor product realm. Indeed, a Banach space $E$ has the $C$-approximation property if and only if $E\otimes_{\pi}F \to (E'\otimes_{\varepsilon}F')'$ is a $C$-isomorphism onto its image. Moreover, these properties are also equivalent to the following inequality:
$\pi \leq C \overleftarrow{\pi}$ on $E\otimes F$ (for all normed spaces $F$).
At first glance the CBAP seems to be a nice extension of the bounded approximation property to operator spaces, and the literature shows a number of examples where the analogy works very well. 
However, the lack of local reflexivity does not directly allow the equivalence between CBAP and the natural operator space versions of the mentioned results related to tensor products. Nevertheless, we will see that a weak version of CBAP introduced in \cite{effros1990approximation} will allow us to translate those approximation properties relative to tensor products of normed spaces into this context.

The following is the obvious adaptation of the notion of W*MAP defined in \cite{effros1990approximation}.

\begin{definition}
    An operator space $E$ has the $C$-W*CBAP if there exists a net of finite-rank maps $\phi_i : E \to E''$ such that $\n{\phi_i}_{\cb} \le C$ for all $i$ and for every $x \in E$, $\phi_i (x) \to \kappa_E(x)$ in the weak* topology.
\end{definition}

Note that clearly the $C$-CBAP implies the $C$-W*CBAP, as convergence in norm implies convergence in the weak* topology.
The same argument as in \cite{effros1990approximation}  shows that if $E$ has the $C$-W*CBAP, then so does $M_n(E)$.

\begin{remark}
Similarly as in \cite{effros1990approximation} it is easy to prove (taking adjoints) that $E$ has the $C$-CBAP if and only if there is a net of weak* continuous finite-rank maps $\psi_i:E' \to E'$ such that $\Vert \psi_i \Vert_{cb}\le C$ for all $i$ and $\psi_i(x')\to x'$ in the weak* topology for every $x'\in E'$. Instead, $E$ has the $C$-W*CBAP if and only if the same condition as above holds with the exception that the weak* continuity of the mappings $\psi_i$ is not assumed/required.
\end{remark}

We now prove a variant of Corollary \ref{cor-characterization-CBAP-with-tau-weak}.

\begin{proposition}\label{prop-characterization-W*CBAP}
Suppose that $E$ is a normed operator space and $C \ge 1$. The following are equivalent:
\begin{enumerate}[(a)]
\item $E$ has the $C$-W*CBAP.
\item $ C \cdot B_{E' \otimes_{\min} E''}$ is  point-weak*-dense in $B_{\CB(E,E'')}$.
\item For any operator space $F$, $C \cdot  B_{E' \otimes_{\min} F'}$ is  point-weak*-dense in $B_{\CB(E,F')}$.
\end{enumerate}
\end{proposition}

\begin{proof}
The implication $(a) \Rightarrow (c)$ is implicit in the proof of \cite[Thm. 2.2]{effros1990approximation}. Suppose that $T : E \to F'$ is a complete contraction. Consider the commutative diagram
$$
\xymatrix{
E'' \ar[r]^{T''} &F''' \ar[d]^{\kappa_F'}\\
E \ar[u]^{\kappa_E} \ar[r]_{T} &F'
}
$$
Note that both $T''$ and $\kappa_F'$ are weak* continuous.
Therefore, if $\phi_i : E \to E''$ is the net of maps in $C \cdot B_{E' \otimes_{\min} E}$ converging to $\kappa_E$ in the point-weak* topology given by the definition of $C$-W*CBAP, it is clear that $\kappa_F'T''\phi_i : E \to F'$ converges to $T$ in the point-weak* topology.

 $(c)\Rightarrow (b)$  is trivial.

Finally, $(b)$ implies that the canonical inclusion $\kappa_E : E \to E''$ can be approximated in the point-weak* topology by a net of finite-rank maps whose cb-norms are at most $C$, which is exactly the definition of $C$-W*CBAP, yielding $(a)$.

\end{proof}

The following result is, in a sense, a quantitative version of \cite[Thm. 2.2]{effros1990approximation}.

\begin{theorem} \label{thm: equivalencias WCBAP}
Let $E$ be an operator space and $C \ge 1$. The following are equivalent:
\begin{enumerate}[(a)]
\item For any operator space $F$, the natural map $E \otimes_{\proj} F \to (E' \otimes_{\min} F')'$ is a complete $C$-isomorphism onto its image.
\item For any operator space $F$, the natural map $E \otimes_{\proj} F \to (E' \otimes_{\min} F')'$ is a $C$-isomorphism onto its image.
\item $E$ has the $C$-W*CBAP.
\item For any complete contraction $T: E \to E''$ there exists a net of finite-rank maps $\phi_i : E \to E''$ such that $\n{\phi_i}_{\cb} \le C$ for all $i$ and for every $x \in E$, $\phi_i (x) \to Tx$ in the weak* topology.
\item For any operator space $F$ and any complete contraction $T: E \to F'$, there is a net of finite-rank maps $\psi_i : E \to F'$ such that $\n{\psi_i}_{\cb} \le C$ for all $i$ and for every $x \in E$, $\psi_i x \to Tx$ in the weak* topology.
\item The natural map $E \otimes_{\proj} E' \to (E' \otimes_{\min} E'')'$ is a complete $C$-isomorphism onto its image.
\item The natural map $E \otimes_{\proj} E' \to (E' \otimes_{\min} E'')'$ is a $C$-isomorphism onto its image.
\end{enumerate}
\end{theorem}

\begin{proof}
$(a) \Rightarrow (b)$ is obvious. Let us assume $(b)$, and let 
$F$ be an operator space. In order to show that the natural map $E \otimes_{\proj} F \to (E' \otimes_{\min} F')'$ is a complete $C$-isomorphism onto its image, it suffices to show that its adjoint $(E' \otimes_{\min} F')'' \to (E \otimes_{\proj} F)' = \CB(E,F')$ is a complete $C$-quotient. In turn, this will follow from proving that for any $n$ the amplification $M_n((E' \otimes_{\min} F')'') \to M_n(\CB(E,F'))$ is a $C$-quotient. Considering the standard identifications
$M_n((E' \otimes_{\min} F')'') = (M_n(E' \otimes_{\min} F'))'' = (E' \otimes_{\min} M_n(F'))''= (E' \otimes_{\min} S_1^n[F]')''$ and
$\CB(E,M_n(F')) = \CB(E,S_1^n[F]')$, we need to show that the natural map
$(E' \otimes_{\min} (S_1^n[F])')'' \to \CB(E,S_1^n[F]')$ is a $C$-quotient. This follows from $(b)$ with $S_1^n[F]$ in place of $F$ and taking adjoints. Thus, $(a) \Leftrightarrow (b)$.

Now, $(c) \Leftrightarrow (d) \Leftrightarrow (e)$ is just a restatement of Proposition \ref{prop-characterization-W*CBAP}.

Let us now prove that $(b) \Leftrightarrow (e)$.
First observe that the natural map $E \otimes_{\proj} F \to (E' \otimes_{\min} F')'$ always has cb-norm one.
Therefore, $(b)$ is equivalent to having
$$
(E \otimes F) \cap B_{( E' \otimes_{\min} F' )'} \subset C \cdot B_{E \otimes_{\proj} F} 
$$
By the bipolar theorem, this in turn is equivalent to
$$
B_{\CB(E,F')} \subset \overline{C \cdot B_{ E' \otimes_{\min} F' }}^{w^*}.
$$
where the closure is in the weak* topology associated with the identification $\CB(E,F') = (E \otimes_{\proj} F)'$. And this is clearly equivalent to the statement of $(e)$.

Finally, it is evident that $(a)\Rightarrow (f)\Rightarrow (g)$ and arguing as in $(b) \Leftrightarrow (e)$ we easily obtain that $(g)\Leftrightarrow (d)$.
\end{proof}

\begin{remark}\label{alpha CBAP}
    In view of our preceding results, it seems plausible to extend the usual concept of the bounded approximation property with respect to a tensor norm (see e.g. \cite[21.7]{Defant-Floret}) to our operator space setting. Namely, a natural definition to consider should be the following: an operator space $E$ has the $\alpha$-W*CBAP (with constant $C>0$) if the natural map 
$$
F\otimes_{\alpha}E \to (F'\otimes_{\alpha'}E')'$$ is a complete $C$-isomorphism onto its image for every operator space $F$.
Nevertheless,  we do not tackle this here and leave it for future research.

\end{remark}

The following two results follow similarly as \cite[Prop. 2.3]{effros1990approximation} and \cite[Prop. 2.4]{effros1990approximation}. We include somewhat alternative proofs for completeness.

\begin{proposition}\label{prop W*CBAP and locally reflexive}
    If an operator space $E$ has the $C$-W*CBAP and is locally reflexive, then $E$ has the $C$-CBAP.
\end{proposition}

\begin{proof}
If $E$ has the $C$-W*CBAP  there is a net of finite-rank maps $\phi_i : E \to E''$ such that $\n{\phi_i}_{\cb} \le C$ for all $i$ and for every $x \in E$, $\phi_i (x) \to \kappa_E(x)$ in the weak* topology. For each $i$, let us denote the range of the mapping $\phi_i$ by $F_i\in \OFIN(E'')$. Since $E$ is locally reflexive for each $i$, there is a net of complete contractions $(\psi_\eta^i)_\eta\subset\CB(F_i,E)$ approximating the inclusion $F_i\hookrightarrow E''$ in the point-weak* topology. By composing we obtain a net of finite rank mappings $(\psi_\eta^i\circ \phi_i)\subset\CB(E,E)$ which converges in the point-weak topology to the identity $id:E\to E$ satisfying $\n{\psi_\eta^i\circ\phi_i}_{\cb} \le C$. A classical convex combinations argument allow us to obtain a net with the same properties but which converges in the point-norm topology to the identity; thus concluding that $E$ has the C-CBAP.
\end{proof}

\begin{proposition}\label{prop W*CBAP from dual}
If $E$ is an operator space such that $E'$ has the $C$-W*CBAP, then so does $E$.    
\end{proposition}

\begin{proof}
    If $E'$ has the $C$-W*CBAP, by Theorem \ref{thm: equivalencias WCBAP} $(a)$ (applied to $F=E$) we obtain that  the map $E' \otimes_{\proj} E \to (E'' \otimes_{\min} E')'$ is a complete $C$-isomorphism with its image. Since both $\proj$ and $\min$ are symmetric we can flip the spaces (by transposition) to obtain that $E \otimes_{\proj} E' \to (E' \otimes_{\min} E'')'$ is a complete $C$-isomorphism with its image. This is exactly what is stated in Theorem \ref{thm: equivalencias WCBAP} $(f)$. Hence, $E$ has the $C$-W*CBAP.
\end{proof}

As a corollary we see that under local reflexivity the $C$-CBAP can be transferred from the dual $E'$ to $E$.
\begin{corollary}
Let $E$ be an operator space which is locally reflexive. If $E'$ has the $C$-CBAP, then so does $E$.
(in fact, if $E'$ has the $C$-W*CBAP then $E$ has $C$-CBAP).
\end{corollary}

As a consequence of Theorem \ref{thm: equivalencias WCBAP} we obtain the following analogue of \cite[Thm. 16.2]{Defant-Floret}:

\begin{theorem}\label{thm:equivalencias WCBAP cofinitamente generada}
Let $E$ be an operator space and $C \ge 1$. The following are equivalent:
\begin{enumerate}[(a)]
\item $E$ has the $C$-W*CBAP.
\item For every operator space $F$ (or only $F=E'$), the  identity $E \otimes_{\overleftarrow{\proj}} F \to E \otimes_{\proj} F$ has $\cb$-norm at most $C$.
\item For every operator space $F$, the  identity $E \otimes_{\overleftarrow{\proj}} F \to E \otimes_{\proj} F$ has norm at most $C$.
\end{enumerate}
% An operator space $E$ has the $C$-W*CBAP if and only if
% $$
% \proj \le C \overleftarrow{\proj} \quad \text{ on } E \otimes F
% $$
% for every operator space $F$ (or only $F=E'$).
In particular, $E$ has the W*CMAP if and only if
$$
\proj = \overleftarrow{\proj} \quad \text{ on } E \otimes F
$$
for every operator space $F$ (or only $F=E'$).
\end{theorem}

\begin{proof}
By the Duality Theorem \ref{duality-theorem} we have the complete isometry $E \otimes_{\overleftarrow{\proj}} F \hookrightarrow (E' \otimes_{\min} F')'$. Also, we know that $\overleftarrow{\proj} \leq \proj$. Therefore, 
 $E \otimes_{\overleftarrow{\proj}} F \to E \otimes_{\proj} F$ has $\cb$-norm at most $C$ (resp. norm at most $C$)
% $$
% \proj \le C \overleftarrow{\proj} \quad \text{ on } E \otimes F
% $$
if and only if the mapping 
$E \otimes_{\proj} F \to (E' \otimes_{\min} F')'$
is a complete $C$-isomorphism (resp. $C$-isomorphism) onto its image. Now,  Theorem \ref{thm: equivalencias WCBAP} gives the result.
\end{proof}

Now we present an operator space version of \cite[Lem. 16.2]{Defant-Floret} which enables us to prove the subsequent corollary.

\begin{lemma} \label{lem:norma cofinitamente generada equivalente}
    Let $\alpha$ be an o.s. tensor norm, $E$ and $F$ be normed operator spaces, and $n\in\N$. Then $\alpha_n \le C \overleftarrow{\alpha}_n$ on $E \otimes F$ if and only if 
$$
B_{M_n\big((E \otimes_\alpha F)'\big)} \subset C\cdot \overline{ B_{M_n( E' \otimes_{\alpha'} F' )} }^{\tau_w} \subset M_n\big(\CB(E,F') \big) = \CB\big(E , M_n(F')\big).
$$
\end{lemma}

\begin{proof}
Let $z \in M_n(E \otimes F)$. By the Duality Theorem \ref{duality-theorem},
$$
\overleftarrow{\alpha}_n(z; E, F) =  \sup \big\{  \n{ \mpair{T}{z} }_{M_{nm}} \mid  T \in B_{M_m( E' \otimes_{\alpha'} F' )} \big\},
$$
whereas clearly
$$
\alpha_n(z; E, F) = \sup  \big\{  \n{ \mpair{T}{z} }_{M_{nm}} \mid  T \in B_{M_m(( E \otimes_{\alpha} F )')} \big\}
$$
By the bipolar theorem, the inclusion
$$
B_{M_n\big((E \otimes_\alpha F)'\big)} \subset C \cdot\overline{ B_{M_n( E' \otimes_{\alpha'} F' )} }^{\tau_w}
$$
is then equivalent to $\alpha_n \le C \overleftarrow{\alpha}_n$ on $M_n(E \otimes F)$.
\end{proof}

From Theorem \ref{thm:equivalencias WCBAP cofinitamente generada} and Lemma \ref{lem:norma cofinitamente generada equivalente} we derive the following corollary. Compare to Proposition \ref{prop-characterization-W*CBAP}.

%{\color{red} Escribir el siguiente resultado con \'items? Valen ítems tipo los de Cor. \ref{cor-characterization-CBAP-with-tau-weak}?}

\begin{corollary}
Suppose that $E$ is a normed operator space and $C \ge 1$. The following are equivalent:
\begin{enumerate}[(a)]
\item $E$ has the $C$-W*CBAP.
\item $ C \cdot B_{E' \otimes_{\min} E''}$ is  $\tau_w$-dense in $B_{\CB(E,E'')}$.
\item For any operator space $F$, $C \cdot  B_{E' \otimes_{\min} F'}$ is $\tau_w$-dense in $B_{\CB(E,F')}$.
\end{enumerate}
%    An operator space $E$ has the $C$-W*CBAP if and only if for every dual operator space $F'$, $C\cdot B_{ E' \otimes_{\min} F' }$ is $\tau_w$-dense in $B_{\CB(E,F')}$.
\end{corollary}

\begin{proof}
The equivalence between $(a)$ and $(c)$ follows from Theorem \ref{thm:equivalencias WCBAP cofinitamente generada} $(c)$ and Lemma \ref{lem:norma cofinitamente generada equivalente}.
Clearly, $(c)$ implies $(b)$.
Assuming $(b)$, any $T \in B_{\CB(E,E'')}$ is the $\tau_w$-limit of a net in $ C \cdot B_{E' \otimes_{\min} E''}$, which means $T$ is also the point-weak limit of the same net (by an argument analogous to that of Remark \ref{Remark-stable-point-norm-topology}), and hence also the point-weak* limit. Now Proposition  \ref{prop-characterization-W*CBAP} implies $(a)$.
\end{proof}

\section{Mapping ideals} \label{mapping ideals}

We introduce now the definition of a mapping ideal, the operator space version of Banach operator ideals defined by Pietsch and also widely developed by Defant and Floret.
The notion in our framework comes from \cite[Sec. 12.2]{Effros-Ruan-book}, though we have to strengthen Effros and Ruan's definition in order to keep the essence of the traditional Banach space concept and allow the natural relationship  with tensor norms.
As in \cite{Defant-Floret} the definition of ideal is restricted to the category of complete spaces. Let us  denote by $\OBAN$ the class of all Banach operator spaces.

\begin{definition}
A \emph{(normed) mapping ideal} $(\mathfrak{A},\mathbf{A})$ is an assignment, for each pair of operator spaces $E, F \in \OBAN$, of a linear space $\mathfrak{A}(E,F) \subseteq \CB(E,F)$  together with an operator space structure $\mathbf{A}$ on $\mathfrak{A}(E,F)$ such that
\begin{enumerate}[(a)]
\item The identity map $\mathfrak{A}(E,F) \to \CB(E,F)$ is a complete contraction.
\item For every $x'\in M_n(E')$ and $y\in M_m(F)$ the mapping $x'\otimes y$ belongs to $M_{nm}(\mathfrak{A}(E,F))$ and $\mathbf{A}_{nm}(x'\otimes y)=\|x'\|_{M_n(E')} \|y\|_{M_m(F)}$.
\item The ideal property: whenever $T \in M_n(\mathfrak{A}(E,F))$, $r \in \CB(E_0,E)$ and $s \in \CB(F,F_0)$, it follows that $s_n \circ T \circ r$ belongs to $ M_n(\mathfrak{A}(E_0,F_0))$ with
$$
\mathbf{A}_n( s_n \circ T \circ r ) \le \n{s}_{\cb} \mathbf{A}_n(T) \n{r}_{\cb}.
$$

\end{enumerate}
\end{definition}

\begin{remark} \label{OFIN-OBAN}
Item (b) of the previous definition is the \textsl{new} requirement, not considered in \cite[Sec. 12.2]{Effros-Ruan-book}. 
In the Banach space realm, the definition of normed operator ideal includes that the ideal norm of the identity $id_{\mathbb  C}:  \mathbb C \to \mathbb C$  is one \cite[Definition 6.1.1.]{Pietsch-67}. This requisite is in fact the key to relate normed ideals and finitely generated  tensor norms.
Unfortunately there is no  immediate  translation of this requirement into the operator space framework, but nevertheless there is a way to tackle this issue.
By \cite[Proposition 6.1.5]{Pietsch-Operator-Ideals} the aforementioned condition about the norm of the identity (together with the other properties of normed operator ideals) implies that ---and in fact it is equivalent to--- the ideal norm of any elementary tensor $x'\otimes y$ (for $x' \in E'$ and $y \in F$) is exactly $\Vert x' \Vert \Vert y \Vert$. This has inspired  item (b) on our definition. Note the similarity between this statement and equation \eqref{norma producto matrices} in the definition of o.s. tensor norms.
 
Let us show now that (b) can be somehow relaxed.

First, note that since $\|x'\otimes y\|_{cb}=\|x'\|_{M_n(E')} \|y\|_{M_m(F)}$, condition (a) implies that $\mathbf{A}_{nm}(x'\otimes y)\ge\|x'\|_{M_n(E')} \|y\|_{M_m(F)}$. Hence, we can replace in (b) the symbol $=$ by $\le$. 

Also, it is enough that this inequality is satisfied for every $E,F\in\OFIN$. Indeed, let us show that the inequality is valid  for every Banach operator space when it is so for every finite-dimensional operator space. Let $x'=(x'_{ij})\in M_n(E')$ and $y=(y_{ij})\in M_m(F)$ with $E,F\in\OBAN$. Consider $L=\bigcap_{i,j}\ker(x'_{ij})\in\OCOFIN(E)$ and $u'\in M_n((E/L)')$ such that $x'=(q_L^E)'_n(u')$. Since $(q_L^E)'$ is a complete isometry, we have $\|u'\|_{M_n((E/L)')}=\|x'\|_{M_n(E')}$. Choosing $N\in\OFIN(F)$ such that $y\in M_m(N)$ by means of condition (c) of the definition we obtain
$$
\mathbf{A}_{nm}(x'\otimes y;E,F)\le \mathbf{A}_{nm}(u'\otimes y; E/L,N)\le
\|u'\|_{M_n((E/L)')}\|y\|_{M_m(N)}=\|x'\|_{M_n(E')}\|y\|_{M_m(F)}.
$$

Therefore, we can exchange condition (b) of the previous definition by the following:

\begin{itemize}
    \item[(b')] For every $E,F\in\OFIN$, $x'\in M_n(E')$ and $y\in M_m(F)$ the mapping $x'\otimes y$ belongs to $M_{nm}(\mathfrak{A}(E,F))$ and $\mathbf{A}_{nm}(x'\otimes y)\le\|x'\|_{M_n(E')} \|y\|_{M_m(F)}$.
\end{itemize}

\end{remark}

Of course, the class of completely bounded mappings $\CB$ is obviously a mapping ideal.

Let us recall the definitions of some other mapping ideals that have already appeared in the literature. We explain in each case (whenever it is not obvious) why condition (b') is satisfied. Let $E$ and $F$ be Banach  operator spaces.
\begin{enumerate}[(i)]
\item \emph{Finite rank mappings} $\mathcal F$: 
 the space $\mathcal{F}(E,F)$ consists of finite rank mappings from $E$ to $F$. This is a mapping ideal with the $\cb$-norm. 
 \item \emph{Completely approximable mappings} $\mathcal{A}$:  the space $\mathcal{A}(E,F)$ is the closure of $\mathcal F(E,F)$ inside $\CB(E,F).$
\item \emph{Completely nuclear mappings} $\mathcal N$ \cite[12.2]{Effros-Ruan-book}:  the space $\mathcal{N}(E,F)$ is the image of the mapping
\begin{equation}\label{def nucleares}
    \Phi: E'\widehat{\otimes}_{\proj}F \to E'\widehat{\otimes}_{\min}F \subset \CB(E,F)
\end{equation}
with the quotient operator space structure given by the canonical identification 
$$
\mathcal{N}(E,F) =\frac{ E' \widehat{\otimes}_{\proj} F}{\ker  \Phi}.
$$
The nuclear operator space norm is denoted by $\nu$. Equation \eqref{norma producto matrices} for the $\proj$ norm shows that (b') is satisfied for $\mathcal{N}$.

\item \emph{Completely integral mappings}  $\mathcal I$ \cite[12.3]{Effros-Ruan-book}:  the space $\mathcal I(E,F)$ is formed by those mappings such that
\begin{equation*}
    \iota(T)=\sup \{(\nu(T|_M):M \in \OFIN(E) \}<\infty. 
\end{equation*}
Given $T=(T_{ij}) \in M_n(\mathcal I(E,F))$ it is defined
$$\iota_n(T)=\sup \{(\nu_n(T|_M):M \in \OFIN(E) \}.$$
This gives an operator space structure on $\mathcal I(E,F)$. Since $\mathcal{I}$ and $\mathcal{N}$ coincide on $\OFIN$, condition (b') is satisfied for $\mathcal{I}$.
\item \emph{Completely $p$-summing $\Pi_p$} \cite[Chapter 5]{Pisier-Asterisque-98}: For $1 \leq p < \infty$, a linear map $T:E \to F$ belongs to $\Pi_p(E,F)$ if the mapping
\begin{equation*}
    i_{S_p} \otimes T: S_p \otimes_{\min}E \to S_p[F]
\end{equation*}
is bounded (equivalently, completely bounded), and we denote its norm by $\pi_p(T)$. The operator space structure of $\Pi_p(E,F)$ is inherited from $\CB( S_p \otimes_{\min}E, S_p[F])$. 
As a consequence of  \cite[Thm. 3.1]{CD-Chevet-Saphar-OS} for any $E,F\in\OFIN$ we have the following complete isometries:
$$\Pi_p(E,F)= (E\otimes_{d_{p'} }F')'= E'\otimes_{(d_{p'} )'}F.$$ 
Equation \eqref{norma producto matrices} for the $(d_{p'} )'$ norm shows that (b') is satisfied for $\Pi_p$.
\item \emph{Completely right $p$-nuclear $\mathcal N^p$ \cite[Definition 2.1]{chavez2019operator}}: We say that a linear mapping $T:E \to F$ belongs to $\mathcal N^p(E,F)$ if it is in the range of the canonical inclusion
\begin{equation*}
J^p: E' \widehat{\otimes}_{d_p }F \to E' \widehat{\otimes}_{\min}F \subset \CB(E,F).    
\end{equation*}
We endow $\mathcal N^p(E,F)$ with the operator space quotient structure $E' \widehat{\otimes}_{d_p }F /Ker J^p.$ Equation \eqref{norma producto matrices} for the $d_p $ norm shows that (b') is satisfied for $\mathcal{N}^p$.
\item \emph{Operator $p$-compact mappings $\mathcal{K}_p $ \cite[Definition 3.2]{chavez2019operator}}: A mapping $T \in \mathcal{K}_p (E,F)$ if there exist $G\in\OBAN$, a completely right $p$-nuclear mapping $\Theta \in \mathcal{N}^p(G,F)$ and a completely bounded mapping $R \in \CB(E,G/\ker\Theta)$ with $\Vert R \Vert_{\cb} \leq 1$ such that the following diagram commutes

\begin{equation}\label{p-compact o.s.}
\xymatrix{
E \ar[r]^T  \ar[rd]_R    & F  & \ar[l]_{\Theta } \ar@{->>}[ld]^{\pi}  G \\
 &   G/\ker\Theta  \ar[u]^{\widetilde{\Theta}},  &
}
\end{equation}

where $\pi$ and $\widetilde{\Theta}$  stand for the natural quotient mapping and the natural monomorphism associated to $\Theta$ respectively.

 For $T \in  \mathcal{K}_p (E,F)$, we also define
$$\kappa_p (T):= \inf \{ \nu^p(\Theta) \},$$ where the infimum runs over all possible completely right $p$-nuclear mappings $\Theta \in \mathcal{N}^p(G,F)$  as in  \eqref{p-compact o.s.}.
 It should be mentioned \cite[Prop. 3.8.]{chavez2019operator} that $\mathcal{K}_p $ is exactly the surjective hull of $\mathcal N^p$ (see Proposition \ref{surjective hull} below). This provides the  operator space structure of $\mathcal{K}_p $ and shows that condition (b') is satisfied for this class. 

\item \emph{Completely $p$-nuclear mappings $\mathcal{N}_p$ \cite{Junge-Habilitationschrift}:}
A linear map $u : E \to F$ is \emph{completely $p$-nuclear} (denoted $u \in \mathcal{N}_p(E,F)$) if there exists a factorization of $u$ as
$$
\xymatrix{
E \ar[r]^{\alpha} &S_\infty \ar[r]^{M(a,b)} &S_p \ar[r]^{\beta} &F
}
$$
where  $a,b \in S_{2p}$, $M(a,b)$ is the multiplication operator $x \mapsto a x b$, and $\alpha, \beta$ completely bounded maps. The completely $p$-nuclear norm of $u$ is defined as
$$
\nu_p(u) = \inf \big\{ \n{\alpha}_{\cb} \n{a}_{S_{2p}} \n{b}_{S_{2p}} \n{\beta}_{\cb} \big\}
$$
where the infimum is taken over all factorizations of $u$ as above.
From \cite[Cor. 3.1.3.9 ]{Junge-Habilitationschrift},
when $E,F\in\OBAN$
and $1 \le p < \infty$, trace duality yields an isometric isomorphism between $\mathcal{N}_p(E,F)'$ and $\Pi_{p'}(F,E'')$.
This allows us to endow $\mathcal{N}_p(E,F)$ with an operator space structure. Also, condition (b') holds since for $E,F\in\OFIN$,
$\mathcal{N}_p(E,F)=F\otimes_{d_p }E'$.

\item \emph{Completely $(q,p)$-mixing mappings $\mathfrak{M}_{q,p}$ \cite{Yew-08,CD-completely-q-p-mixing}}:
Let $1\le p \le q < \infty$. A mapping $u: E \to F$ is said to be \emph{completely $(q,p)$-mixing} (denoted $u\in\mathfrak{M}_{q,p}(E,F)$) if there exists a constant $K$ such that for any $G \in \OBAN$ and any completely $q$-summing map $v : F \to G$, the composition $v \circ u$ is a completely $p$-summing map and $\pi_p(v\circ u) \le K \pi_q (v)$. The \emph{completely $(q,p)$-mixing norm} of $u$ is the smallest such $K$ and is denoted by $\mathfrak{m}_{q,p}(u)$.
By \cite[Thm. 4.3]{CD-completely-q-p-mixing}, when $F \subseteq B(H)$ we have that $u\in\mathfrak{M}_{q,p}(E,F)$ if and only if 
for all $n$ and all $(x_{ij})$ in $M_n(E)$ we have
$$
\sup \left\{ \n{ \big(a(ux_{ij})b \big) }_{S_p^n[S_q(H)]} \;:\; a,b \in B_{S_{2q}(H)} ,\; a,b \ge 0 \right\}\le C \n{(x_{ij})}_{S_p^n\otimes_{\min} E},
$$
so in this way $\mathfrak{M}_{q,p}(E,F)$ inherits an operator space structure from $\CB( S_p\otimes_{\min} E, \ell_\infty(S_p[S_q(H)]))$. 
To verify condition (b'), suppose $E,F\in\OFIN$, $x'\in M_n(E')$ and $y\in M_m(F)$.
Let $(x_{ij}) \in M_k(E)$, and consider $a,b \in B_{S_{2q}(H)}$.
Observe that for each $i,j$ we have $a\big( (x'\otimes y)x_{ij}\big)b = (x'\otimes ayb)(x_{ij})$.
Now, since we have already observed that condition (b') is satisfied for the completely $p$-summing maps, the norm of $x'\otimes ayb$ in $M_{nm}( \CB(S_p \otimes_{\min} E, S_p[S_q(H)]))$ is at most $\n{x'}_{M_n(E')} \n{ayb}_{M_m(S_q(H))}$.
Since the multiplication map $w \mapsto awb$ from $\mathcal{B}(H)$ to $S_q(H)$ has $\cb$-norm at most $\n{a}_{S_{2q}}\n{b}_{S_{2q}}$ \cite[Prop. 5.6]{Pisier-Asterisque-98}, taking the supremum over $a,b$ yields condition (b') for $\mathfrak{M}_{q,p}$.
\end{enumerate}

Naturally, the list above is not meant to be exhaustive.
Without going into the details, we point out that \cite{Junge-Habilitationschrift} defines a number of other potential examples of mapping ideals including:
completely $p$-integral mappings (Def. 3.1.3.2)
the completely injective hull of the completely $p$-nuclear mappings (Def. 3.1.3.5),
completely $(p,k)$-mappings and $(p,k)$-integral maps (Def. 3.1.3.6),
mappings admitting completely bounded factorizations through diagonal operators from $\ell_p$ to $\ell_q$ and their noncommutative analogue, namely two-sided multiplication operators from $S_p$ to $S_q$ (Sec. 3.1.4).
While initially only a norm is defined, characterizations are proved that allow us to obtain o.s. structures.
Observe that by Proposition \ref{prop-comp-injective-hull}, the aforementioned completely injective hull of the completely $p$-nuclear mappings is indeed a mapping ideal.
Note that this is an operator space version of the classical ideal of quasi $p$-nuclear operators (often denoted $\mathcal{QN}_p$) introduced in \cite{Persson-Pietsch} and further studied in e.g. \cite{Reinov-APp,delgado2010operators}.
For all the other aforementioned examples from \cite{Junge-Habilitationschrift} we have not verified that condition (b') is satisfied, but expect that to be the case.

The following definition concerns the completeness of $\mathfrak{A}(E,F)$.
 
\begin{definition}
A mapping ideal $\mathfrak{A}$ is called a \textit{Banach mapping ideal} if $\mathfrak{A}(E,F)$ is a Banach space, for every $E,F \in \OBAN$.
\end{definition}

Note that all the examples listed above are Banach mapping ideals, except for the ideal of finite rank mappings.

As in the Banach space setting, each o.s. tensor norm allows us to construct a mapping ideal ``dual'' to it.
\begin{example} \label{U-alpha}
Given an o.s. tensor norm $\alpha$, for $E,F \in \OBAN$ let
$$
\mathfrak{A}_\alpha(E,F):= (E\otimes_\alpha F')' \cap \CB(E,F).
$$ Then, $\mathfrak{A}_\alpha$ is a mapping ideal.
\end{example}

Indeed,  the complete contraction $E\otimes_{\proj} F'\hookrightarrow E\otimes_\alpha F'$ implies that the identity map $\mathfrak{A}_\alpha(E,F)\subset \CB(E,F)$ is completely contractive too. Also, for $E,F\in\OFIN$, $\mathfrak{A}_\alpha(E,F)= E'\otimes_{\alpha'} F$ and therefore condition (b') of the definition holds. On the other hand, the metric mapping property of $\alpha$ readily implies the ideal property of $\mathfrak{A}_\alpha$. 

We now consider ideals that ``do not change when followed by a complete isometry''. To be precise:

\begin{definition}
A mapping ideal $\mathfrak{A}$ is said to be \emph{completely injective} if for each complete isometry $i : F \to G$ and $T \in M_n(\CB(E,F))$ the following equivalence holds: $T \in M_n(\mathfrak{A}(E,F))$ if and only if $i_n \circ T \in M_n(\mathfrak{A}(E,G))$, with $\mathbf{A}_n(T) = \mathbf{A}_n(i_n \circ T) $.
\end{definition}

It straightforward that the mapping ideal $\CB$ is completely injective. As a consequence we also have that $\Pi_p$ is completely injective. Indeed, let $T \in M_n(\CB(E,F))$ and $i: F \to G$ a complete isometry such that $i_n \circ T \in  M_n(\Pi_p (E,G))$. 
Since $$ M_n(\Pi_p (E,G)) \hookrightarrow M_n(\CB(S_p \otimes_{\min}E, S_p[G]))$$ and $i$ induces a complete isometry from $S_p[F]$ to $S_p[G]$ \cite[Cor. 1.2]{Pisier-Asterisque-98}, we obtain that $T \in  M_n(\Pi_p (E,F))$.

\bigskip

\begin{remark}
Observe that to prove that a mapping ideal $\mathfrak{A}$ is completely injective it is enough to check the following two conditions:
\begin{enumerate}[(a)]
    \item for each complete isometry $i : F \to G$ and $T \in \CB(E,F)$ such that $i \circ T \in \mathfrak{A}(E,G)$ we have that $ T \in \mathfrak{A}(E,F)$.
    \item $\mathbf{A}_n(T) \leq \mathbf{A}_n(i_n \circ T) $ for every complete isometry $i : F \to G$ and $T \in M_n(\mathfrak{A}(E,F))$. 
    \end{enumerate}
\end{remark}

A normed operator space $E$ has the \emph{complete $C$-extension property} if for every $T \in \CB(E,F)$ and a complete injection $i:E \to G$ there is $\widetilde T \in \CB(G,F)$ such that the following diagram commutes
$$
\xymatrix{
G  \ar@{.>}[rd]^{\widetilde T} &\\
E \ar@{^{(}->}[u]^{i} \ar[r]_{T} & F
}	
$$
with $\Vert \widetilde T \Vert_{\cb} \leq C \Vert T \Vert_{\cb}.$

Recall that $B(H)$ has the complete $1$-extension property \cite[Thm. 1.6]{Pisier-Operator-Space-Theory}.

\bigskip
Given a mapping ideal $(\mathfrak{A},\mathbf{A})$, the following proposition allows us to talk about the smallest completely injective ideal that contains it. 

\begin{proposition}\label{prop-comp-injective-hull}
Let $(\mathfrak{A},\mathbf{A})$ be a mapping ideal. Then there is a (unique) smallest completely injective mapping ideal $\mathfrak{A}^{\inj}$ which contains $\mathfrak{A}$ (called the \emph{completely injective hull} of $\mathfrak{A}$). Moreover, if $V$ is a completely injective operator space and $i : F \to V$ is a complete isometry, then $T \in M_n(\CB(E,F))$ is in $M_n(\mathfrak{A}^{\inj}(E,F))$ if and only if $i_n \circ T \in M_n(\mathfrak{A}(E,V))$ with
$$
\mathbf{A}^{\inj}_n(T) = \mathbf{A}_n(i_n \circ T).
$$

\end{proposition}

\begin{proof}
Given $F\in\OBAN$  we denote by $i_F: F \hookrightarrow B(H_F)$ any  embedding that provides $F$ the operator space structure. We define
$$\mathfrak{A}^{\inj}(E,F):=\{T \in \CB(E,F) : i_F \circ T \in \mathfrak{A}(E,B(H_F))\},$$ endowed with the operator space norm  defined for  $T \in M_n(\CB(E,F))$ as 
$$
\mathbf{A}^{\inj}_n(T) := \mathbf{A}_n((i_F)_n \circ T).
$$
To see that this is well defined, let us prove that if $i: F \to V$ is any complete isometry of $F$ into a completely  injective space
$V$ (which has to be complete as explained in Section \ref{injections-and-projections}) then 
for every $n \in \mathbb N$ and $T \in M_n(\CB(E,F))$ such that  $i_n \circ T \in M_n(\mathfrak{A}(E,V))$ we have the equality
\begin{equation} \label{igualdad normas}
\mathbf{A}_n(i_n \circ T)=\mathbf{A}_n((i_F)_n \circ T).
\end{equation}
Indeed, by the  injectivity of $B(H_F)$ there is a mapping $\widetilde i_F: V \to B(H_F)$ such that $ i_F = \widetilde i_F \circ i $ with $\Vert \widetilde i_F \Vert_{\cb}=1$. 
Then, by the ideal property
$$
\mathbf{A}_n((i_F)_n \circ T) = \mathbf{A}_n((\widetilde i_F )_n \circ  i_n \circ T) \leq \mathbf{A}_n(i_n \circ T).
$$
On the other hand, by the the  injectivity of $V$, we can argue as before to show the existence of a mapping  $\widetilde i: B(H_F) \to V$ such that 
$i=\widetilde i \circ i_F$  with $\Vert \widetilde i \Vert_{\cb}=1$ and therefore
$$
\mathbf{A}_n(i_n \circ T) = \mathbf{A}_n((\widetilde i )_n \circ  (i_F)_n \circ T) \leq \mathbf{A}_n((i_F)_n \circ T).
$$
This shows Equation \eqref{igualdad normas}.

To see that $\mathfrak{A}^{\inj}$ is a mapping ideal, we will only prove the ideal property since the other conditions easily follow. 
Let $T \in M_n(\mathfrak{A}^{\inj}(E,F))$, $r \in \CB(E_0,E)$ and $s \in \CB(F,F_0)$. 
Then,
$$
\mathbf{A}^{\inj}_n(s_n\circ T \circ r) = \mathbf{A}_n((i_{F_0})_n \circ s_n \circ T \circ r) \leq \mathbf{A}_n((i_{F_0} \circ s)_n \circ T ) \Vert r \Vert_{\cb}
$$
% By the extension property of $B(H)$ \cite[Theorem 1.6]{Pisier-Operator-Space-Theory},
Once again the   injectivity of $B(H_{F_0})$ provides us with a mapping $\widetilde s \in \CB(B(H_F), B(H_{F_0}))$ satisfying
$i_{F_0}\circ s = \widetilde s \circ i_F$, with $\Vert \widetilde s \Vert_{\cb} \leq \Vert s \Vert_{\cb}$. Thus,
$$
\mathbf{A}^{\inj}_n(s_n\circ T \circ r) \leq \mathbf{A}_n((\widetilde s \circ i_F)_n \circ T ) \Vert r \Vert_{\cb} \leq \Vert s \Vert_{\cb} \mathbf{A}_n( (i_F)_n \circ T ) \Vert r \Vert_{\cb} = \Vert s \Vert_{\cb} \mathbf{A}_n^{\inj}( T ) \Vert r \Vert_{\cb}.
$$
We now see that $\mathfrak{A}^{\inj}$ is completely injective. Indeed, let $i : F \to G$ a complete isometry, $T \in \CB(E,F)$ and suppose that $i \circ T \in \mathfrak{A}^{\inj}(E,G)$. If $i_G : G \to B(H_G)$ is a completely isometric embedding,  the extension property of $B(H_G)$ shows the existence of a mapping $\widetilde{i_F}$ such that  
\begin{equation}
   \widetilde{i_F} \circ (i_G \circ i) = i_F. 
\end{equation}
Now, the definition of the ideal implies $i_G \circ  i \circ T \in  \mathfrak{A}(E,B(H_G))$. Then,  $\widetilde{i_F} \circ (i_G \circ i) \circ T = i_F \circ T$ belongs to $\mathfrak{A}(E,B(H_F))$. Hence, $T \in \mathfrak{A}^{\inj}(E,F).$

Now, if $T \in M_n(\mathfrak{A}^{\inj}(E,F))$ and $i : F \to G$ is a complete isometry, we have
\begin{align*}
   \mathbf{A}_n^{\inj}(T) & = \mathbf{A}_n((i_F)_n \circ T)= \mathbf{A}_n((\widetilde{i_F} \circ i_G \circ i)_n \circ T)  \\ 
    & \leq \mathbf{A}_n( (i_G)_n \circ(i_n \circ T))   = \mathbf{A}_n^{\inj}(i_n \circ T) .
\end{align*}
That $\mathfrak{A}^{\inj}$ is the smallest completely injective ideal which contains $\mathfrak{A}$ is trivial by definition of $\mathfrak{A}^{\inj}$. The last part of the statement follows from what was done at the beginning of the proof to show that $\mathfrak{A}^{\inj}(T)$ was well defined.
\end{proof}

\begin{example}
For the mapping ideal $\mathcal I$ of completely integral mappings,  the previous proposition says that $T\in\mathcal I^{\inj}(E,F)$ if and only if $i_F\circ T\in\mathcal I(E,B(H_F)) $. By \cite[Prop. 15.5.1]{Effros-Ruan-book} $\mathcal I(E,B(H_F))=\Pi_1 (E,B(H_F))$.
Since we have already explained that $\Pi_1 $ is a completely injective mapping ideal we derive that $\mathcal{I}^{\inj}=\Pi_1 $ (as it happens in the Banach space framework).
\end{example}

The following simple result shows that the $\inj$-procedure respects completeness.

\begin{proposition} \label{Uinj completo}
Let  $(\mathfrak{A},\mathbf{A})$  be a Banach mapping ideal then $(\mathfrak{A}^{\inj},\mathbf{A}^{\inj})$ is also a Banach mapping ideal.
\end{proposition}

\begin{proof}
Let $E,F\in \OBAN$, we only have to check that $\mathfrak{A}^{\inj}(E,F)$ is complete.
Let $(T_k)_{k \in \mathbb{N}}$ be a Cauchy sequence in $\mathfrak{A}^{\inj}(E,F)$. Note that $(T_k)_{k \in \mathbb{N}}$ is also a Cauchy sequence in $\CB(E,F)$ and then it converges (in the $cb$-norm) to an operator $T \in \CB(E,F)$.
On the other hand,  $(i_F \circ T_k)_{k \in \mathbb{N}}$ is a Cauchy sequence in $\mathfrak{A}(E,B(H_F))$ which converges to an operator $S \in \mathfrak{A}(E,B(H_F))$. To finish, it remains to prove that $S = i_F \circ T$. Indeed,  note that  $T_k \to T$ pointwise and therefore $i_F \circ T_k \to i_F \circ T$ pointwise as well. Now the result follows by the uniqueness of the limit. 

\end{proof}

As a corollary we have an analogous version of \cite[Corollary 9.7]{Defant-Floret}
for the operator space setting. 
\begin{corollary}
Let $F \in \OBAN$ with the complete $C$-extension property and $(\mathfrak{A},\mathbf{A})$ be a mapping ideal. Then for every $n \in \mathbb N$ and every $E \in \OBAN$ we have
$M_n(\mathfrak{A}(E,F))= M_n(\mathfrak{A}^{\inj}(E,F))$. Moreover, $$\mathbf{A}_n(T) \leq C \mathbf{A}_n^{\inj}(T),$$
for all $T \in M_n(\mathfrak{A}^{\inj}(E,F))$.
\end{corollary}
Similarly, we can consider ideals that behave well when composed with complete metric surjections.

\begin{definition}
A mapping ideal $(\mathfrak{A},\mathbf{A})$ is said to be \emph{completely surjective} if for each complete metric surjection $q : G \twoheadrightarrow E$ and $T \in M_n(\CB(E,F))$ the following equivalence holds: $T \in M_n(\mathfrak{A}(E,F))$ if and only if $ T \circ q \in M_n(\mathfrak{A}(G,F))$, with $\mathbf{A}_n(T) = \mathbf{A}_n(T \circ q) $.
\end{definition}

It is immediate to see that the mapping ideal $\CB$ is completely surjective. 

\bigskip

Given a mapping ideal $(\mathfrak{A},\mathbf{A})$, the following proposition defines the smallest completely surjective ideal that contains it. 
This was already proved in \cite[Prop. 3.6]{chavez2019operator}, we include a detailed proof for completeness. 

%Recall that a normed operator space $E$ is called \emph{completely projective} if, for any $\varepsilon>0$, any completely bounded map $T:E \to F/S$ into a quotient space (here $F$ is any operator space and $S \subset F$ any closed space) admits a lifting $\widetilde T: E \to F$ with $\Vert \widetilde{T} \Vert_{\cb} \leq (1+ \varepsilon) \Vert T \Vert_{\cb}$.
%A typical example of a completely projective operator space is the finite-dimensional Schatten class $S_1^n$.
%\textcolor{red}{Esto ya est\'a en los preliminares, creo que no hace falta repetirlo.
%Tal vez el siguiente p\'arrafo tambi\'en podr\'ia estar en los preliminares.
%Y tenemos un problema de consistencia: a veces decimos nada m\'as projective, pero a veces completely projective. Lo mismo para injective, locally reflexive.}

%Every Banach operator space $E$ can be seen as a quotient of a completely projective space. Indeed, there is a set $I$ and a family $(n_i)_{i \in I} \subset \mathbb{N}$ such that $E$ is the quotient of $\ell_1(\{S_1^{n_i} : i \in I \})$ (see e.g., \cite[Prop. 2.12.2]{Pisier-Operator-Space-Theory}). We denote the latter space by $Z_E$ and $q_E : Z_E \twoheadrightarrow E$ the corresponding complete quotient mapping. The fact that $Z_E$ is completely projective can be tracked in \cite[Chap. 24]{Pisier-Operator-Space-Theory}.

\begin{proposition} \label{surjective hull}
Let $(\mathfrak{A},\mathbf{A})$ be a mapping ideal. Then there is a (unique) smallest completely surjective mapping ideal $\mathfrak{A}^{\sur}$ which contains $\mathfrak{A}$ (called the \emph{completely surjective hull} of $\mathfrak{A}$). Moreover: if $G$ is a completely projective Banach operator space, $E \in \OBAN$ and
$q : G \twoheadrightarrow E$ is a complete metric surjection, then $T \in M_n(\CB(E,F))$ is in $M_n(\mathfrak{A}^{\sur}(E,F))$ if and only if $T \circ q \in M_n(\mathfrak{A}(G,F))$ with
$$
\mathbf{A}^{\sur}_n(T) = \mathbf{A}_n(T \circ q).
$$
\end{proposition}
\begin{proof}
 We define $$\mathfrak{A}^{\sur}(E,F) = \{ T \in \CB(E,F) : T  \circ q_E \in \mathfrak{A}(Z_E,F)\},$$
 endowed with the operator space structure defined in the statement.
First we see that it is well defined.  To see this, suppose that $G$ is a completely projective Banach operator space and
$q : G \twoheadrightarrow E$ is a complete metric surjection  and $T \in M_n(\CB(E,F))$ such that $T \circ q \in M_n(\mathfrak{A}(G,F))$, we have to see that
\begin{equation} \label{igualdad sur}
\mathbf{A}_n(T \circ q_E)= \mathbf{A}_n(T \circ q).
\end{equation}

By the complete projectivity of $Z_E$ there is a mapping $\widetilde q_E: Z_E \to G$ such  that $q_E = q \circ \widetilde q_E$ with $\Vert \widetilde q_E \Vert_{\cb} \leq 1$. Then,
$$
\mathbf{A}_n(T \circ q_E) = \mathbf{A}_n(T \circ q \circ \widetilde q_E) \leq \mathbf{A}_n(T \circ q).
$$
On the other hand, by the projectivity of $G$ there is a mapping $\widetilde q: G \to Z_E$ such  that $q = q_E \circ \widetilde q$ with $\Vert \widetilde q \Vert_{\cb} \leq 1.$ As before,
$$
\mathbf{A}_n(T \circ q)  = \mathbf{A}_n(T \circ q_E \circ \widetilde q) \leq \mathbf{A}_n(T \circ q_E).
$$

This shows Equation \eqref{igualdad sur}.

To see that $\mathfrak{A}^{\sur}$ is a mapping ideal, we will only prove the ideal property since the other conditions easily follow. 

Let $T \in M_n(\mathfrak{A}^{\sur}(E,F))$, $r \in \CB(E_0,E)$ and $s \in \CB(F,F_0)$. 
Since $Z_{E_0}$ is completely projective, given $\varepsilon >0$, there is a lifting $L_{\varepsilon} \in \CB(Z_{E_0},Z_E)$ of $r \circ q_{E_0}$ with $\Vert L_{\varepsilon} \Vert_{\cb} \leq (1+\varepsilon) \Vert r \Vert_{\cb}$ such that  the following diagram commutes

$$
\xymatrix{
E_0   \ar[r]^r    & E   \ar[r]^{T_{ij} }   & F  \ar[r]^s & F_0 \\
Z_{E_0}  \ar@{->>}[u]^{q_{E_0}}  \ar@{.>}[r]_{L_{\varepsilon}}  & Z_E   \ar@{->>}[u]_{q_{E}} &  &
}.
$$

Then, for every $1 \leq i, j \leq n$ we have $s \circ T_{ij} \circ r \circ q_{E_0} = s \circ T_{ij} \circ q_{E} \circ L_{\varepsilon} \in \mathfrak{A}(Z_{E_0},F_0)$ and hence  $ s \circ T_{ij}\circ  r \in \mathfrak{A}^{\sur}(E_0,F_0)$.
Moreover, by the ideal property of $\mathfrak{A}$,
\begin{align*}
 \mathbf{A}_n^{\sur}( s_n \circ T \circ r ) & = \mathbf{A}_n( s_n \circ  T \circ  r \circ q_{E_0} ) =  \mathbf{A}_n( s_n \circ  T \circ  q_{E} \circ  L_{\varepsilon} ) \\
 & \leq  \mathbf{A}_n( s_n \circ  T \circ  q_{E}  ) (1+\varepsilon) \Vert r \Vert_{\cb}  \\
 & \leq (1+\varepsilon) \Vert s \Vert_{\cb} \mathbf{A}_n^{\sur}( T ) \Vert r \Vert_{\cb}.
\end{align*}
We now prove that the mapping ideal $\mathfrak{A}^{\sur}$ is surjective. Let $q : G \twoheadrightarrow E$ be a complete quotient mapping and $T \in \CB(E,F)$ such that $T\circ q \in \mathfrak{A}^{\sur}(G,F)$.
Since $Z_E$ is completely projective and $q\circ q_G : Z_G \twoheadrightarrow E$ is a complete quotient mapping, given $\varepsilon >0$, there is a lifting $L_{\varepsilon}: Z_E \to Z_G$ of $q_E$,  with cb-norm less than or equal to $1+ \varepsilon$.

We have $T\circ q_E = T\circ  q \circ q_G \circ L_{\varepsilon} \in \mathfrak{A}(Z_E,F)$ since $T\circ q \circ q_G \in \mathfrak{A}(Z_G,F)$. Then, $T \in \mathfrak{A}^{\sur}(E,F)$.
Also, if $T \in M_n(\mathfrak{A}^{\sur}(E,F))$ we obtain

\begin{align*}
 \mathbf{A}^{\sur}(T) & =  \mathbf{A}( T\circ q_E  ) =  \mathbf{A}( T\circ q\circ  q_G \circ L_{\varepsilon}  ) \leq (1 + \varepsilon) \mathbf{A}( T\circ q \circ q_G ) \\
 &  = (1 + \varepsilon) \mathbf{A}^{\sur}( T\circ q ) \leq (1+ \varepsilon) \mathbf{A}^{\sur}( T ).
\end{align*}

That $\mathfrak{A}^{\sur}$ is the smallest completely projective ideal is trivial by definition of $\mathfrak{A}^{\sur}$. The last part of the statement follows from what was done at the beginning of the proof to show that $\mathfrak{A}^{\sur}(T)$ was well defined.
\end{proof}

It was proven in \cite[Prop. 3.8]{chavez2019operator} that $$\mathcal{K}_p  = (\mathcal N^p)^{\sur}.$$ 
As a consequence, we obviously have that $\mathcal{K}_p $ is a surjective mapping ideal.

As a corollary of the previous proposition we have an analogous version of \cite[Corollary 9.8]{Defant-Floret}
for the operator space setting.

\begin{corollary}
Let $E$ be a completely projective Banach operator space and $(\mathfrak{A},\mathbf{A})$ be a mapping ideal. Then for every every $F \in \OBAN$ we have that
$\mathfrak{A}(E,F)$  and $\mathfrak{A}^{\sur}(E,F)$ are completely isometric. 
\end{corollary}

As for the $\inj$-procedure, we have that the $\sur$-procedure also respects completeness. The proof is analogous to that  of Proposition \ref{Uinj completo} so we omit it.

\begin{proposition} \label{Usur completo}
Let  $(\mathfrak{A},\mathbf{A})$  be a Banach mapping ideal then $(\mathfrak{A}^{\sur},\mathbf{A}^{\sur})$ is also a Banach mapping ideal.
\end{proposition}

We now define the dual procedure of a mapping ideal.
\begin{definition}
If $(\mathfrak{A},\mathbf{A})$ is a mapping ideal, define for $E,F \in\OBAN$
$$
\mathfrak{A}^{\dual}(E,F) := \{ T \in \CB(E,F) \, : \, T' \in \mathfrak{A}(F',E') \}
$$
and
$$
\mathbf{A}_n^{\dual}(T) := \mathbf{A}_n(T'), \quad T \in M_n(\mathfrak{A}^{\dual}(E,F)).
$$
It is readily seen that $(\mathfrak{A}^{\dual},\mathbf{A}^{\dual})$ is also a mapping ideal.
\end{definition}

It is clear from the definition that the ideal of completely approximable mappings satisfies that $\mathfrak{A} = \mathfrak{A}^{\dual}$. On the other hand, by \cite[Proposition 12.2.5]{Effros-Ruan-book}, if $T\in \mathcal N(E,F)$ we have that $T \in \mathcal N^{\dual}(E,F)$ with $\nu^{\dual}(T)=\nu(T') \leq \nu(T).$
In fact, more is true: the same proof shows that the map $\mathcal N(E,F) \to \mathcal N(F',E')$ given by $T \mapsto T'$ is a complete contraction,
and thus the identity map $\mathcal N(E,F) \to \mathcal N^{\dual}(E,F)$ is a complete contraction as well.
In the case when $E'$ has the operator approximation property
the map $\Phi$ in \eqref{def nucleares} is injective \cite[Thm. 11.2.5]{Effros-Ruan-book},
and therefore the same proof as in \cite[Proposition 12.2.6]{Effros-Ruan-book} shows that the two maps above are in fact complete isometries.

Let us also mention that mappings belonging to the ideals $(\Pi_p )^{\dual}$ have already been considered in the literature.
For example, \cite[Thm. 6.5]{Pisier-Asterisque-98} characterizes mappings that admit a completely bounded factorization through Pisier's operator Hilbert space $OH$ in terms of conditions involving both $\Pi_2 $ and $(\Pi_2 )^{\dual}$,
and \cite[Cor. 7.2.2]{Pisier-Asterisque-98} similarly characterizes mappings admitting a completely bounded factorization through finite-dimensional Schatten $p$-spaces with conditions that involve   $\Pi_{p'} $ and  $(\Pi_p )^{\dual}$.

%------------------------------------------------------
%------------------------------------------------------
%------------------------------------------------------
\section{Maximal o.s. mapping ideals}
\label{represent sec}

In the Banach space setting there is an intrinsic relation between  tensor norms and operator ideals. 
As mentioned before, the same happens in the non-commutative context. In this section we develop the theory of maximal mapping ideals, showing  that the latter and the theory of o.s. tensor norms  are two sides of the same coin. 

If $(\mathfrak{A},\mathbf{A})$ is a mapping ideal, then
$$
M \otimes_{\alpha} N := \mathfrak{A}(M', N)
$$
defines an o.s. tensor norm $\alpha$ on $\OFIN$; in other words, if $z \in M_n(M \otimes N)$ and $T_z \in M_n(\CB(M',N))$ is the associated matrix of linear operators, then
$$
\alpha_n(z ; M, N) := \mathbf{A}_n(T_z : M' \to N).
$$

Since the natural maps $\mathfrak{A}(M',N) \to \CB(M', N)$ and $M \otimes_{\min} N \to \CB(M' , N)$ are respectively a complete contraction and a complete isometry, it follows that
$$
M \otimes_{\alpha} N \to M \otimes_{\min} N
$$
is a complete contraction.

Note that equation \eqref{norma producto matrices} follows immediately from condition (b) in the definition of mapping ideals. Also,
the ideal property of $\mathfrak{A}$ implies the complete metric mapping property of $\alpha$:
for any linear maps  $r \in\CB(M,M_0)$ and $s \in \CB( N ,N_0)$, note that for any $z \in M_n(M \otimes N)$
$$
s_n \circ T_z \circ r' = T_{(r \otimes s)(z) }.
$$
and thus
$$
\alpha_n\big( (r\otimes s)(z) \big) = \mathbf{A}_n( T_{(r \otimes s)(z) }) = \mathbf{A}_n(s_n \circ T_z \circ r' ) \le \n{s}_{\cb} \mathbf{A}_n(T_z) \n{r'}_{\cb} = \n{s}_{\cb}\n{r}_{\cb} \alpha_n(z).
$$

Conversely, let $\alpha$ be an o.s. tensor norm on $\OFIN$ satisfying the complete metric mapping property.
Define for $M,N \in \OFIN$  and $T \in M_n(\CB(M,N))$
$$
\mathbf{A}_n(T : M \to N) := \alpha_n(z_T ; M', N).
$$
where $z_T \in M_n(M'\otimes N)$ is the associated matrix of tensors.
Let $r \in \CB(M_0,M)$ and $s \in CB(N, N_0)$. Again, condition (b) is related with equation \eqref{norma producto matrices}.
Note that
$$
z_{s_n \circ T \circ r} = \big( r' \otimes s  \big) z_T
$$
and thus, by the complete metric mapping property of $\alpha$,
$$
\mathbf{A}_n( s_n \circ T \circ r ) = \alpha_n( z_{s_n \circ T \circ r} ) \le \n{r'}_{\cb} \n{s}_{\cb} \alpha_n(z_T) = \n{r}_{\cb} \n{s}_{\cb} \mathbf{A}_n(T)
$$

Therefore, in the finite-dimensional setting there is a full correspondence between mapping ideals and o.s. tensor norms. The rest of this section is devoted to the study of what can be said in the infinite-dimensional setting.
Naturally, the answer is satisfactory for mapping ideals that can be well-approximated by finite-dimensional pieces, the idea that motivates the next definition.

\begin{definition}\label{def-maximal-hull}
Assume $(\mathfrak{A},\mathbf{A})$ is a mapping ideal. For $E,F \in \OBAN$ and $T \in M_n(\CB(E,F))$, define
$$
\mathbf{A}^{\max}_n(T)  = \sup\left\{ \mathbf{A}_n( (q^F_L)_n \circ T \circ i^E_M ) \colon M \in \OFIN(E), L \in \OCOFIN(F)\right\}
$$
and
$$
\mathfrak{A}^{\max}(E,F)= \left\{ T \in \CB(E,F) \, : \,  \mathbf{A}^{\max}_1(T) < \infty  \right\}.
$$
We call $(\mathfrak{A}^{\max}, \mathbf{A}^{\max})$ the \emph{maximal hull} of  $(\mathfrak{A},\mathbf{A})$, and
 $(\mathfrak{A},\mathbf{A})$ is called \emph{maximal} if  $(\mathfrak{A},\mathbf{A}) = (\mathfrak{A}^{\max}, \mathbf{A}^{\max})$.
\end{definition}

\begin{proposition}
Let $(\mathfrak{A}, \mathbf{A})$ be a mapping ideal. Then $(\mathfrak{A}^{\max}, \mathbf{A}^{\max})$ is a Banach mapping ideal which, of course,  is maximal. 
\end{proposition}

\begin{proof}

It is clear that each $\mathbf{A}^{\max}_n$ is a norm, and that $\mathbf{A}^{\max}$ satisfies Ruan's axioms (say, as in \cite[Prop. 2.3.6]{Effros-Ruan-book}).

Using that the identity $\mathfrak{A}(E,F) \to \CB(E,F)$ is a complete contraction, and  the ideal property for the $\cb$-norm, for $T \in M_n(\CB(E,F))$ we have that 
\[
\mathbf{A}_n( (q^F_L)_n \circ T \circ i^E_M ) \le \n{(q^F_L)_n \circ T \circ i^E_M}_{M_n(\CB(M,F/L))} \le \n{T}_{M_n(\CB(E,F))}, 
\]
which implies that the identity $\mathfrak{A}^{\max}(E,F) \to \CB(E,F)$ is also a complete contraction.
Similarly, the ideal property for $(\mathfrak{A},\mathbf{A})$ shows that all finite-rank operators belong to $\mathfrak{A}^{\max}$ since they belong to $\mathfrak{A}$. Moreover, given $x'\in M_n(E')$ and $y\in M_m(F)$ we have
$$
\mathbf{A}_{nm}( (q^F_L)_{nm} \circ x'\otimes y \circ i^E_M ) =\mathbf{A}_{nm}(  x'\circ i^E_M\otimes (q^F_L)_{m}(y)) = \|x'\circ i^E_M\|_{\cb} \|(q^F_L)_{m}(y)\|_{\cb},
$$  showing that $\mathbf{A}_{nm}^{\max}( x'\otimes y )\le \|x'\|_{\cb} \|y\|_{\cb}$.

Suppose now that $r \in \CB(E_0,E)$ and $s \in \CB(F,F_0)$.
Let $M_0 \in \OFIN(E_0)$ and $L_0 \in \OCOFIN(F_0)$.
Define $M = r(M_0) \in \OFIN(E)$ and $L = s^{-1}(L_0)\in \OCOFIN(F)$.
Note that $r = i^E_M \circ r$, and $q^{F_0}_{L_0} \circ s = \widetilde{s} \circ q^F_L$ for some $\widetilde{s} \in \CB(F/L,F_0/L_0)$ with $\| \widetilde{s}\|_{\cb} = \| {s}\|_{\cb}$.
Therefore, by the ideal property for $(\mathfrak{A},\mathbf{A})$
\begin{multline*}
    \mathbf{A}_n\big( (q^{F_0}_{L_0})_n \circ s_n \circ T \circ r \circ i^{E_0}_{M_0}  \big) =  \mathbf{A}_n\big( \widetilde{s}_n\circ (q^{F}_{L})_n \circ T \circ i^E_M \circ r \circ i^{E_0}_{M_0}  \big) \\
    \le \| \widetilde{s}\|_{\cb}  \mathbf{A}_n\big(  (q^{F}_{L})_n \circ T \circ i^E_M    \big) \| r \circ i^{E_0}_{M_0} \|_{\cb} \le \|s\|_{\cb} \mathbf{A}^{\max}_n(T) \|r\|_{\cb},
\end{multline*}
which yields the ideal property for $(\mathfrak{A}^{\max}, \mathbf{A}^{\max})$.

Let us now check that $\mathfrak{A}^{\max}(E,F)$ is complete when $(\mathfrak{A},\mathbf{A})$ is a  mapping ideal.
For that, let $(T_k)_{k=1}^\infty$ be a sequence in $\mathfrak{A}^{\max}(E,F)$ with $\sum_{k=1}^\infty \mathbf{A}^{\max}_1(T_k) < \infty$. Since $\mathbf{A}^{\max}_1$ dominates the $\cb$-norm, the series $\sum_{k=1}^\infty T_k$ converges in $\CB(E,F)$ to an operator $T$.
Let us now show that the convergence also holds with respect to $\mathbf{A}^{\max}_1$.
Fix $M \in \OFIN(E)$ and $L \in \OCOFIN(F)$.
For $m \in \N$, note that
\[
\mathbf{A}_1\Big( q^F_L \circ \Big(\sum_{k=m}^\infty T_k\Big) \circ i^E_M \Big)
\le \sum_{k=m}^\infty  \mathbf{A}_1\big( q^F_L \circ  T_k \circ i^E_M \big) \le \sum_{k=m}^\infty \mathbf{A}^{\max}_1(T_k),
\]
from where $\mathbf{A}^{\max}_1\big( \sum_{k=m}^\infty T_k \big) \le \sum_{k=m}^\infty \mathbf{A}^{\max}_1(T_k)$ and therefore $\sum_{k=1}^\infty T_k$ converges to $T$ with respect to  $\mathbf{A}^{\max}_1$.
\end{proof}

\begin{example}\rm
 The following mapping ideals are maximal.
 \begin{enumerate}[(i)]
 \item Completely bounded mappings $\CB$.
 \item Completely integral mappings $\mathcal I$: See Example \ref{integrales maximal}, below.
     \item 
 Completely $p$-summing mappings $\Pi_{p} $: Let $E,F \in \OBAN$, $T \in M_n( \Pi_{p} (E,F) )$ and $\eps>0$.
Assume $(\pi_{p} )_n(T) = 1$.
Since $\Pi_{p} (E,F)$ is canonically completely isometrically embedded into $\CB( S_{p}\otimes_{\min} E, S_{p}[F] )$,
by a density argument there exist $k \in \N$ and $ x \in M_n( S^k_{p} \otimes_{\min} E ) $ with $\n{x} = 1$ such that
$$
\n{ (id_{S^k_{p}} \otimes T)(x) }_{M_n( S^k_{p}[F] )} > 1-\varepsilon,
$$
Since $S^k_{p}[F]$ is completely isometrically embedded into $S^k_{p}[F''] = (S_{p'}[F'])'$,
there exists $y \in M_{\ell} (S_{p'}^k[F'])$ with $\n{y} = 1$ and such that
$$
\n{ \mpair{y}{(id_{S^k_{p}} \otimes T)(x)} }_{M_{n\ell}} > 1-\varepsilon.
$$
Consider $y$ as an $k\ell \times k \ell$ matrix over $F'$, and let $L$ be the intersection of these $(k\ell)^2$ functionals in $F'$.
Note that $L \in \OCOFIN(F)$, and that $y$ is an $k\ell \times k \ell$ matrix over the annihilator $L^0$ of $L$ in $F$.
Since $(F/L) = L^0$ completely isometrically, it follows that $S^k_{p'}[F/L] = S^k_{p'}[L^0]$ completely isometrically as well. Therefore, the norm of $y$ as an element of $M_{\ell} (S_{p'}^k[(F/L)'])$ is also equal to one.
If $M$ is the finite-dimensional subspace of $E$ spanned by the $(nk)^2$ vectors that constitute $x$, it follows that
$$
(\pi_{p} )_n\big( (q^F_L)_n \circ T \circ i^E_M \big) > 1-\varepsilon.
$$
 \end{enumerate}
 \end{example}

\begin{definition}\label{defn-associated}
A   mapping ideal $(\mathfrak{A},\mathbf{A})$ and  a finitely-generated o.s. tensor norm $\alpha$ (on $\ONORM$) are said to be \emph{associated}, denoted $(\mathfrak{A},\mathbf{A}) \sim \alpha$ if for every $M,N \in \OFIN$ we have a complete isometry
$$
\mathfrak{A}(M,N) = M' \otimes_{\alpha} N,
$$
given by the canonical map $T \mapsto z_T$.
\end{definition}

\begin{remark}\label{misma norma asociada}
Notice that since this definition is based only on finite-dimensional spaces; two different mapping ideals can be associated to the same o.s. tensor norm. In particular,  $\mathfrak{A}\sim\alpha$ if and only if $\mathfrak{A}^{\max}\sim\alpha$. 
\end{remark}

For example, $\CB, \mathcal A \sim \min$ or $\mathcal N, \mathcal I \sim \proj$.  
Also, by the mere definition, $\mathcal N_p  \sim d_p $.

Observe that the constructions in Definitions \ref{def-finite-and-cofinite-hulls} and  \ref{def-maximal-hull}  establish a one-to-one correspondence between maximal mapping ideals and finitely-generated o.s. tensor norms.
The main result of this section is the following theorem, which shows that the finite-dimensional duality that defines this correspondence can be extended to the infinite-dimensional setting.

\begin{theorem}[Representation Theorem for maximal mapping ideals]\label{representation-theorem}
Let $(\mathfrak{A},\mathbf{A})$ be a maximal mapping ideal associated to the o.s. tensor norm $\alpha$. Then for any $E,F \in \OBAN$ there is a complete isometry
 \begin{align}
 \mathfrak{A}(E,F) &= \left(E \otimes_{\alpha'} F' \right)' \cap \CB(E,F) \label{eqn-representation-thm-1}
 \end{align}
\end{theorem}

\begin{proof}
In order to prove \eqref{eqn-representation-thm-1}, we need to show that for $T \in M_n(\CB(E,F))$, the following holds: $T$ belongs to $ M_n( \mathfrak{A}(E,F) )$ if and only if the associated matrix of bilinear maps  $\beta_{\kappa_F \circ T}$ is in $M_n\big((E \otimes_{\alpha'} F')' \big)$,  with equal norms.
Note that the inequality $\n{\beta_{\kappa_F \circ T}} \le C$ is equivalent to
\begin{equation}\label{eqn-characterization-by-duality}
\n{ \mpair{ \beta_{\kappa_F \circ T} }{ z } }_{M_{nm}}  \le C (\alpha')_m(z ; E, F') \text{ for all } z \in M_m( E \otimes_{\alpha'} F' ).
\end{equation}
By the maximality of $\mathfrak{A}$, it follows that
$T \in M_n( \mathfrak{A}(E,F) )$ with $\mathbf{A}_n(T) \le C$ if and only if
\begin{equation}\label{eqn-characterization-by-maximality}
\mathbf{A}_n( (q^F_L)_n \circ T \circ i^E_M ) \le C \qquad \text{ for any } M \in \OFIN(E) \text{ and } L \in \OCOFIN(F).
\end{equation}
Since $\mathfrak{A}(M,F/L) = (M \otimes_{\alpha'} L^0)'$ completely isometrically,
and recalling that $L^0$ varies over all spaces in $\OFIN(F')$ when $L$ varies over all spaces in $\OCOFIN(F)$, condition \eqref{eqn-characterization-by-maximality}
is equivalent to
\begin{multline}\label{eqn-characterization-by-maximality-II}
\n{ \mpair{ \beta_{(q^F_L)_n \circ T \circ i^E_M} }{z} }_{M_{mn}}  \le C (\alpha')_m(z ; M, L^0)\\
 \text{ for any } M \in \OFIN(E), \;  L^0 \in \OFIN(F') \text{ and } z \in M_m( M \otimes_{\alpha'} L^0 ).
\end{multline}
Notice that for $z \in M_m( M \otimes_{\alpha'} L^0 )$,
$$
\mpair{ \beta_{\kappa_F \circ T} }{ z }  = \mpair{ \beta_{(q^F_L)_n \circ T \circ i^E_M} }{z}.
$$
Therefore \eqref{eqn-characterization-by-maximality-II} is equivalent to \eqref{eqn-characterization-by-duality} because $\alpha'$ is finitely-generated,
finishing the proof.
\end{proof}
In fact the converse is also true. We write the statement in the following proposition and leave the proof as an exercise. Recall that in Example \ref{U-alpha} we define a mapping ideal ``dual'' to a given o.s. tensor norm.

\begin{proposition}
Let $\alpha$ be a finitely-generated o.s. tensor norm. The mapping ideal $(\mathfrak{A}_{\alpha'},\mathbf{A}_{\alpha'})$ (given by
 \begin{align}
 \mathfrak{A}_{\alpha'}(E,F) &:= \left(E \otimes_{\alpha'} F' \right)' \cap \CB(E,F) \label{eqn-representation-thm-2}
 \end{align}
for any $E,F \in \OBAN$)  is  maximal and its associated o.s. tensor norm is $\alpha.$ 
\end{proposition}

We use this fact to see, in the following example, that the ideal of completely integral mappings is maximal. 

\begin{example} \label{integrales maximal}

$\mathcal{I}$  is a maximal mapping ideal. 
\end{example}

\begin{proof} 
We will give two proofs of this (for the second, we present only a sketch of it). Let's start with the first one.

Recall that the o.s. tensor norm associated to $\mathcal I$ is $\proj$. By the previous proposition, to see that $\mathcal I$ is maximal we need to check that for every $E,F \in \OBAN$ there is a complete isometry
\begin{equation}
    \mathcal{I}(E,F)= \left(E \otimes_{\min} F' \right)' \cap \CB(E,F).
\end{equation}
By \cite[Lemma 12.3.3]{Effros-Ruan-book} the canonical inclusion $$S_0: \mathcal I(E,F) \hookrightarrow (E \otimes_{\min}F')'$$ is a complete isometry. Thus, it remains to see that if $\Psi \in (E \otimes_{\min}F')' \cap \CB(E,F)$ then $\Psi=S_0(T)$, for some $T\in \mathcal I(E,F).$
Without loss of generality, we suppose that $\Vert \Psi\Vert_{(E \otimes_{\min}F')'}< 1.$
By Remark \ref{extension-sin-loc-refl} the following inclusion is completely isometric: $$E \otimes_{\min}F' \hookrightarrow (E' \otimes_{\proj}F)' = \CB(E',F').
$$ Then, by the Averson-Wittstock Hahn-Banach theorem \cite[Theorem 4.1.5]{Effros-Ruan-book}, $\Psi$ extends isometrically to $(E' \otimes_{\proj}F)''$. Therefore, Goldstine's theorem provides us of a net $(\Psi_{\gamma})_{\gamma}$ such that $\Vert \Psi_{\gamma} \Vert_{E' \otimes_{\proj}F}< 1,$ which $w^*$-converges to $\Psi$ (that is in the topology  $\sigma((E' \otimes_{\proj}F)'',(E' \otimes_{\proj}F)')$.
We consider $\Phi(\Psi_{\gamma})_{\gamma} \subset \mathcal N(E,F)$ with $\nu(\Phi(\Psi_{\gamma}))<1$, where $\Phi$ is the mapping defined in \eqref{def nucleares}.
Note that given $x \in E$, $y'\in F'$ 
$$\langle y', \Phi(\Psi_{\gamma}) \rangle \to \langle \Psi, x \otimes y' \rangle.$$
Now, the expression $\langle T(x), y' \rangle:= \langle \Psi, x \otimes y'\rangle$ defines a mapping $T \in \CB(E,F)$ satisfying that $\Phi(\Psi_{\gamma})$ converges to $T$ in the point-weak topology. Finally, appealing to \cite[Lemma 12.3.1]{Effros-Ruan-book} we get that $T$ is completely integral and $\Psi=S_0(T)$. This concludes the first proof.

We now present a sketch of an alternative argument for the fact that $\mathcal{I}$ is maximal, 
whose perspective could be enriching.
First, observe that \cite[Ex. 17.1]{Defant-Floret} also holds in the operator space setting. That is:
suppose that for each $\gamma\in\Gamma$, $(\mathfrak{A}^\gamma,\mathbf{A}^\gamma)$ is a maximal mapping ideal.
Define for $E,F \in \OBAN$ and $T \in M_n(\CB(E,F))$,
$$
\mathbf{A}_n(T)  = \sup\left\{ \mathbf{A}^\gamma_n(T) \, : \, \gamma \in \Gamma \right\} \quad
\text{and}
\quad
\mathfrak{A}(E,F)= \left\{ S \in \CB(E,F) \, : \,  \mathbf{A}_1(S) < \infty  \right\}.
$$
It is easy to see that $(\mathfrak{A},\mathbf{A})$ is also a maximal mapping ideal.
Now observe that \cite[Lemma 12.3.1]{Effros-Ruan-book} says that the ideal of completely integral mappings can be obtained by applying the above procedure to the ideals $(\mathfrak{A}^G,\mathbf{A}^G)$ defined by the completely isometric embeddings
$$
\mathfrak{A}^G(E,F) \hookrightarrow \CB( E \otimes_{\min} G, F \otimes_{\proj} G ),
$$
where $G$ ranges over all Banach operator spaces (or just the finite-dimensional ones).
Since each $\mathfrak{A}^G$ is clearly maximal, it follows that so is $\mathcal{I}$. This concludes the sketch of the second proof.
\end{proof}

% \begin{proof}
% To see that $(\mathfrak{A},\mathbf{A})$ is maximal, take $T \in \CB(E,F)$ such that 
% $$\mathbf{A}^{\max}(T)=\sup\left\{ \mathbf{A}( q^F_L \circ T \circ i^E_M ) \colon M \in \FIN(E), L \in \COFIN(F)\right\}< \infty.$$
% Let us see that $T \in \left(E \otimes_{\alpha'} F' \right)'.$
% Consider $z\in E \otimes F$ such that $z \in M \otimes L^{\circ}$, where $M\in FIN(E)$ and $L \in COFIN(F).$
% Then,
% \begin{align}
%     \vert \langle T,z \rangle \vert &= \vert q^F_L \circ T \circ i^E_M (z) \vert  \leq \Vert q^F_L \circ T \circ i^E_M \Vert_{\left(M \otimes_{\alpha'} L^{\circ} \right)'} \alpha'(z,M\otimes L^{\circ}) \\
%     & \leq \mathbf{A}^{\max}(T) \alpha'(z,M\otimes L^{\circ}).
% \end{align}
% Taking infimum over all possible pairs $M$ and $L$ as above, we have that
% $$\vert \langle T,z \rangle \vert \leq \mathbf{A}^{\max}(T) \alpha'(z,E\otimes F').$$
% Thus, $T \in \left(E \otimes_{\alpha'} F' \right)'.$
% \end{proof}

\begin{proposition}\label{Inclusion maximal}
Let $(\mathfrak{A},\mathbf{A})$ and $(\mathfrak B, \mathbf{B})$ be Banach mapping ideals. Suppose that for every $M,N \in \OFIN$ we have that the inclusion $\Vert \mathfrak A (M,N) \hookrightarrow  \mathfrak B (M,N)  \Vert_{\cb} \leq c$. Then
     $\mathfrak{A}^{\max} \subset \mathfrak{B}^{\max}$ and
     $\Vert \mathfrak A^{\max} (E,F) \hookrightarrow  \mathfrak B^{\max} (E,F)  \Vert_{\cb} \leq c$ for every $E,F\in\OBAN$. 
\end{proposition}

\begin{proof}
Let $T \in M_n(\CB(E,F))$.
Fix $M\in\OFIN(E)$ and $L\in\OCOFIN(F)$. From the assumption, note that
$\mathbf{B}_n( (q^F_L)_n \circ T \circ i^E_M ) \le c \mathbf{A}_n( (q^F_L)_n \circ T \circ i^E_M )$. Taking the supremum over all $M$ and $L$ yields $\mathbf{B}_n^{\max}(T) \le c \mathbf{A}_n^{\max}(T)$, giving the result.
\end{proof}

\begin{proposition}
Let $(\mathfrak{A},\mathbf{A})$ be a mapping ideal associated to a finitely generated o.s. tensor norm $\alpha$. Then $(\mathfrak{A}^{\max},\mathbf{A}^{\max})$ is the largest mapping ideal associated to $\alpha$.
\end{proposition}

\begin{proof}
The ideal property of $(\mathfrak{A},\mathbf{A})$ yields that the identity $\mathfrak A (E,F) \to  \mathfrak A^{\max} (E,F)$ is completely contractive for any $E,F\in\OBAN$. For $M,N\in\OFIN$ we additionally get that the identity $\mathfrak{A}^{\max} (M,N) \to  \mathfrak A(M,N)$ is also completely contractive since for $T \in \CB(M,N)$ we have that $T = (q^N_{\{0\}})_n \circ T \circ i^M_M$. Therefore, the identity $\mathfrak{A}^{\max} (M,N) \to  \mathfrak A(M,N)$ is a complete isomorphism and thus $(\mathfrak{A}^{\max},\mathbf{A}^{\max})$ is associated to $\alpha$.

Let $(\mathfrak B, \mathbf{B})$ be any mapping ideal associated to $\alpha$, and let $T \in M_n(\mathfrak{B}(E,F))$.
Let $M\in\OFIN(E)$ and $L\in\OCOFIN(F)$. By the ideal property for $(\mathfrak B, \mathbf{B})$ we have that 
\[
\mathbf{A}_n( (q^F_L)_n \circ T \circ i^E_M ) = \mathbf{B}_n( (q^F_L)_n \circ T \circ i^E_M ) \le \mathbf{B}_n(T),
\]
from where $\mathbf{A}^{\max}_n(T) \le \mathbf{B}_n(T)$, showing that $\mathfrak{B}(E,F) \subseteq \mathfrak{A}^{\max}(E,F)$ (and in fact the inclusion $\mathfrak{B}(E,F) \to \mathfrak{A}^{\max}(E,F)$ is completely contractive).

\end{proof}

\begin{remark} \label{isometria-proj-min}
Given a finite dimensional operator space $M$ and an arbitrary $E \in \OBAN$ there is a canonical complete isometry:
$$
M\otimes_{\proj} E\hookrightarrow (M'\otimes_{\min} E')'.
$$
Indeed, since $M$ is finite dimensional we have $M\otimes_{\proj} E=\mathcal N(M',E)=\mathcal I(M',E)$. Now, the result follows through the embedding $S_0$ referred to in Example \ref{integrales maximal}. By transposition, there is also the following complete isometry:
$$
E\otimes_{\proj} M\hookrightarrow (E'\otimes_{\min} M')'.
$$
\end{remark}

Another classical version of the Representation Theorem states, in particular, that a maximal ideal taking values in a dual space is itself a dual space. To get the operator space version, we need the additional hypothesis of local reflexivity.

\begin{corollary} \label{Cor:Representation theorem}
Let $(\mathfrak{A},\mathbf{A})$ be a maximal mapping ideal and $\alpha$ its associated o.s. tensor norm.
If $F$ is  locally reflexive, then
$$
 \mathfrak{A}(E,F') = (E \otimes_{\alpha'} F)'  %\label{eqn-representation-thm-1}\\
$$
\end{corollary}

\begin{proof}
Consider the diagram
$$
\xymatrix{
\varphi \in (E \otimes_{\alpha'} F)' \ar@{->}[r] \ar@{^{(}->}[d] &(E \otimes_{\proj} F)' \ar@{^{(}->}[d] \ar@{=}[r] & \CB(E,F')  \\
\varphi^\wedge \in (E \otimes_{\alpha'} F'')'  \ar@{->}[r] &(E \otimes_{\proj} F'')'. & \\
}	
$$
The vertical arrows are complete isometries thanks to the Extension Lemma \ref{extension-lemma},
whereas the horizontal arrows are continuous because $\alpha' \le \proj$.
The desired conclusion follows from \eqref{eqn-representation-thm-2}.
\end{proof}

Note that the hypothesis of local reflexivity cannot be omitted from the previous result. Indeed, taking $\mathfrak A=\mathcal I$ (the mapping ideal of completely integral mappings) which is associated to the o.s. tensor norm $\proj$,  we know by \cite[Theorem 14.3.1]{Effros-Ruan-book} that the equality $\mathcal I(E,F') = (E \otimes_{\min} F)' $ holds for every $E \in \OBAN$ if and only if $F$ is locally reflexive.

On the other hand, whenever $\alpha'$ is a $\lambda$-o.s. tensor norm the Extension Lemma \ref{extension-lemma} is valid  without the hypothesis of local reflexivity. Thus, the previous corollary also holds without this hypothesis.

Along similar lines, in this same setting we can get a version of \cite[Thm. 17.15]{Defant-Floret}.

\begin{theorem}
Let $(\mathfrak{A},\mathbf{A})$ be a maximal mapping ideal associated to the o.s. tensor norm $\alpha$.
Suppose that $\alpha'$ is a $\lambda$-o.s. tensor norm.
Then for any $E,F \in \OBAN$ and $T\in\CB(E,F)$, the following are equivalent:
\begin{enumerate}[(a)]
    \item $T \in \mathfrak{A}(E,F)$.
    \item For any $G \in \OBAN$ (or only $G=F'$) the map
    $$
    T \otimes id_G : E \otimes_{\alpha'} G \to F \otimes_{\proj} G
    $$
    is  bounded.
\end{enumerate}
In this case,
$$
\mathbf{A}(T) = \n{T \otimes id_{F'} : E \otimes_{\alpha'} F' \to F \otimes_{\proj} F'} \ge \n{T \otimes id_G : E \otimes_{\alpha'} G \to F \otimes_{\proj} G}.
$$
\end{theorem}

\begin{proof}
Suppose that $T \in \mathfrak{A}(E,F)$. Let $\varphi \in (F\otimes_{\proj}G)' = \CB(F,G')$.
Note that $L_\varphi \circ T \in \mathfrak{A}(E,G') = (E \otimes_{\alpha'} G)'$, where we are using that the previous corollary holds in general because $\alpha'$ is a $\lambda$-o.s. tensor norm.
We then have for $z \in E \otimes G$, using \cite[Formula 17.15.(1)]{Defant-Floret},

\begin{multline*}
    |\pair{\varphi}{  (T \otimes id_G)(z) }| = |\pair{\beta_{L_\varphi \circ T}}{z}| \le \mathbf{A}(L_\varphi \circ T) \alpha'(z; E,G) \le \n{\varphi} \mathbf{A}(T) \alpha'(z; E,G).
\end{multline*}
Taking the supremum over  $\n{\varphi} \le 1$ yields
\[
\proj( (T \otimes id_G)(z); F,G )\le  \mathbf{A}(T) \alpha'(z; E,G).
\]

Now, assume that (b) is satisfied for $G=F'$.
%Since $\alpha'$ is a $\lambda$-o.s. tensor norm, once again the previous corollary applies and thus from the Representation Theorem we have a completely isometric inclusion
Recall that from the Representation Theorem, $\mathfrak{A}(E,F)$ coincides with $(E \otimes_{\alpha'} F')' \cap \CB(E,F)$. Since $T \in \CB(E,F)$, it will suffice to show that
$\beta_{\kappa_F \circ T} \in (E \otimes_{\alpha'} F')'$.
If we now let $z \in E \otimes F'$, using \cite[Formula 17.15.(2)]{Defant-Floret},
\begin{multline*}
    |\pair{\beta_{\kappa_F \circ T}}{z}| = |\pair{\tr_F}{(T \otimes id_{F'})(z)}| \le \proj( (T \otimes id_{F'})(z); F, F' )\\
    \le \alpha'(z;E,F')  \n{T \otimes id_G : E \otimes_{\alpha'} F' \to F \otimes_{\proj} F'}. 
\end{multline*}
\end{proof}

Once we have the theorem above, we  get a version of \cite[Prop. 17.20]{Defant-Floret}.
Notice that this is a generalization of \cite[Prop. 4]{Janson-Kumar-spectra}, where the same result is obtained for some $\lambda$-o.s. tensor norms.

\begin{proposition}\label{prop-lambda-AP}
For any $\lambda$-o.s. tensor norm and $E,F\in\OBAN$, one of which has the OAP, the natural map
$$
i : E\widehat{\otimes}_\lambda F \to E \widehat{\otimes}_{\min} F
$$
is injective.
\end{proposition}

\begin{proof}
Suppose that $F$ has the OAP.
Let $z \in E\widehat{\otimes}_\lambda F$ be such that $i(z)=0$.
Observe that we need to show that $\pair{\varphi}{z} = 0$ whenever $\varphi \in (E\widehat{\otimes}_\lambda F)' \hookrightarrow \CB(E,F')$.

Note that since $\lambda$ is finitely generated, we have $\lambda=(\lambda')'$ and therefore the previous theorem applies, so
\[
L_\varphi \otimes id_F : E\widehat{\otimes}_\lambda F \to F'\widehat{\otimes}_{\proj} F
\]
is continuous. In the diagram
\[
\xymatrix{
E\widehat{\otimes}_\lambda F \ar[d]_{L_\varphi \widehat{\otimes}_{\lambda,\proj} id_F} \ar[r]^i &  E \widehat{\otimes}_{\min} F \ar[d]^{L_\varphi \widehat{\otimes}_{\min} id_F}\\
F' \widehat{\otimes}_{\proj} F \ar[r] &F' \widehat{\otimes}_{\min} F\\
}
\]
the bottom row is injective because $F$ has OAP \cite[Thm. 11.2.5]{Effros-Ruan-book}, which implies
\[
(L_\varphi \widehat{\otimes}_{\lambda,\proj} id_F)(z) = 0 \in F' \widehat{\otimes}_{\proj} F.
\]
Once again using using \cite[Formula 17.15.(2)]{Defant-Floret}, we get
\[
\pair{\varphi}{z} = \pair{\tr F}{ (L_\varphi \widehat{\otimes}_{\lambda,\proj} id_F)(z) } = 0.
\]

\end{proof}

\begin{example}
For the Haagerup o.s. tensor norm $h$ consider as in Example \ref{U-alpha} the mapping ideal $\mathfrak A_h$. Then, $\mathfrak A_h$ is a maximal mapping ideal defined as
$$
\mathfrak{A}_h(E,F):= (E\otimes_h F')' \cap \CB(E,F)
$$ for every $E,F \in \OBAN$.

Since $\alpha'=h$ is a $\lambda$-o.s. tensor norm, we have
$$
 \mathfrak{A}_h(E,F') = (E \otimes_{h} F)'  
 $$
  for every $E,F \in \OBAN$.
Also, recall that the dual of $E \otimes_{h} F$ is identified with the space $\MB(E\times F)$ of multiplicatively bounded bilinear mappings (see, for instance, \cite[Prop. 9.2.2]{Effros-Ruan-book}) then we have the following complete isometry:
$$
 \mathfrak{A}_h(E,F') = \MB(E\times F).  
 $$
\end{example}

Combining the Representation Theorem \ref{representation-theorem} and \ref{Cor:Representation theorem} with the Duality Theorem \ref{duality-theorem} we obtain:

\begin{theorem}[Embedding Theorem]\label{embedding-theorem}
Let $(\mathfrak{A},\mathbf{A})$ be a maximal mapping ideal associated to the o.s. tensor norm $\alpha$. Then for any $E,F \in \OBAN$ there are complete isometries:
\begin{align*}
E \otimes_{\overleftarrow{\alpha}} F  &\hookrightarrow \mathfrak{A}(E',F)
\\
E' \otimes_{\overleftarrow{\alpha}} F  &\hookrightarrow  \mathfrak{A}(E,F) \qquad \textrm{ whenever } E \textrm{ is locally reflexive}\\
E' \otimes_{\overleftarrow{\alpha}} F'  &\hookrightarrow  \mathfrak{A}(E,F') \qquad \textrm{ whenever } E \textrm{ and } F \textrm{ are locally reflexive.}
\end{align*}
\end{theorem}

\begin{remark} \label{remark-embedding-theorem-lambda} \rm
As in the comments after the Duality Theorem \ref{duality-theorem} and Corollary \ref{Cor:Representation theorem}, if $\alpha'$ is a $\lambda$-o.s. tensor norm then the Embedding Theorem \ref{embedding-theorem} is valid without the hypothesis of local reflexivity. In particular, this holds
in the cases $\overleftarrow{\alpha}=\overleftarrow{\min}=\min$ with $\mathfrak A=\CB$ and $\overleftarrow{\alpha}=\overleftarrow{h}=h$ with $\mathfrak A=\mathfrak A_h$. 
\end{remark}

Given an o.s. tensor norm $\alpha$, we denote its transpose o.s. tensor norm $\alpha^t$. That is, for every pair of operator spaces $\alpha^t$ fulfills that 
$$ E \otimes_{\alpha} F = F \otimes_{\alpha^t} E$$ is a complete isometry with the natural identification.
As mentioned in the introduction, $\proj=\proj^t$, $\min=\min^t$ but $h \neq h^t$. Also, it follows form the definition that $(d_p )^t=g_p $ and $(g_p )^t=d_p .$

In the classical setting the dual of a maximal ideal is also maximal, and their associated tensor norms are transposes of each other.
In the operator space setting, in order to get a similar result we once again need the additional hypothesis of local reflexivity.

\begin{corollary}\label{cor-associate-transpose} 
Let $(\mathfrak{A},\mathbf{A})$ be a maximal mapping ideal associated to the finitely-generated o.s. tensor norm $\alpha$. Then $\mathfrak{A}^{dual}$ is associated to $\alpha^t$. 
Moreover, if $(\mathfrak{B},\mathbf{B})$ is the maximal mapping ideal associated to $\alpha^t$ and $E$ is a locally reflexive Banach operator space, then $\mathfrak{B}(E,F) = \mathfrak{A}^{\dual}(E,F)$.
%Then $(\mathfrak{A}^{\dual},\mathbf{A}^{\dual})$ is a maximal mapping ideal, and is associated with the transposed tensor norm $\alpha^t$.
\end{corollary}

\begin{proof}
If $\mathfrak{A} \sim \alpha$ and $M,N$ are finite dimensional operator spaces then $\mathfrak{A}^{\dual}(M,N)=M' \otimes_{\alpha^t}N.$ This shows that $\mathfrak{A}^{\dual} \sim \alpha^t.$
Also, by the Representation Theorem \ref{representation-theorem} and Corollary \ref{Cor:Representation theorem} we have complete isometries
\begin{align*}
\mathfrak{B}(E,F) &=  (E \otimes_{(\alpha^t)'} F')' \cap \CB(E,F)
= \big\{  T \in \CB(E,F) \, : \, \beta_{T} \in (E \otimes_{(\alpha^t)'} F')'   \big\} \\
&= \big\{  T \in \CB(E,F) \, : \, \beta_{T'} \in ( F' \otimes_{\alpha'} E )'   \big\} \\
&= \big\{  T \in \CB(E,F) \, : \, T' \in \mathfrak{A}(F',E') \big\} =  \mathfrak{A}^{\dual}(E,F)
\end{align*}
and therefore, for $S \in M_n(\mathfrak{B}(E,F))$ we have $\mathbf{B}_n(S) = \mathbf{A}_n(S')$.
This is the desired conclusion.
\end{proof}

If $(\mathfrak{A},\mathbf{A})$ is a maximal mapping ideal, it should be noted that $\mathfrak{A}^{\dual}$ is not necessarily  maximal as it happens in the Banach space case.
Indeed, $\mathfrak A=\mathcal I$ the maximal mapping ideal of completely integral mappings, which is associated to the o.s. tensor norm $\proj$, gives us the example.
To see this, note that $\mathcal I^{\dual}(M,N)=M'\otimes_{\proj}N$, for every $M,N \in \OFIN$. Thus, if $\mathcal I^{\dual}$ were a maximal mapping ideal it should coincide with $\mathcal I$ (since $\proj^t=\proj$).
If $W$ is any non-locally reflexive o.s. then by \cite[Theorem 14.3.1]{Effros-Ruan-book} there is $L \in \OFIN$ such that the inclusion $\mathcal N(W,L^*)\hookrightarrow\mathcal I(W,L^*)$ is not isometric.
Let $T \in \mathcal N(W,L^*)$ such that $\iota(T)<\nu(T)$. We have that $T' \in \mathcal N(L,W^*)=\mathcal I(L,W^*)$ (since $L$ is finite-dimensional) and $$ \iota^{\dual}(T) = \iota( T')  = \nu (T') = \nu (T ) > \iota (T ).$$

Recall that for finite-dimensional spaces $E$ and $F$ and linear maps $T : M \to N$, $S : N \to M$, their trace duality pairing is given by
$$
\pair{T}{S} = \tr( S \circ T ).
$$
More generally, for $T \in M_n(\CB(M,N))$ and $S \in M_m(\CB(N,M))$, their trace duality (matrix) pairing is the matrix in $M_{nm}$ given by 
$$
\mpair{T}{S} = [\tr( S_{kl} \circ T_{ij} )].
$$

\begin{definition}
Let $(\mathfrak{A},\mathbf{A})$ be a mapping ideal.
For  $M,N \in \OFIN$, $T \in M_n(\CB(M,N))$ and $S \in M_m(\CB(N,M))$, we denote by $S \circ T$ the induced $mn \times mn$ matrix of mappings $M \to M$.
For $T \in M_n(\CB(E,F))$, define
\begin{multline*}
\mathbf{A}^*_n(T)  = \sup\big\{ \n{  \mpair{(q^F_L)_n \circ  T \circ i^E_M}{S}   }_{M_{mn}}  \, : \,\\
  M \in \OFIN(E), L \in \OCOFIN(F), S \in M_m(\CB(F/L,M)), \mathbf{A}_m(S) \le 1 \big\}
\end{multline*}
and
$$
\mathfrak{A}^{*}(E,F)= \left\{ T \in \CB(E,F) \, : \,  \mathbf{A}^{*}_1(T) < \infty  \right\}.
$$
\end{definition}

As the following proposition shows,  $(\mathfrak{A}^*,\mathbf{A}^*)$ is a maximal mapping ideal which is associated with the so-called adjoint o.s. tensor norm.
Thus, we call   $(\mathfrak{A}^*,\mathbf{A}^*)$ \emph{the adjoint mapping ideal} of $(\mathfrak{A},\mathbf{A})$.

\begin{proposition}
\label{prop-associate-adjoint}
Let $(\mathfrak{A},\mathbf{A})$ be a maximal mapping ideal associated to the finitely-generated o.s. tensor norm $\alpha$.
Then $(\mathfrak{A}^{*},\mathbf{A}^{*})$ is a maximal mapping ideal, and it is associated with the adjoint o.s. tensor norm $\alpha^* :=(\alpha^t)'$.
\end{proposition}

\begin{proof}
Let $(\mathfrak{B},\mathbf{B})$ be the maximal mapping ideal associated to $\alpha^*$. Now we show that $(\mathfrak{B},\mathbf{B})$ is in fact $(\mathfrak{A}^*,\mathbf{A}^*)$.
First, note that, for $M,N \in \OFIN$,
$$
\big( \mathfrak{A}(N,M) \big)' = (N' \otimes_{\alpha} M)' = (M \otimes_{\alpha^t} N') = M' \otimes_{(\alpha^t)'} N = M' \otimes_{\alpha^*} N.
$$
Hence,
$$
\mathfrak{B}(M,N) = \big( \mathfrak{A}(N,M) \big)'
$$
completely isometrically, where the duality pairing is given by the canonical trace duality.
This means that for $T \in M_n(\CB(M,N))$ we have
$$
\mathbf{B}_n(T) = \sup\{  \n{ \mpair{T}{S} }_{M_{nm}}   \, : \, S\in M_m(\mathfrak{A}(N,M)),  \mathbf{A}_m( S ) \le 1 \}.
$$
Since $\mathfrak{B}$ is maximal, it follows that for any $E,F \in \OBAN$,
and $T \in M_n(\CB(E,F))$,
\begin{multline*}
\mathbf{B}_n(T) = \sup\big\{ \n{  \mpair{(q^F_L)_n \circ  T \circ i^E_M}{S}   }_{M_{mn}} \, : \,\\
  M \in \OFIN(E), L \in \OCOFIN(F), S \in M_m(\CB(F/L,M)), \mathbf{A}_m(S) \le 1 \big\}
\end{multline*}
which coincides with the definition of $\mathbf{A}^{*}_n(T)$.
\end{proof}

\begin{proposition}\label{prop-associated-to-p-summing}
Let $1 \le p \le \infty$.
The o.s. tensor norm $(g_{p'})^*$ is associated to the mapping ideal $\Pi_{p}$.
\end{proposition}

\begin{proof}
Note that $(g_p )^* = ((g_p )^t)'$ is a finitely-generated o.s. tensor norm, since dual o.s. tensor norms are defined as a finite hull.

Finally, by the duality between the completely $p$-summing maps and the Chevet-Saphar o.s. tensor norms \cite{CD-Chevet-Saphar-OS}, for any $E,F \in \OFIN$
we have a canonical duality
$$
\Pi_{p}(E,F) = ( E \otimes_{d_{p'} } F' )' = E' \otimes_{(d_{p'} )'} F = E' \otimes_{(g_{p'} )^*} F,
$$
this concludes the proof.
\end{proof}

\section{Minimal o.s. mapping ideals}\label{minimal sec}

In the theory Banach spaces another important topic regarding the interplay between tensor products and ideals is the class of minimal operator ideals \cite[Section 22]{Defant-Floret}.
In this section we construct an operator space analogue of this concept;
an instance of this procedure already appears in the definition of nuclear mappings, as will become clear in Theorem \ref{thm-minimal-ideal} below.

Given a Banach mapping ideal $\mathfrak{A}$, we will describe the smallest Banach mapping ideal which coincides with $\mathfrak{A}$ on $\OFIN$. This mapping ideal is called the \emph{the minimal kernel} of $\mathfrak{A}$ and will be denoted by $\mathfrak{A}^{\min}$. In other words, if $\mathfrak{A}$ is a Banach mapping ideal associated to the finitely generated o.s. tensor norm $\alpha$, then $\mathfrak A^{\min}$ will be the smallest Banach mapping ideal also associated to $\alpha.$

In the Banach space framework, the representation theorem of minimal Banach ideals \cite[22.2]{Defant-Floret} gives the connection with tensor products. In our setting, we will use this association to provide an o.s. structure to the mapping ideal $\mathfrak{A}^{\min}$. 
\bigskip

\begin{definition}Let $(\mathfrak{A},\mathbf{A})$ be a mapping ideal. For $E,F \in \ONORM$ we define the class $\mathfrak{A}^{\min}(E,F)$ of all mappings $T:E \to F$ which admit a factorization of the form 

\begin{equation} \label{factorization minimal}
\xymatrix{
E\ar[d]_R \ar[r]^T & F \\
 G \ar[r]_{\widehat T} & H \ar[u]_S,\\
} 
\end{equation}
where $G,H \in \OBAN$ and $R \in \mathcal{A}(E,G)$, $\widehat T \in \mathfrak{A}(G,H)$ and $S \in \mathcal{A}(H,F)$.
For such a mapping $T$ we denote
$$
\mathbf{A}^{\min}( T ) := \inf \Vert R \Vert_{\cb}  \mathbf{A}( \widehat{T} ) \Vert S \Vert_{\cb},
$$
where the infimum runs over all the possible factorizations of $T$ as above. 

\end{definition}

It is clear that:
\begin{example}\label{example: Fmin}
$\mathcal{F}^{\min}=\mathcal{F}$.
\end{example}

We now see that $\mathfrak{A}^{\min}(E,F)$ has certain structure.

\begin{proposition}
Let $(\mathfrak{A},\mathbf{A})$ be a Banach mapping ideal and $E,F \in \ONORM$. Then the class $\mathfrak{A}^{\min}(E,F)$ is a vector space and $ \mathbf{A}^{\min}$ is a norm. Moreover, if $E,F \in \OBAN$ then $(\mathfrak{A}^{\min}(E,F), \mathbf{A}^{\min})$ is a Banach space.
\end{proposition}

\begin{proof}
It is clear that $ \mathbf{A}^{\min}(T)\ge 0$, for all $T\in \mathfrak{A}^{\min}(E,F)$. Also, since for each factorization of $T$ as in the definition $T=S\circ \widehat{T}\circ R$ we know that $\|\widehat{T}\|_{\cb}\le  \mathbf{A}(\widehat{T})$ and $\|T\|_{\cb}\le \Vert R \Vert_{\cb} \Vert \widehat T \Vert_{\cb} \Vert S \Vert_{\cb}$ we derive that
$\|T\|_{\cb}\le  \mathbf{A}^{\min}(T)$. Hence,   $ \mathbf{A}^{\min}(T)=0$ implies $T=0$.

Positive homogeneity is clear. Thus, to complete the first assertion we just need to check the triangle inequality. Let $T_1, T_2\in \mathfrak{A}^{\min}(E,F)$. For a given $\varepsilon>0$ there are factorizations 
\begin{equation*} 
\xymatrix{
E\ar[d]_{R_1} \ar[r]^{T_1} & F \\
 G_1 \ar[r]_{\widehat T_1} & H_1 \ar[u]_{S_1}\\
} \qquad
\xymatrix{
E\ar[d]_{R_2} \ar[r]^{T_2} & F \\
 G_2 \ar[r]_{\widehat T_2} & H_2 \ar[u]_{S_2}\\
}
\end{equation*}
with $\|R_j\|_{\cb}=\|S_j\|_{\cb}=1$ and $ \mathbf{A}(\widehat{T_j})\le (1+\varepsilon) \mathbf{A}^{\min}(T_j)$, for $j=1,2$. Now, consider the factorization
\begin{equation*} 
\xymatrix{
E\ar[d]_{R} \ar[r]^{T_1+T_2} & F \\
 G_1\oplus_\infty G_2 \ar[r]_{\widehat T} & H_1\oplus_1 H_2 \ar[u]_{S},\\
} 
\end{equation*}
where $R(x)=(R_1(x), R_2(x))$, $\widehat{T}(g_1,g_2)=(\widehat{T_1}(g_1),\widehat{T_2}(g_2))$ and $S(h_1,h_2)=S_1(h_1) + S_2(h_2)$, for every $x\in E, g_j\in G_j, h_j\in H_j$, $j=1,2$.

Note that $\|R\|_{\cb}\le \max\{\|R_1\|_{\cb}, \|R_2\|_{\cb}\}=1$. By duality, the same argument yields $\|S\|_{\cb}= \|S'\|_{\cb}\le \max\{\|S'_1\|_{\cb}, \|S'_2\|_{\cb}\}=1$. Writing $\widehat{T}= i_1\circ \widehat T_1\circ q_1+ i_2\circ \widehat T_2\circ q_2$, where, for each $j=1,2$, $i_j:H_j\to H_1\oplus_1 H_2$ is the canonical inclusion and $q_j:G_1\oplus_\infty G_2\to G_j$ is the canonical projection, we obtain
$\mathbf{A}(\widehat{T})\le \mathbf{A}(\widehat{T_1})+\mathbf{A}(\widehat{T_2})$.

In order to conclude $ \mathbf{A}^{\min}(T_1+T_2)\le (1+\varepsilon) \left(\mathbf{A}^{\min}(T_1)+ \mathbf{A}^{\min}(T_2)\right)$ we just need to check that $R\in\mathcal A(E, G_1\oplus_\infty G_2)$ and $S\in\mathcal A (H_1\oplus_1 H_2, F)$. The first one follows easily from the fact that $R_1$ and $R_2$ are approximable. The second one is derived by duality.

Now, it remains to prove that for $E,F\in \OBAN$, $\mathfrak{A}^{\min}(E,F)$ is complete. Let $\{T_k\}_k\subset \mathfrak{A}^{\min}(E,F)$ such that $\sum_{k=1}^\infty \mathbf{A}^{\min}(T_k) <\infty$. Take $(t_k)_k\subset\mathbb R$ with $1\le t_k \uparrow\infty$ such that
$$
\sum_{k=1}^\infty t_k\mathbf{A}^{\min}(T_k)\le (1+\varepsilon)\sum_{k=1}^\infty\mathbf{A}^{\min}(T_k).
$$

Note that the inequality $\|\cdot\|_{\cb}\le \mathbf{A}^{\min}$ tells us that a linear mapping $T=\sum_{k=1}^\infty T_k$ is well defined and that the series converges in $\CB(E,F)$.

For each $k$, take a factorization
\begin{equation*} 
\xymatrix{
E\ar[d]_{R_k} \ar[r]^{T_k} & F \\
 G_k \ar[r]_{\widehat T_k} & H_k \ar[u]_{S_k}\\
} 
\end{equation*}
with $\|R_k\|_{\cb}=\|S_k\|_{\cb}=1$ and $\mathbf{A}(\widehat{T}_k)\le (1+\varepsilon) \mathbf{A}^{\min}(T_k)$.  Now, consider
\begin{equation*} 
\xymatrix{
E\ar[d]_{R} \ar[r]^{T} & F \\
 c_0(G_k) \ar[r]_{\widehat T} & \ell_1(H_k) \ar[u]_{S},\\
} 
\end{equation*}
where $R(x)=(\frac{R_k(x)}{\sqrt{t_k}})$, $\widehat{T}((g_k)_k)=((t_k\widehat{T}_k(g_k))_k)$ and $S((h_k)_k)=\sum_{k=1}^\infty \frac{S_k(h_k)}{\sqrt{t_k}}$, for every $x\in E, (g_k)_k\in c_0(G_k), (h_k)_k\in \ell_1(H_k)$. It is easy to see that each mapping is well defined.

Let us see that $\widehat{T}$ belongs to $\mathfrak{A}(c_o(G_k), \ell_1(H_k))$. Note that, pointwise, 
\begin{equation}\label{T gorro}
  \widehat{T}= \sum_{j=1}^\infty t_j i_j\circ \widehat T_j\circ q_j,  
\end{equation} where, for each $j$, $i_j:H_j\to \ell_1(H_k)$ is the canonical inclusion and $q_j:c_0(G_k)\to G_j$ is the canonical projection. Since
$$
 \sum_{j=1}^\infty \mathbf{A}(t_j i_j\circ \widehat T_j\circ q_j) \le \sum_{j=1}^\infty t_j \mathbf{A}( \widehat T_j) \le (1+\varepsilon) \sum_{j=1}^\infty t_j \mathbf{A}^{\min}(  T_j ) <\infty
$$ we obtain $\widehat{T}\in \mathfrak{A}(c_o(G_k), \ell_1(H_k))$. Moreover, the series \eqref{T gorro} converges in $\mathfrak A$ and $\mathbf{A}(\widehat T)\le \sum_{j=1}^\infty t_j \mathbf{A}( \widehat T_j )$.

Now we prove that $R\in\mathfrak A(E, c_0(G_k))$. Since $t_k\ge 1$ it is clear that $\|R:E\to \ell_\infty (G_k)\|_{\cb}\le 1$. Also, $t_k\uparrow \infty$ implies that $R$ takes values in $c_0(G_k)$. For a given $\varepsilon>0$ let $N\in\mathbb N$ such that $\frac{1}{\sqrt{t_N}}<\frac{\varepsilon}{2}$. For each $1\le j\le N$, take $\widetilde{R}_j\in\mathcal F(E,G_j)$ satisfying $\|R_j-\widetilde R_j\|_{\cb}\le \frac{\varepsilon}{2N}$. Let $\widetilde{R}:E\to c_0(G_k)$ be given by
$$
x\mapsto \left(\frac{\widetilde R_1(x)}{\sqrt{t_1}},\dots, \frac{\widetilde R_N(x)}{\sqrt{t_N}},0,0,\dots\right).
$$ Clearly, $\widetilde R\in\mathcal F(E, c_0(G_k))$ and $\|R-\widetilde R\|_{\cb}\le \varepsilon$. Therefore, $R\in\mathcal A(E, c_0(G_k))$ with $\|R\|_{\cb}\le 1$.

We finally check that $S\in \mathcal A(\ell_1(H_k), F)$. Note that $S':F'\to\ell_\infty(H'_k)$ is given by $S'(y')=(\frac{S_k'(y')}{\sqrt{t_k})_k})$, so the same argument as above yields $\|S\|_{\cb}=\|S'\|_{\cb}\le 1$. Also, analogously as before we can prove that for each $\varepsilon>0$ there is $\widetilde{S}\in\mathcal F(\ell_1(H_k), F)$ such that $\|S-\widetilde S\|_{\cb}\le\varepsilon$. 

From the arguments above we certainly obtained that $T\in \mathfrak{A}^{\min}(E,F)$ with $\mathbf{A}^{\min}(T)\le \sum_{k=1}^\infty \mathbf{A}^{\min}(T_k)$. This implies that $T=\sum_{k=1}^\infty T_k$ converges in $\mathfrak{A}^{\min}(E,F)$, which finishes the proof.
\end{proof}

In order to relate the tensor product $E'\otimes_{\alpha} F$ with $\mathfrak{A}^{\min}(E,F)$ as in the Banach space setting, we  need the following lemma.

\begin{lemma} \label{contraction}
Let $(\mathfrak{A},\mathbf{A})$ be a Banach  mapping ideal associated to the finitely generated o.s. tensor norm $\alpha$. Then for any $E,F \in \ONORM$ the natural map 
\begin{equation}\label{mapaalminimal}
J: E'\otimes_{\alpha} F \to \mathfrak{A}^{\min}(E,F)
\end{equation}
is a contraction of normed spaces, i.e.,
$\mathbf{A}^{\min}( J(z))  \leq \alpha(z;E',F), \; \forall z \in E'\otimes F.$

\end{lemma}

\begin{proof}
Let $z\in E'\otimes F$, $M\in\OFIN(E')$, $N\in\OFIN(F)$, $u\in M\otimes N$ with $z= (i_M^{E'}\otimes i_N^F) (u)$.

From the diagram
\begin{equation*} 
\xymatrix{
E'\otimes_\alpha F \ar[r]^{J} & \mathfrak{A}^{\min}(E,F) & i_N^F\circ T\circ q_{ ^\circ M}^E\\
 M\otimes_\alpha N \ar[u]_{i_M^{E'}\otimes i_N^F} \ar@{=}[r] & \mathfrak{A}(E/^\circ M, N) \ar[u] & T\ar@{~>}[u]\\
} 
\end{equation*}
we get
$$
\mathbf{A}^{\min}(J(z))= \mathbf{A}^{\min} (i_N^F\circ T_u\circ q_{ ^\circ M}^E)\le \mathbf{A}(T_u)=\alpha(u;M,N).
$$
Since $\alpha$ is finitely generated we obtain
$\mathbf{A}^{\min}(J(z))\le \alpha(z;E',F)$.
\end{proof}

\begin{lemma}\label{lemma-approx-Amin}
Let $T \in \mathfrak{A}^{\min}(E,F)$ with factorization $T = S \circ \widehat T \circ R$ as in \eqref{factorization minimal}. If $\widetilde R \in \mathcal A(E,G)$ and $\widetilde S \in \mathcal A(H,F)$ then
$$
\mathbf{A}^{\min}( T - \widetilde S \circ \widehat T \circ \widetilde R ) \leq \Vert S - \widetilde S \Vert_{\cb} \mathbf{A}( \widehat T ) \Vert R - \widetilde R \Vert_{\cb}
$$
\end{lemma}
\begin{proof}
This is clear by the definition of the norm and the fact that
$$T- \widetilde S \circ \widehat T \circ \widetilde R= (S -\widetilde S) \circ \widehat T \circ (R -\widetilde R).$$
\end{proof}
We are now ready to prove one of the main results of this section, which will allow us to provide an o.s. structure for $\mathfrak{A}^{\min}$.

\begin{theorem}\label{thm-minimal-ideal}
 Let $(\mathfrak{A},\mathbf{A})$ be a Banach mapping ideal associated to the finitely generated o.s. tensor norm $\alpha$. Then for any $E,F \in \OBAN$ the natural map
 $$
  J: E' \widehat{\otimes}_{\alpha} F \to \mathfrak{A}^{\min}(E,F)
 $$
 is a metric surjection (at the Banach space level).
 \end{theorem}
 
\begin{proof}
Using Lemma \ref{contraction}, we know that $J$ is a contraction.
Take $R \in \mathcal{A}(E,G)$, $\widehat T \in \mathfrak{A}(G,H)$ and $S \in \mathcal{A}(H,F)$.
If $z \in E'\otimes F$ satisfies $J(z)=S \circ \widehat T \circ R$ consider the factorization

$$
\xymatrix{
E \ar[r]^{S \circ \widehat T \circ R} \ar[d]_{\overline{R}} & F \\
 M=im(\widetilde R) \ar@{.>}[r]_{T_1} \ar@{^{(}->}[d]_{i_1} & im(S)=N \ar@{^{(}->}[u]_{i_2}\\
 G \ar[r]_{\widehat T} & H \ar[u]_{\overline{S}} 
}
$$
Thus,

\begin{align}\label{ec desigualdad}
\begin{split}
   \alpha(z;E',F) & = \alpha((\overline{R}'\otimes i_2)z_{T_1};E',F) \leq \Vert \overline{R}' \Vert_{\cb} \Vert i_2 \Vert_{\cb}   \alpha(z_{T_1};M',N) \\
  & =  \Vert \overline{R} \Vert_{\cb}  \mathbf{A}( T_1 ) \leq  \Vert \overline{R} \Vert_{\cb} \Vert i_1\Vert_{\cb} \mathbf{A}(\widehat T ) \Vert \overline{S} \Vert_{\cb} \\
  & = \Vert R \Vert_{\cb}  \mathbf{A}( \widehat T )\Vert S \Vert_{\cb}. 
\end{split}
\end{align} 

Now, let $T \in \mathfrak{A}^{\min}(E,F)$. Given $\varepsilon >0$, find a factorization $T = S \circ \widehat T \circ R$ with $\Vert R \Vert_{\cb} \mathbf{A}(\widehat T ) \Vert S \Vert_{\cb} \leq (1+\varepsilon) \mathbf{A}^{\min}( T )$.
Choose $(R_n)_{n \in \mathbb{N}} \subset \mathcal{A}(E,G)$ and $(S_n)_{n \in \mathbb{N}} \subset \mathcal{A}(H,F)$ such that 
$ \Vert R - R_n \Vert_{\cb} \to 0$ and $ \Vert S - S_n \Vert_{\cb} \to 0$ as $n \to \infty$.
By Lemma \ref{lemma-approx-Amin}, we know that $S_n \circ \widehat T \circ R_n$ converges to $T$ in $\mathfrak{A}^{\min}(E,F)$. Now, pick $z_n \in E'\otimes F$ with $J(z_n)= S_n \circ \widehat T \circ R_n$. Then,
\begin{align*}
J(z_n - z_m) &= J(z_n) - J(z_m) \\
&= S_n \circ \widehat T \circ R_n - S_m \circ \widehat T \circ R_m  \\
&= (S_n - S_m) \circ \widehat T \circ R_n + S_m \circ \widehat T \circ (R_n - R_m)
\end{align*}
By reasoning as in \eqref{ec desigualdad} and using the triangle inequality we have
\begin{align*}
\alpha(z_n - z_m; E',F) \leq
\Vert S_n - S_m \Vert_{\cb} \mathbf{A}( \widehat T ) \Vert R_n \Vert_{\cb} +  \Vert S_m \Vert_{\cb} \mathbf{A}( \widehat T ) \Vert R_n - R_m \Vert_{\cb}.  
\end{align*}
Therefore, $(z_n)_{n \in \mathbb{N}}$ is a Cauchy sequence in $E' \widehat \otimes_{\alpha} F$ with limit, say, $z$.
since $J$ is continuous,
$$J(z) = \lim_{n \to \infty} J(z_n) = \lim_{n \to \infty} S_n \circ \widehat T \circ R_n,$$ where the limits are in $\mathfrak{A}^{\min}(E,F).$ We have checked above that the latter limit is in fact $T$. In other words, given $T \in \mathfrak{A}^{\min}(E,F)$ we have found $z \in E' \widehat \otimes_{\alpha} F$ such that $J(z)=T$ and, moreover, by \eqref{ec desigualdad} again
\begin{align*}
    \alpha(z;E',F) & =\lim_{n \to \infty} \alpha(z_n;E',F) \\ 
    & \leq \Vert S_n \Vert_{\cb}  \mathbf{A}(\widehat T ) \Vert R_n \Vert_{\cb} \\
    & \leq \mathbf{A}^{\min}( T ) (1+\varepsilon).
    \end{align*}
This shows that the mapping $J$ is a metric surjection (since we already knew that $\Vert J \Vert \leq 1)$.
\end{proof}
 Using the previous theorem we can now provide an operator space structure for $\mathfrak{A}^{\min}(E,F)$.

\begin{definition}
Given a Banach mapping ideal $(\mathfrak{A},\mathbf{A})$, we endow $\mathfrak{A}^{\min}(E,F)$ with the unique o.s. structure such that the natural map $J$ becomes a complete quotient mapping. Namely, given $T \in M_n(\mathfrak{A}^{\min}(E,F))$, we define $$ \mathbf{A}_n^{\min} (T)  := \inf \alpha_n(z; E',F),$$
 where the infimum runs all over $z\in M_n(E'\widehat \otimes F)$ such that $J_n(z)=T.$

\end{definition}

 With this structure we have obviously the following:

 \begin{theorem}[Representation Theorem for Minimal Mapping Ideals]\label{Representation minimal}
 Let $(\mathfrak{A},\mathbf{A})$ be a Banach mapping ideal associated to the finitely generated o.s. tensor norm $\alpha$. Then for any $E,F \in \OBAN$ the natural operator
\begin{equation} 
\xymatrix{
  J: E' \widehat{\otimes}_{\alpha} F\ar@{->>}[r] &  \mathfrak{A}^{\min}(E,F)
 }
 \end{equation}
 is a complete quotient mapping.
In particular, if $M,N \in  \OFIN$ then  we have a complete isometry $\mathfrak{A}^{\min}(M,N)= M' \widehat{\otimes}_{\alpha} N $.
 \end{theorem}

\begin{theorem} \label{minimal y norma asociada}
Let $(\mathfrak{A},\mathbf{A})$ be a Banach mapping ideal associated to the finitely generated o.s. tensor norm $\alpha$. Then $(\mathfrak A^{\min},\mathbf{A}^{\min})$ is also a Banach mapping ideal associated to $\alpha$. Moreover, $\mathfrak A^{\min}$ is the smallest Banach mapping ideal associated to $\alpha.$
\end{theorem}

\begin{proof}
To see that $(\mathfrak A^{\min},\mathbf{A}^{\min})$ is a mapping ideal, we first see that the identity map
$\mathfrak{A}^{\min}(E,F) \to \CB(E,F)$ is a complete contraction.
% we have to check first that it contains all the  finite rank mappings.
% This is straightforward using Example \ref{example: Fmin}. Indeed, let $E,F \in \OBAN$ and  $T \in \mathcal{F}(E,F)$ a finite rank mapping. Since  $\mathcal{F}^{\min}=\mathcal{F}$ we know there are $G,H \in \OBAN$ and $R \in \mathcal{A}(E,G)$, $\widehat T \in \mathcal{F}(G,H)$ and $S \in \mathcal{A}(H,F)$ such that $T$ can be written as $S \circ \widehat T \circ R$. Observe that $\widehat T \in \mathcal{F}(G,H) \subset \mathfrak{A}(G,H)$ (since $\mathcal{F} \subset \mathfrak{A}$), then $T \in \mathfrak{A}(E,F)$ by the ideal property of $\mathfrak{A}$.

Consider the following commutative diagram:

\begin{equation} 
\xymatrix{
E' \widehat{\otimes}_{\alpha} F\ar@{->>}[r]^J \ar[d] &  \mathfrak{A}^{\min}(E,F) \ar@{.>}[d]
\\
 E'\widehat \otimes_{\min}F \ar@{^{(}->}[r]^{i}
 & \CB(E,F),\\
} 
\end{equation}

Now, by the Representation theorem for minimal mapping ideals we know that $J$ is complete quotient. On the other hand, the  inclusion $i$ is a complete isometry (see Remark \ref{extension-sin-loc-refl}) and the canonical mapping $E' \widehat{\otimes}_{\alpha} F \to E'\widehat \otimes_{\min}F$ is a complete contraction. By simply following the diagram, this implies that the descending arrow $\mathfrak{A}^{\min}(E,F) \dashrightarrow \CB(E,F)$ is a complete contraction. 

Note that condition (b') immediately follows from the last assertion of the previous theorem: $\mathfrak{A}^{\min}(M,N)= M' \widehat{\otimes}_{\alpha} N $.

Let us now focus on the ideal property. Consider $T \in M_n(\mathfrak{A}^{\min}(E,F))$, $r \in \CB(E_0,E)$ and $s \in \CB(F,F_0)$. 
By the ideal property of $\mathcal{A}$ it is easy to see that $s_n \circ T \circ r$ lies in $M_n(\mathfrak{A}^{\min}(E_0,F_0))$.
To bound the norm consider the following commutative diagram:

\begin{equation} 
\xymatrix{
E' \widehat{\otimes}_{\alpha} F \ar@{->>}[r] \ar[d]_{r'\otimes s} &  \mathfrak{A}^{\min}(E,F) \ar[d] & T \ar@{|->}[d]
\\
 E_0' \widehat{\otimes}_{\alpha} F_0 \ar@{->>}[r]
 & \mathfrak{A}^{\min}(E_0,F_0) & s\circ T \circ r.\\
} 
\end{equation}
Since the horizontal arrows are complete quotients and the mapping $$r'\otimes s : E' \widehat{\otimes}_{\alpha} F \to  E_0' \widehat{\otimes}_{\alpha} F_0 $$ has $\cb$-norm bounded by $\Vert s \Vert_{\cb} \Vert r \Vert_{\cb}$, it follows that 
$$
\mathbf{A}_n^{\min}(  s_n \circ T \circ r ) \le \n{s}_{\cb} \mathbf{A}_n^{\min}( T ) \n{r}_{\cb}.
$$
By Theorem \ref{Representation minimal} we know that if $M,N \in  \OFIN$, then there is a complete isometry $\mathfrak{A}^{\min}(M,N)= M' \widehat{\otimes}_{\alpha} N $, therefore the mapping ideal $\mathfrak{A}^{\min}$ is associated to $\alpha$.

Finally, we  show that $\mathfrak{A}^{\min}$ is the smallest Banach mapping ideal associated to $\alpha$. Since for any Banach mapping ideal $\mathfrak{B}$ we have $\mathfrak{B}^{\min} \subset \mathfrak{B}$ it is enough to see that if $\mathfrak{B}$ is  also associated to $\alpha$ then $\mathfrak{A}^{\min}=\mathfrak{B}^{\min}$. 
Indeed, this is straightforward since we have  the complete quotients $J:E' \widehat{\otimes}_{\alpha} F \twoheadrightarrow \mathfrak{A}^{\min}(E,F) $ and  $J:E' \widehat{\otimes}_{\alpha} F \twoheadrightarrow \mathfrak{B}^{\min}(E,F)$.
\end{proof} 
 
Now we present an operator space version of \cite[Prop. 22.1]{Defant-Floret}.

\begin{proposition}
Let $(\mathfrak{A}, \mathbf{A})$ be a Banach mapping ideal associated to the finitely generated o.s. tensor norm $\alpha$.
Then,

\begin{enumerate}
    \item $\mathfrak{A}^{\min}(M,N)= \mathfrak{A}^{\max}(M,N)= \mathfrak{A}(M,N) = M'\otimes_{\alpha}N$ holds completely isometrically for every $M,N\in\OFIN$.

     \item $\mathfrak{A}^{\min\ \max }=\mathfrak{A}^{\max}$.
     
     \item $\mathfrak{A}^{\max\ \min}=\mathfrak{A}^{\min}$.
     
      \item Given $E,F\in\OBAN$ and $T\in M_n(\mathfrak{A}^{\min}(E,F))$ there are $T_k\in M_n(\mathcal F(E,F))$ such that $\mathbf{A}_n^{\min}(T_k-T)\to 0$.

     \item Let $\mathfrak B$ be a Banach mapping ideal. Suppose that for every $M,N \in \OFIN$ we have that the inclusion $\Vert \mathfrak A (M,N) \hookrightarrow  \mathfrak B (M,N)  \Vert_{\cb} \leq c$ then
     $\mathfrak{A}^{\min} \subset \mathfrak{B}^{\min}$ and
     $\Vert \mathfrak A^{\min} (E,F) \hookrightarrow  \mathfrak B^{\min} (E,F)  \Vert_{\cb} \leq c$ for every $E,F\in\OBAN$.

\end{enumerate}
\end{proposition}

\begin{proof}
 $(1)$, $(2)$ and $(3)$ follow directly from the Representation Theorems for maximal and minimal ideals.
% $(1)$ 
% Let $\alpha$ be the o.s. tensor norm associated to $\mathfrak{A}$ (and therefore also associated to $\mathfrak{A}^{\max}$). Then we have completely isometric isomorphisms $\mathfrak{A}^{\max}(M,N)=  \mathfrak{A}(M,N) = M' \widehat{\otimes}_{\alpha} N $ for $M,N\in\OFIN$. Note that by the previous theorem 
% $\mathfrak{A}^{\min}(M,N) =M' \widehat{\otimes}_{\alpha} N $ and the result follows.

% $(3)$ By  Remark \ref{misma norma asociada} we have that $\mathfrak{A}$ and $\mathfrak{A}^{\max}$ have the same associated tensor norm $\alpha$. The result now follows by the Representation theorem for minimal mapping ideals \ref{Representation minimal}.

% $(4)$ It is immediate from the fact that both $\mathfrak{A}$ and $\mathfrak{A}^{\min}$ have the same associated tensor norm $\alpha$ (see Theorem \ref{minimal y norma asociada}) together with Remark \ref{misma norma asociada} and the Representation theorem for maximal mapping ideals \ref{representation-theorem}.

$(4)$ This can easily be derived from Lemma \ref{lemma-approx-Amin}.

$(5)$ 
Let $\beta$ be the o.s. tensor norms  $\mathfrak{B}$. The hypothesis translates into $\Vert M' \widehat{\otimes}_{\alpha} N \hookrightarrow  M' \widehat{\otimes}_{\beta} N  \Vert_{\cb} \leq c$ for every $M,N \in \OFIN$. Since both norms are finitely generated, this completely bounded inclusion can be extended to $\OBAN$.
Also, by Proposition \ref{Inclusion maximal}, we know that $\mathfrak{A}^{\max} \subset \mathfrak{B}^{\max}$. Now, using $(3)$, we have  $\mathfrak{A}^{\min} \subset \mathfrak{B}^{\min}$.
To conclude, consider the following diagram
    
\begin{equation} 
\xymatrix{
E' \widehat{\otimes}_{\alpha} F \ar@{->>}[r]^{J_{\alpha}} \ar@{^{(}->}[d]_{\iota} &  \mathfrak{A}^{\min}(E,F) \ar@{.>}[d] 
\\
 E' \widehat{\otimes}_{\beta} F \ar@{->>}[r]^{J_{\beta}}
 & \mathfrak{B}^{\min}(E,F).\\
} 
\end{equation}
The result follows from the fact that $J_{\alpha}$ and $J_{\beta}$ are complete quotients and $\Vert \iota \Vert_{\cb} \leq c$.  
\end{proof}

\begin{definition}
We say that a Banach mapping ideal $(\mathfrak{A},\mathbf{A})$ is minimal whenever $(\mathfrak{A},\mathbf{A})=(\mathfrak A^{\min},\mathbf{A}^{\min})$.
\end{definition}

We now provide some examples

\begin{enumerate}[(i)]
    \item $\mathcal A = \CB^{\min}$. Moreover, given $E,F\in\OBAN$ we have the completely isometric isomorphism $\mathcal A (E,F)=E'\widehat{\otimes}_{\min} F$.
    
   Indeed, it is known that the canonical mapping $i:E'\widehat{\otimes}_{\min} F\hookrightarrow \CB(E,F)$ is completely isometric and by the Representation Theorem, $J:E'\widehat{\otimes}_{\min} F\twoheadrightarrow \CB^{\min}(E,F)$ is a complete quotient. Then,  the inclusion $\CB^{\min}(E,F)\hookrightarrow\CB(E,F)$ has to  be completely isometric. That means that the norm in $\CB^{\min}(E,F)$ is $\|\cdot\|_{\cb}$. Since $\mathcal A$ is contained in any Banach mapping ideal whose norm is the cb-norm, it should be $\mathcal A = \CB^{\min}$.
   Also, $J$ is a completely isometric isomorphism and so $\mathcal A (E,F)=E'\widehat{\otimes}_{\min} F$.

    \item $\mathcal N = \mathcal I^{\min}$.
    
    Indeed, since $\mathcal I\sim\proj$ we have a complete quotient $J:E'\widehat{\otimes}_{\proj} F\twoheadrightarrow \mathcal I^{\min}(E,F)\subset \CB(E,F)$ and this coincides with the definition of $\mathcal N(E,F)$.
    \item The ideal of completely right-$p$ nuclear mappings $\mathcal N^p$ is minimal.
    
    Indeed, it is clear from its definition that $\mathcal N^p=(\mathcal N^p)^{\min}$ because $\mathcal N^p\sim d_p $ and $J^p:E'\widehat{\otimes}_{d_p } F\twoheadrightarrow \mathcal N^p(E,F)\subset \CB(E,F)$ is a complete quotient.
    
\end{enumerate}
%------------------------------------------------------
%------------------------------------------------------
%------------------------------------------------------
\section{Completely projective and completely injective o.s. tensor norms}
\label{capsulas inyectivas}

In this section we return to the abstract theory of o.s. tensor norms. In particular, we study o.s. tensor norms which behave well with respect to complete injections or projections. 
We know that the the minimal  o.s. tensor norm   respects complete injections \cite[Prop. 8.1.5]{Effros-Ruan-book} whereas  the projective o.s. tensor norm  respects complete projections \cite[Prop. 7.1.7]{Effros-Ruan-book}. 
A well-known property of the Haagerup o.s. tensor norm is that it respects both complete projections and injections \cite[Prop. 9.2.5]{Effros-Ruan-book}. It should be noted that this cannot happen in the Banach space/normed space framework as a consequence of Grothendieck's inequality \cite[20.20]{Defant-Floret}.
Based on the classical definitions, we consider in this section two natural procedures on o.s. tensor norms: the completely injective hull and the completely projective hull.
Similar constructions were considered in \cite{blecher1991tensor}.

We begin by recalling from Section \ref{injections-and-projections} the definitions of completely injective/projective o.s. tensor norms.

\begin{definition}
An o.s. tensor norm  $\alpha$ on $\ONORM$ is called \emph{completely right-injective} if for all $E,F,G \in \ONORM$ and for all complete injections $i : F \hookrightarrow G$ the mapping
$$
id_E \otimes i : E \otimes_{\alpha} F \to E \otimes_\alpha G
$$
is a complete injection 
and it is called \emph{completely right-projective} if for all  $E,F,G \in \ONORM$ and for all complete metric surjections $q :G \twoheadrightarrow F$, and all spaces $E$ the mapping
$$
id_E \otimes q : E \otimes_{\alpha} G \to E \otimes_\alpha F
$$
is a complete metric surjection.
Left versions are defined analogously, and we say that $\alpha$ is \emph{completely injective} (resp. \emph{completely projective}) when it is both completely left- and completely right-injective (resp. projective).
Analogous definitions will be used for other classes of spaces besides $\ONORM$ ($\OFIN$, for example).
\end{definition}

To provide some examples of these definitions, in \cite[Prop. 3.4]{CD-Chevet-Saphar-OS} it is shown that the Chevet-Saphar tensor norms $d_p $ are completely right-projective and the Chevet-Saphar tensor norms $g_p $ are completely left-projective. Also, by \cite[Prop. 6.1]{Pisier-Asterisque-98}, the norm $d_2 $ is completely left-injective and $g_2 $ is completely right-injective. It should also be mentioned that, by \cite[Prop. 6.2]{Wiesner}, any $\lambda$-o.s. tensor norm is completely projective.
For some conditions implying that a $\lambda$-o.s. tensor norm is completely injective, see \cite[Prop. 1.2]{Janson-Kumar-Luthra}.

\bigskip
For $M,N \in \OFIN$, the complete isometry
$$
M \otimes_{\alpha'} N = (M' \otimes_\alpha N')'
$$
implies the following:

\begin{remark} \label{rmk: dualidad entre injectividad y proyectividad}
$\alpha$ is completely right-projective  on $\OFIN$ if and only if $\alpha'$ is completely right-injective  on $\OFIN$.
\end{remark}

 In the Banach space world, the corresponding duality can be extended beyond the finite-dimensional case. Surprisingly, this property cannot be extended into $\OBAN$ for general o.s. tensor norms as mentioned after Remark \ref{consecuencias min} (c). This is an important difference between the classical and the non-commutative theory of tensor norms. Namely, in $\OBAN$, the fact that an o.s. tensor norm is completely right-injective does not imply that its dual norm is completely right-projective. 
  Nevertheless, the converse is valid (see Corollary \ref{cor: proy implica dual iny} below).

 \begin{proposition}\label{prop: injectividad en dim finita}
  If $\alpha$ is finitely-generated and completely right-injective on $\OFIN$ then $\alpha$ is completely right-injective on $\ONORM$.
  \end{proposition}

  \begin{proof}
  Let  $i : F \hookrightarrow G$ be a complete injection and $z \in M_n(E \otimes F)$.
  Since $\alpha$ is finitely-generated, given $\varepsilon>0$, consider $M \in \OFIN(E)$ and $N \in \OFIN(G)$ such that 
  $$\alpha_n((id_E\otimes i)_n (z); M,  N) \leq (1+\varepsilon)\alpha_n((id_E\otimes i)_n (z); E,G).$$
 
  Given $\widetilde{M} \in \OFIN(E)$ and $\widetilde{N} \in \OFIN(F)$ such that $z \in M_n(\widetilde{M} \otimes \widetilde{N})$ we have:
   \begin{align*}
   \alpha_n(z; E,F) & \leq \alpha_n(z; M+\widetilde{M},  \widetilde{N}) = \alpha_n((id_E\otimes i)_n (z); M + \widetilde{M},  N + i(\widetilde{N})) \\
   & = \alpha_n((id_E\otimes i)_n (z); M,  N) \leq (1+\varepsilon)\alpha_n((id_E\otimes i)_n (z); E, G),
  \end{align*}
   where equalities are due to the fact that $\alpha$ is  completely right-injective on $\OFIN$. This concludes the proof since we always have $\alpha_n((id_E\otimes i)_n (z); E,G) \leq \alpha_n(z; E,F)$.
  \end{proof}

As a consequence of Remark \ref{rmk: dualidad entre injectividad y proyectividad}, Proposition \ref{prop: injectividad en dim finita}  and the fact that  $\alpha'$ is finitely-generated we have:

\begin{corollary}
 \label{cor: proy implica dual iny}
If $\alpha$ is completely right-projective on $\OFIN$ then $\alpha'$ is completely right-injective on $\ONORM$.
\end{corollary}

As mentioned above, the reciprocal of the previous statement does not hold in general. Nevertheless, under an additional hypothesis on the tensor norm we can actually get the converse (see Proposition \ref{inclusion normas}).

%------------------------------------------------------
\subsection{Completely injective hulls}

As in the Banach space setting we now describe an analogous theory for the injective hulls of given o.s. tensor norm.

\begin{theorem}\label{thm:norma right injective asociada}
Let $\alpha$ be an o.s. tensor norm on $\ONORM$. Then there is a unique completely right-injective o.s. tensor norm $\alpha \backslash$ on $\ONORM$ such that $\beta \le \alpha\backslash$ for all completely right-injective o.s. tensor norms $\beta \le \alpha$.
For all normed operator spaces $E$ and $F \subseteq B(H)$,
$$
E \otimes_{\alpha \backslash} F \to E \otimes_{\alpha} B(H)
$$
is a complete isometry.
\end{theorem}

\begin{proof}
Suppose that we have two completely isometric embeddings $i_j : F \to B(H_j)$, $j=1,2$ and define operator space structures  $\alpha_j$ on $E \otimes F$ induced by
$$
id_E \otimes i_j : E \otimes F \to E \otimes_\alpha B(H_j).
$$
By the injectivity of $B(H_j)$, there exist completely bounded contractions $R_1 : B(H_1) \to B(H_2)$ and $R_2 : B(H_2) \to B(H_1)$ such that $i_1 = R_2 \circ i_2$ and $i_2 = R_1 \circ i_1$.
The map $(id_E \otimes R_1) \circ (id_E \otimes i_1)$ shows that the identity $E \otimes_{\alpha_1} F \to E \otimes_{\alpha_2} F$ is completely contractive, and analogously so is $E \otimes_{\alpha_2} F \to E \otimes_{\alpha_1} F$.
Therefore, there is no ambiguity if we define $\alpha\backslash$ on $E \otimes F$ to be the operator space structure on $E \otimes F$ induced by the particular embedding
$$
E \otimes F \to E \otimes_{\alpha} B(H).
$$
From the injectivity of $\min$, the fact that $\min \le \alpha$ and the diagram
$$
	\xymatrix{
	E \otimes_{\alpha\backslash} F \ar@{.>}[d] \ar[r]  & E \otimes_\alpha B(H) \ar[d]\\
	E \otimes_{\min} F \ar[r]  & E \otimes_{\min} B(H)
	}
$$
it follows that $\min \le \alpha\backslash$. Similarly, from $\alpha \le \proj$ and
$$
	\xymatrix{
	E \otimes_{\proj} F \ar[r] \ar@{.>}[d]   & E \otimes_{\proj} B(H) \ar[d]\\
	E \otimes_{\alpha\backslash} F \ar[r]  & E \otimes_\alpha B(H)
	}
$$
it follows that $\alpha\backslash \le \proj$.
Thus, it is clear that $\alpha\backslash$ is a reasonable operator space cross-norm.

It is not hard to see that $\alpha\backslash$ has the complete metric mapping property.
Let $S \in \CB(E_1, E_2)$ and $T \in \CB(F_1,F_2)$ and consider completely isometric embeddings $F_j \subseteq B(K_j)$.
By the injectivity of $B(H)$, there exists a completely bounded map $\widehat{T} : B(K_1) \to B(K_2)$ such that $\Vert {\widehat{T}} \Vert_{\cb} = \n{T}_{\cb}$ and
$$
	\xymatrix{
	B(K_1) \ar[r]^{\widehat{T}}  & B(K_2) \\
	F_1 \ar[u]  \ar[r]^{T}  & F_2 \ar[u]
	}
$$
commutes.
After tensorizing,
$$
	\xymatrix{
	E_1 \otimes_{\alpha} B(K_1) \ar[r]^{S \otimes \widehat{T}}  & E_2 \otimes_{\alpha} B(K_2) \\
	E_1 \otimes_{\alpha} F_1 \ar[u]  \ar[r]^{S\otimes T}  & E_2 \otimes_{\alpha} F_2 \ar[u]
	}
$$
from where we clearly see that, by the complete metric mapping property of $\alpha$,
$$
\n{ S\otimes T : E_1 \otimes_{\alpha\backslash} F_1 \to E_2 \otimes_{\alpha\backslash} F_2  }_{\cb} \le \n{S}_{\cb} \n{T}_{\cb}.
$$

Suppose that $i : F \to G$ and  $i_G : G \to B(H)$ are complete injections. Since we had already shown that $\alpha\backslash$ does not depend on the particular embedding into a $B(H)$ and $i_G\circ i:F\to B(H)$ is a complete injection we have, for any $n\in \N$ and $z \in M_n(E \otimes F)$,
\begin{align*}
\alpha\backslash_n((id_E\otimes i)_n (z) ; E, G) &\le \alpha\backslash_n(z ; E, F) = \alpha_n((id_E\otimes i_G\circ i)_n (z) ; E, B(H)) \\ & =  \alpha\backslash_n((id_E\otimes i)_n (z) ; E, G),
\end{align*}
showing that $\alpha\backslash$ is completely-right injective.

Finally, if $\beta$ is a completely  right-injective o.s. tensor norm dominated by $\alpha$ then the diagram
$$
	\xymatrix{
	E \otimes_{\alpha\backslash} F \ar[r] \ar@{.>}[d]   & E \otimes_\alpha B(H) \ar[d]\\
	E \otimes_{\beta} F \ar[r]  & E \otimes_{\beta} B(H)
	}
$$
shows $\beta \le \alpha\backslash$.
\end{proof}

\begin{lemma} \label{finitamente generada la capsula inyectiva}
Let $\alpha$ be a finitely-generated o.s. tensor norm then $ \alpha \backslash$ is also finitely-generated.
\end{lemma}
\begin{proof}
Let $E$ and $F$ in $\ONORM$ and $z \in M_n(E \otimes F)$. 
By the definition of the o.s. tensor norm $\alpha \backslash$ if
$F \subseteq B(H)$ we have
$$(\alpha \backslash)_n(z; E, F) = \alpha_n(z; E, B(H)).$$
Since $\alpha$ is finitely-generated, given $\varepsilon>0$ there are $\widetilde M \in \OFIN(E)$ and $\widetilde N \in \OFIN(B(H))$ such that

$$ \alpha_n(z; \widetilde M, \widetilde N) \leq (1+\varepsilon) \alpha_n(z; E, B(H)).
$$

Now, given $M\in \OFIN(E)$ and $N\in \OFIN(F)$ such that $z\in M \otimes N$ we have that 
\begin{align*}
(\alpha\backslash)_n (z;  (M+\widetilde M),  N) & = (\alpha\backslash)_n (z;  (M+\widetilde M), (N + \widetilde N))  \\
& \leq \alpha_n (z;  (M+\widetilde M),  (N +\widetilde N) )  \leq \alpha_n (z;  \widetilde M, \widetilde N) \\
&\leq (1+\varepsilon) \alpha_n(z; E, B(H)) = (1+\varepsilon) (\alpha  \backslash)_n (z; E, F).
\end{align*}
 This concludes the proof.
\end{proof}

\begin{definition}
\label{inyectiva conmutativa}
The o.s. tensor norm  $\alpha \backslash$ given in Theorem \ref{thm:norma right injective asociada} is called the \emph{completely right-injective hull} of $\alpha$.
One can define analogously the \emph{completely left-injective hull} of $\alpha$, denoted $\slash \alpha$, and the \emph{completely injective hull}
$$
\slash \alpha \backslash := (\slash \alpha) \backslash = \slash (\alpha \backslash).
$$
\end{definition}

Note that there is no ambiguity in the previous equality since $ (\slash \alpha) \backslash = \slash (\alpha \backslash)$. Indeed, let $E, F \in \ONORM$ and suppose that $E \subset B(H_E)$ and $F \subset B(H_F)$; by  Theorem \ref{thm:norma right injective asociada} and its left version, the o.s. tensor norms  $(\slash \alpha)\backslash$ and $\slash (\alpha \backslash)$ can be computed on $E \otimes F$ through the complete isometry
$$ E \otimes F \hookrightarrow B(H_E) \otimes_{\alpha} B(H_F).$$

%------------------------------------------------------
\subsection{Completely projective hulls}

Now it is time to describe the completely projective hulls of a given o.s. tensor norm.

To do this we need some preliminary results that will be useful later. Let us start by proving an operator space version of the perturbation step in the proof of \cite[Lemma 20.2.(2)]{Defant-Floret} without appealing to local reflexivity.
The proof is inspired by \cite[Lemma 2.13.2]{Pisier-Operator-Space-Theory}.

\begin{lemma} \label{lem: perturbation step}
Let $F \in \ONORM$, and let $\widetilde{F}$ be its completion.
Given $N \in \OFIN(\widetilde{F})$ and $\varepsilon>0$, there exists a linear map $R : N \to F$ with $\n{R}_{\cb} \le 1+\varepsilon$ and $Ry=y$ for all $y \in F \cap N$. 
\end{lemma}
\begin{proof}
Let $(x_j,x_j^*)_{j=1}^m$ be a biorthogonal system for the space $N$, that is, a basis $x_1,\dotsc,x_m$ of $N$ and functionals $x'_1,\dotsc,x'_m \in N'$ satisfying $x'_k(x_j) = \delta_{kj}$ (we don't assume anything about their norms).
Without loss of generality, we may suppose that $\{x_j\}_{j=1}^{m_0}$ is a basis for $F \cap N$ (with $m_0=0$ if $F\cap N = \{0\}$).
For $1 \le j \le m_0$ let $y_j = x_j$, and for $m_0+1\le j \le m$ choose $y_j \in F$ close enough to $x_j$ so that
$$
\sum_{j=m_0+1}^m \n{x'_j} \n{x_j-y_j} < \varepsilon.
$$
Define $R : N \to F$ as $y \mapsto \sum_{j=1}^m x'_j(y) y_j$.
Since the identity map $I : N \to \widetilde{F}$ can be written as $y \mapsto \sum_{j=1}^m x'_j(y) x_j$, it is clear that $Ry=y$ for all $y \in F \cap N$, and
$$
\n{R}_{\cb} \le \n{I}_{\cb} + \sum_{j=m_0+1}^m \n{x'_j} \n{x_j-y_j} < 1 + \varepsilon.
$$
This concludes the proof.
\end{proof}

Let us quote an important lemma regarding metric surjections from the classical theory  \cite[Quotient Lemma 7.4]{Defant-Floret}. This can be translated to the operator space setting  canonically adapting the proof. Its statement is the following:

\begin{lemma} \label{Quotient lemma}
Let $E, F\in\ONORM$,  $E_0\subset E$ a dense subspace, $q\in \CB(E,F)$ a surjective mapping  and consider $q_0=q|_{E_0}: E_0\to q(E_0)$ its restriction. Then, $q_0$ is a complete metric surjection if and only if $\overline{\ker q_0}=\ker q$ and $q$ is a complete metric surjection.
\end{lemma}

Another ingredient needed for the construction of the completely projective hull is the following concept of  right-finite hull of an o.s. tensor norm:

\begin{definition}\label{def:right-finite hull}
Given normed operator spaces $E$ and $F$, an o.s. tensor norm $\alpha$ on $\OFIN$ and $u \in M_n(E \otimes F)$, the \textit{right-finite hull} of $\alpha$ is given by
$$
\alpha^{\rightarrow}_n(u ; E, F) = \inf\left\{\alpha_n(u; E,F_0) \colon  F_0 \in\OFIN(F), u \in M_n(E \otimes F_0)\right\}.
$$
\end{definition}

It is clear that $\alpha^{\rightarrow}$ is a o.s. tensor norm and $\alpha\le \alpha^{\rightarrow}\le \overrightarrow{\alpha}$.

We are now ready to present the operator space version of \cite[Lemma 20.2]{Defant-Floret}:

\begin{lemma} \label{lema similar df20.20}
Let $\alpha$ be an o.s. tensor norm on $\ONORM$.
\begin{enumerate}[(a)]
    \item If $\alpha$ is completely right-projective on $\ONORM\times\OBAN$ then $\alpha=\alpha^{\rightarrow}$ on $\ONORM\times\OBAN$.
    \item If $\alpha=\alpha^{\rightarrow}$ on $\ONORM\times\OBAN$ then $\alpha=\alpha^{\rightarrow}$ on $\ONORM\times\ONORM$ and for all $E,F\in\ONORM$ with $\widetilde{F}$ the completion of $F$, the following canonical embedding is a complete isometry:
    $$
    E\otimes_\alpha F\hookrightarrow E\otimes_\alpha \widetilde{F}.
    $$
    \item If $\alpha$ is completely right-projective on $\ONORM\times\OBAN$ then it is completely right-projective on $\ONORM\times\ONORM$.
\end{enumerate}
\end{lemma}

\begin{proof}
$(a)$ For $G\in\OBAN$, by the construction mentioned in Section~\ref{injections-and-projections} there is a complete metric surjection $q_G:Z_G\twoheadrightarrow G$. Since $Z_G$ is an $\ell_1$ sum of finite dimensional spaces it is clear that it has the completely metric approximation property. By the Approximation Lemma \ref{approximation-lemma}, for every normed operator space $E$ we have that $\alpha=\alpha^{\rightarrow}$ on $E\otimes Z_G$.

Hence, for $z\in M_n(E\otimes G)$ and $\varepsilon >0$ there exist $N\in\OFIN(Z_G)$ and $u\in E\otimes N$ such that $Id_E\otimes q_G(u)=z$ and
$$
\alpha_n(u; E, N) \le (1+\varepsilon) \alpha_n(z; E, G).
$$
Therefore,
$$
\alpha_n(z; E, G)\le \alpha^{\rightarrow}_n(z; E, G) \le \alpha_n(z; E, q_G(N))\le \alpha_n(u; E, N) \le (1+\varepsilon) \alpha_n(z; E, G).
$$

$(b)$ Let  $z\in M_n(E\otimes F)$. By the metric mapping property,
$
\alpha^{\rightarrow}_n(z;E, \widetilde{F}) \le \alpha^{\rightarrow}_n(z; E, F). 
$
For $\varepsilon >0$ let $N\in \OFIN(\widetilde{F})$ with $z\in M_n(E\otimes N)$ such that $\alpha_n(z; E, N)\le (1+\varepsilon) \alpha^{\rightarrow}_n(z; E, \widetilde{F})$. By Lemma \ref{lem: perturbation step} there is a completely bounded linear map $R:N\to F$ with $\|R\|_{\cb}\le 1+\varepsilon$ and $Ry=y$ for all $y\in F\cap N$. Note that $z\in M_n(E\otimes (F \cap  N))$ because each entry of the matrix can be seen as an operator from $E'$ to $F$ and also from $E'$ to $N$. Thus, we have $(Id_E\otimes R)_n(z)=z\in M_n(E\otimes R(N))$.
Consequently, 
$$
\alpha^{\rightarrow}_n(z; E, F) \le \alpha_n((Id_E\otimes R)_n(z); E, R(N)) \le (1+\varepsilon) \alpha_n(z; E,N) 
\le (1+\varepsilon)^2 \alpha^{\rightarrow}_n(z; E,\widetilde{F}).
$$

$(c)$ Let $E,F,G\in\ONORM$ and $q:F\twoheadrightarrow G$ a complete metric surjection. By Lemma \ref{Quotient lemma} the completion $\widetilde{q}:\widetilde{F}\twoheadrightarrow \widetilde{G}$ is also a complete metric surjection with $\ker\widetilde{q}=\overline{\ker q}$. We know from $(a)$ and $(b)$ that the following embeddings are complete isometries:
$$
    E\otimes_\alpha F\hookrightarrow E\otimes_\alpha \widetilde{F}\quad \textrm{ and } \quad  E\otimes_\alpha G\hookrightarrow E\otimes_\alpha \widetilde{G},
    $$
    so, $Id_E\otimes \widetilde{q}(E\otimes_\alpha F)=E\otimes_\alpha G$.
Now, our assumption is that $Id_E\otimes \widetilde{q}:E\otimes_\alpha \widetilde{F} \to E\otimes_\alpha \widetilde{G}$ is a complete metric surjection. Thus, applying again Lemma \ref{Quotient lemma}, in order to obtain that $Id_E\otimes q:E\otimes_\alpha F \to E\otimes_\alpha G$ is a complete metric surjection we just need to check that $\ker(Id_E\otimes \widetilde{q})\subset\overline{\ker(Id_E\otimes q)}$ (where the closure is considered inside $E\otimes_\alpha \widetilde{F}$). Indeed, this is true because 
$$\ker(Id_E\otimes\widetilde{q})= E\otimes \ker\widetilde{q}=E\otimes \overline{\ker q} \subset \overline{E\otimes \ker q}=\overline{\ker(Id_E\otimes q)}.$$
\end{proof}

\begin{remark} \label{rmk:right-proj es right fin gen}
Observe that, using the previous lemma,  if $\alpha$ is completely right-projective on $\ONORM$ then it is finitely-generated from the right, i.e., $\alpha=\alpha^{\rightarrow}$.
\end{remark}

\begin{theorem} \label{thm: projective associate}
Let $\alpha$ be an o.s. tensor norm on $\ONORM$. Then there is a unique completely right-projective o.s. tensor norm $\alpha \slash$ on $\ONORM$ such that $\beta \ge \alpha\slash$ for all completely right-projective o.s. tensor norms $\beta \ge \alpha$.

For $E \in \ONORM$ and $F \in \OBAN$, if $Z_0$ is a  completely projective Banach operator space and $q : Z_0 \to F$ is a complete metric surjection, then
$$
E \otimes_{\alpha} Z_0 \to E \otimes_{\alpha \slash} F
$$
is a complete metric surjection.
\end{theorem}
We recall that, given a complete operator space $E$, the completely projective space $Z_E$ introduced in Section~\ref{injections-and-projections} is locally reflexive (as mentioned at the end of Section \ref{locally reflexive}).
With this property at hand we can now prove Theorem \ref{thm: projective associate}.

\begin{proof}[Proof of Theorem \ref{thm: projective associate}]
Let $F \in \OBAN$. Suppose that we have two complete quotients $q_j : Z_j \twoheadrightarrow F$, $j=1,2$ where the spaces $Z_j$ are completely projective operator spaces, and suppose that we have operator space structures  $\alpha_j$ on $E \otimes F$ making
$$
id_E \otimes q_j : E \otimes_\alpha Z_j \to E \otimes_{\alpha_j} F
$$
into complete quotients.
Fix $\varepsilon>0$. By the projectivity of $Z_i$, there exist operators $L_1 : Z_1 \to Z_2$ and $L_2 : Z_2 \to Z_1$ such that $\n{L_i}_{\cb} \le 1+\varepsilon$ satisfying $q_2 \circ L_1 = q_1$ and $q_1 \circ L_2 = q_2$.
Observe that by the metric mapping property of $\alpha$,
$id_E \otimes L_1 : E \otimes_\alpha Z_1 \to E \otimes_\alpha Z_2$ and
$id_E \otimes L_2 : E \otimes_\alpha Z_2 \to E \otimes_\alpha Z_1$ have $\cb$-norm at most $1+\varepsilon$.
Since $id_E \otimes q_j : E \otimes_\alpha Z_j \to E \otimes_{\alpha_j} F$ is a complete quotient, it follows from \cite[Prop. 2.4.1]{Pisier-Operator-Space-Theory} that the identity mappings $E \otimes_{\alpha_1} F \to E \otimes_{\alpha_2} F$ and $E \otimes_{\alpha_2} F \to E \otimes_{\alpha_1} F$ have $\cb$-norm at most $1+\varepsilon$. Letting $\varepsilon\to 0$, we conclude that
$E \otimes_{\alpha_1} F$ and  $E \otimes_{\alpha_2} F$ are canonically completely isometric.
Therefore, there is no ambiguity if we define $\alpha\slash$ on $E \otimes F$ to be the operator space structure on $E \otimes F$ induced by the particular quotient
$$
E \otimes_\alpha Z_F \twoheadrightarrow E \otimes_{\alpha/} F
$$
where $Z_F$ is the projective space introduced in Section~\ref{injections-and-projections}.

From the projectivity of $\proj$, the fact that $\alpha \le \proj$ and the diagram
$$
	\xymatrix{
	E \otimes_{\proj} Z_F \ar[d] \ar@{->>}[r]  & E \otimes_{\proj} F \ar@{.>}[d]  \\
	E \otimes_{\alpha} Z_F \ar@{->>}[r]  & E \otimes_{\alpha\slash} F
	}
$$
together with \cite[Prop. 2.4.1]{Pisier-Operator-Space-Theory} it follows that $\alpha\slash \le \proj$.

Similarly, from the inequality $\min \le \alpha$ and the diagram
$$
	\xymatrix{
	E \otimes_{\alpha} Z_F \ar@{->>}[r] \ar[d]  & E \otimes_{\alpha\slash} F \ar@{.>}[d]  \\
	E \otimes_{\min} Z_F  \ar[r]  & E \otimes_{\min} F
	}
$$
it follows that $\min \le \alpha\slash$.
Therefore, $\alpha\slash$ is a reasonable operator space cross-norm.

Let us now show that that $\alpha\slash$ has the complete metric mapping property.
Let $S \in \CB(E_1, E_2)$ and $T \in \CB(F_1,F_2)$ and consider complete quotients $Z_j \twoheadrightarrow F_j$ with the spaces $Z_j$ being completely projective.
Therefore, given $\varepsilon>0$ there exists a completely bounded map $\widehat{T} : Z_1 \to Z_2$ such that $\Vert \widehat{T} \Vert_{\cb} \le (1+\varepsilon) \n{T}_{\cb}$ and such that the  diagram 
$$
	\xymatrix{
	Z_1 \ar@{->>}[d] \ar[r]^{\widehat{T}}  & Z_2 \ar@{->>}[d] \\
	F_1   \ar[r]^{T}  & F_2
	}
$$
commutes.
After tensorizing, we obtain
$$
	\xymatrix{
	E_1 \otimes_{\alpha}Z_1 \ar@{->>}[d] \ar[r]^{S \otimes \widehat{T}}  &E_2 \otimes_{\alpha} Z_2 \ar@{->>}[d] \\
	E_1 \otimes_{\alpha}F_1   \ar[r]^{S\otimes T}  & E_2 \otimes_{\alpha} F_2
	}
$$
from where we clearly see, appealing to the complete metric mapping property of $\alpha$, that
$$
\n{ S\otimes T : E_1 \otimes_{\alpha\slash} F_1 \to E_2 \otimes_{\alpha\slash} F_2  }_{\cb} \le (1+\varepsilon) \n{S}_{\cb} \n{T}_{\cb}.
$$
Letting $\varepsilon\to0$ yields the complete metric mapping property of $\alpha\slash$.

Now, let $G \in \OBAN$ and suppose that $q : F \twoheadrightarrow G$ is a complete quotient.
If $q_F : Z_F \twoheadrightarrow F$ is the standard complete quotient,  observe that the composition
$q \circ q_F : Z_F \twoheadrightarrow G$ is also a complete quotient.
For $z \in M_n(E \otimes G)$, by what was shown at the beginning of the proof we have
\begin{align*}
\alpha\slash_n(z ; E, G) &= \inf\big\{ \alpha_n(w; E , Z_F) \,:\, (q \circ q_F)_nw = z \big\} \\
&= \inf\big\{ \alpha_n(w; E , Z_F) \,:\, (q_F)_nw = u, q_nu = z \big\} \\
&= \inf\big\{ \inf\{ \alpha_n(w; E , Z_F) \,:\, (q_F)_nw = u \} \,:\, q_nu=z \big\} \\
&= \inf \big\{ \alpha\slash_n(u;E,F) \,:\,  q_nu=z  \big\}.
\end{align*}
That is, $id_E \otimes q : E \otimes_{\alpha\slash} F \to E \otimes_{\alpha\slash} G$ is a complete quotient and therefore $\alpha\slash$ is completely right-projective in $\ONORM \times \OBAN$.

According to Lemma \ref{lema similar df20.20} $(a)$ we know that the norm $ \alpha \slash$ is finitely-generated from the right in $\ONORM \times \OBAN$. Thus, by Lemma \ref{lema similar df20.20} $(b)$ and $(c)$ this norm can be extended to $\ONORM \times \ONORM$ and this extension is completely right-projective on $\ONORM$.
Moreover, by definition $\alpha \leq \alpha \slash$ on $\ONORM \times \OFIN$, which implies that $$ \alpha \leq \alpha^{\rightarrow} \leq (\alpha \slash)^{\rightarrow} = \alpha \slash.$$

Finally, let $\beta$ be any completely right-projective norm that dominates $\alpha$. Then, $\alpha \slash \leq \beta$ in $\ONORM\times \OFIN$ by definition.  Remark \ref{rmk:right-proj es right fin gen} says that $\beta=\beta^{\rightarrow}$ on $\ONORM$ then, by Lemma \ref{lema similar df20.20} $(a)$, we have
$$ \alpha  \slash = (\alpha \slash)^{\rightarrow} \leq  \beta^{\rightarrow} = \beta.$$
\end{proof}

\begin{lemma} \label{finitamente generada la capsula proyectiva}
If $\alpha$ be a finitely-generated o.s. tensor norm in $\ONORM$, then $\alpha/$ is also finitely-generated.
\end{lemma}
\begin{proof}
Let $z \in M_n(E \otimes F)$. By Lemma~\ref{lema similar df20.20} we know that $\alpha /$ is finitely-generated on the right. Then, for a given $\varepsilon>0$ there is a subspace $N \in \OFIN(F)$ satisfying
$$ (\alpha /)_n(z; E, N) \leq (1+\varepsilon) (\alpha /)_n(z; E, F).$$
Since $N$ is finite dimensional it is complete. Hence, there is a completely projective Banach operator space $Z_N$ such that 
$$id_E \otimes q_N: E\otimes_{\alpha}Z_N \twoheadrightarrow E\otimes_{\alpha /}N$$ is a complete quotient.
Then there is an element $w \in M_n(E \otimes Z_N)$ with $(id_E \otimes q_N)_n (w) = z$
and
$$ \alpha_n(w; E, Z_N) \leq (1+\varepsilon) (\alpha /)_n (z;E, N).$$
Now, due to $\alpha$ being finitely-generated there exists $M \in \OFIN(E)$ such that 
$$ \alpha_n(w; M, Z_N) \leq (1+\varepsilon) \alpha_n(w; E, Z_N).$$
Note that $id_M \otimes q_N: E\otimes_{\alpha}Z_N \to E\otimes_{\alpha /}N$ has norm one and $(id_M \otimes q_N)_n (w) = z$, therefore
$$( \alpha /)_n(z;M, N) \leq \alpha_n (w;M, Z_N) \leq (1+\varepsilon)^3 \alpha_n(z;E, F).$$

\end{proof}

The following lemma will be useful for our purposes.
\begin{lemma} \label{lem: finit a la derecha}
Let $\alpha$ be an o.s. tensor norm finitely-generated from the right then $\backslash \alpha$ is finitely-generated.
\end{lemma}

\begin{proof}
Let $E,F\in\ONORM$ and $z\in M_n(E\otimes F)$. By the left version of Remark~\ref{rmk:right-proj es right fin gen} we know that $\backslash \alpha$ is finitely-generated from the left. Then, given $\varepsilon>0$ there exists $M\in\OFIN(E)$ such that $z\in M_n(M\otimes F)$ and
$$
(\backslash\alpha)_n(z; M,F)\le (1+\varepsilon)(\backslash\alpha)_n(z; E,F).
$$
 Applying the left version of Theorem~\ref{thm: projective associate} (and using that  $M\in\OFIN\subset\OBAN$) we have that there exists $u\in M_n(Z_{M}\otimes F)$ satisfying $(q_{M}\otimes id_F)_n(u)=z$ and 
$$
\alpha_n(u; Z_{M},F)\le (1+\varepsilon)(\backslash\alpha)_n(z; M,F).
$$
Since $\alpha$ is finitely-generated from the right, there is a space $N\in\OFIN(F)$ such that $u\in M_n(Z_{M}\otimes N)$ and
$$
\alpha_n(u; Z_{M},N)\le (1+\varepsilon)\alpha_n(u; Z_{M},F).
$$
The mapping
$$
q_{M}\otimes id_{N}:Z_{M}\otimes_\alpha N\twoheadrightarrow M\otimes_{\backslash\alpha}N
$$ satisfies $(q_{M}\otimes id_{N})_n(u)=z$. This implies that $z\in M_n(M\otimes N)$ and 
$$
(\backslash\alpha)_n(z; M,N)\le \alpha_n(u; Z_{M},N),
$$ which completes the proof.
\end{proof}

\begin{remark}
All the previous results about completely right-injective and completely right-projective o.s. tensor norms have left versions with analogous proofs. Specifically, there are valid left statements of Remark~\ref{rmk: dualidad entre injectividad y proyectividad},  Proposition~\ref{prop: injectividad en dim finita}, Corollary~\ref{cor: proy implica dual iny},  Theorem~\ref{thm:norma right injective asociada}, Lemma~\ref{finitamente generada la capsula inyectiva}. Also, after defining the left-finite hull of an o.s. tensor norm, there are  left adaptations of Lemma~\ref{lema similar df20.20}, Remark~\ref{rmk:right-proj es right fin gen}, Theorem~\ref{thm: projective associate},  Lemma~\ref{finitamente generada la capsula proyectiva} and Lemma~\ref{lem: finit a la derecha}.
\end{remark}

\begin{definition}\label{proyectiva conmutativa}
The o.s. tensor norm  $\alpha \slash$ given in Theorem \ref{thm: projective associate} is called the \emph{completely right-projective hull} of $\alpha$.
One can define analogously the \emph{completely left-projective hull} of $\alpha$, denoted $\backslash \alpha$, and the \emph{completely projective hull}
$$
\backslash \alpha \slash := (\backslash \alpha) \slash = \backslash (\alpha \slash).
$$

\end{definition}

Note that there is no ambiguity in the previous equality since  $\backslash (\alpha \slash)$ and $(\backslash \alpha) \slash$ are equal. Indeed, let $E,F \in OBAN$ and $Z_E$ and $Z_F$ completely projective operator spaces such that 
$$ q_E: Z_E \twoheadrightarrow E\quad\textrm{and}\quad q_F: Z_E \twoheadrightarrow F$$ are complete quotients then we known that both norms  $(\backslash \alpha) \slash$ and  $\backslash (\alpha \slash)$ coincide on $E\otimes F$ and can be computed through the complete quotient
$$ q_E \otimes q_F: Z_E \otimes_{\alpha} Z_F \twoheadrightarrow E \otimes_{\backslash \alpha \slash} F.$$

Using Remark \ref{rmk:right-proj es right fin gen} we know that $\alpha \slash$ is finitely-generated from the right and therefore, by Lemma \ref{lem: finit a la derecha}, the o.s. tensor norm $\backslash (\alpha \slash)$ is finitely-generated. The same argument can be used to prove that $(\backslash \alpha) \slash$ is finitely-generated. Since both norms $\backslash (\alpha \slash)$ and $(\backslash \alpha) \slash$ coincide on $\OBAN$ and they are finitely-generated, they must be equal on $\ONORM$.

\bigskip

The operator space $S_1(H)$ is not completely projective \cite{Blecher-standard}, but with respect to completely right- (or left-) projective hulls of a finitely generated o.s. tensor norm it behaves as if it was. The reason behind this is the fact that
$S_1(H)$ is an $\mathscr{OS}_{1,1+}$ space (see Section \ref{usual notation}):
each finite-dimensional subspace of $S_1(H)$ is contained in a larger finite-dimensional subspace which is ``almost a copy'' of $S_1^k$ for some $k$:

\begin{proposition} \cite[Prop. 5.3]{Effros-Ruan-Grothendieck-Pietsch}\label{subespacios-S1}
Given $M\in\OFIN(S_1(H))$ and $\varepsilon>0$ there exist a subspace $\widetilde{M}\in\OFIN(S_1(H))$, a number  $k\in\mathbb N$ and a complete isomorphism $T:\widetilde{M}\to S_1^k$ such that $\widetilde{M}\supset M$, $\|T\|_{cb}\le 1+\varepsilon$ and $\|T^{-1}\|_{cb}\le 1+\varepsilon$.
\end{proposition}

Now we can state and prove the result mentioned above.

\begin{proposition}\label{S1 casi proyectivo}
Let $\alpha$ be a finitely-generated o.s. tensor norm and let $E\in\ONORM$. Then,
$$
E\otimes_\alpha S_1(H) = E\otimes_{\alpha/} S_1(H) \quad\textrm{ and } \quad S_1(H) \otimes_\alpha E= S_1(H) \otimes_{\backslash\alpha} E.
$$
\end{proposition}

\begin{proof}
Since $\alpha\le \alpha/$ in order to prove the first identity (the other is obtained by transposition) we have just to show  for any $z\in M_n(E\otimes S_1(H))$ that $(\alpha/)_n(z;E,S_1(H))\le \alpha_n(z;E,S_1(H))$. Given $\varepsilon>0$ there is $M\in\OFIN(S_1(H))$ satisfying $z\in M_n(E\otimes M)$ and \begin{equation}\label{desig-1}
    \alpha_n(z;E,M)\le (1+\varepsilon) \alpha_n(z;E,S_1(H)).
\end{equation}
By Proposition \ref{subespacios-S1} there is $\widetilde{M}\supset M$, $k\in\mathbb N$ and a complete isomorphism $T:\widetilde{M}\to S_1^k$ such that  $\|T\|_{cb}\le 1+\varepsilon$ and $\|T^{-1}\|_{cb}\le 1+\varepsilon$. Then
\begin{equation}\label{desig-2}
     \alpha_n((id\otimes T)_n(z);E,S_1^k) \le (1+\varepsilon) \alpha_n(z;E,\widetilde M)\le (1+\varepsilon) \alpha_n(z;E,M).
\end{equation}
Now, applying $T^{-1}$ and using that $S_1^k$ is completely projective we have
\begin{equation}\label{desig-3}
     (\alpha/)_n(z;E,\widetilde M) \le (1+\varepsilon) \alpha_n((id\otimes T)_n(z);E,S_1^k).
\end{equation}

Finally, since $(\alpha/)_n(z;E,S_1(H))\le(\alpha/)_n(z;E,\widetilde M)$ compounding the identities \eqref{desig-1}, \eqref{desig-2} and \eqref{desig-3} we arrive to 
$$
(\alpha/)_n(z;E,S_1(H))\le (1+\varepsilon)^3 \alpha_n(z;E,S_1(H)),
$$
which concludes the proof.
\end{proof}

As a consequence of the previous result we present the following  operator space version of a classical Banach space identity.

\begin{example}\label{Ejemplo d infinito}
$d_\infty =\min/$.

\rm Indeed, since both norms are finitely-generated (by Proposition \ref{prop-dp-and-gp-finitely-generated} and Lemma \ref{lem: finit a la derecha}) it is enough to prove that $M\otimes_{d_\infty }N=M\otimes_{\min/}N$ for every $M,N\in\OFIN$.
Using that every separable operator space is completely isometric to a quotient of $S_1$ \cite[Cor. 2.12.3]{Pisier-Operator-Space-Theory} and that both $d_\infty $ and $\min/$ are completely right-projective, we know that there is a complete quotient $q:S_1\twoheadrightarrow N$ which produces two other complete quotients 
$$
id\otimes q: M\otimes_{d_\infty }S_1 \twoheadrightarrow M\otimes_{d_\infty } N \quad \textrm{and}\quad id\otimes q: M\otimes_{\min/}S_1 \twoheadrightarrow M\otimes_{\min/} N.
$$
Hence, the result is obtained once we prove $M\otimes_{d_\infty }S_1=M\otimes_{\min/}S_1$ for every $M\in\OFIN$. This is certainly true because $M\otimes_{d_\infty }S_1=M\otimes_{\min}S_1$ (by \cite[Th. 3.7]{CD-Chevet-Saphar-OS}) and $M\otimes_{\min/}S_1=M\otimes_{\min}S_1$ (by Proposition \ref{S1 casi proyectivo}).
\end{example}

Another consequence of Propositions \ref{subespacios-S1} and \ref{S1 casi proyectivo} is the following statement that will be useful later. The proof is obtained proceeding as in Example \ref{Ejemplo d infinito} along with Proposition \ref{S1 casi proyectivo}.

\begin{corollary}\label{Cor: capsula proyectiva con S1}
Let $E, F\in\ONORM$ with $F$ separable and let $\alpha$ be a finitely-generated o.s. tensor norm. If $z\in M_n(E\otimes_{\alpha/}F)$ and $\varepsilon>0$ there exist a number $k$, a mapping $R:S_1^k\to F$ and a matrix $u\in M_n(E\otimes_{\alpha}S_1^k)$ such that $(id\otimes R)_n(u)=z$, $\|R\|_{cb}\le 1+\varepsilon$ and
$$
\alpha_n(u;E, S_1^k)\le (1+\varepsilon) (\alpha/)_n(z;E,F).
$$
\end{corollary}

%-----------------------------------------------------

\section[Mapping ideals associated to completely injective/projective hulls]{Mapping ideals associated to completely injective/projective hulls and accessibility} \label{mapping accessibility sec}

In Corollary \ref{cor: proy implica dual iny} we have shown that the dual of a completely right-projective (completely left-projective) o.s. tensor norm is completely right-injective (completely left-injective). In the Banach space setting the reciprocal is valid but in the operator space context this is no longer true (see Remark \ref{consecuencias min} below). Note that 
all known  proofs of this fact for Banach space tensor norms use elements of the theory  mentioned in Subsection \ref{Relevant differences}. 

Hence, in order to produce a result relating completely
 injective/projective hulls with duality we have to appeal for an extra hypothesis of accessibility.

Even if previously we presented in detail results about right hulls  and obtain by analogy the left statements, now, \textsl{to see the other side of the moon}, we chose to proceed first  with left hulls and to derive by similarity the right versions.

\begin{proposition} \label{inclusion normas}
Let $\alpha$ be a finitely-generated o.s. tensor norm.
\begin{enumerate}[(a)]
 \item If $\alpha$ is left-accessible then $ (\backslash \alpha')' = / \alpha$    and  $(/ \alpha)' = \backslash \alpha'$.
\item If $\alpha$ is right-accessible then $ ( \alpha'/)' =  \alpha\backslash$    and  $(\alpha\backslash)' =  \alpha'/$.
\item If $\alpha$ is  totally accessible then $ (\backslash \alpha'/)' = / \alpha\backslash$    and  $(/ \alpha\backslash)' = \backslash \alpha'/$.  
    
\end{enumerate}

 \end{proposition}

 \begin{proof}
 
We will only prove $(a)$  since $(b)$ and $(c)$  follow analogously. To see the first identity, it suffices to check it in $\OFIN$ since both norms are finitely-generated. Indeed, the norm $(\backslash \alpha')'$ is finitely-generated since it is a dual norm and  $/ \alpha$ is so by the left version of  Lemma~\ref{finitamente generada la capsula inyectiva}.

 Let $M$ and $N$ be finite dimensional operator spaces. We have to see that 
 \begin{equation} \label{ec1}
  \big(M \otimes_{ \backslash \alpha '} N' \big)' =  M' \otimes_{ / \alpha} N.
 \end{equation}

 To check \eqref{ec1} we consider the following diagram

 $$ \xymatrix{   (M \otimes_{\backslash \alpha'} N' )' \ar@{^{(}->}[rr]^{ \; \; (A)} \ar[d]_{(\spadesuit)}
 & & (Z_M \otimes_{\backslash \alpha '} N')'  \ar@2{-}[d]^{(B)} \\
 M' \otimes_{/ \alpha} N  \ar@{^{(}->}[r]^{\;\;\;\;\;\;\;\;\;\;} &  Z_M' \otimes_{/ \alpha} N \overset{(C)}{=} Z_M' \otimes_{ \alpha} N  \ar@{^{(}->}[r]^{\;\;\;\;\;\;\;\;\;\; \;  \; \; (D)} & (Z_M \otimes_{\alpha '} N')'}, \\$$

 Note that the arrow $(A)$ is completely isometric because it is the transpose of the complete quotient given by
 $$ \xymatrix{  Z_M \otimes_{\backslash \alpha '} N' \ar@{->>}[rr] & & M \otimes_{\backslash \alpha'} N'.}$$
 Equalities in $(B)$ and $(C)$ hold since $Z_M$ is completely projective and  $Z_M'$ is completely injective \cite[Corollary 24.6]{Pisier-Operator-Space-Theory}, respectively.
 The mapping $(D)$ is a complete isometry  by the Duality Theorem \ref{duality-theorem} (recall that $Z_M$ is locally reflexive) and  the equality $\alpha = \overset{\leftarrow}{\alpha}$
 in $Z_M'\otimes N$ which holds since $\alpha$ is left-accessible. Therefore, the mapping $(\spadesuit)$ is a completely isometric isomorphism.
 
The second identity of $(a)$ follows by taking adjoints and using that, by the left version of Lemma~\ref{finitamente generada la capsula proyectiva}, $\backslash\alpha'$ is finitely-generated.
\end{proof}

\begin{remark} \label{remark-inclusion-normas}
Looking at the bottom line of the previous commutative diagram, we have that if $\alpha$ is a left-accessible finitely-generated o.s. tensor norm and $M$ and $N$ are finite dimensional operator spaces then we have the following complete isometry:
$$
M' \otimes_{/ \alpha} N \hookrightarrow (Z_M \otimes_{\alpha '} N')'.
$$
Similarly, we have for $\alpha$ right-accessible and finitely-generated:
$$
M' \otimes_{\alpha\backslash} N \hookrightarrow (M \otimes_{\alpha '} Z_{N'})'.
$$
\end{remark}

\begin{example}\label{example:dual}
Since $\proj$ is left- and right-accessible and finitely-generated, Proposition \ref{inclusion normas} tells us that
$$
(\backslash\min)'=/\proj \quad\textrm{and}\quad (\min/)'=\proj\backslash.
$$
\end{example}

It would be interesting to know whether the equality
 $(\backslash\min /)'=/\proj \backslash$ holds. This can not be easily derived from Proposition \ref{inclusion normas} $(a)$ since we do not know if $\proj\backslash$ is left-accessible, or from Proposition \ref{inclusion normas} $(c)$ since we know that $\proj$ is not totally accessible. Later, in Remark \ref{rmk:equivalencia Banach} we will see that these two norms, $(\backslash\min /)'$ and $ /\proj \backslash$ are equivalent at the Banach space level. The question about if they are equal (or, at least, equivalent) as o. s. tensor norms remains open.

\bigskip

As we have seen, accessibility is related with the duality between injective and projective hulls of o.s. tensor norms. This property also appears when describing  the associated norms of surjective and injective hulls of mapping ideals (see Theorem \ref{theorem A} below). First, we need the following lemma.

 \begin{lemma} \label{lemat 1}
  Let $(\mathfrak{A}, \mathbf{A})$ be a mapping ideal.
\begin{enumerate}[(a)]
    \item If $\beta$ is a finitely-generated o.s. tensor norm associated to  $\mathfrak{A}^{\sur}$ then $\beta$ is completely  left-injective.
    \item If $\beta$ is a finitely-generated o.s. tensor norm associated associated to  $\mathfrak{A}^{\inj}$ then $\beta$ is completely  right-injective.
\end{enumerate}
   \end{lemma}

 \begin{proof}
 
$(a)$ We suppose first that $E_1, E_2, F$ are finite dimensional and $i:E_1 \to E_2$ is a complete isometry.
 Then,
 \begin{align*}
  E_1 \otimes_{\beta} F &= \mathfrak{A}^{\sur}(E_1',F) \\
  E_2 \otimes_{\beta} F &= \mathfrak{A}^{\sur}(E_2',F).
 \end{align*}
Through these identifications the mapping $i\otimes id_F : E_1 \otimes_{\beta} F \to E_2 \otimes_{\beta} F$ looks like
\begin{align*}
  \mathfrak{A}^{\sur}(E_1',F) & \longrightarrow \mathfrak{A}^{\sur}(E_2',F) \\
   T & \longmapsto T \circ i'.
 \end{align*}

 Since $\mathfrak{A}^{\sur}$ is surjective  and $i': E_2' \twoheadrightarrow E_1'$ is a complete quotient mapping, we have $\mathbf{A}_n^{\sur}(T) =  \mathbf{A}_n^{\sur}(T \circ i') $ for all $T \in M_n(\mathfrak{A}^{\sur}(E_1',F))$, which concludes the proof in the finite dimensional case.

Now, let $E_1, E_2, F$ be normed  operator spaces and $i:E_1 \to E_2$ be a complete isometry.
For $z \in M_n(E_1 \otimes_{\beta} F)$ we must show that
 \begin{equation}
  \beta_n( i \otimes id_F(z); E_2, F) =
  \beta_n( z; E_1, F).
 \end{equation}

 The metric mapping property gives us that $\beta_n( i \otimes id_F(z); E_2, F) \le
  \beta_n( z; E_1, F)$. For the reverse inequality, let $M_1 \subset E_1, M_2 \subset E_2, N \subset F$  be finite dimensional operator spaces such that $z \in M_n(M_1 \otimes N)$ and $i \otimes id_F(z) \in M_n(M_2 \otimes_{\beta} N)$.
  Then, 
  \begin{equation}
  \beta_n( z; E_1, F) \leq \beta_n( z; M_1, N) = \beta_n( i \otimes id_F(z);i(M_1)+ M_2, N) \leq \beta_n( i \otimes id_F(z); M_2, N).
  \end{equation}
 Since $\beta$ is finitely-generated, the equality follows.

$(b)$ As in the previous item, we begin by assuming that $E, F_1, F_2$ are finite dimensional operator spaces and $i:F_1 \to F_2$ is a complete isometry.
Then,
 \begin{align*}
  E \otimes_{\beta} F_1 &= \mathfrak{A}^{\inj}(E',F_1) \\
  E \otimes_{\beta} F_2 &= \mathfrak{A}^{\inj}(E',F_2).
 \end{align*}

Through these identifications the mapping $id_E\otimes i : E \otimes_{\beta} F_1 \to E \otimes_{\beta} F_2$ looks like
\begin{align*}
  \mathfrak{A}^{\inj}(E',F_1) & \longrightarrow \mathfrak{A}^{\inj}(E',F_2) \\
   T & \longmapsto i\circ T.
 \end{align*}
 
Now, the equality  $\mathbf{A}_n^{\inj} (T) = \mathbf{A}_n^{\inj} (i_n\circ T) $ for all $T \in M_n(\mathfrak{A}^{\inj}(E',F_1))$ concludes the proof in the finite dimensional case.
 The case for arbitrary operator spaces is easily derived as in item $(a)$.
 \end{proof}

 \begin{remark} \label{desig norm}
 Let $\alpha, \beta, \gamma$ be  finitely-generated o.s. tensor norms associated to the mapping ideals $\mathfrak{A}$, $\mathfrak{A}^{\sur}$ and $\mathfrak{A}^{\inj}$ respectively. By the previous lemma we have
 \begin{equation}
  \beta \leq / \alpha \quad \mbox{ and } \quad \gamma \leq \alpha \backslash.
 \end{equation}

 \end{remark}

 \begin{theorem}\label{theorem A}
  Let $(\mathfrak{A}, \mathbf A)$ be a mapping ideal and $\alpha$ its associated  o.s. tensor norm. 
  \begin{enumerate}[(a)]
      \item If $\alpha$ is left-accessible then the  o.s. tensor norm $/ \alpha$ is associated to $\mathfrak{A}^{\sur}$.
       \item If $\alpha$ is right-accessible then the  o.s. tensor norm $\alpha \backslash$ is associated to $\mathfrak{A}^{\inj}$.
  \end{enumerate}
  \end{theorem}

  \begin{proof}
$(a)$ Suppose first that $\mathfrak{A}$ is a maximal  mapping ideal and let $M,N$ be finite-dimensional operator spaces.
 Let $q_M: Z_M \twoheadrightarrow M$ be the canonical complete quotient mapping. Then, for every $T \in M_n(\mathfrak{A}^{\sur}(M,N))$ we have
$$
 \mathbf A_n^{\sur}(T)  = \mathbf A_n(T\circ q_M)   = \Vert T\circ q_M \Vert_{M_n((Z_M \otimes_{\alpha '} N')')}
 =( / \alpha)_n( T ; M',  N).
 $$
 The second equality follows from the Representation Theorem \ref{representation-theorem} and the third comes from Remark \ref{remark-inclusion-normas}.
 Therefore, we have the complete isometry $$\mathfrak{A}^{\sur}(M,N) \overset{1}{=} M' \otimes_{/ \alpha} N.$$
 In particular,  $ / \alpha$ is the associate o.s. tensor norm of $\mathfrak{A}^{\sur}.$

 Now, let $\mathfrak{A}$ be an arbitrary mapping ideal. Note that if $\mathfrak{A} \sim \alpha$ then $\mathfrak{A}^{\max} \sim \alpha$ and thus, by the first part of the proof $(\mathfrak{A}^{\max})^{\sur} \sim / \alpha$.

 The inclusion $\mathfrak{A} \subset \mathfrak{A}^{\max}$ (in the o.s. setting) implies $\mathfrak{A}^{\sur} \subset (\mathfrak{A}^{\max})^{\sur}$. Denoting by $\beta$ an  o.s. tensor norm associated to $\mathcal {A}^{\sur}$ we derive that
 \begin{equation}
  / \alpha \leq \beta.
 \end{equation}

 Now the equality follows by Remark \ref{desig norm}.
 
 $(b)$ Suppose first that $\mathfrak{A}$ is a maximal mapping ideal and let $M,N$ be finite-dimensional operator spaces.  Let $q_{N'}: Z_{N'} \twoheadrightarrow N'$ be the canonical complete quotient mapping. Then, for every $T \in M_n(\mathfrak{A}^{\inj}(M,N))$ we have
$$
  \mathbf A_n^{\inj}(T) = \mathbf A_n( (q_{N'}')_n \circ T )  = \Vert (q_{N'}')_n \circ T \Vert_{M_n((M \otimes_{\alpha '} Z_{N'})')}
 = (\alpha \backslash)_n( T ; M',  N).
 $$
 The second equality follows from the Representation Theorem \ref{representation-theorem} and the third comes from Remark \ref{remark-inclusion-normas}.
 Therefore, we have the complete isometry $$\mathfrak{A}^{\inj}(M,N) \overset{1}{=} M' \otimes_{ \alpha\backslash} N.$$
 In particular,  $ \alpha\backslash$ is an o.s. tensor norm associated to $\mathfrak{A}^{\inj}.$
 
 The proof for an arbitrary mapping ideal $\mathfrak{A}$ runs as in the previous item.
\end{proof}

\begin{remark}\label{rmk: bad news}
An important property in the Banach space framework \cite[Prop. 20.9]{Defant-Floret} says that if $\alpha$ and $\beta$ are Banach-space tensor norms with $\beta$ completely right-injective, the following are equivalent:
\begin{enumerate}[(a)]
\item $\beta = \alpha\backslash$
\item $E \otimes_{\beta} \ell_\infty^n = E \otimes_{\alpha} \ell_\infty^n$ for every normed space $E$.
\end{enumerate}

The proof of this equivalence depends crucially on the fact that in the Banach space setting the injective hull is given by the embedding into an $\ell_\infty(I)$ space, which is an $\mathcal{L}_{\infty,1+}$ space (meaning for every $\varepsilon>0$ any finite-dimensional subspace of $\ell_\infty(I)$ is contained in a larger finite-dimensional subspace which is $1+\varepsilon$ isomorphic to an $\ell_\infty^n$), so the $\mathcal{L}_p$-local technique lemma can be applied.

The analogous argument in the operator space setting, however, does not work.
While we do have a local technique lemma (Lemma \ref{local-technique-lemma}), it only applies to $\mathscr{OS}_{p,C}$ spaces (see definition in Section~\ref{usual notation}).
In particular, note that $\mathscr{OS}_{\infty,C}$ spaces are clearly exact (recall that an operator space $F$ is called $C$-exact if every finite-dimensional subspace of $F$ is $C$-completely isomorphic to a subspace of an $M_n$ space).
Since in the operator space setting the completely injective hull is given by embedding into a $B(H)$ space, and $B(H)$ is not exact when $H$ is infinite-dimensional, there is no hope of adapting the proof of the aforementioned Banach space result to the operator space setting.
% the injective associate is given by embedding into a $B(H)$ space and $B(H)$ is not a $\mathscr{OS}_{\infty,C}$ space (which is the counterpart of an $\mathcal{L}_{\infty,C}$-space in the Operator Space context).
%Of course, things would get fixed if any operator space could be embedded into some other injective operator space that is an $\mathscr{OS}_{\infty,C}$ space. That is not true, because $\mathscr{OS}_{\infty,C}$ spaces (and thus their subspaces) are clearly exact. 
\end{remark}

What we can get is the following replacement:

\begin{proposition}
Let $\alpha$ and $\beta$ be o.s. tensor norms with $\beta$ completely right-injective.
If $E \otimes_{\beta} M_n = E \otimes_{\alpha} M_n$ for all $n\in\N$,
then $E \otimes_{\beta} F = E \otimes_{\alpha\backslash} F$ for any $F \subseteq \mathcal{K}(\ell_2)$.
\end{proposition}

Now we study the relation between accessibility of a tensor norm with some properties of its associated mapping ideal. As it is our habit in the operator space framework, we consider local reflexivity versions of the definitions.

%\begin{remark}
%It would actually be more natural to consider definitions of accessibility that take into account $\overrightarrow{\alpha}_n$, $\overleftarrow{\alpha}_n$ for all $n \in \N$.
%Proposition \ref{prop-equivalences-accessibility} below would remain valid, as long as Definition \ref{defn-accessible-mapping-ideal} of accessibility for a mapping ideal is adjusted accordingly for matrices of mappings instead of individual mappings.
%\end{remark}

\begin{definition}\label{defn-accessible-mapping-ideal}
A mapping ideal $(\mathfrak{A},\mathbf{A})$ is called \emph{right-accessible} \emph{(locally right-accessible)} if for all $M \in \OFIN$, $F \in \OBAN$ ($F\in\OLOC$), $n\in\N$, $T \in M_n(\CB(M,F))$ and $\varepsilon>0$ there are $N \in \OFIN(F)$ and $S \in M_n(\CB(M,N))$ such that $\mathbf{A}_n(S) \le (1+\varepsilon)\mathbf{A}_n(T)$ and the following diagram commutes
$$
\xymatrix{
M_n(M) \ar[dr]_S \ar[r]^T &M_n(F) \\
 &M_n(N) \ar[u]_{(i^F_N)_n}\\
}
$$
$(\mathfrak{A},\mathbf{A})$ is called \emph{left-accessible} \emph{(locally left-accessible)}  if for all $E \in \OBAN$  ($E\in\OLOC$), $N\in\OFIN$, $n\in\N$, $T \in M_n(\CB(E,N))$ and $\varepsilon>0$, there exist $L \in \OCOFIN(E)$ and $S \in M_n(\CB(E/L,N))$ such that $\mathbf{A}_n(S) \le (1+\varepsilon) \mathbf{A}_n(T)$ and the following diagram commutes
$$
\xymatrix{
M_n(E) \ar[r]^T \ar[d]_{(q^E_L)_n} &M_n(N)\\
M_n(E/L) \ar[ru]_S &\\
}
$$
A mapping ideal which is both left and right-accessible (locally left- and locally right-accessible) is called \emph{accessible} \emph{(locally accessible)}.
Moreover, $(\mathfrak{A},\mathbf{A})$ is called \emph{totally accessible} \emph{(locally  totally accessible)}   if for every matrix of finite-rank operators $T : M_n(E) \to M_n(F)$ between complete
operator spaces (complete locally reflexive operator spaces) and $\varepsilon>0$, there are $L \in \OCOFIN(E)$, $N \in \OFIN(F)$ and $S \in M_n(\CB(E/L,N))$ such that
$\mathbf{A}_n(S) \le (1+\varepsilon) \mathbf{A}_n(T)$ and the following diagram commutes
$$
\xymatrix{
M_n(E) \ar[r]^T \ar[d]_{(q^E_L)_n} &M_n(F)\\
M_n(E/L) \ar[r]_S &M_n(N) \ar[u]_{(i^F_N)_n}\\
}
$$
\end{definition}

The necessity of these \emph{local} versions of the accessibility definitions is justified by Example \ref{example-accessible-ideals} $(2)$, where we show that the mapping ideal $\mathcal I$ is locally left-accessible but not left-accessible.

\begin{remark}\label{remark-injective-and-projective-implies-accessible}
It is clear that every completely injective mapping ideal is right-accessible, and every completely projective mapping ideal is left-accessible.
Assume  $T : M_n(E) \to M_n(F)$ is a matrix of finite-rank operators.
Let $\im(T)$ denote the span of the union of the spaces $(T_{ij}(E))_{i,j=1}^n$, and observe that $\im(T) \in \OFIN(F)$.
Denoting $\ker(T) = \bigcap_{i,j=1}^n \ker(T_{ij})$, note $\ker(T) \in \OCOFIN(E)$ and moreover there is a canonical
factorization
$$
\xymatrix{
M_n(E) \ar[r]^T \ar[d] &M_n(F)\\
M_n(E/\ker(T)) \ar[r]_S &M_n(\im(T)) \ar[u]\\
}
$$
This shows that a mapping ideal that is both completely injective and completely projective is totally accessible.
\end{remark}
As an example, since we have already noted that $\CB$ is both completely injective and completely projective, we deduce that the mapping ideal $\CB$ is totally accessible.

\begin{remark}\label{remark-maximal-accessibility}
Note that if $(\mathfrak{A}^{\max},\mathbf{A}^{\max})$ is right-accessible (left-accessible, totally accessible, locally accessible) then $(\mathfrak{A},\mathbf{A})$ is right-accessible (left-accessible, totally accessible, locally accessible). Indeed, this is clear from the definitions of the different variants of accessibility along with the completely contractive inclusion $\mathfrak{A}(E,F)\subset \mathfrak{A}^{\max}(E,F)$ for $E,F\in\OBAN$ and the completely isometric equality $\mathfrak{A}(M,N)= \mathfrak{A}^{\max}(M,N)$ for $M,N\in\OFIN$.
\end{remark}

There is a reason why we are defining notions of accessibility for both cross-norms and mapping ideals: in the Banach space setting, for a maximal mapping ideal $\mathfrak{A}$ associated with a (finitely-generated) cross-norm $\alpha$, the two notions of accessibility coincide.
We will now show that, in the operator space world, there is a relationship between (some of) the notions of accessibility for o.s. tensor norms and for operator ideals, but, as usual, local reflexivity is involved. Actually, for right accessibility the equivalence is just as in the Banach space case but for left accessibility we could only prove a weaker statement.

\begin{proposition}\label{prop-equivalences-accessibility}
Let $(\mathfrak{A},\mathbf{A})$ be a mapping ideal and $\alpha$ be its associated finitely-generated o.s. tensor norm.
Then:

\begin{enumerate}[(a)]
    \item  $\alpha$  is right-accessible (locally right-accessible)  if and only if $(\mathfrak{A},\mathbf{A})$ is right-accessible (locally right-accessible).
    \item If $\alpha$ is left-accessible (totally accessible) then $(\mathfrak{A},\mathbf{A})$ is locally left-accessible (locally totally accessible).
\end{enumerate}

\end{proposition}

\begin{proof}
$(a)$ Assume that $\alpha$ is right-accessible. Due to Remark \ref{remark-maximal-accessibility} we may consider that $(\mathfrak{A},\mathbf{A})$ is maximal.
Let $M \in \OFIN$,   $F\in\OBAN$ and
 $T \in M_n(\CB(M,F))$. Denote by $z_T \in M_n(M' \otimes F)$ the matrix of tensors corresponding to $T$.
By the definition of right-accessibility and the Embedding Theorem \ref{embedding-theorem} (using the fact that $M$ is finite-dimensional and therefore locally reflexive) we have
\begin{equation}\label{igualdadnorma}
\overrightarrow{\alpha}_n( z_T ; M', F) = \overleftarrow{\alpha}_n( z_T ; M', F) = \mathbf{A}_n(T).
\end{equation}
This implies that there are $\widetilde{M} \in \OFIN(M')$, $N \in \OFIN(F)$ and $u \in M_n(\widetilde{M} \otimes N)$ such that
$$
\alpha_n(u ; \widetilde{M}, N) \le (1+\varepsilon)\mathbf{A}_n(T) \quad \text{and} \quad (i^{E'}_{\widetilde{M}} \otimes i^F_N)_n (u) = z_T;
$$
moreover, note that we can assume $\widetilde{M} =M'$.
Hence, if $T_u  \in M_n(\CB(M, N))$ is the matrix of mappings corresponding to $u$, it satisfies $\mathbf{A}_n(T_u) \le (1+\varepsilon) \mathbf{A}_n(T)$ and
$(i^F_N)_n \circ T_u = T$.

% For the general case we just note that we also have the equality in equation \eqref{igualdadnorma}. Indeed,
% since we always have the complete contractions

% $$ E'\otimes_{\alpha} F \rightarrow \mathfrak{A}(E,F) \rightarrow \mathfrak{A}^{\max}(E,F).$$
% and $\alpha= \overleftarrow{\alpha}$ on the tensor $E'\otimes F$, the equality follows from the Embedding Theorem \ref{embedding-theorem} again.

Conversely, assume that $(\mathfrak{A},\mathbf{A})$ is right-accessible.
Since $\alpha$ is finitely-generated, $\alpha = \overrightarrow{\alpha}$. We always have $\overleftarrow{\alpha} \le \alpha$, so all we need to show is
$$
\alpha(\cdot; M',F) \le \overleftarrow{\alpha}(\cdot; M',F)
$$
for finite-dimensional $M$ and arbitrary $F$.
Given $z \in M_n(M' \otimes F)$ and $\varepsilon > 0$, by the right-accessibility of $\mathfrak{A}$ there are $N \in \OFIN(F)$ and $S \in M_n(\CB(M,N))$ such that
$$
\mathbf{A}_n(S) \le (1+\varepsilon)\mathbf{A}_n(T_z) \quad \text{and} \quad (i^F_N)_n \circ S = T_z
$$
where $T_z \in M_n(\CB(M,F))$ is the matrix of mappings corresponding to $z$.
If $z_S \in M_n(E' \otimes N)$ is the matrix of tensors corresponding to $S$,
since $M' \otimes_\alpha N = \mathfrak{A}(M,N)$ completely isometrically, it follows that
$$
\alpha_n(z_S; M', N) = \mathbf{A}_n(S) \le (1+\varepsilon)\overleftarrow{\alpha}_n(z ; E' , F)
$$
and $(i_{M'} \otimes i^F_N)_n(z_S) = z$,
showing that $\alpha$ is right-accessible.

The case of local right-accessibility follows in the same way just replacing $F \in \OBAN$ by $F\in \OLOC$. 

$(b)$ We prove the case of left-accessibility; the other one is similar. Again, by Remark \ref{remark-maximal-accessibility} we may consider that $(\mathfrak{A},\mathbf{A})$ is maximal. Let $E\in\OLOC$, $N\in\OFIN$ and $T\in M_n(\CB(E,N))$. Denoting by $z_T\in M_n(E'\otimes N)$ the matrix of tensors corresponding to $T$ and applying the definition of left-accessibility and the Embedding Theorem \ref{embedding-theorem} we have
$$
\overrightarrow{\alpha}_n( z_T ; E', N) = \overleftarrow{\alpha}_n( z_T ; E', N) = \mathbf{A}_n(T).
$$
This implies that there is $M \in \OFIN(E')$  and $u \in M_n(M \otimes N)$ such that
$$
\alpha_n(u ; M, N) \le (1+\varepsilon)\mathbf{A}_n(T) \quad \text{and} \quad (i^{E'}_{M} \otimes id)_n (u) = z_T.
$$

Hence, if $T_u  \in M_n(\CB(E/^0M, N))$ is  the matrix of mappings corresponding to $u$, we arrive to $\mathbf{A}_n(T_u) \le (1+\varepsilon) \mathbf{A}_n(T)$ and
$ T_u\circ q_{^0M}^E = T$, which finishes the proof.
\end{proof}

\begin{remark} \label{remark-lambda-accessibility}
If $\alpha'$ is a $\lambda$-o.s. tensor norm, then the conclusion of Proposition \ref{prop-equivalences-accessibility} $(b)$ is obtained without the word ``locally''. This is clear following the previous argument since, as we comment in Remark \ref{remark-embedding-theorem-lambda}, in this case the Embedding Theorem \ref{embedding-theorem} is valid with the local reflexivity hypothesis. On the other hand, in general it is not possible to obtain the conclusion of Proposition \ref{prop-equivalences-accessibility} $(b)$  without the word ``locally''. Indeed, in Example \ref{example-accessible-ideals} (2) we see that $\mathcal I$ is not left-accessible even though the associated o.s. tensor norm $\proj$ is left-accessible.
\end{remark}

\begin{example} \label{example-accessible-ideals}
\begin{enumerate}
    \item The mapping ideal $\mathfrak A_h$ is totally accessible.
    
    {\rm Since $h$ is totally accessible and $h'=h$ is a $\lambda$-o.s. tensor norm, the conclusion follows from Proposition \ref{prop-equivalences-accessibility} and Remark \ref{remark-lambda-accessibility}.}
   
    \item The mapping ideals $\mathcal N$ and $\mathcal I$ are right-accessible and locally left-accessible. The mapping ideal $\mathcal I$ is not left-accessible.
    
    \rm Since $\proj$ is both right and left accessible we obtain the first assertion from Proposition \ref{prop-equivalences-accessibility}. Now, suppose that $\mathcal I$ is left-accessible. Then, for every $E \in \OBAN$,   $N\in\OFIN$, $T \in \CB(E,N)$ and $\varepsilon>0$, there exist $L \in \OCOFIN(E)$ and $S \in \CB(E/L,N)$ such that $T=S\circ q_L^E$ and $\iota(S) \le (1+\varepsilon) \iota(T)$. Due to the fact that $E/L, N\in \OFIN$ we know that $\nu(T)\le \nu(S)=\iota(S)$ which implies $\nu(T)\le \iota(T)$. Since the other inequality is always valid we obtain the  isometry $\mathcal N(E,N)=\mathcal I(E,N)$. By \cite[Prop. 14.3.1]{Effros-Ruan-book} this isometry is valid for every $N\in\OFIN$ if and only if $E$ is locally reflexive. Therefore, the mapping ideal $\mathcal I$ can not be left-accessible.
\end{enumerate}

\end{example}

We are not aware of whether there exists a locally right-accessible mapping ideal which is not right-accessible. It would be interesting to have such an example or a proof of its nonexistence.

\section{Natural o.s. tensor norms}\label{normasnaturales sec}

Grothendieck's \emph{R\'{e}sum\'{e}} contained the list of all \textit{natural} tensor norms. These norms come from applying a finite number of \emph{natural} operations to the projective and injective tensor norms. They are obtained by taking left/right projective and injective hulls in some order (see Sections 15 and 20 in \cite{Defant-Floret}).
Grothendieck proved that there were at most fourteen possible natural norms, but he did not know the exact dominations among them, or if there was a possible reduction on the table of natural norms (this was, in fact, one of the open problems posed in the \emph{R\'{e}sum\'{e}}). This was solved, several years later, thanks to very deep ideas of Gordon and Lewis.
All these results are now classical and can be found for example in \cite[Section 27]{Defant-Floret} and \cite[4.4.2]{diestel2008metric}. 

One of the strengths of Grothendieck's result is that most of his fourteen tensor norms (at least ten) are \textit{really natural}, since they turn out to be equivalent to the most relevant tensor norms: those related to the ideals of bounded, integral, absolutely $r$-summing ($r=1,2$), $r$-factorable ($r=1,2,\infty$) and $2$-dominated operators. These tensor norms appear \emph{naturally} in the theory by their own interest, and it is a remarkable thing that they can be obtained from the projective/injective norm by means of the \emph{natural} operations introduced by Grothendieck.

In the operator space setting a study of natural norms in the sense of Grothendieck was suggested by Blecher in \cite{blecher1991tensor} while in \cite{kirchberg1995exact, pisier2020tensor} completely injective and completely projective hulls of tensor norms are considered in the $C^*$-algebra framework. However, only a few natural operator space tensor norms have been previously  considered as related to mapping ideals.  That is the case of $\min /$, $\backslash\min$ and $/\proj \backslash$. 
According to Example \ref{Ejemplo d infinito} we have that $d_\infty  = \min/$, and this was also observed in \cite{Effros-Ruan-Grothendieck-Pietsch} (right after Cor. 5.5). By transposing, $g_\infty  = \backslash\min$. The o.s. tensor norm $/\proj\backslash$ was introduced in \cite{dimant2015bilinear} under the name $\eta$ and the dual of $E\otimes_{\eta}F$  can be identified with the operator space bilinear ideal of extendible elements (i.e. those which extend to any larger space).
As can be seen in \cite{dimant2015bilinear} the o.s. tensor norm $/\proj\backslash$ is connected to deep results on noncommutative versions of Grothendieck's inequality \cite{pisier2002grothendieck,haagerup2008effros}, and it actually is equivalent at the Banach space level to $h \cap h^t$. Similarly $\backslash\min/$ is equivalent to $h+h^t$ at the Banach space level (see Lemma \ref{Lema:equivalencia nivel Banach} for the details).

The lack of a proper tensor product version of Grothendieck's inequality along with (or maybe motivated by) the existence of Haagerup tensor norm which is both completely injective and completely projective has a strong impact in the behavior of \textsl{natural} operator space tensor norms.

In this section, we consider Grothendieck's natural norms in the operator space framework. In other words, those norms obtained from $\min$ or $\proj$ after applying left or right injective/projective hulls finitely many times. 

We present a first overview of the list, proving many dominations and non-equivalences. We also state an interesting list of open questions about them as well.
On the positive side, we completely describe the list of all natural norms that come from applying to $\min$ or $\proj$ two-sided symmetric operations (injective or projective hulls).
Precisely, the list consists of six o.s. tensor norms. Again, this differs from the Banach space case where there are four.

We begin with some results from  Banach space theory that remain valid for operator spaces and later we get into the differences between both contexts.

As in the Banach space setting the process of taking alternatively injective and projective hulls (at the same side or at both sides of a norm) finishes after three steps. This is shown in the following lemma which is the o.s.version of \cite[Prop. 2.6.3]{diestel2008metric} (see also \cite[Lemma 3.3]{carando2012natural}).
\begin{lemma}\label{lemma two-sided hulls}
Let $\alpha$ be an o.s. tensor norm then
\begin{enumerate}[(a)]
    \item $\backslash (/ (\backslash (/ \alpha \backslash )/ )\backslash )/ = \backslash (/ \alpha \backslash ) /$ and $/(\backslash (/ (\backslash\alpha /)\backslash )/)\backslash = /(\backslash\alpha /)\backslash$.
    \item $\alpha \backslash / \backslash / =  \alpha \backslash /$ and $\alpha /\backslash /\backslash = \alpha /\backslash$.
    \item $\backslash / \backslash / \alpha  = \backslash / \alpha $ and $/\backslash /\backslash\alpha  = /\backslash\alpha $.
\end{enumerate}
\end{lemma}

\begin{proof} We just prove the first equality of $(a)$ since all the arguments are  similar.  It is clear that $ / (\backslash (/ \alpha \backslash) /) \backslash  \le \backslash (/ \alpha \backslash )/$ so we can apply completely projective hulls to both sides to obtain $\backslash (/ (\backslash (/ \alpha \backslash) / )\backslash) /\le \backslash (/ \alpha \backslash )/$. The other inequality follows analogously: we start from $/\alpha \backslash\le \backslash (/\alpha \backslash ) /$ and apply completely  injective hulls to both sides and then completely projective hulls arriving to $\backslash (/ \alpha \backslash ) /\le\backslash (/ (\backslash (/ \alpha \backslash ) / )\backslash ) /$.

\end{proof}

As we have seen in Definitions \ref{inyectiva conmutativa} and \ref{proyectiva conmutativa} we can commute the order when taking completely injective (or completely projective) hulls on both sides of a norm. The situation is different if we apply a completely injective hull on one side and a completely projective hull on the other. The following lemma (which has an analogous statement for the Banach space setting) provides an example of this fact through an equality of norms which will be important for the future description of natural norms.

\begin{lemma} \label{lema:identidad proj}
The following o.s. tensor norms are equal:
$$
(/\proj)/=\proj= \backslash(\proj\backslash).
$$
\end{lemma}

\begin{proof}
By Lemma \ref{finitamente generada la capsula inyectiva} and Lemma \ref{finitamente generada la capsula proyectiva}, $(/\proj)/$ is finitely generated. So, to prove the first equality (the other is obtained by transposing) it is enough to work with $E\in\ONORM$ and $F\in\OBAN$.
Let  $q_F:Z_F\twoheadrightarrow F$ be a complete quotient. Then the following two mappings are complete quotients too:
$$
id_E\otimes q_F:E\otimes_{/\proj}Z_F\twoheadrightarrow E\otimes_{(/\proj)/}F \quad\textrm{and}\quad id_E\otimes q_F:E\otimes_{\proj}Z_F\twoheadrightarrow E\otimes_{\proj}F.
$$

Since $E\otimes_{/\proj}Z_F= E\otimes_{\proj}Z_F$ (because $Z_F'$ is completely injective -see Remark \ref{rmk:left injective}) we deduce that $E\otimes_{(/\proj)/}F=E\otimes_{\proj}F$.
\end{proof}

\bigskip

The $\min$ counterpart of the previous lemma (which is valid in the Banach space framework) is not true in the operator space world. To prove this we appeal to the following result which can be also obtained (after dualization) as a byproduct of the proof of  \cite[Prop. 15.4.3]{Effros-Ruan-book} together with the fact that the set of  \textsl{exactly integral mappings} is contained in the set of completely 1-summing mappings.

%\begin{lemma}\label{lemma:Exactly_integral}
%Let $E\in\OBAN$ and $M\in\OFIN$ with $M\subset E$. If there exists $\lambda>0$ such that  $\|id: \Pi_1 (E,M)\to \mathcal I(E,M)\|_{cb}\le \lambda$  then the operator space $M$ is $\lambda$-exact.
%\end{lemma}

\begin{lemma}\label{lemma:Exactly_integral}
Let $E\in\OBAN$ and $F\in\OFIN$ with $F'\subset E$. If there exists $C>0$ such that  $\|id: E\otimes_{\min}F\to E\otimes_{\min/}F\|\le C$  then the operator space $F'$ is $C$-exact.
\end{lemma}

\begin{proof}
Since $F$ is finite-dimensional, there is a canonical completely isometric isomorphism $\Phi: E\otimes_{\min}F\to \CB (F',E)$. Let $z\in E\otimes_{\min}F$ such that $\Phi(z)=id$, which leads to $\min(z; E,F)=1$ and $\min/(z; E,F)\le C$. Given $\varepsilon>0$, by Corollary \ref{Cor: capsula proyectiva con S1} there exist $k\in\mathbb N$, $R:S_1^k\to F$ and $u\in E\otimes_{\min}S_1^k$ such that $(id\otimes R)(u)=z$, $\|R\|_{cb}\le 1+\varepsilon$ and 
$$
\min (u; E, S_1^k)\le (1+\varepsilon) \min/ (z;E,F)\le C (1+\varepsilon).
$$
Translating this through $\Phi$ we obtain the following commutative diagram:
$$
\xymatrix{
M_k  \ar[rd]^{\Phi(u)} &\\
F' \ar[u]^{R'} \ar[r]_{id} & E
}	
$$ where $\|R'\|_{cb}\le 1+\varepsilon$ and $\|\Phi(u)\|_{cb}\le C(1+\varepsilon)$. Hence, $F'$ is $C$-exact.
\end{proof}

The statement of the next lemma is (maybe) a little bit stronger than the one we have promised. We show that the norms $/(\min/\backslash)$ and $(/\backslash\min)\backslash$ differ from $\min$. Since $/(\min/\backslash)\le /(\min/)$ and $(/\backslash\min)\backslash\le (\backslash\min)\backslash$ the  result follows.

\begin{lemma}\label{lema:no identidad min}
Neither $/(\min/\backslash)$ nor $(/\backslash\min)\backslash$ are equivalent to $\min$.
In particular, neither $/(\min/)$ nor $(\backslash\min)\backslash$ are equivalent to $\min$.
\end{lemma}

\begin{proof}
Let us consider the operator space $S_1^n$ for any $n> 2$ and a Hilbert space $H$ such that $S_1^n\subset \mathcal B(H)$. Recall that $S_1^n=M_n'$.

Since $\mathcal{B}(H)$ and $M_n$ are completely injective   we know that $\mathcal B(H)\otimes_{/(\min /\backslash)}M_n=\mathcal B(H)\otimes_{\min /}M_n$. Thus, to prove that $/(\min /\backslash)$ and $\min$ are not equivalent it is enough to see that $\mathcal B(H)\otimes_{\min /}M_n$ and $\mathcal B(H)\otimes_{\min }M_n$ are not uniformly completely isomorphic. By Lemma \ref{lemma:Exactly_integral} this is true
since there is no $C$ such that $S_1^n$ is $C$-exact for all $n$ \cite[Thm. 14.5.4]{Effros-Ruan-book}.

%Taking into account that $M_n$ is finite dimensional we know
%$$
%\mathcal I(\mathcal B(H),S_1^n)=(\mathcal B(H)\otimes_{\min} M_n)' \quad\textrm{and} %\quad \Pi_1 (\mathcal B(H),S_1^n) = (\mathcal B(H)\otimes_{\min /} M_n)'.
%$$

%Now, since there is no $\lambda$ such that $S_1^n$ is $\lambda$-exact for all $n$ \cite[Thm. 14.5.4]{Effros-Ruan-book},
%Lemma \ref{lemma:Exactly_integral} tells us that  $\mathcal I(\mathcal B(H),S_1^n)$ and $\Pi_1 (\mathcal B(H),S_1^n)$ are not  uniformly completely isomorphic,  obtaining the desired result.
\end{proof}

\begin{remark}\label{consecuencias min}
From Lemmas \ref{lema:identidad proj} and \ref{lema:no identidad min} we derive several consequences:

\begin{enumerate}[(a)]
    \item The o.s. tensor norm $\min/$ is not left-accessible (and the o.s. tensor norm $\backslash\min$ is not right-accessible).
    
    Indeed, since we know from Example \ref{example:dual} that $(\min/)'=\proj\backslash$,
    if $\min/$ was left-accessible by appealing to Proposition \ref{inclusion normas} we would have that $/(\min/)= (\backslash (\proj\backslash))'=\proj ' =\min$ which is not true.

    \item If $\alpha=\min/$ and $\beta=\backslash\min$ since both are finitely generated the previous item shows that
    $$
    (/\alpha)'\not= \backslash\alpha' \quad\textrm{and}\quad (\beta\backslash)'\not=\beta'/.
    $$
    Also, note that $/(\min/\backslash)=/(\min/)\backslash$ and $(/\backslash\min)\backslash=/(\backslash\min)\backslash$. Then,
    by Lemma \ref{lema:no identidad min} and observing that $\backslash(\proj\backslash)/= \backslash (/\proj)/=\proj$ we obtain
    $$
    (/\alpha\backslash)'\not= \backslash\alpha'/ \quad\textrm{and}\quad (/\beta\backslash)'\not=\backslash\beta'/.
    $$

     \item The o.s. tensor norm $\gamma=/(\min/) $ is completely  left-injective, but its dual $\gamma'$ is not completely left-projective. Analogously, the o.s. tensor norm $(\backslash\min)\backslash $ is completely right-injective but its dual  is not completely right-projective. Also,  the o.s. tensor norms $/(\backslash\min)\backslash $ and $/(\min/)\backslash $ are (both sides) completely injective but their duals are not (both sides) completely projective .
    
    Indeed, since $\min\le \gamma\le \min/$ we have that $\proj\ge \gamma'\ge \proj \backslash$. If $\gamma'$ were completely left-projective we would have $\proj\ge \gamma'\ge \backslash(\proj \backslash)=\proj$ which is not true. 
\end{enumerate}
\end{remark}

The previous remark shows that the dual of a completely injective operator space tensor norm is not necessarily completely projective, exhibiting that the definition of projective hulls (as the dual of the injective hull) given in \cite{blecher1991tensor} is not adequate.

\bigskip

As in the Banach space setting we obtain a smaller norm if we apply first a completely projective hull on one side and then the completely injective hull on the other, than if we do so the other way around.

\begin{lemma}\label{Lemma:desigualdad proy-iny}
Let $\alpha$ be an o.s. tensor norm. Then:
$$
/(\alpha/)\le(/\alpha)/ \quad \textrm{and} \quad (\backslash\alpha)\backslash\le \backslash(\alpha\backslash).
$$
\end{lemma}

\begin{proof}
We have just to prove the first inequality because the other is obtained by transposing. Due to Remark \ref{rmk:right-proj es right fin gen} along with a slight refinement of Lemma \ref{finitamente generada la capsula inyectiva}, both norms $/(\alpha/)$ and $(/\alpha)/$ are finitely-generated from the right. Thus, it is enough to check the inequality for $E\otimes F$ with $E\in\ONORM$ and $F\in\OBAN$. Consider a complete isometry $i_E:E\hookrightarrow \mathcal B(H_E)$ and a complete quotient $q_F:Z_F\twoheadrightarrow F$.

By the projectivity of $Z_F$, the mapping 
$id_{\mathcal{B}(H_E)}\otimes  q_F: \mathcal B(H_E)\otimes_{\alpha}Z_F\to \mathcal B(H_E)\otimes_{\alpha /}F$ is a complete quotient; hence a complete contraction. This implies (by the injectivity of $\mathcal B(H_E)$) that $id_E\otimes  q_F: E\otimes_{/\alpha}Z_F\to E\otimes_{/(\alpha /)}F$ is a complete contraction.

Now the conclusion follows from the following commutative diagram:
$$
\xymatrix{
E\otimes_{(/\alpha )/}F \ar[r] &E\otimes_{/(\alpha /)}F\\
E\otimes_{/\alpha }Z_F
\ar[u]
\ar[ru] &\\
}
$$

\end{proof}

Note that the inequalities of the previous lemma are not in general equalities. Indeed, 
$/(\proj/)=/\proj$  is  not equivalent to $\proj$. So, from Lemma \ref{lema:identidad proj} we derive $/(\proj/)\not=(/\proj)/$ and $(\backslash\proj)\backslash\not= \backslash(\proj\backslash)$.

The same happens in the Banach space setting: $/(\pi/)\not=(/\pi)/$ and $(\backslash\pi)\backslash\not= \backslash(\pi\backslash)$. Be aware of a false assertion  about this topic given in   \cite[Cor. 2.4.17]{diestel2008metric} (which is deduced from the erroneous  statement of \cite[Prop. 2.4.16]{diestel2008metric}). 

Let us denote by \textsl{``$\proj$ family''} (resp. \textsl{``$\min$ family''}) the set of all the o.s. tensor norms produced by applying one-sided completely injective and/or completely projective hulls any number of times to the norm $\proj$ (resp. $\min$). Note that if we translate this to the Banach space context we obtain that the union of both  families is the set of Grothendieck's natural tensor norms.

It is clear that $/\proj\backslash$ is the largest completely injective o.s. tensor norm and $\backslash\min /$ is the smallest completely projective o.s. tensor norm. The following simple lemma says that they are also the smallest and the largest members of their respective families.

\begin{lemma}\label{lemma:smallest y largest}
\begin{enumerate}[(a)]
    \item The smallest norm of the $\proj$ family is $/\proj\backslash$.
    \item The largest norm of the $\min$ family is $\backslash\min /$.
\end{enumerate}
\end{lemma}

\begin{proof}
The proof of $(a)$ follows easily from the following self explanatory fact: if $/\proj\backslash\le\alpha$ for some o.s. tensor norm $\alpha$, then
$$
/\proj\backslash\le\alpha\backslash,\quad
/\proj\backslash\le
/\alpha,\quad  /\proj\backslash\le\alpha/, \quad /\proj\backslash\le\backslash\alpha.
$$
The proof of $(b)$ is similar.
\end{proof}

\begin{remark}\label{Remark:Grothendieck}\textbf{Some tensor norm consequences of Grothendieck's inequality that are not longer valid in the operator space framework.}

\begin{enumerate}[(a)]
    \item It follows from Grothendieck's inequality  that $/\pi\backslash$ is dominated by $\backslash\varepsilon /$ (i.e. there exists a constant $C>0$ such that $/\pi\backslash\le C\backslash\varepsilon /$).
    
    In the operator space setting, it is not possible for a member of the $\min$ family to dominate a member of the $\proj$ family. Indeed, since the Haagerup tensor norm $h$ is both completely injective and completely projective we have that
    $$
    \backslash\min /\le h\le /\proj\backslash.
    $$
    The lack of symmetry of $h$ prevents it from being equivalent to any  symmetric tensor norm (as $ \backslash\min /$ or $/\proj\backslash$). To check this recall \cite[Prop. 9.3.4]{Effros-Ruan-book} that for any Hilbert space $H$ if we denote by $H_c$ the associated column space we have that $H_c\otimes_h H_c'=\mathcal K(H)$ while $H_c'\otimes_h H_c=\mathcal B(H)$.
    
    By means of the previous lemma, we get that any member of the $\min$ family is smaller than (and not equivalent to) any member of the $\proj$ family.
    
    \item As a byproduct of the inequality referred in (a) it is obtained that the norms $\backslash(/\pi\backslash)/$ and $\backslash\varepsilon /$ are equivalent; the same happens with $/(\backslash\varepsilon /)\backslash$ and $/\pi\backslash$.
    
    The operator space version of these equivalences is clearly not valid due to the arguments explained in the previous item. This fact certainly has an impact in the quantity of \textsl{natural} operator space tensor norms.
    
    \item In the Banach space world the equivalences stated in the previous item obviously produce that if $\alpha$ is an injective tensor norm then $\backslash\alpha/$ is equivalent to $\backslash\varepsilon /$ and if $\alpha$ is a projective tensor norm then $/\alpha \backslash$ is equivalent to $/\pi\backslash$. 
    
    In the operator space universe  these equivalences are obviously not true. From our explanation in (a) it is clear that $\backslash\min /\le h\le \backslash(/\proj\backslash)/$ are three non-equivalent completely projective hulls of completely injective norms, and $/(\backslash\min /)\backslash\le h\le /\proj\backslash$ are three non-equivalent completely injective hulls of completely projective norms.
    
\end{enumerate}

\end{remark}

\begin{remark}\textbf{Open question.} \label{rmk:open question}

In the Banach space setting the impact of Grothendieck's inequality in the description of natural tensor norms is decisive. We have pointed out in the previous remark several relations that are not longer valid in the operator space framework.

There is another equivalence of Banach space tensor norms (also derived from Grothendieck's inequality) that is crucial for the resulting number of 14 natural tensor norms, namely:

\centerline{The norms $/((/\pi\backslash)/)$ and $(/\pi\backslash)/$ are equivalent.}

To recall the sketch of the argument we use the symbol $\sim$ to denote equivalent tensor norms and we denote by $w_2$ the Hilbertian tensor norm and by $d_2$ the Chevet-Saphar tensor norm. By Grothendieck's inequality, $/\pi\backslash\sim w_2$ and $w_2/\sim d_2$. Since $d_2$ is completely  left-injective the desired equivalence is proved.

In the operator space setting we cannot follow the same path, but nevertheless we can ask:

\centerline{Are the norms $/((/\proj\backslash)/)$ and $(/\proj\backslash)/$  equivalent?}

\medskip

Observe that clearly $/((/\proj\backslash)/)\le(/\proj\backslash)/$ so the question can be equivalently posed in the following way:

\centerline{Is the norm  $(/\proj\backslash)/$  completely left-injective?}

\end{remark}

If the answer to the open question is positive we  obtain a large number of coincidences between members of the $\proj$ family.
Indeed, an affirmative answer would turn into identities the  following inequalities derived from Lemma \ref{Lemma:desigualdad proy-iny}:

\begin{enumerate}
    \item[(a)] $/(\proj\backslash /)\le  (/\proj\backslash)/$ and $(\backslash /\proj)\backslash\le \backslash (/\proj\backslash)$.
    \item[(b)] $\backslash/(\proj\backslash /)\le \backslash (/\proj\backslash)/$ and $(\backslash /\proj)\backslash/ \le \backslash (/\proj\backslash)/$.
\end{enumerate}
Also, if there are identities in (a) the following inequalities turn out also to be identities:
\begin{enumerate}
    \item[(c)] $/\proj \backslash\le (/\proj \backslash)/\backslash $ and $/\proj \backslash\le /\backslash(/\proj \backslash)$. 
\end{enumerate}

\subsection{The $\proj$ family}
Even though we are not able to give a full description of the $\proj$ family, we devote this subsection to present a picture of what we know.

The $\proj$ family begins of course with the $\proj$ norm. Then, we enumerate the norms taking into account how many procedures (of completely left or completely right-injective/projective hull) we have applied to $\proj$ in order to obtain them. The list, \textsl{up to four procedures} is the following:

\begin{enumerate}
    \item[(0)] $\proj$.
    \item  $/\proj\ \bullet \ \proj\backslash$.
    \item $/\proj\backslash\ \bullet \ \backslash /\proj\ \bullet \ \proj \backslash /$.
    \item $\backslash(/\proj\backslash)\ \bullet\ (/\proj\backslash)/ \bullet \ (\backslash /\proj)\backslash\ \bullet \ /(\proj \backslash /)$.
    \item $\backslash(/\proj\backslash)/ \ \bullet\ /\backslash(/\proj\backslash) \ \bullet\ (/\proj\backslash)/\backslash \bullet \ (\backslash /\proj)\backslash /\ \bullet \ \backslash /(\proj \backslash /)$.
\end{enumerate}

The norms (derived from $\proj$ after at most four procedures) that do not appear in the previous list are identical to one of those that do appear in the list.  Indeed, from Lemma \ref{lema:identidad proj}, Lemma \ref{Lemma:desigualdad proy-iny}, Lemma \ref{lemma:smallest y largest} and the definitions of completely injective and completely projective hulls we have:

\begin{enumerate}
    \item[(2)]  $(/\proj)/= \backslash(\proj\backslash) =\proj.$
    \item[(3)] $(\backslash /\proj)/=  \backslash(\proj \backslash /))=\proj$.
    \item[(4)] $  /((/\proj\backslash)/)=/( \proj \backslash /)\ \bullet \ (\backslash(/\proj\backslash))\backslash =(\backslash / \proj )\backslash$.
    \item[(4)] $ ( /( \proj \backslash /)) /= (/\proj\backslash)/ \ \bullet \ \backslash((\backslash / \proj )\backslash) = \backslash(/\proj\backslash)$.
    \item[(4)] $/(\backslash /\proj)\backslash = /(\proj \backslash /)\backslash = /\proj \backslash$.
\end{enumerate}

Thus, the members of the $\proj$ family \textsl{up to four procedures} are 15. We do not know whether they are all non-equivalent. In fact, as we mentioned above, if the answer of the open question is positive we derive six identities between members of the $\proj$ family, reducing the list from 15 to 9 members.

Moreover, an affirmative answer to the open question would also imply that there exist no more members of the $\proj$ family: a fifth procedure produces no new elements. Indeed, in this situation, the fourth procedure only generates one element: $\backslash(/\proj\backslash)/ = (\backslash /\proj)\backslash /= \backslash /(\proj \backslash /)$. And it is easily seen that a completely left or completely right-injective hull of this element (under the present hypothesis of an affirmative answer to the open question) yields the norm $/\proj \backslash$.

Now, we present the dominations valid between the  15 norms of our list and later we will provide examples to show that many of these norms are non-equivalent. Each arrow $\alpha\to\beta$  in the next diagram means that $\beta\le\alpha$. Double arrows are those where we  prove in the sequel that the dominations are strict. Tensor norms connected by doted arrows are equivalent if the Open question \ref{rmk:open question} has an affirmative answer. We do not know if the dominations connected by standard arrows are in fact strict.

\tiny{

\begin{equation}\label{normasnaturales}
 \xymatrixrowsep{0.7in}
\xymatrixcolsep{0.05in}
 \xymatrix{ & & & *+[F]{ \begin{array}{c}
                    \proj
                   \end{array}} \ar@{=>}[d] \ar@{=>}[ld] \ar@{=>}[rd] & & & 
                   \\
                   & & *+[F]{ \begin{array}{c}
                    \proj\backslash /
                   \end{array}} \ar@{=>}[lld] \ar[rrd] \ar@{=>}[rrrd] & *+[F]{ \begin{array}{c}
                    \backslash(/\proj\backslash)/
                   \end{array}} \ar[llld] \ar@{-->}[lld] \ar[rrrd] \ar@{-->}[rrd] & *+[F]{ \begin{array}{c}
                    \backslash /\proj 
                   \end{array}} \ar@{=>}[llld] \ar[lld] \ar@{=>}[rrd]& & \\
                   *+[F]{ \begin{array}{c}
                     (/\proj\backslash)/
                   \end{array}} \ar@{-->}[rd] \ar[rrd] &
                   *+[F]{ \begin{array}{c}
                     \backslash /(\proj \backslash /)
                   \end{array}}  \ar[d] & *+[F]{ \begin{array}{c}
                    /\proj
                   \end{array}}  \ar@{=>}[ld] & & *+[F]{ \begin{array}{c}
                    \proj\backslash
                   \end{array}} \ar@{=>}[rd] & 
                   *+[F]{ \begin{array}{c}
                    \ (\backslash /\proj)\backslash /
                   \end{array}} \ar[d] & *+[F]{ \begin{array}{c}
                    \backslash(/\proj\backslash)\ 
                   \end{array}} \ar[lld] \ar@{-->}[ld] & \\ 
                   & 
                   *+[F]{ \begin{array}{c}
                    /(\proj\backslash /)
                   \end{array}} \ar[rrd]  & *+[F]{ \begin{array}{c}
                     (/\proj\backslash)/\backslash
                   \end{array}} \ar@{-->}[rd]& & *+[F]{ \begin{array}{c}
                    /\backslash(/\proj\backslash) 
                   \end{array}} \ar@{-->}[ld]& 
                   *+[F]{ \begin{array}{c}
                    (\backslash /\proj)\backslash
                   \end{array}} \ar[lld] &  & \\
                   & & & *+[F]{ \begin{array}{c}
                    /\proj\backslash
                   \end{array}} & & & 
                   }
 \end{equation}
}

\normalsize

The dominations stated in the previous diagram are easily derived from the definitions of completely injective and completely projective hulls plus the results  already proved in Lemmas \ref{lema:identidad proj},  \ref{Lemma:desigualdad proy-iny}  and \ref{lemma:smallest y largest}.

The following two remarks look into representations of elements belonging to the dual of tensor products with particular tensor norms. They will be useful to show that some of the previous dominations can not be turned into equivalences.

\begin{remark}\label{rmk:weakly compact}
If $T\in\CB (E,F')$ belongs to
$ (E\otimes_{/\proj\backslash}F)'$ and $i_E:E\to \mathcal B(H_E)$ and $i_F:F\to \mathcal B(H_F)$ are complete isometries then there exists a completely bounded mapping $\widetilde{T}:\mathcal B(H_E)\to \mathcal B(H_F)'$ such that the following diagram commutes:

$$
\xymatrix{
E \ar[r]^T \ar[d]_{i_E} & F'\\
\mathcal B(H_E) \ar[r]_{\widetilde{T}} & \mathcal B(H_F)' \ar[u]_{i_F'}\\
}
$$

Appealing to \cite[Thm. 18.1]{Pisier-Grothendiecks-theorem} we know that $\widetilde{T}$ (and hence $T$) is weakly compact. 
\end{remark}

\begin{remark}\label{rmk:left injective}
If $F'$ is a completely  injective operator space then $E\otimes_{/\proj}F=E\otimes_{\proj}F$. 

Indeed, for $T\in\CB (E,F')$   we have that 
$T\in (E\otimes_{/\proj}F)'$ if for any complete isometry $i_E:E\to \mathcal B(H_E)$ there exists a completely bounded mapping $\widetilde{T}:\mathcal B(H_E)\to  F'$ such that the following diagram commutes:

$$
\xymatrix{
E \ar[r]^T \ar[d]_{i_E} & F'\\
\mathcal B(H_E) \ar[ru]_{\widetilde{T}} & \\
}
$$

The injectivity of $F'$ implies that this happens to any mapping $T\in\CB (E,F')$ (and that the norm is maintained).

Of course, an analogous result holds for the completely right-injective hull: if $E'$ is a completely injective operator space then $E\otimes_{\proj\backslash}F= E\otimes_{\proj}F$.
\end{remark}

\begin{example}
The following pairs of norms are not equivalent: $\proj$ and $\backslash /\proj$; $\proj$ and $\proj\backslash /$; $\proj\backslash$ and $(\backslash /\proj)\backslash$; $/\proj$ and $/(\proj\backslash /)$.

\rm We carry out the argument for the  first and third pairs of norms. The other two are obtained by transposing.

Let us consider the completely projective space $Z=\ell_1(\{S_1^n:\ n\in\N\})$ and recall that its dual $Z'=\ell_\infty(\{M_n:\ n\in\N\})$ is an injective space. Then we have the following identities:

$$
Z \otimes_{\proj}Z'= Z \otimes_{\proj\backslash}Z'\quad \textrm{and}\quad Z \otimes_{(\backslash /\proj)\backslash}Z'= Z \otimes_{\backslash /\proj}Z'= Z \otimes_{ /\proj}Z'=Z \otimes_{ /\proj\backslash} Z'.
$$

 Thus, the desired results are proved if we show that $Z \otimes_{ \proj}Z' \not= Z \otimes_{ /\proj\backslash} Z'$.

Let $\kappa_Z\in\CB (Z,Z'')$ be the canonical inclusion. Then, $\kappa_Z\in (Z \otimes_{ \proj}Z')'$. Now, Remark \ref{rmk:weakly compact} says that if $\kappa_Z\in (Z \otimes_{/ \proj \backslash}Z')'$ then $\kappa_Z$ would be weakly compact. Since this is not possible because $Z$ is not reflexive, we deduce that
$$
\kappa_Z\in (Z \otimes_{ \proj}Z')'\setminus (Z \otimes_{/ \proj \backslash}Z')'
$$
and hence $Z \otimes_{ \proj}Z' \not= Z \otimes_{ /\proj\backslash} Z'$.
\end{example}

\begin{example}
The following pairs of norms are not equivalent: $\backslash /\proj$ and $\backslash(/\proj\backslash)$; $\proj\backslash /$ and $(/\proj\backslash)/$.

\rm Let us consider again the completely projective space $Z=\ell_1(\{S_1^n:\ n\in\N\})$ and let $i_Z:Z\hookrightarrow \mathcal B(H)$ be a complete isometry, for a suitable Hilbert space $H$. Let us recall that $S_1(H)$ is an operator space pre-dual of $\mathcal B(H)$, and the latter is a completely injective space.
Now we have, by Remark \ref{rmk:left injective},
\begin{multline*}
Z \otimes_{\proj}S_1(H)= Z \otimes_{/\proj}S_1(H) = Z \otimes_{\backslash /\proj}S_1(H)\quad \textrm{and}\\ \quad  Z \otimes_{ /\proj\backslash}S_1(H)= Z \otimes_{\backslash (/\proj\backslash )}S_1(H).
\end{multline*}

It is clear that $i_Z\in (Z \otimes_{\proj}S_1(H))'$ but it cannot belong to $(Z \otimes_{ /\proj\backslash}S_1(H))'$ because it is not weakly compact (since it is an isometry from a non reflexive space).
\end{example}

\begin{example}\label{example: no equivalente proj}
The following pairs of norms are not equivalent:  $\proj$ and  $\backslash(/\proj\backslash)/$; $\backslash /\proj$ and $\backslash /(\proj\backslash /)$; $\proj\backslash /$ and $(\backslash
 /\proj)\backslash /$.

\rm We consider the Banach space $\ell_1(\mathbb Z)$ with the maximal operator space structure. This is a completely projective space (being an $\ell_1$-sum of the projective space $\C$) and thus its dual is completely injective \cite[Cor. 24.6]{Pisier-Operator-Space-Theory}. Then,

$$
\ell_1(\mathbb Z) \otimes_{\proj}\ell_1(\mathbb Z)= \ell_1(\mathbb Z) \otimes_{\proj\backslash}\ell_1(\mathbb Z)= \ell_1(\mathbb Z) \otimes_{\proj\backslash/}\ell_1(\mathbb Z).
$$
Also we have
$$
\ell_1(\mathbb Z) \otimes_{\backslash( /\proj\backslash)/}\ell_1(\mathbb Z)= \ell_1(\mathbb Z) \otimes_{/\proj\backslash}\ell_1(\mathbb Z).
$$

Now we can borrow an example from Banach space theory \cite[Ex. 1.1]{zalduendo2005extending}.
Let $T:\ell_1(\mathbb Z)\to \ell_\infty(\mathbb Z)$ be given by 

$$
T(x)=\left(\sum_{m\in\mathbb Z} \mathrm{sgn}(m) x_{m-n}\right)_{n\in\N}.
$$

Then $T\in (\ell_1(\mathbb Z) \otimes_{\proj}\ell_1(\mathbb Z))'$ (because  in $\ell_1(\mathbb Z)$ with the maximal operator space structure the $\proj$ tensor product coincides with the Banach $\pi$ tensor product).

Once again we have that $T\not\in (\ell_1(\mathbb Z) \otimes_{/\proj\backslash}\ell_1(\mathbb Z))'$ because $T$ is not weakly compact \cite[page 50]{zalduendo2005extending}.

\end{example}

\subsection{The $\min$ family} 

Surprisingly the $\min$ family is not the \textsl{mirror reflection} of the $\proj$ family as it happens in the Banach space setting.
The reason, of course, comes from Lemmas \ref{lema:identidad proj} and \ref{lema:no identidad min}.

Due to this fact there are probably more members in the $\min$ family than in the $\proj$ family. Hence, in order to have a manageable set, we chose to present a diagram of the $\min$ family just \textsl{up to three procedures}. 

The list is the following:

\begin{enumerate}
    \item[(0)] $\min$.
    \item  $\backslash\min\ \bullet \ \min/$.
    \item $\backslash\min/\ \bullet \ /\backslash \min\ \bullet \ \min /\backslash\ \bullet\ (\backslash\min)\backslash \ \bullet\  /( \min /)$.
    \item $/(\backslash\min/)\ \bullet\ (\backslash\min/)\backslash \bullet \ (/ \backslash\min)/\ \bullet \ \backslash(\min /\backslash ) \ \bullet\ \backslash/(\min/)\ \bullet\ (\backslash\min)\backslash / \bullet \ (/ \backslash\min)\backslash\ \bullet \ /(\min /\backslash )$.
    \end{enumerate}

The norms (derived from $\min$ after at most three procedures) that do not appear in the previous list are identical to one of those that do appear in the list.  Indeed, from Lemma \ref{lema:identidad proj} and the definitions of completely injective and completely projective hulls we have:

\begin{enumerate}

    \item[(3)] $( /(\min /))/=  \min /\ \bullet \ \backslash((\backslash\min)\backslash)=\backslash
    \min$.
    \item[(3)] $  ( /(\min /))\backslash=/( \min /\backslash )\ \bullet \ /((\backslash\min)\backslash) =(/\backslash  \min )\backslash$.
    \end{enumerate}

Thus, the members of the $\min$ family \textsl{up to three procedures} are 16. We do not know whether they are all non-equivalent. As in the diagram for the $\proj$ family, each   arrow $\alpha\to\beta$   means that $\beta\le\alpha$ and doubled arrows are the only ones where we can prove that the dominations are strict.

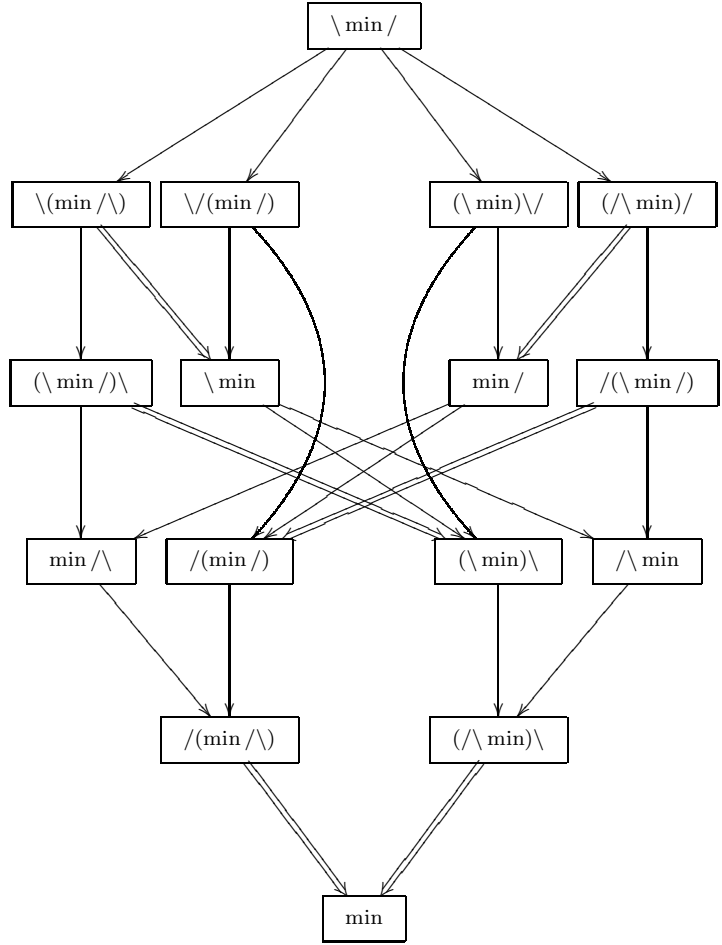
\tiny{

\begin{equation}\label{normasnaturalesmin}
 \xymatrixrowsep{0.7in}
\xymatrixcolsep{0.05in}
 \xymatrix{ & & & *+[F]{ \begin{array}{c}
                    \backslash \min /
                   \end{array}} \ar[dl] \ar[lld] \ar[rrd] \ar[rd] & & & 
                   \\
                   &  *+[F]{ \begin{array}{c}
                    \backslash (\min / \backslash) 
                   \end{array}} \ar[d] \ar@{=>}[rd] & *+[F]{ \begin{array}{c}
                    \backslash /  (\min / )
                   \end{array}} \ar[d] \ar@/^3pc/[dd] &   & *+[F]{ \begin{array}{c}
                      (\backslash \min  ) \backslash /
                   \end{array}} \ar[d] \ar@/_3pc/[dd]  &  *+[F]{ \begin{array}{c}
                      (/ \backslash   \min  ) /
                   \end{array}} \ar@{=>}[ld] \ar[d]\\
                    &  *+[F]{ \begin{array}{c}
                    (\backslash \min / ) \backslash
                   \end{array}} \ar[d] \ar@{=>}[rrrd] & *+[F]{ \begin{array}{c}
                     \backslash  \min 
                   \end{array}} \ar[rrd] \ar[rrrd] &   & *+[F]{ \begin{array}{c}
                      \min  /
                   \end{array}} \ar[llld] \ar[lld] &  *+[F]{ \begin{array}{c}
                     / ( \backslash   \min  / ) 
                   \end{array}} \ar[d] \ar@{=>}[llld]\\ &  *+[F]{ \begin{array}{c}
                    \min / \backslash 
                   \end{array}} \ar[rd] & *+[F]{ \begin{array}{c}
                     /  (\min / )
                   \end{array}} \ar[d] &   & *+[F]{ \begin{array}{c}
                      (\backslash \min  ) \backslash 
                   \end{array}} \ar[d] &  *+[F]{ \begin{array}{c}
                      / \backslash   \min 
                   \end{array}} \ar[ld]\\
                    &  &   *+[F]{ \begin{array}{c}
                     /  (\min / \backslash )
                   \end{array}} \ar@{=>}[rd]  &   & *+[F]{ \begin{array}{c}
                       (/ \backslash \min  ) \backslash
                   \end{array}} \ar@{=>}[ld]&  &\\
                    & & & *+[F]{ \begin{array}{c}
                    \min
                   \end{array}} & & & 
                   }
 \end{equation}
}
\normalsize

\begin{example}
The following pairs of norms are not equivalent: $\backslash\min$ and $\backslash (\min/\backslash)$; $(\backslash\min)\backslash$ and $(\backslash \min/)\backslash$; $\min/$ and $(/\backslash\min)/$; $/(\min /)$ and $/(\backslash\min /)$.

\rm As usual, we just argue about the first two non-equivalences. The other two follow by transposition. By Lemma \ref{lemma:Exactly_integral} we know that if  $\|id:S_1^n\otimes_{\min} M_n\to S_1^n\otimes_{\min/} M_n\|\le C$   then  $S_1^n$ is $C$-exact. Since there is no $C$ such that  $S_1^n$ is $C$-exact for every $n$ we derive that   $S_1^n\otimes_{\min} M_n$ and $S_1^n\otimes_{\min/} M_n$ are not uniformly completely isomorphic. Since $S_1^n$ is completely projective and $M_n$ is completely injective we have
$S_1^n\otimes_{\min} M_n=S_1^n\otimes_{\backslash\min} M_n=S_1^n\otimes_{(\backslash\min)\backslash} M_n$ and $S_1^n\otimes_{\min/} M_n= S_1^n\otimes_{\backslash (\min/\backslash)} M_n= S_1^n\otimes_{(\backslash \min/)\backslash} M_n$
from where the result follows.
\end{example}

\subsection{Two-sided natural o.s. tensor norms} \label{twosided}

In this section we show that those natural o.s. tensor norms that come from applying to $\min$ or $\proj$ two-sided hulls operations (injective $/\cdot\backslash$, or projective $\backslash \cdot /$ hulls)
are exactly six.
This is somewhat of a surprise: as stated in \cite{blecher1991tensor}, ``there is no reason to suppose that this [process] will not give a large number of inequivalent norms'' (due to the lack of a direct analog of Grothendieck's inequality in the o.s. setting).
Now, Lemma \ref{lemma two-sided hulls} $(a)$ shows that there are only 6 possible members of the two-sided family and  the dominations between them are easily seen:
$$  \min  \leq \  /(\backslash \min /) \backslash\leq \backslash \min / \leq /\proj\backslash \leq \backslash(/\proj\backslash)/ \leq \proj 
$$

Also, in Remark \ref{Remark:Grothendieck} we observe that $\backslash \min /$ and $ /\proj\backslash$ are not equivalent. From Example \ref{example: no equivalente proj} (see also diagram \eqref{normasnaturales}) we know that $\backslash(/\proj\backslash)/$ and $\proj$ are not equivalent. We prove in Lemma \ref{lema:no identidad min} that $\min$ and $(/\backslash\min)\backslash$ are not equivalent. Since $(/\backslash\min)\backslash = /(\backslash\min)\backslash\le /(\backslash\min/)\backslash $ we obtain that $\min$ and $/(\backslash\min/)\backslash$ could not be equivalent either.

To conclude that there are six two-sided natural norms all we need to show is the non-equivalence of the following two pairs of norms: $/(\backslash \min /)\backslash $  and $ \backslash \min /$; $\backslash(/\proj\backslash)/ $ and $ /\proj\backslash$. In the Banach space setting this kind of statement is clear since a projective norm could not be equivalent to an injective one. In the operator space framework this argument is not available, so we need a different proof.

\begin{lemma}\label{Lema:equivalencia nivel Banach}
    The norms $ /\proj\backslash$ and  $ \backslash \min /$ are equivalent at the Banach space level to $h\cap h^t$ and $h+h^t$, respectively.
\end{lemma}
 \begin{proof}
   In \cite{dimant2015bilinear} it is showed, as a consequence of results from \cite{pisier2002grothendieck} and \cite{haagerup2008effros}, that for any operator spaces $E$ and $F$ the spaces $E \otimes_{/\proj\backslash} F$ and $E \otimes_{h\cap h^t} F$ are isomorphic.
   By chasing down the constants in the proof one can see that they are universal, so we conclude that $/\proj\backslash$ is equivalent to $h\cap h^t$ at the Banach space level. Also, \cite[Cor. 0.8 (iii)]{pisier2002grothendieck} says that there is an isomorphism between $A'\otimes_{\min} B'$ and $A'\otimes_{h+h^t} B'$ for any $C^*$-algebras $A$ and $B$, with constants independent of the spaces involved. Now, if $E$ and $F$ are operator spaces, we have that the completely projective spaces $Z_E$ and $Z_F$ are duals of $C^*$-algebras (recall that they are defined as $\ell_1$-sums of $S_1^n$ spaces, so they are the duals of the corresponding $c_0$-sums of $M_n$ spaces \cite[Sec. 2.6]{Pisier-Operator-Space-Theory}). Therefore, we have that $E\otimes_{\backslash\min/}F$ and $E\otimes_{h+h^t}F$
are isomorphic with constants independent of $E$ and $F$. This means that the o.s. tensor norms $\backslash\min/$ and $h+h^t$ are equivalent at the Banach space level.
 \end{proof}

\begin{remark}\label{rmk:equivalencia Banach}
A close look at the definition of dual tensor norms reveals that if two o.s. tensor norms are equivalent at the Banach space level, then so are their dual o.s. tensor norms (because the norm on the dual of an operator space depends only on the norm of the operator space, and not on the norms of the higher matricial levels).
Therefore, recalling that $(h\cap h^t)'=h+h^t$ (see Proposition \ref{prop-duality-symmetrized-Haagerup}), as a consequence of the previous lemma we obtain that $(\backslash\min /)'$ is equivalent, at the Banach space level, to $/\proj \backslash$.
\end{remark}

\begin{lemma}\label{Lema:no-equivalencia-two-sided}
    The norms $/(\backslash \min /)\backslash $  and $ \backslash \min /$ are not equivalent. The norms $\backslash(/\proj\backslash)/ $ and $ /\proj\backslash$ are not equivalent.
\end{lemma}

\begin{proof}

In \cite[Remark 13]{Oikhberg-Pisier} it is proved that $h+h^t$ is not completely injective. In fact, their proof shows more: there is a sequence of pairs of complete isometries $T_n : E_n \to \widetilde{E}_n$ and $R_n : F_n \to \widetilde{F}_n$ so that the inverses of 
$$
T_n \otimes R_n : E_n \otimes_{h+h^t} F_n \to (T_n \otimes R_n)(E_n \otimes F_n) \subset \widetilde{E}_n \otimes_{h+h^t} \widetilde{F}_n
$$
have norms tending to infinity, so
$h+h^t$ cannot be equivalent at the Banach space level to a completely injective o.s. tensor norm. As a consequence of Lemma \ref{Lema:equivalencia nivel Banach}, neither can $\backslash\min/$, and therefore 
$\backslash\min/$ and $/(\backslash \min /)\backslash $ cannot be equivalent.

 Also, we know from Corollary \ref{cor: proy implica dual iny} that the dual of a completely projective tensor norm is completely injective. Thus, if $h\cap h^t$ were equivalent at the Banach space level to a completely projective tensor norm, from the previous remark we would obtain that $(h\cap h^t)'=h+h^t$ would be equivalent at the Banach space level to a completely injective tensor norm. From the previous paragraph we know that this is false, so therefore $/\proj\backslash$ cannot be equivalent at the Banach space level to a completely projective tensor norm. This, finally, means that $/\proj\backslash$ and $\backslash(/\proj\backslash)/$ are not equivalent.

\end{proof}

Hence, we can represent the two-sided family in the following diagram:

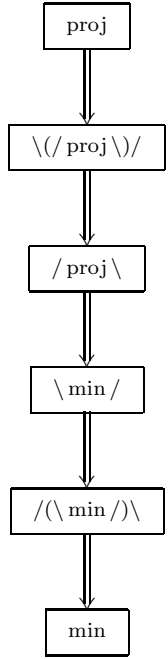
\tiny{

\begin{equation}\label{normasnaturalestwosided}
 \xymatrixrowsep{0.4in}
\xymatrixcolsep{0.05in}
 \xymatrix{*+[F]{ \begin{array}{c}
                    \proj
                   \end{array}} \ar@{=>}[d] 
                   \\
                   *+[F]{ \begin{array}{c}
                  \backslash(/\proj\backslash)/
                   \end{array}}
                   \ar@{=>}[d] 
                   \\
                   *+[F]{ \begin{array}{c}
                  /\proj\backslash
                   \end{array}}
                   \ar@{=>}[d] 
                   \\
                   *+[F]{ \begin{array}{c}
                  \backslash \min /
                   \end{array}}
                   \ar@{=>}[d] 
                   \\
                   *+[F]{ \begin{array}{c}
                  /(\backslash \min /) \backslash
                   \end{array}}
                   \ar@{=>}[d] 
                   \\
                   *+[F]{ \begin{array}{c}
                  \min
                   \end{array}}
                   }
 \end{equation}
}

\normalsize

This is different from the Banach space case: those natural norms that come from applying to $\varepsilon$ or $\pi$ two-sided hulls operations 
are exactly four. This is how the two-sided natural norms in the classical setting are arranged:

\tiny{

\begin{equation}\label{normasnaturalesclassic}
 \xymatrixrowsep{0.4in}
\xymatrixcolsep{0.05in}
 \xymatrix{ & & & *+[F]{ \begin{array}{c}
                    \pi
                   \end{array}} \ar@{=>}[ld] \ar@{=>}[rd] & & & 
                   \\
                   & & *+[F]{ \begin{array}{c}
                     /\pi \backslash
                   \end{array}}  \ar@{=>}[rd] &   & *+[F]{ \begin{array}{c}
                    \backslash \epsilon / 
                   \end{array}} \ar@{=>}[ld] & & \\
                     & & & *+[F]{ \begin{array}{c}
                    \epsilon
                   \end{array}}  & & & 
                    }
 \end{equation}
}

\normalsize

In both diagrams \eqref{normasnaturalestwosided} and \eqref{normasnaturalesclassic} each   arrow $\alpha\Rightarrow \beta$   means that $\beta\le\alpha$ and that these norms are not equivalent. The norms $ /\pi \backslash$ and $ \backslash \epsilon /$ are not equivalent and there is no domination between them.

Obviously this difference between these frameworks is  expected: in the context of operator spaces we do not have a Grothendieck-type inequality relating the norms   as in the classical Banach context (e.g., $ /\pi \backslash$ and $ / (\backslash \epsilon /) \backslash$ are equivalent and also their duals  $  \backslash \epsilon / $ and $\backslash (/\pi \backslash )/$).

\begin{remark}\label{rmk:symmetric natural tensor norms}
It would be tempting to say that the diagram \eqref{normasnaturalestwosided} shows the complete picture of all \textbf{symmetric} natural o.s. tensor norms. While it is clear that we can only get symmetric o.s. tensor norms by starting with $\min$ and $\proj$ and applying two-sided hulls finitely many times, we do not know whether these are \textbf{all} the symmetric natural o.s. tensor norms. For example recall that  $(/\proj)/= \backslash(\proj\backslash) =\proj$, so taking hulls in a non-symmetric way can still yield a symmetric o.s. tensor norm.
\end{remark}

\section{Conclusions and some open questions}

The main goal of this work is to start a program 
 for the  theory of tensor products and tensor norms for the category of operator spaces. In particular, we  focus on the interplay of these theories and the theory of mapping ideals. 
Of course, many definitions and results in this framework are natural since they have a corresponding one in the context of Banach spaces. But many of them are not!   
As noticed many times, the theory is not just a straightforward translation of what is known in the classical setting and new and challenging questions naturally arise.  
To start with a typical difference, we highlight the role of local reflexivity, which seems to be crucial in many places as is well-known to experts. However, sometimes 
this hypothesis can be avoided: for example, when dealing with $\lambda$-o.s. tensor norms.
An unexpected issue in this category appears when relating the left accessibility  of mapping ideals and associated tensor norms. The fact that their relationship is weaker than the one for its right counterpart is certainly puzzling. 

The most surprising fact we realized during this work is that the $\min$ family can not be obtained by dualizing the $\proj$ family.  This, of course, is very counter intuitive and exhibits, once again, the differences between the classical and the non-commutative theories. All of this is related, in some sense, to the fact that accessibility plays an important role in the duality between injective and projective hulls of a given o.s. tensor norm. This hypothesis does not appear in the theory of Banach spaces.
 Moreover, these families are completely separate: one family dominates the other. This is due to the fact that the Haagerup o.s. tensor norm is simultaneously completely injective and projective (an impossible property in the classical realm).

Summarizing, we have begun  looking with operator space eyes at some relevant properties and results from the classical theory, and found that to fully develop the theory it is important a different and original perspective, involving novel insights, ideas, techniques  or hypotheses. Clearly, there are plenty of unexplored paths for future research. This work aims to contribute to a program that we believe is much vaster.
To conclude we want to explicitly recall here some unsolved issues that appeared through this monograph. 

 \textbf{Problem 1:} All along the monograph there are results where some hypotheses about local reflexivity are not needed in the case of $\lambda$-o.s. tensor norms. It would be welcome to have more examples of $\lambda$-o.s. tensor norms, in addition to the typical ones ($\proj$, $h$ and $h^t$) and the $\odot^\theta$ constructed in Section \ref{Subsection: Examples}.

\textbf{Problem 2:} We introduce in Section \ref{dual section} the notions of locally right-accessible tensor norm and locally left-accessible tensor norm, but we do not have any example that justifies that these concepts are weaker than their non-local siblings. 

\textbf{Problem 3:} We consider in Section \ref{complete bap} a weak version of the complete bounded approximation property, namely the W*CBAP. We do not know any example of a space without CBAP satisfying W*CBAP. In view of Proposition  \ref{prop W*CBAP and locally reflexive} such a space can not be locally reflexive.

\textbf{Problem 4:} The theory of $\alpha$-W*CBAP suggested in Remark \ref{alpha CBAP} seems to be an interesting topic for future development.

\textbf{Problem 5:} In Section \ref{mapping accessibility sec} it would be interesting to have an example of a locally right-accessible mapping ideal which is not right-accessible. Recall that we show in Example \ref{example-accessible-ideals} that $\mathcal I$ is an example for the left version of this question. 

\textbf{Problem 6:} In Proposition \ref{prop-equivalences-accessibility} we see that there is an equivalence between right-accessibility of a mapping ideal and of a tensor norm associated to it. For left-accessibility the result is much weaker and just in one direction. It would be important to know if a statement for the other direction is valid: if the mapping ideal is left-accessible, is it true that the associated tensor norm is (locally) left-accessible?

\textbf{Problem 7:} With respect to natural tensor norms we have barely begun a classification of them and there is a hard road ahead. It is left to be solved whether most of the dominations in diagrams \ref{normasnaturales} and \ref{normasnaturalesmin}  are strict or not. 

\textbf{Problem 8:} A solution to Open question \ref{rmk:open question} will shed some light on the structure of the $\proj$ family after more than 4 procedures. 

\textbf{Problem 9:} Our knowledge of the $\min$ family is so scarce that almost the whole picture remains to be done.

\textbf{Problem 10:} The fact that the $\proj$ and $\min$ families are not a reflected copy of one another, along with the necessity of extra conditions to ensure that the dual of the projective hull of a norm is the injective hull of the dual norm, have as a consequence that we do not know which are the duals for most of the natural norms. The main inquiry of this group is whether the following identity holds: 
 $(\backslash\min /)'=/\proj \backslash$ (see the comment after Example \ref{example:dual}). Recall, also, that from Remark \ref{rmk:equivalencia Banach}, $(\backslash\min /)'$ and $/\proj \backslash$ are equivalent, at the Banach space level. So, it is expected that they are equivalent, or even equal, as o.s. tensor norms.
 
 \textbf{Problem 11:} Since the $\min$ family is certainly larger than the $\proj$ family, there should exist norms in the $\min$ family whose duals do not belong to the $\proj$ family. It would be interesting to identify these norms.

 \textbf{Problem 12:} As stated in Remark \ref{rmk:symmetric natural tensor norms}, it would be nice to know if all the symmetric natural  tensor norms are the two-sided natural tensor norms (i.e.  those that appear in diagram \eqref{normasnaturalestwosided}).

\newcommand{\etalchar}[1]{$^{#1}$}
\providecommand{\bysame}{\leavevmode\hbox to3em{\hrulefill}\thinspace}
\providecommand{\MR}{\relax\ifhmode\unskip\space\fi MR }
\providecommand{\MRhref}[2]{%
  \href{http://www.ams.org/mathscinet-getitem?mr=#1}{#2}
}
\providecommand{\href}[2]{#2}


\begin{thebibliography}{CDDG19}

\bibitem[AK21]{Janson-Kumar-spectra}
Janson Antony and Ajay Kumar, \emph{Spectra of elements in operator space
  tensor products of {$\rm C^*$}-algebras}, Positivity \textbf{25} (2021),
  no.~5, 1973--1987. \MR{4338555}

\bibitem[AKL19]{Janson-Kumar-Luthra}
Janson Antony, Ajay Kumar, and Preeti Luthra, \emph{Operator space tensor
  products and inductive limits}, J. Math. Anal. Appl. \textbf{470} (2019),
  no.~1, 235--250. \MR{3865134}

\bibitem[Ble92a]{Blecher-standard}
David~P. Blecher, \emph{The standard dual of an operator space}, Pacific J.
  Math. \textbf{153} (1992), no.~1, 15--30. \MR{1145913}

\bibitem[Ble92b]{blecher1991tensor}
\bysame, \emph{Tensor products of operator spaces. {II}}, Canad. J. Math.
  \textbf{44} (1992), no.~1, 75--90. \MR{1152667}

\bibitem[BLM04]{Blecher-LeMerdy-book}
David~P. Blecher and Christian Le~Merdy, \emph{Operator algebras and their
  modules---an operator space approach}, London Mathematical Society
  Monographs. New Series, vol.~30, The Clarendon Press Oxford University Press,
  Oxford, 2004, Oxford Science Publications. \MR{2111973 (2006a:46070)}

\bibitem[BP91]{Blecher-Paulsen-Tensors}
David~P. Blecher and Vern~I. Paulsen, \emph{Tensor products of operator
  spaces}, J. Funct. Anal. \textbf{99} (1991), no.~2, 262--292. \MR{1121615
  (93d:46095)}

\bibitem[CD12]{CD-completely-q-p-mixing}
Javier~Alejandro Ch\'{a}vez-Dom\'{\i}nguez, \emph{Completely {$(q,p)$}-mixing
  maps}, Illinois J. Math. \textbf{56} (2012), no.~4, 1169--1183. \MR{3231478}

\bibitem[CD16]{CD-Chevet-Saphar-OS}
Javier~Alejandro Ch{\'a}vez-Dom{\'{\i}}nguez, \emph{The {C}hevet-{S}aphar
  tensor norms for operator spaces}, Houston J. Math. \textbf{42} (2016),
  no.~2, 577--596. \MR{3529971}

\bibitem[CDDG19]{chavez2019operator}
Javier~Alejandro Ch\'{a}vez-Dom\'{\i}nguez, Ver\'{o}nica Dimant, and Daniel
  Galicer, \emph{Operator {$p$}-compact mappings}, J. Funct. Anal. \textbf{277}
  (2019), no.~8, 2865--2891. \MR{3990738}

\bibitem[CG12]{carando2012natural}
Daniel Carando and Daniel Galicer, \emph{Natural symmetric tensor norms}, J.
  Math. Anal. Appl. \textbf{387} (2012), no.~2, 568--581. \MR{2853127}

\bibitem[DF93]{Defant-Floret}
Andreas Defant and Klaus Floret, \emph{Tensor norms and operator ideals},
  North-Holland Mathematics Studies, vol. 176, North-Holland Publishing Co.,
  Amsterdam, 1993. \MR{MR1209438 (94e:46130)}

\bibitem[DFS08]{diestel2008metric}
Joe Diestel, Jan~H. Fourie, and Johan Swart, \emph{The metric theory of tensor
  products}, American Mathematical Society, Providence, RI, 2008,
  Grothendieck's r\'{e}sum\'{e} revisited. \MR{2428264}

\bibitem[DFU15a]{dimant2015biduals}
Ver\'{o}nica Dimant and Maite Fern\'{a}ndez-Unzueta, \emph{Biduals of tensor
  products in operator spaces}, Studia Math. \textbf{230} (2015), no.~2,
  167--187. \MR{3476485}

\bibitem[DFU15b]{dimant2015bilinear}
\bysame, \emph{Bilinear ideals in operator spaces}, J. Math. Anal. Appl.
  \textbf{429} (2015), no.~1, 57--80. \MR{3339064}

\bibitem[DPnS10]{delgado2010operators}
Juan~M. Delgado, C\'andido Pi\~{n}eiro, and Enrique Serrano, \emph{Operators
  whose adjoints are quasi {$p$}-nuclear}, Studia Math. \textbf{197} (2010),
  no.~3, 291--304. \MR{2607494}

\bibitem[DW14]{Defant-Wiesner}
Andreas Defant and Dirk Wiesner, \emph{Polynomials in operator space theory},
  J. Funct. Anal. \textbf{266} (2014), no.~9, 5493--5525. \MR{3182950}

\bibitem[EJR00]{Effros-Junge-Ruan}
Edward~G. Effros, Marius Junge, and Zhong-Jin Ruan, \emph{Integral mappings and
  the principle of local reflexivity for noncommutative {$L^1$}-spaces}, Ann.
  of Math. (2) \textbf{151} (2000), no.~1, 59--92. \MR{1745018 (2000m:46120)}

\bibitem[EK87]{Effros-Kishimoto}
Edward~G. Effros and Akitaka Kishimoto, \emph{Module maps and
  {H}ochschild-{J}ohnson cohomology}, Indiana Univ. Math. J. \textbf{36}
  (1987), no.~2, 257--276. \MR{891774}

\bibitem[ER90]{effros1990approximation}
Edward~G. Effros and Zhong-Jin Ruan, \emph{On approximation properties for
  operator spaces}, Internat. J. Math. \textbf{1} (1990), no.~2, 163--187.
  \MR{1060634}

\bibitem[ER94]{Effros-Ruan-Grothendieck-Pietsch}
\bysame, \emph{The {G}rothendieck-{P}ietsch and {D}voretzky-{R}ogers theorems
  for operator spaces}, J. Funct. Anal. \textbf{122} (1994), no.~2, 428--450.
  \MR{1276165 (96c:46074b)}

\bibitem[ER00]{Effros-Ruan-book}
\bysame, \emph{Operator spaces}, London Mathematical Society Monographs. New
  Series, vol.~23, The Clarendon Press Oxford University Press, New York, 2000.
  \MR{1793753 (2002a:46082)}

\bibitem[ER03]{Effros-Ruan-Hopf}
\bysame, \emph{Operator space tensor products and {H}opf convolution algebras},
  J. Operator Theory \textbf{50} (2003), no.~1, 131--156. \MR{2015023}

\bibitem[Gro53]{grothendieck1956resume}
Alexander Grothendieck, \emph{R\'{e}sum\'{e} de la th\'{e}orie m\'{e}trique des
  produits tensoriels topologiques}, Bol. Soc. Mat. S\~{a}o Paulo \textbf{8}
  (1953), 1--79. \MR{94682}

\bibitem[HM08]{haagerup2008effros}
Uffe Haagerup and Magdalena Musat, \emph{The {E}ffros-{R}uan conjecture for
  bilinear forms on {$C^*$}-algebras}, Invent. Math. \textbf{174} (2008),
  no.~1, 139--163. \MR{2430979}

\bibitem[JNRX04]{Junge-Nielsen-Ruan-Xu}
Marius Junge, Niels~J. Nielsen, Zhong-Jin Ruan, and Quanhua Xu,
  \emph{{$\mathcal C\mathcal O\mathcal L_p$} spaces---the local structure of
  non-commutative {$L_p$} spaces}, Adv. Math. \textbf{187} (2004), no.~2,
  257--319. \MR{2078339 (2005g:46120)}

\bibitem[JP10]{Junge-Parcet-Maurey-factorization}
Marius Junge and Javier Parcet, \emph{Maurey's factorization theory for
  operator spaces}, Math. Ann. \textbf{347} (2010), no.~2, 299--338.
  \MR{2606939}

\bibitem[Jun96]{Junge-Habilitationschrift}
Marius Junge, \emph{Factorization theory for spaces of operators}, Habilitation
  Thesis. Kiel, 1996.

\bibitem[Kir95]{kirchberg1995exact}
Eberhard Kirchberg, \emph{Exact {${\rm C}^*$}-algebras, tensor products, and
  the classification of purely infinite algebras}, Proceedings of the
  {I}nternational {C}ongress of {M}athematicians, {V}ol. 1, 2 ({Z}\"{u}rich,
  1994), Birkh\"{a}user, Basel, 1995, pp.~943--954. \MR{1403994}

\bibitem[LP68]{lindenstrauss1968absolutely}
Joram Lindenstrauss and Aleksander Pe{\l}czy\'nski, \emph{Absolutely summing
  operators in {$L_{p}$}-spaces and their applications}, Studia Math.
  \textbf{29} (1968), 275--326. \MR{231188}

\bibitem[OP99]{Oikhberg-Pisier}
Timur Oikhberg and Gilles Pisier, \emph{The ``maximal'' tensor product of
  operator spaces}, Proc. Edinburgh Math. Soc. (2) \textbf{42} (1999), no.~2,
  267--284. \MR{1697398}

\bibitem[Pie67]{Pietsch-67}
Albrecht Pietsch, \emph{Absolut {$p$}-summierende {A}bbildungen in normierten
  {R}\"aumen}, Studia Math. \textbf{28} (1966/1967), 333--353. \MR{MR0216328
  (35 \#7162)}

\bibitem[Pie78]{pietsch1978operator}
\bysame, \emph{Operator ideals}, Mathematische Monographien [Mathematical
  Monographs], vol.~16, VEB Deutscher Verlag der Wissenschaften, Berlin, 1978.
  \MR{519680}

\bibitem[Pie80]{Pietsch-Operator-Ideals}
\bysame, \emph{Operator ideals}, North-Holland Mathematical Library, vol.~20,
  North-Holland Publishing Co., Amsterdam, 1980, Translated from German by the
  author. \MR{MR582655 (81j:47001)}

\bibitem[Pis83]{pisier1983counterexamples}
Gilles Pisier, \emph{Counterexamples to a conjecture of {G}rothendieck}, Acta
  Math. \textbf{151} (1983), no.~3-4, 181--208. \MR{723009}

\bibitem[Pis86]{pisier1986factorization}
\bysame, \emph{Factorization of linear operators and geometry of {B}anach
  spaces}, CBMS Regional Conference Series in Mathematics, vol.~60, Published
  for the Conference Board of the Mathematical Sciences, Washington, DC; by the
  American Mathematical Society, Providence, RI, 1986. \MR{829919}

\bibitem[Pis95]{pisier1995exact}
\bysame, \emph{Exact operator spaces}, Ast\'{e}risque (1995), no.~232,
  159--186, Recent advances in operator algebras (Orl\'{e}ans, 1992).
  \MR{1372532}

\bibitem[Pis98]{Pisier-Asterisque-98}
\bysame, \emph{Non-commutative vector valued {$L_p$}-spaces and completely
  {$p$}-summing maps}, Ast\'erisque (1998), no.~247, vi+131. \MR{1648908
  (2000a:46108)}

\bibitem[Pis03]{Pisier-Operator-Space-Theory}
\bysame, \emph{Introduction to operator space theory}, London Mathematical
  Society Lecture Note Series, vol. 294, Cambridge University Press, Cambridge,
  2003. \MR{2006539 (2004k:46097)}

\bibitem[Pis12]{Pisier-Grothendiecks-theorem}
\bysame, \emph{Grothendieck's theorem, past and present}, Bull. Amer. Math.
  Soc. (N.S.) \textbf{49} (2012), no.~2, 237--323. \MR{2888168}

\bibitem[Pis20]{pisier2020tensor}
\bysame, \emph{Tensor products of {$C^*$}-algebras and operator spaces---the
  {C}onnes-{K}irchberg problem}, London Mathematical Society Student Texts,
  vol.~96, Cambridge University Press, Cambridge, 2020. \MR{4283471}

\bibitem[PP69]{Persson-Pietsch}
Arne Persson and Albrecht Pietsch, \emph{{$p$}-nukleare und {$p$}-integrale
  {A}bbildungen in {B}anachr\"aumen}, Studia Math. \textbf{33} (1969), 19--62.
  \MR{0243323 (39 \#4645)}

\bibitem[PS02]{pisier2002grothendieck}
Gilles Pisier and Dimitri Shlyakhtenko, \emph{Grothendieck's theorem for
  operator spaces}, Invent. Math. \textbf{150} (2002), no.~1, 185--217.
  \MR{1930886}

\bibitem[Re{\u{\i}}82]{Reinov-APp}
Oleg~I. Re{\u{\i}}nov, \emph{Approximation properties of order {$p$} and the
  existence of non-{$p$}-nuclear operators with {$p$}-nuclear second adjoints},
  Math. Nachr. \textbf{109} (1982), 125--134. \MR{705902}

\bibitem[Rya02]{Ryan}
Raymond~A. Ryan, \emph{Introduction to tensor products of {B}anach spaces},
  Springer Monographs in Mathematics, Springer-Verlag London Ltd., London,
  2002. \MR{MR1888309 (2003f:46030)}

\bibitem[WBF{\etalchar{+}}]{Wittstock}
Gerd Wittstock, Benedikt Betz, Hans-J\"org Fischer, Anselm Lambert, Kim Louis,
  Matthias Neufang, and Ina Zimmermann, \emph{What are operator spaces?},
  \url{https://www.math.uni-sb.de/ag/wittstock/OperatorSpace.pdf}, Accessed:
  2019-10-19.

\bibitem[Wie09]{Wiesner}
Dirk Wiesner, \emph{Polynomials in operator space theory}, Ph.D. thesis, Carl
  von {O}ssietzky {U}niversit\"at {O}ldenburg, 2009.

\bibitem[Yew08]{Yew-08}
Khye~Loong Yew, \emph{Completely {$p$}-summing maps on the operator {H}ilbert
  space {$OH$}}, J. Funct. Anal. \textbf{255} (2008), no.~6, 1362--1402.
  \MR{2565712 (2010k:46026)}

\bibitem[Zal05]{zalduendo2005extending}
Ignacio Zalduendo, \emph{Extending polynomials on {B}anach spaces---a survey},
  Rev. Un. Mat. Argentina \textbf{46} (2005), no.~2, 45--72 (2006).
  \MR{2281672}

\end{thebibliography}
\end{document}